\documentclass[oneside]{book}
\usepackage{amsthm}
\usepackage{amssymb}
\usepackage{MnSymbol}

\usepackage{url}
\usepackage{hyperref}

\usepackage[all,cmtip]{xy}
\usepackage{xspace}
\usepackage{multirow}
\usepackage{enumitem}

\usepackage{graphicx}
\usepackage{epstopdf}
\usepackage{tikz} 
\usepackage{pgf}
\usetikzlibrary{arrows,decorations.pathreplacing}

\usepackage[margin=1.75in]{geometry}
\usepackage{appendix}

\newcommand{\concat}[0]{\textrm{\^{}}}
\newcommand{\cea}[0]{CEA\xspace}
\newcommand{\infi}[0]{\mathtt{in}}
\newcommand{\comp}[0]{\mathtt{c}}
\newcommand{\omegaext}[0]{\omega^+}
\newcommand{\POne}[0]{\textrm{I}\xspace}
\newcommand{\PTwo}[0]{\textrm{II}\xspace}
\newcommand{\loc}[0]{L}
\newcommand{\game}[1]{\mc{G}_{#1}}
\newcommand{\basicgame}[1]{\mc{G}^b_{#1}}
\newcommand{\CLoc}[0]{\textbf{(CLoc)}\xspace}
\newcommand{\COneStar}[0]{\textbf{(C1\textsuperscript{*})}\xspace}
\newcommand{\CTwoStar}[0]{\textbf{(C2\textsuperscript{*})}\xspace}
\newcommand{\CThreeStar}[0]{\textbf{(C3\textsuperscript{*})}\xspace}
\newcommand{\CFourStar}[0]{\textbf{(C4\textsuperscript{*})}\xspace}
\newcommand{\COne}[0]{\textbf{(C1)}\xspace}
\newcommand{\CTwo}[0]{\textbf{(C2)}\xspace}
\newcommand{\CThree}[0]{\textbf{(C3)}\xspace}
\newcommand{\CFour}[0]{\textbf{(C4)}\xspace}
\newcommand{\At}[0]{\textbf{(At)}\xspace}
\newcommand{\Ext}[0]{\textbf{(Ext)}\xspace}
\newcommand{\Surj}[0]{\textbf{(Surj)}\xspace}
\newcommand{\WS}[0]{\textbf{(WS)}\xspace}
\newcommand{\BOne}[0]{\textbf{(B1)}\xspace}
\newcommand{\BTwo}[0]{\textbf{(B2)}\xspace}
\newcommand{\COneDagger}[0]{\textbf{(C1\textsuperscript{\dag})}\xspace}
\newcommand{\CTwoDagger}[0]{\textbf{(C2\textsuperscript{\dag})}\xspace}
\newcommand{\CThreeDagger}[0]{\textbf{(C3\textsuperscript{\dag})}\xspace}
\newcommand{\CFourDagger}[0]{\textbf{(C4\textsuperscript{\dag})}\xspace}
\newcommand{\AtDagger}[0]{\textbf{(At\textsuperscript{\dag})}\xspace}
\newcommand{\ExtOneDagger}[0]{\textbf{(Ext1\textsuperscript{\dag})}\xspace}
\newcommand{\ExtTwoDagger}[0]{\textbf{(Ext2\textsuperscript{\dag})}\xspace}
\newcommand{\CB}[0]{\textbf{(CB\textsuperscript{\dag})}\xspace}
\newcommand{\noreps}[1]{\langle #1 \rangle}
\newcommand{\yes}[1]{\left[#1 \right]}
\newcommand{\no}[1]{\left\llcorner #1 \right\lrcorner}

\newcommand{\mc}{\mathcal }

\DeclareMathSymbol{\mlq}{\mathord}{operators}{``}
\DeclareMathSymbol{\mrq}{\mathord}{operators}{`'}
\DeclareMathOperator{\dom}{dom}
\DeclareMathOperator{\ran}{ran}
\DeclareMathOperator{\dgSp}{dgSp}

\newtheorem{theorem}{Theorem}[chapter]
\newtheorem{lem}[theorem]{Lemma}
\newtheorem{thm}[theorem]{Theorem}
\newtheorem{corollary}[theorem]{Corollary}
\newtheorem{prop}[theorem]{Proposition}
\newtheorem{claim}[theorem]{Claim}

\theoremstyle{definition}
\newtheorem{defn}[theorem]{Definition}
\newtheorem{exam}[theorem]{Example}

\newtheorem{question}[theorem]{Question}
\newcounter{cases}[theorem]
\newtheorem{case}[cases]{Case}
\newtheorem{subcase}{Subcase}

\theoremstyle{remark}
\newtheorem{remark}[theorem]{Remark}

\numberwithin{section}{chapter}
\numberwithin{equation}{chapter}
\numberwithin{subcase}{cases}
\numberwithin	{figure}{chapter}

\makeindex

\begin{document}

\frontmatter

\title{Degree Spectra of Relations on a Cone}

\author{Matthew Harrison-Trainor\\Group in Logic and the Methodology of Science\\University of California, Berkeley\\USA 94720\\\texttt{matthew.h-t@berkeley.edu}}
\date{\today}

\maketitle

\tableofcontents

\chapter*{Abstract}
Let $\mathcal{A}$ be a mathematical structure with an additional relation $R$. We are interested in the degree spectrum of $R$, either among computable copies of $\mathcal{A}$ when $(\mathcal{A},R)$ is a ``natural'' structure, or (to make this rigorous) among copies of $(\mathcal{A},R)$ computable in a large degree \textbf{d}. We introduce the partial order of degree spectra \textit{on a cone} and begin the study of these objects. Using a result of Harizanov---that, assuming an effectiveness condition on $\mc{A}$ and $R$, if $R$ is not intrinsically computable, then its degree spectrum contains all c.e.\ degrees---we see that there is a minimal non-trivial degree spectrum on a cone, consisting of the c.e.\ degrees. We show that this does not generalize to d.c.e.\ degrees by giving an example of two incomparable degree spectra on a cone. We also give a partial answer to a question of Ash and Knight: they asked whether (subject to some effectiveness conditions) a relation which is not intrinsically $\Delta^0_\alpha$ must have a degree spectrum which contains all of the $\alpha$-CEA degrees. We give a positive answer to this question for $\alpha = 2$ by showing that any degree spectrum on a cone which strictly contains the $\Delta^0_2$ degrees must contain all of the 2-CEA degrees. We also investigate the particular case of degree spectra on the structure $(\omega,<)$. This work represents the beginning of an investigation of the degree spectra of ``natural'' structures, and we leave many open questions to be answered.

\section*{Acknowledgments}

The author would like to thank Julia Knight, Alexander Melnikov, Noah Schweber, and especially Antonio Montalb\'an for their helpful comments and conversations. The author was partially supported by the NSERC PGS D, the NSERC Julie Payette Research Scholarship, and the Berkeley Fellowship.


\mainmatter

\chapter{Introduction}
The aim of this monograph is to introduce the study of ``nice'' or ``natural'' relations on a computable structure via the technical device of relativizing to a cone of Turing degrees.

Let $\mc{A}$ be a mathematical structure, such as a graph, poset, or vector space, and $R \subseteq A^n$ an additional relation on that structure (i.e., not in the diagram). The relation $R$ might be the set of nodes of degree three in a graph or the set of linearly independent pairs in a vector space. The basic question we ask in the computability-theoretic study of such relations is: how do we measure the complexity of the relation $R$? One way to measure the complexity of $R$ is the \textit{degree spectrum} of $R$. As is often the case in computability, many examples of relations with pathological degree spectra have been constructed in the literature but these tend to require very specific constructions. In this work, we restrict our attention to natural relations to capture those structures and relations which tend to show up in normal mathematical practice. We find that the degree spectra of natural relations are much better behaved than those of arbitrary relations, but not as well-behaved as one might hope.

The study of relations on a structure began with Ash and Nerode \cite{AshNerode81} who showed that, given certain assumptions about $\mc{A}$ and $R$, the complexity of the formal definition of $R$ in the logic $\mc{L}_{\omega_1 \omega}$ is related to its intrinsic computability. For example, $R$ has a computable $\Sigma_n$ definition if and only if $R$ is intrinsically $\Sigma_n$, that is, for any computable copy $\mc{B}$ of $\mc{A}$, the copy of $R$ in $\mc{B}$ is $\Sigma_n$.

Harizanov \cite{Harizanov87} introduced the degree spectrum of $R$ to capture a finer picture of the relation's complexity. The degree spectrum of $R$ is the collection of all Turing degrees of copies of the relation $R$ inside computable copies $\mc{B}$ of $\mc{A}$. The degree spectra of particular relations have been frequently studied, particularly with the goal of finding as many possible different degree spectra as possible. For example, Harizanov \cite{Harizanov93} has shown that there is a $\Delta^0_2$ (but not c.e.)\ degree \textbf{a} such that $\{0,\textbf{a}\}$ is the degree spectrum of a relation. Hirschfeldt \cite{Hirschfeldt00} has shown that for any $\alpha$-c.e.\ degree $\textbf{b}$, with $\alpha \in \omega \cup \{\omega\}$, $\{0,\textbf{b}\}$ is the degree spectrum of a relation. Hirschfeldt has also shown that for any c.e.\ degree \textbf{c} and any computable ordinal $\alpha$, the set of $\alpha$-c.e.\ degrees less than or equal to \textbf{c} is a degree spectrum. A number of other papers have been published showing that other degree spectra are possible---see for example Khoussainov and Shore \cite{KhoussainovShore98} and Goncharov and Khoussainov \cite{GoncharovKhusainov97}.

These results require complicated constructions and one would not expect relations which one finds in nature to have such degree spectrum. Instead, we expect to find simpler degree spectra such as the set of all c.e.\ degrees, the set of all d.c.e.\ degrees, or the set of all $\Delta^0_2$ degrees. The goal of this paper is to begin to answer the question of what sorts of degree spectra we should expect to find in nature. Since we cannot formally describe what we mean by a relation found in nature, we will prove our results relative to a cone. One expects a result which holds on a cone to hold for any ``nice'' or ``natural'' relations and structures because natural properties tend to relativize. Such structures include vector spaces and algebraically closed fields, but not first-order arithmetic. We hope to be able to convince the reader that the study of relations relative to a cone is an interesting and useful way of approaching the study of relations that one might come across in nature. Our results are the beginning and there is a large amount of work still to be done. An interesting picture is already starting to emerge.

We will introduce the definition, suggested by Montalb\'an, of a (relativized) degree spectrum of a relation. The results in this paper can be viewed as studying the partial ordering of degree spectra on a cone. The following is a simplification of the full definition which will come later in Chapter \ref{Preliminaries}.

\begin{defn}[Montalb\'an]\label{Montalban-Defn}
Let $\mc{A}$ and $\mc{B}$ be structures with relations $R$ and $S$ respectively. We say that $R$ and $S$ have the same degree spectrum ``on a cone'' if, for all sets $C$ on a cone,
\[ \{d(R^{\tilde{\mathcal{A}}} \oplus C) : \tilde{\mathcal{A}} \cong \mathcal{A} \text{ and } \tilde{\mathcal{A}} \leq_T C \} = \{d(S^{\tilde{\mathcal{B}}} \oplus C) : \tilde{\mathcal{B}} \cong \mathcal{B} \text{ and } \tilde{\mathcal{B}} \leq_T C \} \]
where $d(D)$ is the Turing degree of the set $D$.
\end{defn}

We are particularly interested in whether or not there are ``fullness'' results for particular types of degree spectra, by which we mean results which say that degree spectra on a cone must contain many degrees. This is in opposition to the pathological examples of many small (even two-element) degree spectra that can be constructed when not working on a cone. There are a small number of previous ``fullness'' results, though not in the language which we use here, starting with Harizanov \cite{Harizanov91} who proved that (assuming $(*)$ below), as soon as a degree spectrum contains more than the computable degree, it must contain all c.e.\ degrees:

\begin{thm}[{Harizanov \cite[Theorem 2.5]{Harizanov91}}]\label{Harizanov}
Let $\mc{A}$ be a computable structure and $R$ a computable relation on $\mc{A}$ which is not relatively intrinsically computable. Suppose moreover that the effectiveness condition $(*)$ holds of $\mc{A}$ and $R$. Then for every c.e.\ set $C$, there is a computable copy $\mc{B}$ of $\mc{A}$ such that $R^\mc{B} \equiv_T C$.

\begin{itemize}[label=$(*)$]
	\item For every $\bar{a}$, we can computably find $a \in R$ such that for all $\bar{b}$ and quantifier-free formulas $\theta(\bar{z},x,\bar{y})$ such that $\mc{A} \models \theta(\bar{a},a,\bar{b})$, there are $a' \notin R$ and $\bar{b}'$ such that $\mc{A} \models \theta(\bar{a},a',\bar{b}')$
\end{itemize}
\end{thm}

The result is stated using the effectiveness condition $(*)$ which says that $R$ must be a nice relation in some particular way. When we relativize to a cone, the effectiveness condition trivializes and we are left with the statement:

\begin{corollary}[Harizanov]\label{cor:Harizanov1}
Relative to a cone, every degree spectrum either is the computable degree, or contains all c.e.\ degrees.
\end{corollary}

\noindent This result stands in contrast to the state of our knowledge of degree spectra when not working on a cone, where we know almost no restrictions on what sets of degrees may be degree spectra.

Ash and Knight tried to generalize Harizanov's result in the papers \cite{AshKnight95} and \cite{AshKnight97}. They wanted to replace ``c.e.'' by ``$\Sigma^0_\alpha$''. In our language of degree spectra on a cone, they wanted to show that every degree spectrum is either contained in the the $\Delta_\alpha$ degrees or contains all of the $\Sigma_\alpha$ degrees. However, they discovered that this was false: there is a computable structure $\mc{A}$ with a computable relation $R$ where $R$ is intrinsically $\Sigma_\alpha$, not intrinsically $\Delta^0_\alpha$, and for any computable copy $\mc{B}$, $R^\mc{B}$ is $\alpha$-CEA. Moreover, the proof of this relativizes.

So instead of asking whether ``c.e.'' can be replaced by $\Sigma^0_\alpha$, Ash and Knight asked whether ``c.e.'' can be replaced by ``$\alpha$-CEA''. A set $S$ is $n$-CEA if there are sets $S_0,S_1,S_2,\ldots,S_n = S$ such that $S_0$ is c.e., $S_1$ is c.e.\ in and above $S_0$, $S_2$ is c.e.\ in and above $S_1$, and so on. For now, the reader can ignore what this means for an infinite ordinal $\alpha$ (the definition in general will follow in Chapter \ref{Preliminaries}). Ash and Knight were able to show that Harizanov's result can be extended in this manner when the coding is done relative to a $\Delta^0_\alpha$-complete set (note that, modulo $\Delta^0_\alpha$, every $\Sigma^0_\alpha$ set is $\alpha$-CEA):
\begin{thm}[{Ash-Knight \cite[Theorem 2.1]{AshKnight97}}]
Let $\mc{A}$ be an $\alpha$-friendly computable structure with an additional computable relation $R$. Suppose that $R$ is not relatively intrinsically $\Delta^0_\alpha$ and moreover, for all $\bar{c}$, we can find $a \notin R$ which is effectively $\alpha$-free over $\bar{c}$. Then for any $\Sigma^0_\alpha$ set $C$, there is a computable copy $\mc{B}$ of $\mc{A}$ such that
\[ R^\mc{B} \oplus \Delta^0_\alpha \equiv_T C \oplus \Delta^0_\alpha \]
where $\Delta^0_\alpha$ is a $\Delta^0_\alpha$-complete set.
\end{thm}

\noindent But this theorem is not enough to show that, on a cone, every degree spectrum is either contained in the the $\Delta_\alpha$ degrees or contains all of the $\alpha$-CEA degrees. A much better result, and one which would be sufficient, would be to show that $R^\mc{B} \equiv_T C$ rather than $R^{\mc{B}} \oplus \Delta^0_\alpha \equiv_T C \oplus \Delta^0_\alpha$. This was the goal of Knight in \cite{Knight98} where she showed that it could be done with strong assumptions on the relation $R$, namely that one could find $\alpha$-free elements which did not interfere with each other too much. The general question, without these strong assumptions, was left unresolved.

One of our main results in this paper in Chapter \ref{SigmaTwoSection} is a positive answer to this question in the case $\alpha = 2$.

\begin{thm}\label{relativized-2}
Let $\mc{A}$ be a structure, and let $R$ be an additional relation. Suppose that $R$ is not intrinsically $\Delta^0_2$ on any cone. Then, on a cone, the degree spectrum of $R$ contains the 2-\cea sets.
\end{thm}

\noindent The proof uses an interesting method which we have not seen before and which we think is of independent interest. We will describe the method briefly here. During the construction, we are presented with two possible choices of how to continue, but it is not clear which will work. We are able to show that one of the two choices must work, but in order to find out which choice it is we must consider a game in which we play out the rest of the construction against an opponent who attempts to make the construction fail. By finding a winning strategy for this game, we are able to decide which choice to make.

Up to this point, degree spectra on a cone are looking very well-behaved, and in fact one might start to hope that they are linearly ordered. However, this is not the case as we see by considering Ershov's hierarchy. Suppose (once again working on a cone) that there is a computable copy $\mc{B}$ such that $R^{\mc{B}}$ is not of c.e.\ degree. Is it necessarily the case that for every d.c.e.\ set $W$, there is a computable copy $\mc{C}$ such that $R^\mc{C} \equiv_T W$? We will show in Chapter \ref{DCESection} that this is not the case. Moreover, we will show that there is a computable structure $\mc{A}$ with relatively intrinsically d.c.e.\ relations $R$ and $S$ which have incomparable degree spectra relative to every oracle.

\begin{thm}\label{thm-incomparable-dce}
There is a computable structure $\mc{A}$ and relatively intrinsically d.c.e.\ relations $R$ and $S$ such that neither $R$ nor $S$ are intrinsically of c.e.\ degree, even relative to any cone, and the degree spectra of $R$ and $S$ are incomparable relative to any cone (i.e., the degree spectrum of $R$ is not contained in that of $S$, and vice versa).
\end{thm}

In proving this, we will also give a structural condition equivalent to being \textit{intrinsically of c.e.\ degree} (which, as far as we are aware, is a new definition; we mean that in any computable copy, the relation has c.e.\ degree). The structural condition works for relations which are intrinsically d.c.e., and it does not seem difficult to extend it to work in more general cases.

The following is a summary of all that we know about the possible degree spectra of relations on a cone:
\begin{enumerate}
	\item there is a smallest degree spectrum: the computable degree,
	\item there is a smallest degree spectrum strictly containing the computable degree: the c.e.\ degrees (Corollary \ref{cor:Harizanov1}, see \cite{Harizanov91}),
	\item there are two incomparable degree spectra both strictly containing the c.e.\ degrees and strictly contained in the d.c.e. degrees (Theorem \ref{relativized-2}),
	\item any degree spectrum strictly containing the $\Delta^0_2$ degrees must contain all of the 2-CEA degrees (Theorem \ref{thm-incomparable-dce}).
\end{enumerate}
Figure \ref{fig:intro} gives a graphical representation of all that is known. There are many more questions to be asked about degree spectra on a cone. In general, at least at the lower levels, there seem to be far fewer degree spectra on a cone than there are degrees (or degree spectra not on a cone). We expect this pattern to continue. However, an interesting phenomenon is that not all degree spectra are ``named'' (e.g., as the $\Sigma^0_1$ or $\Delta^0_2$ degrees are), though perhaps this is just because we do not understand enough about them to name them. The hope would be to classify and name all of the possible degree spectra.

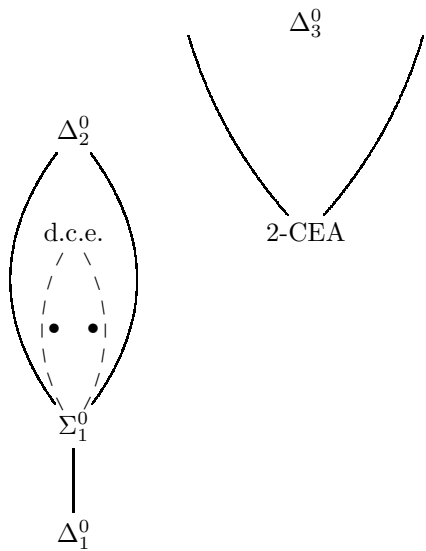
\begin{figure}[htb]
	\[\xymatrix{
&\,&\Delta_{3}^{0}&\,\\\Delta_{2}^{0}&&&\\\mbox{d.c.e.}&&2\mbox{-CEA}\ar@{-}@/^/[uul]\ar@{-}@/_/[uur]&\\\bullet\quad\bullet&&&\\\Sigma^0_{1}\ar@{--}@/_1pc/[uu]\ar@{--}@/^1pc/[uu]\ar@{-}@/_2pc/[uuu]\ar@{-}@/^2pc/[uuu]&&&\\\Delta_{1}^{0}\ar@{-}[u]&&&
}\]
	\caption{A visual summary of everything we know so far (including results from this paper) about the possible degree spectra of relations on a cone. The possible degree spectra are labeled or contained within one of the two enclosed areas (i.e., between $\Sigma^0_1$ and $\Delta^0_2$, or above 2-CEA). The two strictly d.c.e.\ degrees shown are incomparable.}
	\label{fig:intro}
\end{figure}

In this paper, we will also consider the special case of the structure $(\omega,<)$. This special case has been studied previously by Downey, Khoussainov, Miller, and Yu \cite{DowneyKhoussianovMillerYu09}, Knoll \cite{Knoll09}, and Wright \cite{Wright13}. Knoll showed that for the standard copy of $(\omega,<)$, the degree spectrum of any unary relation which is not intrinsically computable consists of exactly the $\Delta^0_2$ degrees. So while the counterexamples of Theorem \ref{thm-incomparable-dce} exist in general, such counterexamples may not exist for particular structures. Wright \cite{Wright13} later independently found this same result about unary relations, and also showed that for $n$-ary relations on $(\omega,<)$, any degree spectrum which contains a non-computable set contains all c.e.\ sets. In Chapter \ref{OmegaSection}, we begin an investigation of what the partial order of degree spectra on a cone looks like when we restrict ourselves to relations on $(\omega,<)$. We show that every relation which is intrinsically $\alpha$-c.e.\ is intrinsically of c.e.\ degree. We also introduce the notion of what it means for a relation to \textit{uniformly} have its degree spectrum, and show that either
\begin{enumerate}
	\item there is a computable relation $R$ on $(\omega,<)$ whose degree spectrum strictly contains the c.e.\ degrees but does not contain all of the d.c.e.\ degrees, or
	\item there is a computable relation $R$ on $(\omega,<)$ whose degree spectrum is all of the $\Delta^0_2$ degrees but does not have this degree spectrum uniformly.
\end{enumerate}
Both (1) and (2) are interesting situations, and one or the other must occur. It seems likely, but very difficult to show, that (1) occurs.

In this monograph, we would like to advocate for this program of studying degree spectra on a cone. We hope that some of the results in this paper will support the position that this is a fruitful view. We believe that the study of degree spectra on a cone is interesting beyond the characterization of the degree spectra of naturally occurring structures. Even when considering structures in general, knowing which results can and cannot be proven on a cone is still illuminating. If a result holds in the computable case, but not on a cone, then that means that the result relies in some way on the computable presentation of the original structure (for example, diagonalization arguments can often be used to produce pathological examples; such arguments tend not to relativize). Understanding why certain a result holds for computable structures but fails on a cone is a way of understanding the essential character of the proof of that result.

We will begin by making some preliminary definitions in Chapter \ref{Preliminaries}. In Chapter \ref{DCESection}, we will consider d.c.e.\ degrees. We begin by introducing the notion of being intrinsically of c.e.\ degree and give a characterization of such relations. We then prove Theorem \ref{thm-incomparable-dce}. In Chapter \ref{OmegaSection}, we apply some of the results from Chapter \ref{DCESection} to the structure $(\omega,<)$ and study the degree spectra of such relations. Chapter \ref{SigmaTwoSection} is devoted to the proof of Theorem \ref{relativized-2} on 2-CEA degrees. Finally, in Chapter \ref{QuestionSection} we will state some interesting open questions and describe what we see as the future of this program. We also include a discussion of the technical details around the definition of degree spectra on a cone, in particular in relation to relativizing Harizanov's results on c.e.\ degrees, in Appendix \ref{CESection}.

\chapter{Preliminaries}
\label{Preliminaries}
In this chapter, we will introduce some background from computability theory and computable structure theory before formally defining what we mean by ``on a cone.''

\section{Computability Theory}

We will assume that the reader has a general knowledge of computability theory. For the most part, the reader will not need to know about computable ordinals. Occasionally we will state general questions involving a computable ordinal $\alpha$, but the reader should feel free to assume that $\alpha$ is a finite number $n$. See Chapter 4 of \cite{AshKnight00} for a reference on computable ordinals.

There are two classes of sets which will come up frequently which we will define here. Ershov's hierarchy \cite{Ershov68a,Ershov68b,Ershov70} generalizes the c.e.\ and d.c.e.\ sets to classify sets by how many times their computable approximation is allowed to change.

\begin{defn}[Ershov's hierarchy]
A set $X$ is \textit{$\alpha$-c.e.}\ if there are functions $g: \omega \times \omega \to \{0,1\}$ and $n: \omega \times \omega \to \{ \beta : \beta \leq \alpha\}$ such that for all $x$ and $s$,
\begin{enumerate}
	\item $g(x,0) = 0$,
	\item $n(x,0) = \alpha$,
	\item $n(x,s+1) \leq n(x,s)$,
	\item if $g(x,s+1) \neq g(x,s)$ then $n(x,s+1) < n(x,s)$, and
	\item $\lim_{s \to \infty} g(x,s) = X(x)$.
\end{enumerate}
\end{defn}

The function $g$ guesses at whether $x \in X$, with $n$ counting the number of changes. We could instead have made the following equivalent definition:

\begin{defn}[Ershov's hierarchy, alternate definition]\label{def:Ershov-alternate}
$X$ is $\alpha$-c.e.\ if there are uniformly c.e.\ families $(A_\beta)_{\beta < \alpha}$ and $(B_\beta)_{\beta < \alpha}$ such that
\[ X = \bigcup_{\beta < \alpha} (A_\beta - \bigcup_{\gamma < \beta} B_\gamma) \]
and if $x \in A_\beta \cap B_\beta$, then $x \in A_\gamma \cup B_\gamma$ for some $\gamma < \beta$.
\end{defn}
See Chapter 5 of \cite{AshKnight00} for more on $\alpha$-c.e.\ sets.

A set $X$ is \textit{c.e.\ in and above} (CEA in) a set $Y$ if $X$ is c.e.\ in $Y$ and $X \geq_T Y$. We can easily generalize this to any finite $n$ by iterating the definition: $X$ is $n$-CEA in $Y$ if there are $X_0 = Y,X_1,\ldots,X_n = X$ such that $X_i$ is CEA in $X_{i-1}$ for each $i = 1,\ldots,n$. We can even generalize this to arbitrary $\alpha$:

\begin{defn}[Ash-Knight \cite{AshKnight95}]
A set $X$ is \textit{$\alpha$-CEA} in a set $Y$ if there is a sequence $(X_\beta)_{\beta \leq \alpha}$ such that
\begin{enumerate}
	\item $X_0$ is recursive,
	\item $X_{\beta+1}$ is CEA in $X_\beta$,
	\item $X_\beta \leq X_\delta$ for $\delta > \beta$ a limit ordinal,
	\item $X_\delta$ is c.e.\ in $\bigoplus_{\beta < \delta} X_\beta$, and
	\item $X_\alpha = X$.
\end{enumerate}
\end{defn}

Ash and Knight note that a set which is $\alpha$-CEA is $\Sigma^0_\alpha$, but that the converse is not necessarily true (one can see that this follows from the existence of a minimal $\Delta^0_2$ set; a minimal $\Delta^0_2$ set is not c.e.\ and hence not 2-CEA).

\section{Computable Structure Theory}

We will consider only countable structures, so we will say ``structure'' when we mean ``countable structure.'' For an introduction to computable structures as well as much of the other background, see \cite{AshKnight00}. We view the atomic diagram $D(\mc{A})$ of a structure $\mc{A}$ as a subset of $\omega$, and usually we will identify $\mc{A}$ with its diagram. A computable presentation (or computable copy) $\mc{B}$ of a structure $\mc{A}$ is another structure $\mc{B}$ with domain $\omega$ such that $\mc{A} \cong \mc{B}$ and the atomic diagram of $\mc{B}$ is computable.

The infinitary logic $\mc{L}_{\omega_1 \omega}$ is the logic which allows countably infinite conjunctions and disjunctions but only finite quantification. If the conjunctions and disjunctions of a formula $\varphi$ are all over computable sets of indices for formulas, then we say that $\varphi$ is computable. We use $\Sigma_\alpha^\infi$ and $\Pi_\alpha^\infi$ to denote the classes of all infinitary $\Sigma_\alpha$ and all $\Pi_\alpha$ formulas respectively. We will also use $\Sigma_\alpha^\comp$ and $\Pi_\alpha^\comp$ to denote the classes of computable $\Sigma_\alpha$ and $\Pi_\alpha$ formulas. These formulas will often involve finitely many constant symbols from the structure. See Chapter 6 of \cite{AshKnight00} for a more complete description of computable formulas.

By a relation $R$ on a structure $\mc{A}$, we mean a subset of $\mc{A}^n$ for some $n$. We say that $R$ is \textit{invariant} if it is fixed by all automorphisms of $R$. It is a theorem, following from the Scott Isomorphism Theorem \cite{Scott65}, that a relation is invariant if and only if it is defined in $\mc{A}$ by a formula of $\mc{L}_{\omega_1 \omega}$. All of the relations that we will consider will be invariant relations. If $\mc{B}$ is a computable copy of $\mc{A}$, then there is a unique interpretation $R^\mc{B}$ of $R$ in $\mc{B}$, either by using the $\mc{L}_{\omega_1 \omega}$-definition of $R$, or using the invariance of $R$ under automorphisms (so that if $f: \mc{A} \to \mc{B}$ is an isomorphism, $f(R)$ is a relation on $\mc{B}$ which does not depend on the choice of the automorphism $f$).

The study of invariant relations began with Ash and Nerode \cite{AshNerode81}. They made the following definition: if $\Gamma$ is some property of sets, then $R$ is intrinsically $\Gamma$ if among all of the computable copies $\mc{B}$ of $\mc{A}$, $R^\mc{B}$ is $\Gamma$. Usually we will talk about relations which are intrinsically computable (or more generally $\Delta_\alpha$), intrinsically c.e. (or more generally $\Sigma_\alpha$ or $\Pi_\alpha$), intrinsically $\alpha$-CEA, or intrinsically $\alpha$-c.e. Ash and Nerode showed that (making some assumptions on the structure and on the relation) a relation is intrinsically c.e.\ if and only if is defined by a $\Sigma_1^\comp$ formula:

\begin{thm}[{Ash-Nerode \cite[Theorem 2.2]{AshNerode81}}]\label{AshNerode}
Let $\mc{A}$ be a computable structure and $R$ a relation on $\mc{A}$. Suppose that for any tuple $\bar{c} \in \mc{A}$ and any finitary existential formula $\varphi(\bar{c},\bar{x})$, we can decide whether or not there is $\bar{a} \notin R$ such that $\mc{A} \models \varphi(\bar{c},\bar{a})$. Then the following are equivalent:
\begin{enumerate}
	\item $R$ is intrinsically c.e.\,
	\item $R$ is defined by a $\Sigma_1^\comp$ formula with finitely many parameters from $\mc{A}$ (we say that $R$ is \textit{formally $\Sigma_1$} or \textit{formally c.e.}).
\end{enumerate}
\end{thm}

In practice, most naturally occurring structures and relations satisfy the effectiveness condition from this theorem. However, there are structures which do not have the effectiveness condition, and some of these structures are counterexamples to the conclusion of the theorem.

Barker \cite{Barker88} later generalized this to a theorem about intrinsically $\Sigma_\alpha$ relations. Ash and Knight \cite{AshKnight96} also proved a result for intrinsically $\alpha$-c.e.\ relations (with the formal definition being of the form of Definition \ref{def:Ershov-alternate} above).

Ash, Knight, Manasse, and Slaman \cite{AshKnightManasseSlaman89} and independently Chisholm \cite{Chisholm90} considered a relativized notion of intrinsic computability. We say that $R$ is \textit{relatively intrinsically $\Sigma_\alpha$} (or $\Pi_\alpha$, etc.) if, in every copy $\mc{B}$ of $\mc{A}$, $R^\mc{B}$ is $\Sigma^0_\alpha(\mc{B})$ ($\Pi^0_\alpha(\mc{B})$, etc.). Then they were able to prove a theorem similar to Theorem \ref{AshNerode} above but without an effectiveness condition: 
\begin{thm}[Ash-Knight-Manasse-Slaman \cite{AshKnightManasseSlaman89}, Chisholm \cite{Chisholm90}]\label{AKMSC}
Let $\mc{A}$ be a computable structure and $R$ a relation on $\mc{A}$. The following are equivalent:
\begin{enumerate}
	\item $R$ is relatively intrinsically $\Sigma_\alpha$,
	\item $R$ is defined by a $\Sigma_\alpha^\comp$ formula with finitely many parameters from $\mc{A}$.
\end{enumerate}
\end{thm}

\noindent These theorems say that the computational complexity of a relation is strongly tied to its logical complexity.

In order to give a finer measure of the complexity of a relation, Harizanov \cite{Harizanov87} introduced the degree spectrum.
\begin{defn}[Degree Spectrum of a Relation]
The degree spectrum of a relation $R$ on a computable structure $\mc{A}$ is the set
\[ \dgSp(R) = \{ d(R^\mc{B}) : \mc{B} \text{ is a computable copy of } \mc{A} \}. \]
\end{defn}

\section{Relativizing to a Cone}

In this section, we will formally describe what we mean by working on a cone, and by the degree spectrum of a relation on a cone. Consider the degree spectrum of a relation. For many natural structures and relations, the degree spectrum of a relation is highly related to the model-theoretic properties of the relation $R$. However, for more pathological structures (and first-order arithmetic), the degree spectra of relations can often be badly behaved. Some examples of such relations were given in the introduction. On the other hand, Theorem \ref{Harizanov} says that many relations have degree spectra which are nicely behaved (of course, there are relations which do not satisfy the effectivity condition from this theorem and which do not satisfy the conclusion---see \cite{Harizanov91}).

This is a common phenomenon in computable structure theory that there are unnatural structures which are counterexamples to theorems which would otherwise hold for natural structures. This unnatural behaviour tends to disappear when the theorem is relativized to a sufficiently high cone; in the case of Harizanov's result, relativizing the conclusion to any degree above $0''$ allows the theorem to be stated without the effectivity condition (since $0''$ can compute what is required by the effectivity condition).

A \textit{Turing cone} is collection of sets of the form $\{X : X \geq_T A\}$ for some fixed set $A$. A collection of sets is Turing invariant if whenever $X$ is in the collection, and $X \equiv_T Y$, then $Y$ is in the collection (i.e., the collection is a set of Turing degrees). Martin \cite{Martin68} noticed that any Turing invariant collection $A$ which is \textit{determined}\footnote{A set $A$ (viewed as set of reals in Cantor space $2^\omega$) is determined if one of the two players has a winning strategy in the Gale-Stewart game $G_A$, where players \POne and \PTwo alternate playing either $0$ or $1$; \POne wins if the combined sequence of plays is in $A$, and otherwise \PTwo wins. See \cite{GaleStewart53} and \cite{Jech03}.} either contains a cone, or contains a cone in its complement. Note that only one of these can happen for a given set $A$, as any two cones intersect non-trivially and contain a cone in their intersection. Moreover, by Borel determinacy (see \cite{Martin75}) every Borel invariant set is determined. Thus we can form a $\{0,1\}$-valued measure on the Borel sets of Turing degrees, selecting as ``large'' those sets of Turing degrees which contain a cone.

Given a statement $P$ which relativizes to any degree \textbf{d}, we say that $P$ holds \textit{on a cone} if the set of degrees \textbf{d} for which the relativization of $P$ to \textbf{d} holds contains a cone. If $P$ defines a Borel set of degrees in this way, then either $P$ holds on a cone or $\neg P$ holds on a cone. If $P$ holds on a cone, then $P$ holds for \textit{most} degrees, or for \textit{sufficiently high} degrees.

We say that $R$ is \textit{intrinsically $\Sigma_\alpha$ on a cone} if for all degrees \textbf{d} on a cone, and all copies $\mc{B}$ of $\mc{A}$ with $\mc{B}$ computable in \textbf{d}, $R^\mc{B}$ is $\Sigma^0_\alpha(\textbf{d})$. Then by relativizing previous results (Theorem \ref{AshNerode} or Theorem \ref{AKMSC}), we see that $R$ is intrinsically $\Sigma_\alpha$ on a cone if and only if it is defined by a $\Sigma^\infi_\alpha$ formula, without any computability-theoretic assumptions on either $\mc{A}$ and $R$ or the $\Sigma^\infi_\alpha$ formula. This is a relativisation to a cone of both intrinsically $\Sigma^0_\alpha$ and relatively intrinsically $\Sigma^0_\alpha$ simultaneously.

Note that when we work on a cone, we do not need to assume that the structure $\mc{A}$ or the relation $R$ are computable, because we can consider only cones with bases above $\mc{A} \oplus R$.

Now we will define what we mean by the degree spectrum of a relation on a cone. First, there is a natural relativisation of the degree spectrum of a relation to a degree \textbf{d}. The \textit{degree spectrum of $R$ below the degree \textbf{d}} is
\[ \dgSp(\mc{A},R)_{\leq \textbf{d}} = \{ d(R^\mc{B}) \oplus \textbf{d} : (\mc{B},R^\mc{B}) \text{ is an isomorphic copy of } (\mc{A},R) \text{ with } \mc{B} \leq_T \textbf{d}\}. \]
An alternate definition would require the isomorphic copy $\mc{B}$ to be Turing equivalent to \textbf{d}, rather than just computable in \textbf{d}:
\[ \dgSp(\mc{A},R)_{\equiv \textbf{d}} = \{ d(R^\mc{B}) \oplus \textbf{d} : (\mc{B},R^\mc{B}) \text{ is an isomorphic copy of } (\mc{A},R) \text{ with } \mc{B} \equiv_T \textbf{d}\}. \]
If $\mc{B} \leq_T \textbf{d}$, then by Knight's theorem on the upwards closure of the degree spectrum of structures (see \cite{Knight86}), there is an isomorphic copy $\mc{C}$ of $\mc{B}$ with $\mc{C} \equiv_T \textbf{d}$ and a $\textbf{d}$-computable isomorphism $f:\mc{B} \to \mc{C}$. Then $R^\mc{C} \oplus \textbf{d} \equiv_T R^\mc{B} \oplus \textbf{d}$. So these two definitions are equivalent.

The proof of Theorem \ref{Harizanov} relativizes to show that for any degree $\textbf{d} \geq 0''$, if $\dgSp(\mc{A},R)_{\leq \textbf{d}}$ contains a degree which is not computable in $\textbf{d}$, then it contains every degree CEA in \textbf{d}. One could also have defined the degree spectrum of a relation to be the set
\[ \dgSp^*(\mc{A},R)_{\leq  \textbf{d}} = \{ d(R^\mc{B}) : (\mc{B},R^\mc{B}) \text{ is an isomorphic copy of } (\mc{A},R) \text{ with } \mc{B} \leq_T \textbf{d}\}. \]
In this case, Harizanov's proof of Theorem \ref{Harizanov} does not relativize. In Appendix \ref{CESection}, we will consider a new proof of Harizanov's result which relativizes in the correct way for this definition of the degree spectrum. However, our proof becomes much more complicated than Harizanov's original proof. For the other results in the paper, we will not consider $\dgSp^*(\mc{A},R)_{\leq  \textbf{d}}$. Though it is quite possible that there are similar ways to extend our proofs, it would distract from main content of those results. We are interested in whether there is any real difference between $\dgSp(\mc{A},R)_{\leq  \textbf{d}}$ and $\dgSp^*(\mc{A},R)_{\leq  \textbf{d}}$, or whether any result provable about one transfers in a natural manner to the other. For example, is it always the case that restricting $\dgSp^*(\mc{A},R)_{\leq  \textbf{d}}$ to the degrees above \textbf{d} gives $\dgSp(\mc{A},R)_{\leq  \textbf{d}}$ for sufficiently high \textbf{d}?

Now we want to make our relativisation of the degree spectrum independent of the degree \textbf{d}. Thus we turn to Definition \ref{Montalban-Defn} due to Montalb\'an, which we will now develop more thoroughly. To each structure $\mc{A}$ and relation $R$, we can assign the map $f_R \colon \textbf{d} \mapsto \dgSp(\mc{A},R)_{\leq\textbf{d}}$. Given two pairs $(\mc{A},R)$ and $(\mc{B},S)$, for any degree \textbf{d}, either $\dgSp(\mc{A},R)_{\leq\textbf{d}}$ is equal to $\dgSp(\mc{B},S)_{\leq\textbf{d}}$, one is strictly contained in the other, or they are incomparable. By Borel determinacy, there is a cone on which exactly one of these happens. Thus we get a pre-order on these functions $f_R$, and taking the quotient by equivalence, we get a partial order on degree spectrum. Denote the elements of the quotient by $\dgSp_{rel}(\mc{A},R)$. We call $\dgSp_{rel}(\mc{A},R)$ the \textit{degree spectrum of $R$ on a cone}.

For many classes $\Gamma$ of degrees which relativize, for example the $\Sigma_\alpha$ degrees, there is a natural way of viewing them in this partial ordering by considering the map $\Gamma: \textbf{d} \to \Gamma(\textbf{d})$. By an abuse of notation, we will talk about such a class $\Gamma$ as a degree spectrum (in fact, it is easy to see for many simple classes of degrees that they are in fact the degree spectrum of some relation on some structure). Thus we can say, for example, that the degree spectrum, on a cone, of some relation contains the $\Sigma^0_\alpha$ degrees, or is equal to the d.c.e.\ degrees, and so on.

In particular, using this notation, we see that Theorem \ref{Harizanov} yields:
\begin{corollary}[Harizanov]\label{cor:Harizanov}
Let $\mc{A}$ be a structure and $R$ a relation on $\mc{A}$. Then either:
\begin{enumerate}
	\item $\dgSp_{rel}(\mc{A},R) = \Delta^0_1$ or
	\item $\dgSp_{rel}(\mc{A},R) \supseteq \Sigma^0_1$.
\end{enumerate}
\end{corollary}
The cone on which this theorem holds is $(\mc{A} \oplus R)''$---i.e., one could replace $\dgSp_{rel}$ with $\dgSp_{\leq \textbf{d}}$ for any degree $\textbf{d} \geq_T (\mc{A} \oplus R)''$.

We also get the following restatements of Theorems \ref{relativized-2} and \ref{thm-incomparable-dce}:

\begin{corollary}
Let $\mc{A}$ be a structure and $R$ a relation on $\mc{A}$. Then either:
\begin{enumerate}
	\item $\dgSp_{rel}(\mc{A},R) \subseteq \Delta^0_2$ or
	\item $\dgSp_{rel}(\mc{A},R) \supseteq \text{2-\cea}$.
\end{enumerate}
\end{corollary}

\begin{corollary}
There is a structure $\mc{A}$ and relations $R$ and $S$ on $\mc{A}$ such that $\dgSp_{rel}(\mc{A},R)$ and $\dgSp_{rel}(\mc{A},S)$ contains the c.e.\ degrees and are contained within the d.c.e.\ degrees, but neither $\dgSp_{rel}(\mc{A},R) \subseteq \dgSp_{rel}(\mc{A},S)$ nor $\dgSp_{rel}(\mc{A},S) \subseteq \dgSp_{rel}(\mc{A},R)$.
\end{corollary}

Note that these two concepts that we have just introduced---intrinsic computability on a cone and degree spectra on a cone---are completely independent of the presentations of $\mc{A}$ and $R$. Moreover, the intrinsic computability of a relation $R$ is completely dependent on its model-theoretic properties. So by looking on a cone, we are able to look at more model-theoretic properties of relations while using tools of computability theory.

The reader should always keep in mind the motivation behind this work. The theorems we prove are intended to be applied to naturally occurring structures. For well-behaved structures, ``property $P$'' and ``property $P$ on a cone'' should be viewed as interchangeable.

This work was originally motivated by a question of Montalb\'an first stated in \cite{Wright13}. There is no known case of a structure $(\mc{A},R)$ where $\dgSp(\mc{A},R)_{\leq \textbf{d}}$ does not have a maximum degree for \textbf{d} sufficiently large. When the degree spectrum does contain a maximum degree, define the function $f_{\mc{A},R}$ which maps a degree \textbf{d} to the maximum element of $\dgSp(\mc{A},R)_{\leq \textbf{d}}$. This is a degree-invariant function\footnote{Technically, the function we are considering maps a set $C$ to some set $D$ which is of maximum degree in $\dgSp(\mc{A},R)_{\leq d(C)}$. For now, we can ignore which sets we choose.}, and hence the subject of Martin's conjecture. Montalb\'an has asked whether Martin's conjecture is true of this function $f_{\mc{A},R}$, that is, is it true that for every structure $\mc{A}$ and relation $R$, there is an ordinal $\alpha < \omega_1$ such that for all $\textbf{d}$ on a cone, $\textbf{d}^{(\alpha)}$ is the maximal element of $\dgSp(\mc{A},R)_{\leq \textbf{d}}$?

Recall the question of Ash and Knight from the introduction, which we stated as: is it true that every degree spectrum is either contained in the the $\Delta_\alpha$ degrees or contains all of the $\alpha$-\cea degrees? If this was true, then Montalb\'an's question would (almost) be answered positively, as every relation $R$ has a definition which is $\Sigma_\alpha^\infi$ and $\Pi_\alpha^\infi$ for some $\alpha < \omega_1$; choosing $\alpha$ to be minimal, if $\alpha$ is a successor ordinal $\alpha = \beta + 1$, then for all degrees $\textbf{d}$ on a cone, there is a complete $\Delta^0_\alpha(\textbf{d})$ degree which is $\beta$-CEA above \textbf{d} and hence is the maximal element of $\dgSp(\mc{A},R)_{\leq \textbf{d}}$. The relativized version of Harizanov's Theorem \ref{Harizanov} answers Ash and Knight's question (and hence Montalb\'an's question) for relations which are defined by a $\Sigma_1^\infi$ formula, and our Theorem \ref{relativized-2} answers these questions for relations which are $\Sigma_2^\infi$-definable\footnote{Note that the proofs of these results also show that $f_{\mc{A},R}$ is uniformly degree-invariant for these relations, which also implies that Martins conjecture holds for these $f_{\mc{A},R}$---see \cite{Steel82} and \cite{SlamanSteel88}. What we mean is that, given sets $C$ and $D$ of degree \textbf{d} (with $C \equiv_T D$), Theorem \ref{Harizanov} (respectively Theorem \ref{relativized-2}) provide $\mc{C},\mc{D}$ isomorphic copies of $\mc{A}$ with $\mc{C} \leq_T C$ and $\mc{D} \leq_T D$ with $R^\mc{C} \equiv_T C'$ and $R^\mc{D} \equiv_T D'$ (resp. $R^\mc{C} \equiv_T C''$ and $R^\mc{D} \equiv_T D''$) and these Turing equivalences are uniform in $C'$ and $D'$ (resp. $C''$ and $D''$). So given an index for the equivalence $C \equiv_T D$, we can effectively find an index for the equivalence $R^\mc{C} \equiv_T R^\mc{D}$.}.

\chapter[D.C.E. Degree Spectra]{Degree Spectra Between the C.E. Degrees and the D.C.E. Degrees}
\label{DCESection}
We know that every degree spectrum (on a cone) which contains a non-comput\-able degree contains all of the c.e.\ degrees. In this section, we will consider relations whose degree spectra strictly contain the c.e.\ degrees. The motivating question is whether any degree spectrum on a cone which strictly contains the c.e.\ degrees contains all of the d.c.e.\ degrees. We will show that this is false by proving Theorem \ref{thm-incomparable-dce} which says that there are two incomparable degree spectra which contain only d.c.e.\ degrees. In the process, we will define what it means to be \textit{intrinsically of c.e.\ degree} (as opposed to simply being c.e.) and give a characterization of the relatively intrinsically d.c.e.\ relations which are intrinsically of c.e.\ degree, and at the same time a sufficient (but not necessary) condition for a relation to not be intrinsically of c.e.\ degree.

\section[Conditions to be Intrinsically of C.E. Degree]{Necessary and Sufficient Conditions to be Intrinsically of C.E. Degree}

We begin by defining what it means to be intrinsically of c.e.\ degree.
\begin{defn}$R$ on $\mc{A}$ is \textit{intrinsically of c.e.\ degree} if in every computable copy $\mc{B}$ of $\mc{A}$, $R^\mc{B}$ is of c.e.\ degree.
\end{defn}
\noindent We can make similar definitions for \textit{relatively intrinsically of c.e.\ degree} and \textit{intrinsically of c.e.\ degree on a cone}. As far as we are aware, these are new definitions.

Any relation which is intrinsically c.e.\ is intrinsically of c.e.\ degree, but the following example shows that the converse implication does not hold (even on a cone).

\begin{exam}\label{simple-example}
Let $\mc{A}$ be two-sorted with sorts $B$ and $C$. There is a relation $S$ in the signature of $\mc{A}$ of type $B \times C$. The sort $B$ is a directed graph, each connected component of which consists of two elements and one directed edge. Each element of $B$ is related via $S$ to zero, one, or two elements of $C$, and the additional relation $R$ (not in the signature of $\mc{A}$) consists of those elements of $B$ which are related to exactly one element of $C$. $\mc{A}$ consists of infinitely many copies of each of these three connected components and nothing else, with the edge adjacency relation and $S$:

\[\xymatrix@C=1em{
0 \ar[r] & 1 & & 1 \ar[r] & 2 & & 2 \ar[r] & 2
}\]

\noindent The numbers show how many elements of $C$ a particular element of $B$ is related to. For example, $0 \to 1$ is a two element connected component with a single directed edge, and the first element is not related to any elements of $C$, while the second element is related to a single element of $C$. In any copy $\mc{B}$ of $\mc{A}$, the set $T^{\mc{B}}$ of elements related to exactly two elements of $C$ is c.e.\ in $\mc{B}$. We claim that this set has the same Turing degree as $R^\mc{B}$. Let $a$ and $b$ be elements in $\mc{A}$, with a directed edge from $a$ to $b$. Then there are three possibilities:
\begin{enumerate}
	\item $a \notin R^{\mc{B}}$, $b \in R^{\mc{B}}$ and $a \notin T^{\mc{B}}$, $b \notin T^{\mc{B}}$,
	\item $a \in R^{\mc{B}}$, $b \notin R^{\mc{B}}$ and $a \notin T^{\mc{B}}$, $b \in T^{\mc{B}}$, or
	\item $a \notin R^{\mc{B}}$, $b \notin R^{\mc{B}}$ and $a \in T^{\mc{B}}$, $b \in T^{\mc{B}}$.
\end{enumerate}
Each of these three possibilities is distinct from the others both in terms of $R$ and also in terms of $T^{\mc{B}}$. So knowing whether $a \in R^\mc{B}$ and $b \in R^\mc{B}$ determines whether $a \in T^{\mc{B}}$ and $b \in T^{\mc{B}}$, and vice versa. Hence $R^\mc{B} \oplus \mc{B} \equiv T^{\mc{B}} \oplus \mc{B}$. Since $T^{\mc{B}}$ is c.e., $R^\mc{B}$ is of c.e.\ degree in $D(\mc{B})$. Note that $R^\mc{B}$ is always d.c.e.\ in $D(\mc{B})$, but one can show using a standard argument that $R^\mc{B}$ is not always c.e.\ in $D(\mc{B})$.
\end{exam}

We will begin by finding a necessary and sufficient condition for a relation to be intrinsically of c.e.\ degree. We will assume, for one of the directions, that the relation is relatively intrinsically d.c.e. A relation which is not intrinsically $\Delta_2$ cannot be intrinsically of c.e.\ degree (and, assuming sufficient effectiveness conditions, the same is true for the relative notions). We leave the question open for relations which are relatively intrinsically $\Delta_2$ but not relatively intrinsically d.c.e.

An important idea in most of the results in this work are the free tuples from the theorem of Ash and Nerode on intrinsically computable relations \cite{AshNerode81}, and other variations.
\begin{defn}
Let $\bar{c}$ be a tuple from $\mc{A}$. We say that $\bar{a} \notin R$ is \textit{free} (or \textit{1-free}) over $\bar{c}$ if for any finitary existential formula $\psi(\bar{c},\bar{x})$ true of $\bar{a}$ in $\mc{A}$, there is $\bar{a}' \in R$ which also satisfies $\psi(\bar{c},\bar{x})$.
\end{defn}
\noindent Such free elements, and many variations, have been used throughout the literature, including in many of the results we referenced in the previous chapters. We will only use 1-free elements in Appendix \ref{CESection}, but we will use other variants in Chapters \ref{DCESection}, \ref{OmegaSection}, and \ref{SigmaTwoSection}.

In the spirit of the definitions made just before Propositions 2.2 and 2.3 of \cite{AshKnight96}, we will make the following definition of what it means to be \textit{difference-free}, or \textit{d-free}. Let $\mc{A}$ be a computable structure and $R$ a computable relation on $\mc{A}$. We begin with the case where $R$ is unary, where the condition is simpler to state. We say that $a \notin R$ is \textit{d-free over $\bar{c}$} if for every $b_1,\ldots,b_n$ and existential formula $\varphi(\bar{c},u,v_1,\ldots,v_n)$ true of $a,b_1,\ldots,b_n$, there are $a' \in R$ and $b_1',\ldots,b_n'$ which satisfy $\varphi(\bar{c},u,v_1,\ldots,v_n)$ such that for every existential formula $\psi(\bar{c},u,v_1,\ldots,v_n)$ true of them, there are $a'',b_1'',\ldots,b_n''$ satisfying $\psi$ with $a'' \notin R$ and $b_i \in R \Leftrightarrow b_i'' \in R$.

Note that this is different from the 2-free elements which are defined just before Propositions 2.2 and 2.3 in \cite{AshKnight96}. The definitions are the same, except that for $a$ to be 2-free over $\bar{c}$, there is no requirement on the $b_i$ and $b_i''$. Note that an element $a$ may be 2-free over $\bar{c}$, but not d-free over $\bar{c}$ (but if $a$ is d-free over $\bar{c}$, then it is 2-free over $\bar{c}$).

Now suppose that $R$ is an $m$-ary relation. We say that $\bar{a}$ is d-free over $\bar{c}$ if for every $\bar{b}$ and existential formula $\varphi(\bar{c},\bar{u},\bar{v})$ true of $\bar{a},\bar{b}$, there are $\bar{a}'$ and $\bar{b}'$ which satisfy $\varphi(\bar{c},\bar{u},\bar{v})$ such that $R$ restricted to tuples from $\bar{c}\bar{a}'$ is not the same as $R$ restricted to tuples from $\bar{c}\bar{a}$ and also such that for every existential formula $\psi(\bar{c},\bar{u},\bar{v})$ true of them, there are $\bar{a}'',\bar{b}''$ satisfying $\psi$ and such that $R$ restricted to $\bar{c}\bar{a}''\bar{b}''$ is the same as $R$ restricted to $\bar{c}\bar{a}\bar{b}$. If $R$ is unary, a tuple $\bar{a}$ is d-free over $\bar{c}$ if and only if one of its entries $a_i$ is.

Under sufficient effectiveness conditions we will show---for a formally d.c.e.\ relation $R$ on a structure $\mc{A}$---that $R$ is not intrinsically of c.e.\ degree if and only if for each tuple $\bar{c}$ there is some $\bar{a}$ which is d-free over $\bar{c}$ (note that under the effectiveness conditions of Proposition 2.2 of Ash and Knight \cite{AshKnight96}, a relation is formally d.c.e.\ if and only if it is intrinsically d.c.e.). In fact, the existence of a tuple $\bar{c}$ over which no tuple $\bar{a}$ is d-free will imply that $R$ is not intrinsically of c.e.\ degree even if $R$ is not formally d.c.e. We will use this in Theorem \ref{thm:weird-relation-on-omega} of Chapter \ref{OmegaSection}.

When stated in terms of degree spectra on a cone, our result is:
\begin{prop}
Let $\mc{A}$ be a structure and $R$ a relation on $\mc{A}$. Then if, for each tuple $\bar{c}$, there is $\bar{a}$ which is d-free over $\bar{c}$, then the degree spectrum $\dgSp_{rel}(\mc{A},R)$ on a cone strictly contains the c.e.\ degrees. Moreover, if $R$ is formally d.c.e., then this is a necessary condition.
\end{prop}

The (relativizations of) the next two propositions prove the two directions of this using the appropriate effectiveness conditions.
\begin{prop}\label{ce-degree-condition}
Let $R$ be a formally d.c.e.\ relation on a computable structure $\mc{A}$. Suppose that there is a tuple $\bar{c}$ over which no $\bar{a}$ is d-free. Assume that given tuples $\bar{a}$ and $\bar{c}$, we can find witnesses $\bar{b}$ and $\varphi(\bar{c},\bar{u},\bar{v})$ to the fact that $\bar{a}$ is not d-free over $\bar{c}$ (and furthermore, given $\bar{a}'$ and $\bar{b}'$ satisfying $\varphi$, find $\psi(\bar{c},\bar{u},\bar{v})$) as in the definition of d-freeness. Then for every computable copy $\mc{B}$ of $\mc{A}$, $R^\mc{B}$ is of c.e.\ degree.
\end{prop}
\begin{proof}
Let $\mc{A}$ and $R$ be as in the statement of the proposition. We will assume that $R$ is unary. The proof when $R$ is not unary uses exactly the same ideas, but is a little more complicated as we cannot ask whether individual elements are or are not in $R^\mc{B}$, but instead we must ask about tuples (including tuples which may include elements of $\bar{c}$). The translation of the proof to the case when $R$ is not unary requires no new ideas, and considering only unary relations will make the proof much easier to understand.

Let $\bar{c} \in \mc{A}$ be such that no $a \notin R$ is d-free over $\bar{c}$. We may omit any reference to $\bar{c}$ by assuming that it is in our language. Let $\mc{B}$ be a computable copy of $\mc{A}$. We must show that $R^\mc{B}$ is of c.e.\ degree. We will use $a$, $a'$, $\bar{b}$, $\bar{b}'$, etc. for elements of $\mc{A}$, and $d$, $e$, etc.\ for elements of $\mc{B}$.

We will begin by making some definitions, following which we will explain the intuitive idea behind the proof. Finally, we will define two c.e.\ sets $A$ and $B$ such that $R^\mc{B} = A - B \equiv_T A \oplus B$.

Since there are no d-free elements, for each $a \notin R$ there is an existential formula $\varphi_{a}(u,v_1,\ldots,v_{n_a})$ and a tuple $\bar{b}^a = b_1^a,\ldots,b_{n_a}^a$ such that
\[ \mc{A} \models \varphi_{a}(a,\bar{b}^a) \]
which witness the fact that $a$ is not d-free.

Now let $a \notin R$, $a' \in R$, and $\bar{b}'$ be such that
\[ \mc{A} \models \varphi_{a}(a',\bar{b}'). \]
By choice of $\varphi_{a}$ and $\bar{b}^a$, there is an existential formula $\psi(u,\bar{v})$ extending $\varphi_{a}(u,\bar{v})$ such that for all $a''$ and tuples $\bar{b}'' = (b_1'',\ldots,b_{n_a}'')$ with
\[ \mc{A} \models \psi(a'',\bar{b}''), \]
if $b_k \in R \Leftrightarrow b_k'' \in R$ for all $k$ then $\bar{a}'' \in R$. Note that $\psi$ depends only on $a$, $a'$, and $\bar{b}'$ (as $\bar{b}$ depends on $a$). Let $\psi_{a,a',\bar{b}'}$ be this formula $\psi$. We can find $\bar{b}^a$, $\varphi_a$, and $\psi_{a,a',\bar{b}}$ effectively using the hypothesis of the theorem.

Let $R$ be defined by $\alpha(u) \wedge \neg \beta(u)$ with $\alpha(u)$ and $\beta(u)$ being $\Sigma^\comp_1$ formulas with finitely many parameters (which we may assume are included in our language). We may assume that every solution of $\psi_{a,a',\bar{b}'}$ for every $a$, $a'$, and $\bar{b}'$ is a solution of $\alpha(u)$ by replacing each formula by its conjunction with $\alpha(u)$. Let $\alpha_s(u)$ and $\beta_s(u)$ be the finitary existential formulas which are the disjunctions of the disjuncts enumerated in $\alpha(u)$ and $\beta(u)$ respectively by stage $s$.

We can effectively enumerate the $\Sigma^\comp_1$ formulas which are true of elements of $\mc{B}$. At each stage $s$, we have a list of formulas which we have found to be true so far. This is the partial existential diagram of $\mc{B}$ at stage $s$, which we denote by $D_{\exists,s}(\mc{B})$. We say that an element $d \in \mc{B}$ appears to be in $R^{\mc{B}}$ at stage $s$ if one of the disjuncts of $\alpha_s(d)$ is in $D_{\exists,s}(\mc{B})$, and no disjunct of $\beta_s(d)$ is in $D_{\exists,s}(\mc{B})$. Otherwise, we say that $d$ appears to be in $\neg R^{\mc{B}}$ at stage $s$.

Now note that $\alpha$ defines a c.e.\ set $\tilde{A}$ in $\mc{B}$, and $\beta$ defines a c.e.\ set $\tilde{B}$, and $R^\mc{B} = \tilde{A} - \tilde{B}$. Then $\tilde{A} \oplus \tilde{B} \geq_T R^\mc{B}$. If in fact we had $\tilde{A} \oplus \tilde{B} \equiv_T R^\mc{B}$, then $R^\mc{B}$ would be of c.e.\ degree. However, $R^\mc{B}$ may not compute $\tilde{A} \oplus \tilde{B}$ because it may not be able to tell the difference between an element of $\tilde{B}$ and an element not in $\tilde{A}$. So we will come up with appropriate sets $A$ and $B$ where $R^\mc{B}$ \textit{can} tell the difference, i.e. $R^\mc{B} \geq_T A$ and $R^\mc{B} \geq_T B$. We can always assume that $B \subseteq A$ and hence $B = A - R^\mc{B}$, so it suffices to show that $R^\mc{B} \geq_T A$.

The set $A$ will consist of the of elements $d \in \mc{B}$ with the following property: for some stages $t > s$,
\begin{enumerate}
	\item $d$ appears to be in $R^{\mc{B}}$ at stage $s$ and at stage $t$, and
	\item for every $a \in \mc{A}$ with $a \notin R$ and $\bar{e}=(e_1,\ldots,e_{n_a}) \in \mc{B}$ we have found at stage $s$ with the property that $\varphi_a(d,\bar{e})$ is in $D_{\exists,s}(\mc{B})$, by stage $t$ we have found an $a' \in \mc{A}$ and $\bar{b}' \in \mc{A}$ with $\psi_{a,a',\bar{b}'}(d,\bar{e})$ in $D_{\exists,t}(\mc{B})$.
\end{enumerate}
When we say that at stage $s$ we have found $a \in \mc{A}$ and $\bar{e} \in \mc{B}$, we mean that $a$ comes from the first $s$ elements of $\mc{A}$ and $\bar{e}$ from the first $s$ elements of $\mc{B}$.

\begin{claim}
$R^\mc{B} \subseteq A$ and thus $R^\mc{B} = A - (A - R^\mc{B})$.
\end{claim}
\begin{proof}
Let $d$ be an element of $R^\mc{B}$. There is some stage $s$ at which $d$ appears to be in $R^\mc{B}$ and moreover, at any stage $t > s$, $d$ still appears to be in $R^\mc{B}$. For any $a \in \mc{A}$ and $\bar{e} \in \mc{B}$ with
\[ \mc{B} \models \varphi_a(d,\bar{e}), \]
there are $a' \in \mc{A}$ and $\bar{b}' \in \mc{A}$ with
\[ \mc{B} \models \psi_{a,a',\bar{b}'}(d,\bar{e}). \]
If $f: \mc{A} \to \mc{B}$ is an isomorphism, then $a' = f^{-1}(d)$ and $\bar{b}' = f^{-1}(\bar{e})$ are one possible choice. This suffices to show that $d \in A$, and so $R^\mc{B} \subseteq A$.

The second part of the claim follows immediately.
\end{proof}

\begin{claim}
$A$ and $A - R^\mc{B}$ are c.e.
\end{claim}
\begin{proof}
$A$ is c.e.\ because to check whether $d \in A$, we search for $s$ and $t$ satisfying the two conditions in the definition of $A$. For a given $s$ and $t$, these two conditions are computable to check.

$A - R^\mc{B}$ is c.e.\ because it is just equal to the elements of $A$ which satisfy the formula $\beta$, as every element of $A$ satisfies $\alpha$.
\end{proof}

\begin{claim}
$R^\mc{B} \geq_T A$.
\end{claim}
\begin{proof}
Given $d$, we want to check (using $R^\mc{B}$ as an oracle) whether $d \in A$. First ask $R^\mc{B}$ whether $d \in R^\mc{B}$. If the answer is yes, then we must have $d \in A$.

Otherwise, $d \notin R^\mc{B}$. Now, since $A$ is c.e., it suffices to show that checking whether $d$ is in its complement is also c.e. Suppose that at some stage $s$, $d$ has not yet been enumerated into $A$, and we find $a \in \mc{A}$ and $\bar{e} \in \mc{B}$ such that:
\begin{enumerate}
	\item $d$ does not appear to be in $R^\mc{B}$ at stage $s$,
	\item $a$ and $\bar{e}$ are from among the first $s$ elements of $\mc{A}$ and $\mc{B}$ respectively,
	\item $\varphi_a(d,\bar{e})$ is in $D_{\exists,s}(\mc{B})$, and
	\item the oracle $R^\mc{B}$ tells us that $e_j \in R^\mc{B}$ if and only if $b_j^a \in R$.
\end{enumerate}
We claim that $d \notin A$. Suppose to the contrary that $d \in A$. Then there is a stage $s'$ at which $d$ appears to be in $R^\mc{B}$, and a stage $t'$ at which $d$ enters $A$ (i.e., $s'$ and $t'$ are the $s$ and $t$ from the definition of $A$). Now at these two stages $s'$ and $t'$, $d$ appears to be in $R^\mc{B}$. Moreover, we know that at stage $s$, $d$ has not yet entered $A$, and so $s < t'$. Then we must have $s < s'$ since between stages $s'$ and $t'$, $d$ appears to be in $R^\mc{B}$, and this is not the case at stage $s$. Then there must be elements $a',\bar{b}' \in \mc{A}$ with $\psi_{a,a',\bar{b}'}(d,\bar{e})$ in $D_{\exists,t'} (\mc{B})$. This is a contradiction, since by choice of $\psi_{a,a'\bar{b}'}$ we cannot have $d \notin R^\mc{B}$ and $e_j \in R^\mc{B} \Leftrightarrow b_j \in R$.

Now, since $d \notin R^\mc{B}$, there exists some $a$ and $\bar{e}$ such that
\[ \mc{B} \models \varphi_a(d,\bar{e}) \]
and $e_j \in R^\mc{B}$ if and only if $b^a_j \in R$. If $f : \mc{A} \to \mc{B}$ is an isomorphism, $a = f^{-1}(d)$ and $\bar{e} = f(\bar{b}^a)$ are one possible choice for $a$ and $\bar{e}$. Then one of two things must happen first---either $d$ is enumerated into $A$, or we find $a$ and $\bar{e}$ as above at some stage $s$ when $a$ does not appear to be in $R^\mc{B}$ and hence $d$ is not in $A$. Since finding such $a$ and $\bar{e}$ is a computable search, the complement of $A$ is c.e.\ in $R^\mc{B}$.
\end{proof}

Now since $R^\mc{B} \geq_T A$, $R^\mc{B} \geq_T A - R^\mc{B}$. Thus $R^\mc{B} \geq_T A \oplus (A - R^\mc{B})$. It is trivial to see that $A \oplus (A - R^\mc{B}) \geq R^\mc{B}$. Finally, $A$ and $A - R^\mc{B}$ are c.e.\ sets, and so their join is of c.e.\ degree. Thus $R^\mc{B}$ is of c.e.\ degree.
\end{proof}

For the second direction, we do not have to assume that $R$ is formally d.c.e. We will use this direction in Section 4.2 and in Chapter 5, and the style of the proof will be a model for Propositions \ref{M-not-A} and \ref{prop:incomp-2}, and Theorem \ref{thm:weird-relation-on-omega} later on.

\begin{prop}\label{ce-condition}
Let $\mc{A}$ be a computable structure and $R$ a computable relation on $\mc{A}$. Suppose that for each $\bar{c}$, there is $\bar{a} \notin R$ d-free over $\bar{c}$. Also, suppose that for each tuple $\bar{c}$, we can effectively find a tuple $\bar{a}$ which is d-free over $\bar{c}$, and moreover we can find $\bar{a}'$, $\bar{b}'$, and $\bar{a}''$ as in the definition of d-freeness. Then there is a computable copy $\mc{B}$ of $\mc{A}$ such that $R^\mc{B}$ is not of c.e.\ degree.
\begin{proof}
We will assume that $R$ is a unary relation, and hence that for each $\bar{c}$, there is $a \notin R$ d-free over $\bar{c}$. We begin by describing a few conventions. We denote by $X[0, \ldots, n]$ the first $n+1$ elements of $X$. We say that a computation $\Phi^X = Y$ has use $u$ if it takes fewer than $u$ steps and it only uses the oracle $X[0, \ldots, u]$.

The proof will be to construct a computable copy $\mc{B}$ of $\mc{A}$ which diagonalizes against every possible Turing equivalence with a c.e.\ set. The construction will be a finite-injury priority construction. We use the d-free elements to essentially run a standard proof that there are d.c.e.\ degrees which are not c.e.\ degrees.

We will construct $\mc{B}$ with domain $\omega$ by giving at each stage $s$ a tentative finite isomorphism $F_s:\omega \to \mc{A}$. In the limit, we get $F:\omega \to \mc{A}$ a bijection, and $\mc{B}$ is the pullback along $F$ of the structure on $\mc{A}$. We will maintain values $n_{e,i,j}[s]$, $u_{e,i,j}[s]$, $t_{e,i,j}[s]$, and $v_{e,i,j}[s]$ which reference computations that have converged. See Figure \ref{fig:prop4-7} for a visual representation of what these values mean.

We will meet the following requirements:
\begin{description}
	\item[$\mathcal{R}_{e,i,j}$] If $\Phi_i$ and $\Phi_j$ are total, then either $R^\mc{B} \neq \Phi_i^{W_e}$ or $W_e \neq \Phi_j^{R^\mc{B}}$.
	\item[$\mathcal{S}_{i}$] The $i$th element of $\mc{A}$ is in the image of $F$.
\end{description}
Put a priority ordering on these requirements.

At each stage, each requirement $R_{e,i,j}$ will be in one of four states: \textsc{initialized}, \textsc{waiting-for-computation}, \textsc{waiting-for-change}, or \textsc{diagonalized}. A requirement will move through these four stages in order, and be satisfied when it enters the state \textsc{diagonalized}. If it is injured, a requirement will return to the state \textsc{initialized}.

\begin{figure}[b]
	\includegraphics{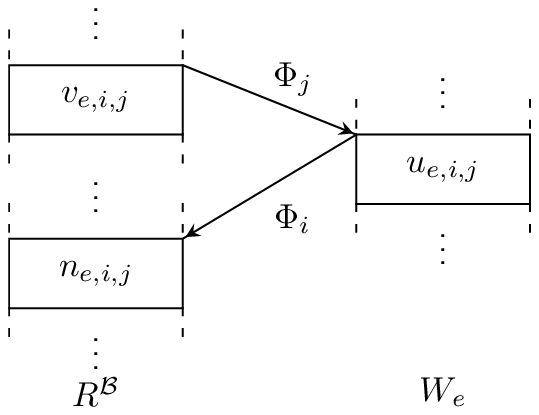}
	\caption{The values associated to a requirement for Proposition \ref{ce-condition}. An arrow shows a computation converging. The computations use an oracle and compute some initial segment of their target. The tail of the arrow shows the use of the computation, and the head shows the length. So, for example, we will have $R^\mc{B}[0,\ldots,n_{e,i,j}] = \Phi_i^{W_e}[0,\ldots,n_{e,i,j}]$ with use $u_{e,i,j}$.}
	\label{fig:prop4-7}
\end{figure}

We are now ready to describe the construction.

\vspace*{10pt}
\noindent\textit{Construction.}
\vspace*{10pt}

At stage 0, let $F_s = \varnothing$ and for each $e,i,j$, let $n_{e,i,j}[0]$, $u_{e,i,j}[0]$, $t_{e,i,j}[0]$, and $v_{e,i,j}[0]$ be $0$ (i.e., undefined).

At a stage $s+1$, let $F_s:\{0,\ldots,\xi_s\} \to \mc{A}$ be the partial isomorphism determined in the previous stage, and let $\mc{B}_s$ be the finite part of the diagram of $\mc{B}$ which has been determined so far. We have an approximation $R^{\mc{B}_s}$ to $R^\mc{B}$ which we get by taking $k \in R^{\mc{B}_s}$ if $F_s(k) \in R$.

We will deal with a single requirement---the highest priority requirement which requires attention at stage $s+1$. We say that a requirement $\mc{S}_i$ \textit{requires attention at stage $s+1$} if the $i$th element of $\mc{A}$ is not in the image of $F_s$. If $\mc{S}_i$ is the highest priority requirement which requires attention, then let $a$ be the $i$th element of $\mc{A}$. Let $F_{s+1}$ extend $F_s$ with $F_{s+1}(\xi_s+1)=a$. Set $\xi_{s+1} = \xi_s + 1$. Injure each requirement of higher priority.

The conditions for a requirement $\mc{R}_{e,i,j}$ to \textit{require attention at stage $s+1$} depends on the state of the requirement. Below, we will list for each possible state of $\mc{R}_{e,i,j}$, the conditions for $R_{e,i,j}$ to require attention, and the action that the requirement takes if it is the highest priority requirement that requires attention.

\begin{description}[font=\sc]

\item[{Initialized}]
The requirement has been initialized, so $n_{e,i,j}[0]$, $u_{e,i,j}[0]$, $t_{e,i,j}[0]$, and $v_{e,i,j}[0]$ are all $0$.

\medskip{}

\begin{description}[font=\it,labelindent=-15pt,leftmargin = 0pt]
	\item[Requires attention] The requirement always requires attention.
	\item[Action] Choose a new element $a$ of $\mc{A}$ which is d-free over the image of $F_{s}$. Note that $a \notin R$. Set $F_{s+1}(\xi_{s} + 1) = a$, $n_{e,i,j}[s+1] = \xi_{s}$, and $\xi_{s+1} = \xi_{s} + 1$.
\end{description}

\item[{Waiting-for-computation}]
We have set $n_{e,i,j}$, so we need to wait for the computations \eqref{eq1:R-W} and \eqref{eq1:W-R} below to converge. Once they do, we can use the fact that $F(n_{e,i,j}) = a \notin R$ was chosen to be d-free to modify $F$ so that $F(n_{e,i,j}) \in R$ to break the computation \eqref{eq1:R-W}.

\medskip{}

\begin{description}[font=\it,labelindent=-15pt,leftmargin = 0pt]
	\item[Requires attention] The requirement requires attention if there is a computation
\begin{equation}\label{eq1:R-W} R^{\mc{B}_s}[0, \ldots, {n_{e,i,j}}[s]] = \Phi_{i,s}^{W_{e,s}}[0, \ldots, {n_{e,i,j}}[s]] \end{equation}
with use $u$, and a computation
\begin{equation}\label{eq1:W-R} W_{e,s}[0, \ldots, u] = \Phi_{j,s}^{R^{\mc{B}_s}}[0, \ldots, u] \end{equation}
with use $v$.

	\item[Action] Set $u_{e,i,j}[s+1] = u$, $t_{e,i,j}[s+1] = s$, and $v_{e,i,j}[s+1] = v$. Let
\begin{align*}
\bar{c} &= (F_{s}(0),\ldots,F_{s}(n_{e,i,j}[s] - 1)) \\
a &= F_{s}(n_{e,i,j}[s]) \\
\bar{b} &= (F_{s}(n_{e,i,j}[s]+1),\ldots,F_{s}(\xi_s)).
\end{align*}
Write $\bar{c} = (c_0,\ldots,c_{n_{e,i,j}[s]-1})$ and $\bar{b} = (b_{n_{e,i,j}[s]+1},\ldots,b_{\xi_s})$. We will have ensured during the construction that $a$ is d-free over $\bar{c}$. So we can find $a' \in R$ and $\bar{b}'$ such that $\bar{c},a',\bar{b'}$ satisfies the same quantifier-free formulas determined so far in $\mc{B}_s$ as $\bar{c},a,\bar{b}$, and so that for any further existential formula $\psi(\bar{c},u,\bar{v})$ true of $\bar{c},a',\bar{b}'$, there are $a'' \notin R$ and $\bar{b}''(b_{n_{e,i,j}[s]+1}'',\ldots,b_{\xi_s}'')$ satisfying $\psi$ and with $b_k'' \in R$ if and only if $b_k \in R$. Define
\begin{align*}
F_{s+1}(k) &= c_k & &\text{for } 0 \leq k < n_{e,i,j}[s] \\
F_{s+1}(n_{e,i,j}[s]) &= a' & &\\
F_{s+1}(k) &= b_k' & &\text{for } n_{e,i,j}[s] < k \leq \xi_s.
\end{align*}
Set the state of this requirement to \textsc{waiting-for-change}. Each requirement of lower priority has been injured. Reset the state of all such requirements to \textsc{initialized} and set all of the corresponding values to $0$.
\end{description}

\item[{Waiting-for-change}]
In the previous state, $F$ was modified to break the computation \eqref{eq1:R-W}. If we are to have $R^{\mc{B}} = \Phi_i^{W_e}$, then $W_e$ must change below the use of that computation. In this state, we wait for this to happen, and then use the fact that $a$ was chosen to be d-free in state \textsc{initialized} to return $R^{\mc{B}}$ to the way it was for the computation \eqref{eq1:W-R}.

\medskip{}

\begin{description}[font=\it,labelindent=-15pt,leftmargin = 0pt]
	\item[Requires attention] Let $u = u_{e,i,j}[s]$, $v = v_{e,i,j}[s]$ and $t = t_{e,i,j}[s]$. The requirement requires attention if
\[ W_{e,s}[0, \ldots, u] \neq W_{e,t}[0, \ldots, u]. \]

	\item[Action] Let
\begin{align*}
\bar{c} &= (F_{s}(0),\ldots,F_{s}(n_{e,i,j}[s] - 1)) \\
a' &= F_{s}(n_{e,i,j}[s]) \\
\bar{b}' &= (F_{s}(n_{e,i,j}[s]+1),\ldots,F_{s}(v)) \\
\bar{d}' &= (F_{s}(v+1),\ldots,F_{s}(\xi_s))
\end{align*}
As before, write $\bar{c} = (c_0,\ldots,c_{n_{e,i,j}[s]-1})$, $\bar{b}' = (b_{n_{e,i,j}[s]+1}',\ldots,b_{v}')$, and $\bar{d}'=(d_{v+1}',\ldots,d_{\xi_s}')$. Now $a'$ was chosen in state \textsc{waiting-for-computation}. So we can choose $a'' \notin R$, $\bar{b}''=(b_{n_{e,i,j}[s]+1}'',\ldots,b_{v}'')$, and $\bar{d}''=(d_{v+1}'',\ldots,d_{\xi_s}'')$ such that $\bar{c}a''\bar{b}''\bar{d}''$ satisfies any formula determined by $\mc{B}_s$ to be satisfied by $\bar{c}a'\bar{b}'\bar{d}'$, and moreover $b_k'' \in R$ if and only if $F_t(k) \in R$ (note that $F_t(k)$ is the value $b_k$ from state \textsc{waiting-for-computation}). Define
\begin{align*}
F_{s+1}(k) &= c_k &&\text{for } 0 \leq k < n_{e,i,j}[s] \\
F_{s+1}(n_{e,i,j}[s]) &= a'' && \\
F_{s+1}(k) &= b_k'' &&\text{for } n_{e,i,j}[s] < k \leq v \\
F_{s+1}(k) &= d_k'' &&\text{for } v < k \leq \xi_s.
\end{align*}
Then we will have
\[ R^{\mc{B}_{s+1}}[0, \ldots, v] = R^{\mc{B}_{t}}[0, \ldots, v]. \]
So
\[ \Phi_j^{R^\mc{B}_{s+1}}[0, \ldots, u] = \Phi_j^{R^\mc{B}_{t}}[0, \ldots, u] = W_{e,t}[0, \ldots, u] \neq W_{e,s+1}[0, \ldots, u] \]
since the use of this computation at stage $t$ was $v$. Set the state of this requirement to \textsc{diagonalized}. Each requirement of lower priority has been injured. Reset the state of all such requirements to \textsc{initialized} and set all of the corresponding values to $\varnothing$.
\end{description}

\item[{Diagonalized}]
In this state, $R^{\mc{B}}$ is the same as it was under the use $v$ in the computation \eqref{eq1:W-R} from state \textsc{waiting-for-computation}. By \eqref{eq1:W-R}, if we are to have $W_e = \Phi_j^{R^\mc{B}}$, then $W_e$ restricted to the elements $0,\ldots,u$ must be the same as it was then. But this cannot happen, because some such element entered $W_e$ during the previous state. So we have satisfied the requirement $\mc{R}_{e,i,j}$.
\medskip{}
\begin{description}[font=\it,labelindent=-15pt,leftmargin = 0pt]
	\item[Requires attention] The requirement never requires attention.
	\item[Action] None.
\end{description}
\end{description}

\noindent Set $\mc{B}_{s+1}$ to be the atomic and negated atomic formulas true of $0,\ldots,\xi_{s+1}$ with G\"odel number at most $s$.

\vspace*{10pt}
\noindent\textit{End construction.}
\vspace*{10pt}

Note that at any stage $s$, the $n_{e,i,j}$ are ordered by the priority of the corresponding requirements. This is because if a requirement is injured, each lower priority requirement is injured at the same time, and then new values of $n_{e,i,j}$ are defined in order of priority. Moreover, if $\mc{R}_{e,i,j}$ is of higher priority than $\mc{R}_{e',i',j'}$ and $v_{e,i,j}$ is defined, then $n_{e,i,j} < v_{e,i,j} < n_{e',i',j'}$. 

If $\mc{R}_{e,i,j}$ is never injured after some stage $s$, then it only acts at most three times---once in each of the stages \textsc{initialized}, \textsc{waiting-for-computation}, and \textsc{waiting-for-change}, in that order---and it never moves backwards through the stages. A requirement $\mc{S}_{i}$ acts only once if it is not injured. So every requirement is injured only finitely many times.

It remains to show that every requirement is eventually satisfied. Suppose to the contrary that some requirement is not satisfied. There must be some least such requirement. First, suppose that it is a requirement $\mc{S}_i$. Then there is a stage $s$ after which each lower requirement never acts. Then at the next stage, $\mc{S}_i$ acts, and is never again injured. So $\mc{S}_i$ is satisfied.

Now suppose that $\mc{R}_{e,i,j}$ is the least requirement which is not satisfied, and let $s$ be a stage after which each lower requirement never acts. So $\mc{R}_{e,i,j}$ is never injured after the stage $s$. Also, since $\mc{R}_{e,i,j}$ is not satisfied, we have
\[ W_e = \Phi_j^{R^\mc{B}} \text{ and } R^\mc{B} = \Phi_i^{W_e}. \]

If $R_{e,i,j}$ was in state \textsc{initialized} at stage $s$, then at a later stage, $n_{e,i,j}$ is defined and the requirement enters stage \textsc{waiting-for-computation}. Eventually, at a stage $t$, the following computations must converge:
\begin{align}
\label{R-W} R^{\mc{B}_t}[0, \ldots, n_{e,i,j}] &= \Phi_{i,t}^{W_{e,t}}[0, \ldots, n_{e,i,j}] && \textrm{with use $u$}\\
\label{W-R} W_{e,t}[0, \ldots, u] &= \Phi_{j,t}^{R^{\mc{B}_t}}[0, \ldots, u]  && \textrm{with use $v$}.
\end{align}
Then $\mc{R}_{e,i,j}$ requires attention at stage $t+1$. We modify $F$ to have $F(n_{e,i,j}) \in R$, breaking computation \eqref{R-W}. Requirement $\mc{R}_{e,i,j}$ also moves to state \textsc{waiting-for-change}.

Since $R^{\mc{B}} = \Phi_i^{W_e}$, eventually at some stage $t'$, $W_e$ must change below the use $u$ of the computation \eqref{R-W}. Then $\mc{R}_{e,i,j}$ requires attention at stage $t'+1$. We modify $F$ by moving $F(n_{e,i,j})$ back to $\neg R$ and ensuring that
\[ R^{\mc{B}_{t'+1}}[0, \ldots, v] = R^{\mc{B}_t}[0, \ldots, v]. \]
But for every stage $t'' > t'$ we have $W_{e,t''}[0, \ldots, u] \neq W_{e,t}[0, \ldots, u]$ since $W_e$ is a c.e.\ set and \[ W_{e,t'}[0,\ldots,u] \neq W_{e,t}[0, \ldots, u].\] This, together with computation \eqref{W-R} contradicts the assumption that $R^{\mc{B}} = \Phi_i^{W_e}$. So every requirement is satisfied.
\end{proof}
\end{prop}

In Proposition \ref{ce-degree-condition}, we showed that a condition about the non-existence of free elements is equivalent to a condition on the possible degrees of the relation $R^\mc{B}$ in computable copies $\mc{B}$ of $\mc{A}$. In particular, we showed that the relation $R^\mc{B}$ is Turing equivalent to the join of c.e.\ sets. In, for example, the theorems of Ash and Nerode \cite {AshNerode81} and Barker \cite{Barker88} there are two parts: first, that a condition on the existence of free tuples is equivalent to a condition on the possible computability-theoretic properties of $R^\mc{B}$; and second, that a condition on the existence of free tuples is equivalent to a syntactic condition on the relation $R$. In Propositions \ref{ce-degree-condition} and \ref{ce-condition}, we are missing this second syntactic part.

We might hope that there is a syntactic condition which is equivalent (under some effectiveness conditions) to being intrinsically of c.e.\ degree. For example, one candidate (and certainly a sufficient condition) would be that there are formally $\Sigma^0_1$ sets $A \supseteq R$ and $B = A - R$ such that $A$ is formally $\Sigma^0_1$ and $\Pi^0_1$ relative to $R$.

In the proof of Proposition \ref{ce-degree-condition}, we found c.e.\ sets $A$ and $B$ in our particular copy $\mc{B}$ of $\mc{A}$ such that $R^\mc{B} = A - B$ and $R^\mc{B} \geq_T A$. These c.e.\ sets were not necessarily definable by $\Sigma^0_1$ formulas, but instead depended on the enumeration of $\mc{B}$. When $R$ was defined by $\alpha(x) \wedge \neg \beta(x)$, whether or not an element $a \notin R^\mc{B}$ which satisfied both $\alpha(x)$ and $\beta(x)$ was in $A$ depended on the order in which we discovered certain facts in $\mc{B}$.

The following example should be taken as (very strong) evidence that we cannot find an appropriate syntactic condition.

\begin{exam}
Consider a structure as in Example \ref{simple-example}, except that the connected components are different. There are infinitely many copies of each of the following five connected components:
\[
\xymatrix@C=1.5em@R=1.5em{
&0&&0&&&1&&&0&&1\\0\ar[ru]\ar[r]\ar[rd]\ar[rdd]&1&1\ar[ru]\ar[r]\ar[rd]\ar[rdd]&1\ar[r]&0&2\ar[ru]\ar[r]\ar[rd]\ar[rdd]&1\ar[r]&0&2\ar[ru]\ar[r]\ar[rd]\ar[rdd]&1&2\ar[ru]\ar[r]\ar[rd]\ar[rdd]&1\\&1&&1&&&1&&&1&&1\\&\vdots&&\vdots&&&\vdots&&&\vdots&&\vdots
}
\]
\noindent The formally d.c.e.\ relation $R$ consists of the nodes which are labeled $1$. Note that the elements at the center of their connected component are definable in both a $\Sigma^0_1$ and a $\Pi^0_1$ way (in a $\Sigma^0_1$ way as they are the only nodes of degree at least three, and in a $\Pi^0_1$ way because they are the unique such nodes in their connected components, and so they are the only node which are not connected to some other node of degree at least three). In particular, given a connected component, we can compute the center.

\begin{claim}
$R$ is relatively intrinsically of c.e.\ degree. 
\end{claim}
\begin{proof}
We will use Proposition \ref{ce-degree-condition}.\footnote{This will also give us our first non-trivial application of Proposition \ref{ce-degree-condition}.} First, we claim that no tuple $\bar{a}$ is d-free over $\varnothing$. Since $R$ is unary, it suffices to show that no single element $a$ is d-free over $\varnothing$. If $a$ is not in the center of its connected component, then there is an existential formula $\varphi(u)$ which witnesses this. If $a' \in R$ also satisfies $\varphi(u)$, then $\varphi(u) \wedge \psi(u)$ is true of $a'$ where $\psi(u)$ is the existential formula which says that $u$ is labeled ``1'' or ``2''. Every solution of $\varphi(u) \wedge \psi(u)$ is in $R$. So $a$ is not d-free over $\varnothing$. Now suppose that $a \notin R$ is the center element of its connected component. If $a$ is labeled ``2'', then it is not d-free over $\varnothing$. So suppose that $a$ is labeled ``0''. There is a $b \notin R$ which is connected to $a$, and an existential formula $\varphi(u,v)$ which says that $u$ is the center element of its connected component and $v$ is connected to $u$. Now if $a' \in R$ satisfies $(\exists v) \varphi(u,v)$, then $a'$ also satisfies $\psi(u)$ which says that there is a chain of length two leading off of $a'$. Now suppose that $a'' \notin R$ satisfies $\psi(u) \wedge (\exists v) \varphi(u,v)$. Let $b''$ be such that we have $\psi(a'',b'')$. Now, since $a''$ satisfies $\psi(u)$, any such $b''$ must be labeled ``1'' and hence be in $R$. But we had $b \notin R$. Thus $a$ cannot have been d-free. So there is no tuple $\bar{a}$ which is d-free over $\varnothing$. The effectiveness conditions of Proposition \ref{ce-degree-condition} are immediate, because everything we did above is computable. Moreover, this relativizes. Thus $R$ is relatively intrinsically of c.e.\ degree.
\end{proof}

Now we argue that there is no syntactical fact about $R$ which explains this. Such a syntactical fact should say that $R$ is intercomputable with a join of formally $\Sigma_1$ sets. So it should say something like: there is a join $S$ of formally $\Sigma_1$ sets which is formally $\Pi_1(R)$, and $R$ is formally $\Sigma_1(S)$ and $\Pi_1(S)$. We will show that there are no non-trivial sets which are formally $\Sigma^0_1$ and $\Pi^0_1(R)$.

Let $A$, $B$, $C$, $D$, and $E$ be the sets of points which are at the center of connected components of the first, second, third, fourth and fifth types respectively. It is not hard to check that the formally $\Sigma^0_1$ sets are:
\begin{enumerate}
	\item $A \cup B \cup C \cup D \cup E$,
	\item $B \cup C \cup D \cup E$,
	\item $B \cup C$,
	\item $C \cup D \cup E$, and
	\item $C$.
\end{enumerate}
The formally $\Sigma^0_1(R)$ sets are those above, and in addition:
\begin{enumerate}
	\item $B$,
	\item $A \cup D$,
	\item $A \cup B \cup D$,
	\item $A \cup C \cup D \cup E$.
\end{enumerate}
Every $\Sigma_1(R)$ set containing $E$ also contains $C$ and $D$. If a set $X$ and its complement $\bar{X}$ are both $\Sigma_1(R)$ and one of them is $\Sigma_1$, then one of them must contain $E$, and hence $C$ and $D$. Then the other must be contained within $A \cup B$. The only possibilities for $X$ and $\bar{X}$ are $B$ and $A \cup C \cup D \cup E$, but neither of these are formally $\Sigma_1$.

So there are no formally $\Sigma_1$ sets which are also formally $\Pi_1(R)$. Thus it seems impossible to have a syntactic condition which is necessary and sufficient for a relation to be intrinsically of c.e.\ degree.
\end{exam}

\section{Incomparable Degree Spectra of D.C.E.\ Degrees}

In this section our goal is to prove the following theorem, which will yield Theorem \ref{thm-incomparable-dce}:

\begin{thm}\label{thm:incomp-dce}
There are structures $\mc{A}$ and $\mc{M}$ with relations $R$ and $S$ respectively which are relatively intrinsically d.c.e.\ such that the degree spectra of $R$ and of $S$ are incomparable even relative to every cone.
\end{thm}

This is a surprising result which is interesting because it says that there is no ``fullness'' result for d.c.e.\ degrees over the c.e.\ degrees, that is, no result that says that any degree spectrum on a cone which strictly contains the c.e.\ degrees contains all of the c.e.\ degrees. Also, it implies that the partial ordering of degree spectra on a cone is not a linear order. The two structures $\mc{A}$ and $\mc{M}$, and the relations $R$ and $S$, are as follows.

\begin{exam}\label{exam1}
Let $\mc{A}$ be two-sorted with sorts $A_1$ and $A_2$. The first sort $A_1$ will be the tree $\omega ^ {< \omega}$ with the relation ``is a child of'' and the root node distinguished. The second sort $A_2$ will be an infinite set with no relations. There will be a single binary relation $U$ of type $A_1 \times A_2$. Every element of $A_2$ will be related by $U$ to exactly one element of $A_1$, and each element of $A_1$ will be related to zero, one, or two elements of $A_2$. The only elements of $A_1$ related to no elements of $A_2$ are those of the form $0^n = \underbrace{0\ldots0}_{n}$. Any other element of the form $\sigma \concat a$ is related to one element of $A_2$ if $a$ is odd, and to two if $a$ is even. The structure $\mc{A}$ consists of these two sorts $A_1$ and $A_2$, the ``is a child of'' relation, the root of the tree, and $U$.

We say that an element of $A_1$ is of \textit{type $\langle n \rangle$} (so possibly of \textit{type $\langle 0 \rangle$}, \textit{type $\langle 1 \rangle$}, or \textit{type $\langle 2 \rangle$}) if it is related by $U$ to exactly $n$ elements of $A_2$. The relation $R$ on $\mc{A}$ is the set of elements of $A_1$ which are related by $U$ to exactly one element of $A_2$, that is, the elements of type $\langle 1 \rangle$.
\end{exam}


\begin{exam}\label{exam2}
Let $\mc{M}$ be a three-sorted model with sorts $M_1$, $M_2$, and $M_3$. The sort $M_1$ will be the tree $\omega ^ {< \omega}$, however, instead of defining the tree with the relation ``is a child of,'' the tree will be given as a partial order. We will have a relation $V$ of type $M_1 \times M_2$ which is defined in the same way as $U$, except with $M_1$ replacing $A_1$ and $M_2$ replacing $A_2$. There will be another relation $W$ on $M_1 \times M_3$ such that each element of $M_3$ is related by $W$ to an element of $M_1$, and each element of $M_1$ is related to either no elements or one element of $M_3$. An element of $M_1$ will be related (via $W$) to an element of $M_3$ exactly if its last entry is odd, but it is not of the form $0^n \concat 1$. 

Once again, we give elements of $M_1$ a ``type''. We say that an element of $M_1$ is of \textit{type $\langle n,m \rangle$} if it is related by $V$ to exactly $n$ elements of $M_2$ and by $W$ to exactly $m$ elements of $M_3$. The possible types of elements are $\langle 0,0 \rangle$, $\langle 1,0 \rangle$, $\langle 1,1 \rangle$, and $\langle 2,0 \rangle$. The relation $S$ on $\mc{M}$ will be defined in the same way as $R$, but again $A_1$ is replaced by $M_1$ and $A_2$ by $M_2$; so $S$ consists of the elements of types $\langle 1,0 \rangle$ and $\langle 1,1 \rangle$. Note that every element of $\mc{B}$ which is in $S$, except for those of the form $0^n 1$, is of type $\langle 1,1 \rangle$, and hence satisfies an existential formula which is only satisfied by elements of $S$.
\end{exam}

Both examples have d-free elements over any tuple $\bar{c}$; these elements are of the form $0^n$ for $n$ large enough that no children of $0^n$ appear in $\bar{c}$ (i.e., the elements of type $\langle 0 \rangle$ in $\mc{A}$ or type $\langle 0,0 \rangle$ in $\mc{M}$). In either structure, any existential formula (over $\bar{c}$) satisfied by $0^n$ is also satisfied by $0^{n-1} 1$, and any existential formula satisfied by $0^{n-1} 1$ is also satisfied by $0^{n-1} b$ for $b$ even. Moreover, the relation $R$ (or $S$) on the subtrees of $0^n$ and $0^{n-1} b$ is the same under the natural identification. Both structures satisfy the effectiveness condition from Proposition \ref{ce-condition}, so for all degrees \textbf{d}, $\dgSp(\mc{A},R)_{\leq \textbf{d}}$ strictly contains $\Sigma^0_1(\textbf{d})$.

Note that in $\mc{A}$, there is an existential formula $\varphi(u)$ which says that $u$ is of type $\langle 2 \rangle$, and an existential formula $\psi(u)$ which says that $u$ is of type $\langle 1 \rangle$ or of type $\langle 2 \rangle$ (i.e., of type $\langle n \rangle$ with $n \geq 1$). Similarly, in $\mc{M}$, for each $n_0$ and $m_0$ there are existential formulas which say that an element $u$ is of type $\langle n,m \rangle$ with $n \geq n_0$ and $m \geq m_0$.

We begin by proving in Proposition \ref{M-not-A} that there is a Turing degree in the (unrelativized) degree spectrum of $S$ which is not in the degree spectrum of $R$ before proving in Proposition \ref{prop:incomp-2} that there is a Turing degree in the degree spectrum of $R$ which is not in the degree spectrum of $S$. After each proposition, we give the relativized version.

\begin{prop}\label{M-not-A}
There is a computable copy $\mc{N}$ of $\mc{M}$ such that no computable copy $\mc{B}$ of $\mc{A}$ has $S^{\mc{N}} \equiv_T R^{\mc{B}}$.
\end{prop}

\begin{proof}
The proof will be very similar to that of Proposition \ref{ce-condition}, though it will not be as easy to diagonalize as both structures have d-free elements. We will construct $\mc{N}$ with domain $\omega$ by giving at each stage a tentative finite isomorphism $F_s:\omega \to \mc{M}$. $F = \lim F_s$ will be a bijection, giving $\mc{N}$ as an isomorphic copy.

We need to diagonalize against computable copies of $\mc{A}$. Given a computable function $\Phi_e$, we can try to interpret $\Phi_e$ as giving the diagram of a computable structure $\mc{B}_e$ isomorphic to $\mc{A}$. At each stage, we get a finite substructure $\mc{B}_{e,s}$ which is isomorphism to a finite substructure of $\mc{A}$ by running $\Phi_{e}$ up to stage $s$, and letting $\mc{B}_{e,s}$ be the greatest initial segment on which all of the relations are completely determined and which is isomorphic to a finite substructure of $\mc{A}$ (because $\mc{A}$ is relatively simple, this can be checked computably). Let $\mc{B}_e$ be the union of the $\mc{B}_{e,s}$. If $\Phi_e$ is total and gives the diagram of a structure isomorphic to $\mc{A}$, then $\mc{B}_e$ is that structure. Otherwise, $\mc{B}_e$ will be some other, possibly finite, structure. For elements of $\mc{B}_{e,s}$, we also have an approximation of their type in $\mc{B}_e$, by looking at how many elements of the second sort they are connected to in $\mc{B}_{e,s}$. By further reducing the domain of $\mc{B}_{e,s}$, we may assume that $\mc{B}_{e,s}$ has the following property since $\mc{A}$ does: all of the elements of $\mc{B}_{e,s}$ which are of type $\langle 0 \rangle$ in $\mc{B}_{e,s}$ are linearly ordered.

We will meet the following requirements:
\begin{description}
	\item[$\mc{R}_{e,i,j}$]If $\mc{B}_e$ is isomorphic to $\mc{A}$, and $\Phi_i^{R^{\mc{B}_e}}$ and $\Phi_j^{S^{\mc{N}}}$ are total, then either $S^{\mc{N}} \neq \Phi_i^{R^{\mc{B}_e}}$ or $R^{\mc{B}_e} \neq \Phi_j^{S^{\mc{N}}}$.
	\item[$\mc{S}_{i}$] The $i$th element of $\mc{M}$ is in the image of $F$.
\end{description}

Note that the d-free elements of both structures are linearly ordered. Suppose that in $\mc{A}$, $p \notin R$ is d-free (so $p=0^\ell$ for some $\ell$), and $q \notin R$ is d-free over $p$ and in the subtree below $p$. Then, using the fact that $p$ is d-free, we can replace it by $p' \in R$ (replacing $q$ by $q'$) while maintaining any existential formula, and then we can replace $p'$ by $p'' \notin R$ (and $q'$ by $q''$). However, $q''$ will no longer be d-free, because it will not be of the form $0^k$. The same is true in $\mc{M}$. This is what we will exploit for both this proof and the proof of the next proposition.

\begin{figure}[htb]
	\includegraphics{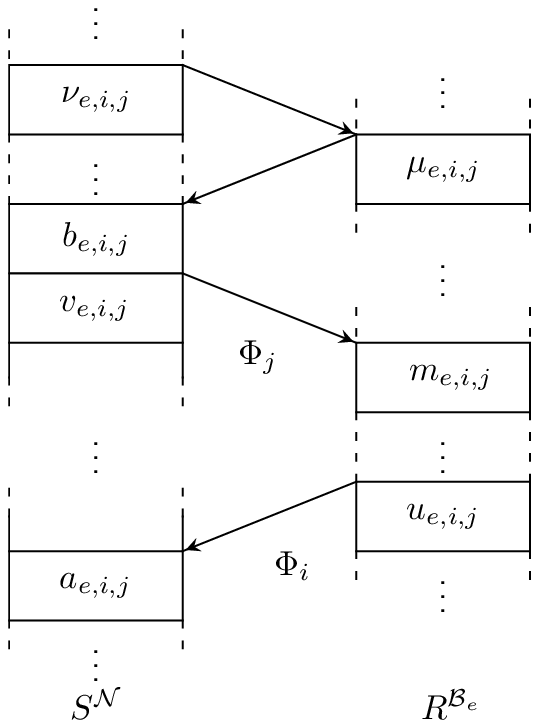}
	\caption{The values associated to a requirement for Proposition \ref{M-not-A}. An arrow shows a computation converging. The computations use an oracle and compute some initial segment of their target. The tail of the arrow shows the use of the computation, and the head shows the length.}
	\label{fig:prop4-11}
\end{figure}

The way we will meet $\mc{R}_{e,i,j}$ will be to put a d-free element $x \notin S$ into $\mc{N}$. If there is no $p$ in $\mc{B}_{e,s}$ which is d-free, we would be able to diagonalize by moving $x$ to $x' \in S$, and then later to $x'' \notin S$ and using appropriate computations as in Proposition \ref{ce-condition}. So we may assume that there is $p$ in $\mc{B}_{e,s}$ which is d-free. Now if $\mc{B}_e$ is an isomorphic copy of $\mc{A}$, we will eventually find a chain $p_0,p_1,\ldots,p_n = p$ from the root node $p_0$ to $p$, where $p_{i+1}$ is a child of $p_i$. Then in $\mc{A}$ every d-free element aside from $p_0,\ldots,p_n = p$ is in the subtree below $p$.

In $\mc{N}$, we just have to respect the tree-order, so no matter how much of the diagram of $\mc{N}$ we have build so far, we can always add a new d-free element $y$ such that $x$ is in the subtree below $y$. Then we will use the fact described above about d-free elements which are in the subtree below another d-free element. By moving $x$ to $x' \in S$ and then to $x'' \notin S$, we can force $p$ to move to $p' \in R$ and then $p'' \notin R$. Then, by moving $y$ to $y' \in S$ and then $y'' \notin S$, we can diagonalize as $\mc{B}_e$ will have no d-free elements which it can use: all of the d-free elements which were below $p$ have now been moved to be below $p''$ are are no longer d-free. We could still move $y$ to $y'$ and then $y''$ as $x$ was in the subtree below $y$ rather than vice versa. As in Proposition \ref{ce-condition}, we will use various computations to force $\mc{B}_e$ to follow $\mc{N}_e$.

The requirement $\mc{R}_{e,i,j}$ will have associated to it at each stage $s$ values $a_{e,i,j}[s]$, $u_{e,i,j}[s]$, $v_{e,i,j}[s]$, $m_{e,i,j}[s]$, and $t_{e,i,j}[s]$, and $b_{e,i,j}[s]$, $\mu_{e,i,j}[s]$, $\nu_{e,i,j}[s]$, and $\tau_{e,i,j}[s]$. These values will never be redefined, but may be canceled. When a requirement is injured, its corresponding values will be canceled. Figure \ref{fig:prop4-11} shows how these values are related.

At each stage, each of the requirements $\mc{R}_{e,i,j}$ will be in one of the following states: \textsc{initialized}, \textsc{waiting-for-first-computation}, \textsc{waiting-for-second-computation}, \textsc{waiting-for-first-change}, \textsc{waiting-for-second-change}, \textsc{waiting-for-third-change}, or \textsc{diagonalized}. Every requirement will move through these linearly in that order.

We are now ready to describe the construction.

\vspace*{10pt}
\noindent\textit{Construction.}
\vspace*{10pt}

At stage 0, let $F_s = \varnothing$ and for each $e$, $i$, and $j$ let $a_{e,i,j}[0]$, $u_{e,i,j}[0]$, and so on be $0$ (i.e., undefined).

At a stage $s+1$, let $F_s:\{0,\ldots,\xi_s\} \to \mc{M}$ be the partial isomorphism determined in the previous stage, and let $D(\mc{N}_s)$ be the finite part of the diagram of $\mc{N}$ which has been determined so far. We have an approximation $S^{\mc{N}}_s$ to $S^\mc{N}$ which we get by taking $k \in S^{\mc{N}}_s$ if $F_s(k) \in S$. For each $e$, we have a guess $R^{\mc{B}_{e}}_s$ at $R^{\mc{B}_{e}}$ using the diagram of the finite structure $\mc{B}_{e,s}$, given by $x \in R^{\mc{B}_e}_s$ if and only if in $\mc{B}_{e,s}$, $x$ is of type $\langle 1 \rangle$ (i.e., related by $U^{\mc{B}_{e,s}}$ to exactly one element of the second sort).

We will deal with a single requirement---the highest priority requirement which requires attention at stage $s+1$. A requirement $\mc{S}_i$ \textit{requires attention at stage $s+1$} if the $i$th element of $\mc{A}$ is not in the image of $F_s$. If $\mc{S}_i$ is the highest priority requirement which requires attention, then let $c$ be the $i$th element of $\mc{A}$. Let $F_{s+1}$ extend $F_s$ with $c$ in its image. Injure each requirement of higher priority.

The conditions for a requirement $\mc{R}_{e,i,j}$ to \textit{require attention at stage $s+1$} depends on the state of the requirement. Below, we will list for each possible state of $\mc{R}_{e,i,j}$, the conditions for $R_{e,i,j}$ to require attention, and the action that the requirement takes if it is the highest priority requirement that requires attention. We will also loosely describe what is happening in the construction, but a more rigorous treatment will follow.

\begin{description}[font=\sc]
\item[{Initialized}]
The requirement has been initialized, so $a_{e,i,j}[0]$, $u_{e,i,j}[0]$, and so on are all $0$.
\medskip{}
\begin{description}[font=\it,labelindent=-15pt,leftmargin = 0pt]
	\item[Requires attention] The requirement always requires attention.
	\item[Action] Let $F_{s+1}$ extend $F_s$ by adding to its image the element $0^{\ell}$, where $\ell$ is large enough that $0^{\ell}$ has no children in $\ran(F_{s})$. Then $0^\ell$ is d-free over $\ran(F_s)$. Let $a_{e,i,j}[s+1]$ be such that $F_{s+1}({a_{e,i,j}[s+1]}) = 0^\ell$. Change the state to \textsc{waiting-for-first-computation}.
\end{description}

\item[{Waiting-for-first-computation}]
We have set $F(a_{e,i,j}) = 0^\ell \notin R$ a d-free element. We wait for the computations \eqref{eq1:S-R} and \eqref{eq1:R-S} below. Then, we use the fact that $\mc{M}$ is given using the tree-order to insert an element $b_{e,i,j}$ in $\mc{N}$ above $a_{e,i,j}$ (so that now the image of $b_{e,i,j}$ under $F$ is $0^\ell$ and the image of $a_{e,i,j}$ under $F$ is $0^{\ell+1}$).
\medskip{}
\begin{description}[font=\it,labelindent=-15pt,leftmargin = 0pt]
	\item[Requires attention] The requirement requires attention if:
\begin{enumerate}
	\item there is a computation
\begin{equation}\label{eq1:S-R} S^\mc{N}_s[0,\ldots,a_{e,i,j}[s]] = \Phi_{i,s}^{R^{\mc{B}_e}_s}[0,\ldots,a_{e,i,j}[s]] \end{equation}
with use $u < s$,
	\item each element $p$ of the first sort in $\mc{B}_{e,s}$ is part of a chain $p_0,p_1,\ldots,p_n = p$ in $\mc{B}_{e,s}$ where $p_0$ is the root node and $p_{i+1}$ is a child of $p_i$, and
	\item there is a computation
\begin{equation}\label{eq1:R-S} R^{\mc{B}_e}_s[0,\ldots,m] = \Phi_{j,s}^{S^\mc{N}_s}[0,\ldots,m] \end{equation}
with use $v < s$ where $m \geq u$ is larger than each $p_i$ above.
\end{enumerate}
	\item[Action] Set $u_{e,i,j}[s+1] = u$, $v_{e,i,j}[s+1] = u$, $m_{e,i,j}[s+1] = m$, and $t_{e,i,j}[s+1] = s$. Let $a = a_{e,i,j}[s]$. We have $F_s(a) = 0^\ell$, where $\ell$ is large enough that no child of $0^\ell$ appears earlier in the image of $F_s$. Set
\[ F_{s+1}(w)=\begin{cases}
0^{\ell + 1}\sigma & F_{s}(w)=0^{\ell}\sigma\\
F_{s}(w) & otherwise
\end{cases}. \]
What we have done is taken every element of $\mc{N}$ which was mapped to the subtree below $0^\ell$, and moved it to the subtree below $0^{\ell + 1}$. Now let $b = b_{e,i,j}[s+1]$ be the first element on which $F_{s+1}$ is not yet defined and set $F_{s+1}(b) = 0^\ell$. So $F_{s+1}(b)$ is the parent of $F_{s+1}(a)$. Any existential formula which was true of the tree below $0^\ell$ is also true of the tree below $0^{\ell + 1}$. Also, for $w \in \dom (F_s)$, $F_{s}(w) \in R$ if and only if $F_{s+1}(w) \in R$. Change the state to \textsc{waiting-for-second-computation}.
\end{description}

\item[{Waiting-for-second-computation}]
In the previous state we defined $b_{e,i,j}$, so now we have to wait for the computations \eqref{eq2:S-R} and \eqref{eq2:R-S} below involving it. Then we modify $F$ so that it now looks like $a_{e,i,j} \in S$, breaking the computation \eqref{eq1:S-R} above.
\medskip{}
\begin{description}[font=\it,labelindent=-15pt,leftmargin = 0pt]
	\item[Requires attention] The requirement requires attention if there are computations
\begin{equation}\label{eq2:S-R} S^\mc{N}_s[0,\ldots,b_{e,i,j}[s]] = \Phi_{i,s}^{R^{\mc{B}_e}_s}[0,\ldots,b_{e,i,j}[s]] \end{equation}
with use $\mu < s$, and
\begin{equation}\label{eq2:R-S} R^{\mc{B}_e}_s[0,\ldots,\mu] = \Phi_{j,s}^{S^\mc{N}_s}[0,\ldots,\mu] \end{equation}
with use $\nu < s$.
	\item[Action] Set $\mu_{e,i,j}[s+1] = \mu$, $\nu_{e,i,j}[s+1] = \nu$, and $\tau_{e,i,j}[s+1] = s$. Let $a = a_{e,i,j}[s]$ and $b = b_{e,i,j}[s]$. We have $F_s(a) = 0^{\ell + 1}$ and $F_s(b) = 0^\ell$, where $\ell$ is large enough that no child of $0^\ell$ appears before $a$ in the image of $F_s$. Choose $x$ odd and larger than any odd number we have encountered so far, and define $F_{s+1}$ by
\[ F_{s+1}(w)=\begin{cases}
0^{\ell}x\concat\sigma & F_{s}(w)=0^{\ell+1}\concat\sigma\\
F_{s}(w) & otherwise
\end{cases}. \]
What we have done is taken every element of $\mc{N}$ which was mapped to the subtree below $0^{\ell+1}$, and moved it to the subtree below $0^{\ell} x$. Note that $F_{s+1}(b) = F_s(b)$ and for $w < a$, $F_s(w) = F_{s+1}(w)$. Any existential formula which was true of the tree below $0^\ell$ is also true of the tree below $0^{\ell}x$ (but not vice versa, since $0^{\ell} x$ is of type $\langle 1 \rangle$ but $0^{\ell+1}$ is of type $\langle 0 \rangle$). Also, for $w \in \dom (F_s)$, $F_{s}(w) \in S$ if and only if $F_{s+1}(w) \in S$ with the single exception of $w = a$. In that case, $F_{s}(a) \notin S$ and $F_{s+1}(a) \in S$. Change to state \textsc{waiting-for-first-change}. 
\end{description}

\item[{Waiting-for-first-change}]
In the previous state, we modified $F$ to break the computation \eqref{eq1:S-R}. If we are to have $S^\mc{N} = \Phi_i^{R^{\mc{B}_e}}$, then $R^{\mc{B}_e}$ must change below its use $u_{e,i,j}$. So some element of $\mc{B}_e$ which was previously of type $\langle 0 \rangle$ becomes of type $\langle 1 \rangle$, or some element which was previously of type $\langle 1 \rangle$ becomes of type $\langle 2 \rangle$. When this happens, we modify $F$ (by changing the image of $a_{e,i,j}$ again) so that $S^\mc{N}$ becomes the same as it was originally (below $\nu$).
\medskip{}
\begin{description}[font=\it,labelindent=-15pt,leftmargin = 0pt]
	\item[Requires attention] The requirement requires attention if
\[ R^{\mc{B}_e}_s[0,\ldots,u_{e,i,j}[s]] \neq R^{\mc{B}_e}_{t_{e,i,j}[s]}[0,\ldots,u_{e,i,j}[s]]. \]
	\item[Action] Let $a = a_{e,i,j}[s]$ and $b = b_{e,i,j}[s]$. We have $F_s(a) = 0^{\ell} x$, where $x$ is odd. Choose $y > 0$ even and larger than any even number we have encountered so far, and define $F_{s+1}$ by
\[ F_{s+1}(w)=\begin{cases}
0^{\ell}y\concat\sigma & F_{s}(w)=0^{\ell}x\concat\sigma\\
F_{s}(w) & otherwise
\end{cases}. \]
This is moving the subtree below $0^{\ell}x$ to the subtree below $0^{\ell}y$. Once again, $F_{s+1}(b) = F_s(b) = 0^\ell$. For $w \in \dom(F_s)$, $w \neq a$, we have $F_{s}(w) \in S$ if and only if $F_{s+1}(w) \in S$. For $w = a$, we have $F_{s}(a) \in S$ and $F_{s+1}(a) \notin S$. Change the state to \textsc{waiting-for-second-change}.
\end{description}

\item[{Waiting-for-second-change}]
In the previous state, we modified $F$ so that $S^\mc{N}$ is the same as it was previously in state \textsc{waiting-for-first-comput\-a\-tion}. By the computation \eqref{eq1:R-S}, $R^{\mc{B}_e}$ must return, below the use $u_{e,i,j}$, to the way it was previously (i.e., as it was when the computation \eqref{eq1:R-S} was found). It must be that the element from state \textsc{waiting-for-first-change} which changed its type then (from type $\langle 0 \rangle$ to type $\langle 1 \rangle$ or from type $\langle 1 \rangle$ to type $\langle 2 \rangle$) must now change its type again, and so it must have gone from type $\langle 0 \rangle$ to type $\langle 1 \rangle$ and now changes from type $\langle 1 \rangle$ to type $\langle 2 \rangle$. Call this element $p$. When this happens, we modify $F$ so that $b_{e,i,j}$ looks like it is in $S^{\mc{N}}$. We can do this because so far we have only modified the image of $a_{e,i,j}$, and $a_{e,i,j}$ was in the subtree below $b_{e,i,j}$. This breaks the computation \eqref{eq2:S-R}.
\medskip{}
\begin{description}[font=\it,labelindent=-15pt,leftmargin = 0pt]
	\item[Requires attention] The requirement requires attention if
\[ R^{\mc{B}_e}_s[0,\ldots,\mu_{e,i,j}[s]] = R^{\mc{B}_e}_{t_{e,i,j}[s]}[0,\ldots,\mu_{e,i,j}[s]]. \]
and also, each of the elements $m_{e,i,j}[s] + 1,\ldots,\mu_{e,i,j}[s]$ is of type $\langle 1 \rangle$ or type $\langle 2 \rangle$.
	\item[Action] Let $b = b_{e,i,j}[s]$. We have $F_s(b) = 0^{\ell}$. Choose $x > 0$ odd and larger than any odd number we have encountered so far, and define $F_{s+1}$ by
\[ F_{s+1}(w)=\begin{cases}
0^{\ell-1}x\concat\sigma & F_{s}(w)=0^{\ell}\concat\sigma\\
F_{s}(w) & otherwise
\end{cases}. \]
This is moving the subtree below $0^{\ell}$ to the subtree below $0^{\ell - 1}x$. For $w \in \dom(F_s)$, $w \neq b$, we have $F_{s}(w) \in S$ if and only if $F_{s+1}(w) \in S$. For $w = b$, we have $F_{s}(b) \notin S$ and $F_{s+1}(b) \in S$. Change the state to \textsc{waiting-for-third-change}.
\end{description}

\item[{Waiting-for-third-change}]
In the previous state, we broke the computation \eqref{eq2:S-R}. If we are to have $S^{\mc{N}} = \Phi_i^{R^{\mc{B}_e}}$, then $R^{\mc{B}_e}$ must change below the use $\mu$ of this computation. But since $S^{\mc{N}}[0,\ldots,v]$ is the same as it was before, by the computation \eqref{eq1:R-S}, $R^{\mc{B}_e}[0,\ldots,u]$ cannot change. So $R^{\mc{B}_e}$ must change on one of the elements $u+1,\ldots,\mu$. Let $p$ be the element we described in the previous state. By (2) from state \textsc{waiting-for-first-computation}, the only elements from among $u+1,\ldots,\mu$ in $\mc{B}_e$ which can be of type $\langle 0 \rangle$ are in the subtree below $p$. So when, in state \textsc{waiting-for-first-change}, $p$ becomes of type $\langle 1 \rangle$, each of the elements in the subtree below $p$ becomes of type $\langle 1 \rangle$ or type $\langle 2 \rangle$. So now, when $R^{\mc{B}_e}$ changes on one of the elements $u+1,\ldots,\mu$, it does so by some such element which was of type $\langle 1 \rangle$ becoming of type $\langle 2 \rangle$. Now modify $F$ so that $S^{\mc{N}}$ looks the same as it did originally (below $\nu$).
\medskip{}
\begin{description}[font=\it,labelindent=-15pt,leftmargin = 0pt]
	\item[Requires attention] The requirement requires attention if
\[ R^{\mc{B}_e}_s[0,\ldots,\mu_{e,i,j}[s]] \neq R^{\mc{B}_e}_{\tau_{e,i,j}[s]}[0,\ldots,\mu_{e,i,j}[s]]. \]
	\item[Action] Let $b = b_{e,i,j}[s]$. We have $F_s(b) = 0^{\ell-1} x$. Choose $y > 0$ even and larger than any even number we have encountered so far, and define $F_{s+1}$ by
\[ F_{s+1}(w)=\begin{cases}
0^{\ell-1}y\concat\sigma & F_{s}(w)=0^{\ell-1}x\concat\sigma\\
F_{s}(w) & otherwise
\end{cases}. \]
This is moving the subtree below $0^{\ell-1}x$ to the subtree below $0^{\ell - 1}y$. For $w \in \dom(F_s)$, $w \neq b$, we have $F_{s}(w) \in S$ if and only if $F_{s+1}(w) \in S$. For $w = b$, we have $F_{s}(b) \in S$ and $F_{s+1}(b) \notin S$. Change the state to \textsc{diagonalized}.
\end{description}

\item[{Diagonalized}]
In the previous state, we made sure that one of the elements $u+1,\ldots,\mu$ of $\mc{B}_e$ which previously looked like it was in $R^{\mc{B}_e}$ is now not in $R^{\mc{B}_e}$, so that that element is now of type $\langle 2 \rangle$ and hence must be in $R^{\mc{B}_e}$. We also modified $F$ so that $S^{\mc{N}}$ is the same as it was in state \textsc{initialized} (below $\nu$). Then, by computation $\eqref{eq2:R-S}$, we cannot have $R^{\mc{B}_e} = \Phi_j^{S^\mc{N}}$. So we have satisfied $\mc{R}_{e,i,j}$.
\medskip{}
\begin{description}[font=\it,labelindent=-15pt,leftmargin = 0pt]
	\item[Requires attention] The requirement never requires attention.
	\item[Action] None.
\end{description}

\end{description}

When a requirement of higher priority than $\mc{R}_{e,i,j}$ acts, $\mc{R}_{e,i,j}$ is injured. When this happens, $\mc{R}_{e,i,j}$ is returned to state \textsc{initialized} and its values $a_{e,i,j}$, $u_{e,i,j}$, etc. are set to $0$. Now injure all requirements of lower priority than the one that acted.

\noindent Set $D(\mc{N}_{s+1})$ to be the pullback along $F_{s+1}$ of the atomic and negated atomic formulas true of $\ran(F_{s+1})$ with G\"odel number at most $s$.

\vspace*{10pt}
\noindent\textit{End construction.}
\vspace*{10pt}

If $\mc{R}_{e,i,j}$ is never injured after some stage $s$, then it acts at most once at each stage, and it never moves backwards through the stages. A requirement $\mc{S}_{i}$ only acts once if it is not injured. So every requirement is injured only finitely many times.

Now we will show that each requirement is satisfied. Each requirement $\mc{S}_i$ is satisfied, because there is a stage $s$ after which each lower requirement never acts, and then at the next stage, $\mc{S}_i$ acts if it is not already satisfied, and is never again injured.

Now suppose that $\mc{R}_{e,i,j}$ is the least requirement which is not satisfied, and let $s$ be the last stage at which it is injured (or $s=0$ if it is never injured). So $\mc{R}_{e,i,j}$ is never injured after the stage $s$. Also, since $\mc{R}_{e,i,j}$ is not satisfied, $\mc{B}_{e}$ is isomorphic to $\mc{A}$, $S^{\mc{N}} = \Phi_i^{R^{\mc{B}_e}}$, and $R^{\mc{B}_e} = \Phi_j^{S^{\mc{N}}}$.

Since $\mc{R}_{e,i,j}$ was just injured at stage $s$, it in state \textsc{initialized}, and so it requires attention at stage $s_1 = s+1$. Then we define $a$ such that $F_{s_1+1}(a) \notin S$ and change to state \textsc{waiting-for-first-computation}.

Since $S^\mc{N} = \Phi_i^{R^{\mc{B}_e}}$, $R^{\mc{B}_e} = \Phi_j^{S^\mc{N}}$, and $\mc{B}_e$ is isomorphic to $\mc{A}$, at some stage $t > s_1$, $\mc{R}_{e,i,j}$ will require attention. Thus we get $u$, $m$, and $v$ such that
\begin{equation} \label{S-R-t-1} S^{\mc{N}}_t[0,\ldots,a] = \Phi_{i,t}^{R^{\mc{B}_e}_t}[0,\ldots,a] \end{equation}
with use $u$ and
\begin{equation} \label{R-S-t-1} R^{\mc{B}_e}_t[0,\ldots,m] = \Phi_{j,t}^{S^{\mc{N}}_t}[0,\ldots,m] \end{equation}
with use $v$. Each element $p$ from among $0,\ldots,a$ in $\mc{B}_e$ has a chain in $\mc{B}_{e,t}$ from the root node to itself, and each of these nodes is in $\{0,\ldots,m\}$. We define $b > m$ and $F_{t+1}$ such that $F_{t+1}(b) \notin S$. Then we change to state \textsc{waiting-for-second-computation}.

Once again, using the fact that $R^{\mc{B}_e} = \Phi_i^{S^\mc{N}}$ and $S^\mc{N} = \Phi_j^{R^{\mc{B}_e}}$, at some stage $\tau > t$, we will have
\begin{equation} \label{S-R-tau-1} S^{\mc{N}}_\tau[0,\ldots,b] = \Phi_{i,\tau}^{R^{\mc{B}_e}_\tau}[0,\ldots,b] \end{equation}
with use $\mu$ and
\begin{equation} \label{R-S-tau-1} R^{\mc{B}_e}_\tau[0,\ldots,\mu] = \Phi_{j,\tau}^{S^{\mc{N}}_\tau}[0,\ldots,\mu] \end{equation}
with use $\nu$. Note that \eqref{S-R-tau-1} implies that
\begin{equation}\label{R-R-t-tau} R^{\mc{B}_e}_{\tau}[0,\ldots,u] =  R^{\mc{B}_e}_{t}[0,\ldots,u]. \end{equation} Then $\mc{R}_{e,i,j}$ requires attention at stage $\tau+1$. We define $F_{\tau+1}$ so that $F_{\tau+1}(a) \in S$. Then we have
\begin{equation}\label{R-R-t1-tau} S^{\mc{N}}_{\tau + 1}[0,\ldots,u] \neq S^{\mc{N}}_{\tau}[0,\ldots,u] \end{equation}
The state is changed to \textsc{waiting-for-first-change}. 

Now by \eqref{S-R-tau-1}, \eqref{R-R-t1-tau}, and the fact that $S^\mc{N} = \Phi_i^{R^{\mc{B}_e}}$, at some stage $s_1 > \tau$, we have
\begin{equation} \label{R-s1-t} R^{\mc{B}_e}_{s_1}[0,\ldots,u] \neq R^{\mc{B}_e}_t[0,\ldots,u] = R^{\mc{B}_e}_\tau[0,\ldots,u] \end{equation}
and so $\mc{R}_{e,i,j}$ requires attention at stage $s_1+1$. $F_{s_1+1}$ is defined such that
\begin{equation}\label{eq6} S^\mc{N}_{s_1+1}[0,\ldots,\nu] = S^\mc{N}_{\tau}[0,\ldots,\nu]. \end{equation}
Then the state is changed to \textsc{waiting-for-second-change}.

Since $R^{\mc{B}_e} = \Phi_j^{S^{\mc{N}}}$ and using \eqref{R-S-tau-1} and \eqref{eq6}, at some stage $s > s_1$, we have
\begin{equation}\label{eq7} R^{\mc{B}_e}_{s}[0,\ldots,\mu] = R^{\mc{B}_e}_{\tau}[0,\ldots,\mu]. \end{equation}
Now by \eqref{R-s1-t} and \eqref{eq7}, there must be some $p \in \{0,\ldots,u\}$ such that $p \notin R^{\mc{B}_e}_{\tau}$, $p \in R^{\mc{B}_e}_{s_1}$, and $p \notin R^{\mc{B}_e}_{s}$. Then in $\mc{B}_{e,s}$, $p$ must be of type $\langle 2 \rangle$. All of the elements of $\mc{B}_e$ from $\{m+1,\ldots,\mu\}$ which looked like they were of type $\langle 0 \rangle$ at stage $t$ were in the subtree below $p$, so there is a stage $s' > s$ at which each of them is of type $\langle 1 \rangle$ or type $\langle 2 \rangle$. Then, at some stage $s_2 > s'$, we still have
\[ R^{\mc{B}_e}_{s_2}[0,\ldots,\mu] = R^{\mc{B}_e}_{\tau}[0,\ldots,\mu]. \]
So $\mc{R}_{e,i,j}$ requires attention at stage $s_2+1$. $F_{s_2+1}$ is defined so that $F_{s_2+1}(b) \in S$. Since we had $F_{s_2}(b) \notin S$,
\[ S^\mc{N}_{s_2+1}[0,\ldots,v] = S^\mc{N}_{\tau}[0,\ldots,v] = S^\mc{N}_{t}[0,\ldots,v].\]
The state is changed to \textsc{waiting-for-third-change}.

So by \eqref{R-S-t-1} and \eqref{S-R-tau-1} and since $R^{\mc{B}_e} = \Phi_j^{S^{\mc{N}}}$, at some stage $s_3 > s_2$, we have 
\[ R^{\mc{B}_e}_{s_3}[0,\ldots,\mu] \neq R^{\mc{B}_e}_{\tau}[0,\ldots,\mu] \]
but
\[ R^{\mc{B}_e}_{s_3}[0,\ldots,m] = R^{\mc{B}_e}_{\tau}[0,\ldots,m] \]
so that
\[ R^{\mc{B}_e}_{s_3}[m+1,\ldots,\mu] \neq R^{\mc{B}_e}_{\tau}[m+1,\ldots,\mu] \]
So $\mc{R}_{e,i,j}$ requires attention. $F_{s_3+1}$ is defined so that 
\[ S^\mc{N}_{s_3+1}[0,\ldots,\mu] = S^\mc{N}_{\tau}[0,\ldots,\mu].\]
The state is changed to \textsc{diagonalized}.

Since $R^{\mc{B}_e} = \Phi_j^{S^\mc{N}}$, at some stage $s_4 > s_3$, by \eqref{R-S-tau-1}, we have
\[ R^{\mc{B}_e}_{s_4}[0,\ldots,\mu] = R^{\mc{B}_e}_{\tau}[0,\ldots,\mu]. \]
Then in $\mc{B}_e$, there must be $q \in \{m+1,\ldots,\mu\}$ such that $F_{s_2}(q) \notin R$, $F_{s_3}(q) \in R$, and $F_{s_4}(q) \notin R$. Then in $\mc{B}_{e,s_2}$ at stage $s_2$, $q$ must have looked like it was of type $\langle 0 \rangle$. We have already established that all of the elements in $\{m+1,\ldots,\mu\}$ were of type $\langle 1 \rangle$ or of type $\langle 2 \rangle$. This is a contradiction. Hence all of the requirements are satisfied.
\end{proof}

The proposition relativizes as follows:

\begin{corollary}\label{cor:prop-1-rel}
For every degree $\textbf{d}$, there is a copy $\mc{N}$ of $\mc{M}$ with $\mc{N} \leq_T \textbf{d}$ such that no copy $\mc{B}$ of $\mc{A}$ with $\mc{B} \leq_T \textbf{d}$ has $S^{\mc{N}} \oplus \textbf{d} \equiv_T R^{\mc{B}} \oplus \textbf{d}$.
\end{corollary}

Now we have the proposition in the other direction, in the unrelativized form:

\begin{prop}\label{prop:incomp-2}
There is a computable copy $\mc{B}$ of $\mc{A}$ such that no computable copy $\mc{N}$ of $\mc{M}$ has $R^{\mc{B}} \equiv_T S^{\mc{N}}$.
\end{prop}

\begin{proof}
We will construct a computable copy $\mc{B}$ of $\mc{A}$ with domain $\omega$. We will diagonalize against every possible Turing equivalence with a computable copy $\mc{N}$ of $\mc{M}$. We will build $\mc{B}$ with an infinite injury construction using subrequirements, where each subrequirement is injured only finitely many times.

We will construct $\mc{B}$ by giving at each stage a tentative finite isomorphism $F_s:\omega \to \mc{A}$. $F = \lim F_s$ will be a bijection, giving $\mc{B}$ as an isomorphic copy.
The proof will be very similar in style to the proof of Propositions \ref{ce-condition} and \ref{M-not-A}, but there are some significant complications.

We need to diagonalize against computable copies of $\mc{M}$. As in the previous proposition, given a computable function $\Phi_e$, we can try to interpret $\Phi_e$ as giving the diagram of a computable structure $\mc{N}_e$ isomorphic to $\mc{M}$. At each stage, we get a finite substructure $\mc{N}_{e,s}$ isomorphic to a substructure of $\mc{M}$. If $\Phi_e$ is total and gives the diagram of a structure isomorphic to $\mc{M}$, then $\mc{N}_e = \bigcup \mc{N}_{e,s}$ is that structure. Otherwise, $\mc{N}_e$ will be some structure which may be finite and may or may not be isomorphic to $\mc{M}$. Also, recall that elements of $\mc{N}_{e,s}$ have a type which approximates their type in $\mc{N}_e$, and that we can assume that our approximation $\mc{N}_{e,s}$ has the the following property: all of the elements of $\mc{N}_{e,s}$ which are of type $\langle 0,0 \rangle$ in $\mc{N}_{e,s}$ are linearly ordered.

We will meet the following requirements:
\begin{description}
	\item[$\mc{R}_{e,i,j}$]If $\mc{N}_e$ is isomorphic to $\mc{M}$, and $\Phi_i^{S^{\mc{N}_e}}$ and $\Phi_j^{R^{\mc{B}}}$ are total, then either $R^{\mc{B}} \neq \Phi_i^{S^{\mc{N}_e}}$ or $S^{\mc{N}_e} \neq \Phi_j^{R^{\mc{B}}}$.
	\item[$\mc{S}_{i}$] The $i$th element of $\mc{A}$ is in the image of $F$.
\end{description}
The strategy for satisfying a requirement $\mc{R}_{e,i,j}$ is as follows. The requirement $\mc{R}_{e,i,j}$ will have, associated to it at each stage $s$, values $a_{e,i,j}[s]$, $u_{e,i,j}[s]$, $v_{e,i,j}[s]$, and $t_{e,i,j}[s]$. Also, for each $n$, there will be values $b^n_{e,i,j}[s]$, $\mu^n_{e,i,j}[s]$, $\nu^n_{e,i,j}[s]$, and $\tau^n_{e,i,j}[s]$. See Figure \ref{fig:prop4-12} for a depiction of what these values mean. If $\mc{R}_{e,i,j}$ is not satisfied, then $\mc{N}_e$ will be isomorphic to $\mc{M}$, and we will have \[R^{\mc{B}} = \Phi_i^{S^{\mc{N}_e}} \text{ and } S^{\mc{N}_e} = \Phi_j^{R^{\mc{B}}}.\] Thus given, for example, a value for $n_{e,i,j}$, there will be some $u_{e,i,j}$ such that \[ R^{\mc{B}}[0,\ldots,n_{e,i,j}] = \Phi_j^{S^{\mc{N}_e}}[0,\ldots,n_{e,i,j}] \] with use $u_{e,i,j}$. In the overview of the proof which follows below, we will just assume that these values (and the corresponding computations) always exist, since otherwise $\mc{R}_{e,i,j}$ is trivially satisfied. We also write $n$ for $n_{e,i,j}$ etc.

\begin{figure}[b]
	\includegraphics{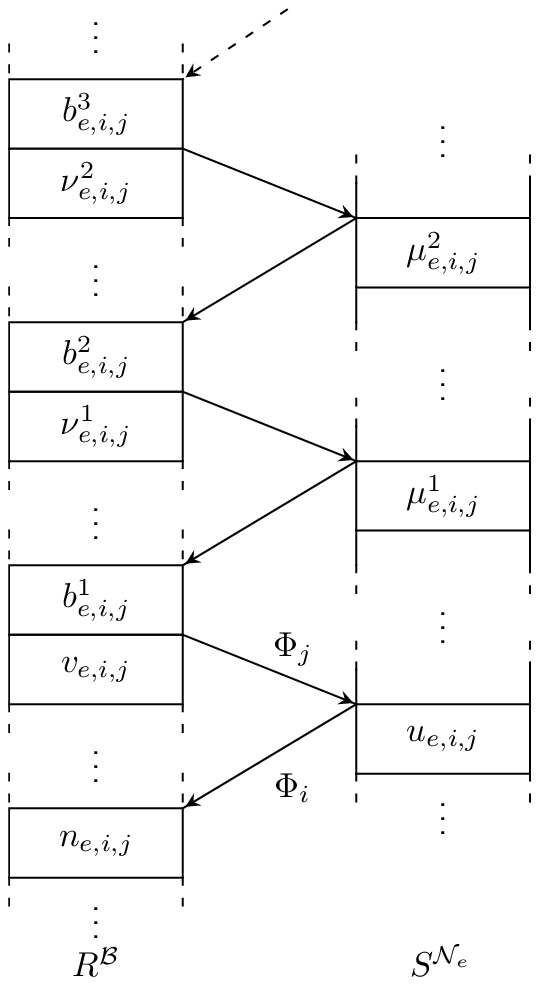}
	\caption{The values associated to a requirement for Proposition \ref{prop:incomp-2}. An arrow shows a computation converging. The computations use an oracle and compute some initial segment of their target. The tail of the arrow shows the use of the computation, and the head shows the length of the output. For example, $R^{\mc{B}}[0,\ldots,n_{e,i,j}] = \Phi_j^{S^{\mc{N}_e}}$ with use $u_{e,i,j}$.}
	\label{fig:prop4-12}
\end{figure}

We begin by mapping $n_{e,i,j}$ in $\mc{B}$ to the d-free element $0^\ell$ in $\mc{A}$. Because of the computation
\[ R^{\mc{B}}[0,\ldots,n_{e,i,j}] = \Phi_j^{S^{\mc{N}_e}}, \]
there must be at least one element in $\mc{N}_e$ below the use $u_{e,i,j}$ of this computation which still looks like it could be d-free in $\mc{N}_e$, i.e.\ of type $\langle 0,0 \rangle$. If not, we could modify $F$ to map $n_{e,i,j}$ to $0^{\ell-1} x$ where $x$ is odd (so that $F(n_{e,i,j}) \in R$) and then later modify it again to map $n_{e,i,j}$ to $0^{\ell-1} y$ where $y$ is even (and so $F(n_{e,i,j}) \notin R$) to immediately diagonalize against the computation above (while maintaining, as usual, any existential formulas). So one of the elements $0,\ldots,u_{e,i,j}$ of $\mc{N}_e$ must look like an element of the form $0^m$, i.e., be of type $\langle 0,0 \rangle$. All of the elements in $\mc{M}$ of the form $0^m$ (i.e., of type $\langle 0,0 \rangle$) are linearly ordered; if $\mc{N}_e$ is isomorphic to $\mc{M}$, the same is true in $\mc{N}_e$. Let $p$ be the element from $0,\ldots,u_{e,i,j}$ which is of type $\langle 0,0 \rangle$ and which is furthest from the root node of all such elements.

Now map $b^1_{e,i,j} \in \mc{B}$ to an element of the form $0^{\ell_1}1$ where $\ell_1 \geq \ell$. We claim that one of the elements $u_{e,i,j}+1,\ldots,\mu^1_{e,i,j}$ of $\mc{N}_e$ has to be:
\begin{enumerate}
	\item[(i)] not in the subtree below $p$ and
	\item[(ii)] of type $\langle 0,0 \rangle$ or type $\langle 1,0 \rangle$.
\end{enumerate}
Otherwise, we will be able to diagonalize to satisfy the requirement $\mc{R}_{e,i,j}$ in the following way. First, modify $F$ to map $n_{e,i,j}$ to $0^{\ell-1} x$ where $x$ is odd and then to $0^{\ell-1}y$ where $y$ is even. Then $b^1_{e,i,j}$ is now mapped to an element of the form $\rho 1$ for some $\rho$. This will force $p$, or some other element between $p$ in the root node, to enter $S^{\mc{N}_e}$ and then leave $S^{\mc{N}_e}$. Then $p$ must be look like an element of the form $0^m z$ for $z$ even, or $\tau z$ where $\tau$ is \textit{not} of the form $0^m$. In either case, by assumption each of the elements $u_{e,i,j}+1,\ldots,\mu^1_{e,i,j}$ of $\mc{N}_e$ must now either be of type $\langle 1,1\rangle$ or type $\langle 2,0 \rangle$ (if one of these elements was not before, then it is now as it was in the subtree below $p$ and $p$ is not of type $\langle 0,0 \rangle$). So all of these elements satisfy some existential formula which forces them to be in $S^{\mc{N}_e}$ (if they are of type $\langle 1,1 \rangle$) or which forces them out of $S^{\mc{N}_e}$ (if they are of type $\langle 2,0 \rangle$). Recall that $F$ is now mapping $b^1_{e,i,j}$ to $\rho 1$ for some $\rho$. We also have the computations
\[ R^{\mc{B}}[0,\ldots,b^1_{e,i,j}] = \Phi_j^{S^{\mc{N}_e}}[0,\ldots,b^1_{e,i,j}]\]
with use $\mu^1_{e,i,j}$ and
\[ S^{\mc{N}_e}[0,\ldots,u_{e,i,j}] = \Phi_i^{R^{\mc{B}}}[0,\ldots,u_{e,i,j}] \]
with use $v_{e,i,j}$. So by now modifying $F$ to map $b^1_{e,i,j}$ to an element of the form $\rho z$ for $z$ even, we now have $b^1_{e,i,j} \notin R^{\mc{B}}$. We break the first computation, causing $S^{\mc{N}_e}$ to change below the use $\mu^1_{e,i,j}$. Because of the second computation, $S^{\mc{N}_e}$ must stay the same on the elements $0,\ldots,u_{e,i,j}$. Thus $S^{\mc{N}_e}$ must change on the elements $u_{e,i,j}+1,\ldots,\mu_{e,i,j}$. But this cannot happen as remarked before. Thus we have diagonalized and satisfied $\mc{R}_{e,i,j}$.

Thus, if we cannot diagonalize to satisfy $\mc{R}_{e,i,j}$ in this way, one of the elements $u_{e,i,j}+1,\ldots,\mu^1_{e,i,j}$ of $\mc{N}_e$ satisfies (i) and (ii) above. Choose one such element, and call it $q^1$. Defining $b^2_{e,i,j}$, $b^3_{e,i,j}$, and so on in the same way, we get $q^2$, $q^3$, and so on. Thus we find, in $\mc{N}_e$, infinitely many elements satisfying (i) and (ii). But if $\mc{N}_e$ is isomorphic to $\mc{M}$, there can only be finitely many such elements: if $p$ is the isomorphic image of $0^m$, then the only elements satisfying (i) and (ii) are the isomorphic images of $0^{m'}$ and $0^{m'} 1$ for $m' < m$. Thus we force $\mc{N}_e$ to be non-isomorphic to $\mc{R}_{e,i,j}$ and satisfy $\mc{R}_{e,i,j}$ in that way.

The requirement $\mc{R}_{e,i,j}$ will have subrequirements $\mc{R}^n_{e,i,j}$ for $n \geq 1$. The main requirement will choose $a$, while each subrequirement will choose $b^n$. A subrequirement will act only when all of the previous requirements have chosen their values $b^n$. The main requirement $\mc{R}_{e,i,j}$ will monitor $\mc{N}_e$ to see whether it has given us elements $p$ and $q^n$, and if not, it can attempt to diagonalize. Either the parent requirement $\mc{R}_{e,i,j}$ will at some point diagonalize and be satisfied, or each subrequirement will be satisfied guaranteeing that $\mc{N}_e$ is not isomorphic to $\mc{M}$.

Because the subrequirements use infinitely many values $b^n$, we need to assign the subrequirements a lower priority than the main requirement in order to give the other requirements a chance to act. $\mc{R}_{e,i,j}$ will be of higher priority than its subrequirements $\mc{R}^n_{e,i,j}$, and the subrequirements will be of decreasing priority as $n$ increases. The subrequirements will be interleaved in this ordering, so that, for example, the ordering might begin (from highest priority to lowest priority):
\[ \mc{R}_{e_1,i_1,j_1} > \mc{R}^1_{e_1,i_1,j_1} > \mc{R}_{e_2,i_2,j_2} > \mc{R}^2_{e_1,i_1,j_1} > \mc{R}^1_{e_2,i_2,j_2} > \mc{R}^3_{e_1,i_1,j_1} > \cdots . \]

The requirement $\mc{R}_{e,i,j}$ will have, associated to it at each stage $s$, the values $a_{e,i,j}[s]$, $u_{e,i,j}[s]$, $v_{e,i,j}[s]$, and $t_{e,i,j}[s]$. A subrequirement $\mc{R}^n_{e,i,j}$ will be associated with the values $b^n_{e,i,j}[s]$, $\mu^n_{e,i,j}[s]$, $\nu^n_{e,i,j}[s]$, and $\tau^n_{e,i,j}[s]$. These values will never be redefined, but may be canceled. When a requirement is injured, its corresponding values will be canceled, with one exception. If $\mc{R}_{e,i,j}$ finds an opportunity to diagonalize using $b^n_{e,i,j}[s]$, then it will protect the values $b^n_{e,i,j}[s]$, $\mu^n_{e,i,j}[s]$, $\nu^n_{e,i,j}[s]$, and $\tau^n_{e,i,j}[s]$ using its own priority.

At each stage, each requirement $\mc{R}_{e,i,j}$ will be in one of the following states: \textsc{initialized}, \textsc{waiting-for-computation}, \textsc{next-subrequirement}, \textsc{waiting-for-change}, \textsc{diagonalized}, \textsc{waiting-for-first-change-$n$}, \textsc{waiting-for-second-change-$n$}, or \textsc{diagonalized-$n$}. The requirement will move through these in the following order (where, at the branch, the requirement will move along either the left branch or the right branch, and if it moves along the right branch it does so for some specific value of $n$):

\[\xymatrix@C=-5pc@R=0.75pc{ & \textsc{initialized}\ar[d]\\
 & \textsc{waiting-for-computation}\ar[d]\\
 & \textsc{next-subrequirement}\ar[dddl]\ar[dddr]\\
\\
\\
\textsc{waiting-for-change}\ar[d] &  & \textsc{waiting-for-first-change-\ensuremath{n}}\ar[d]\\
\textsc{diagonalized} &  & \textsc{waiting-for-second-change-\ensuremath{n}}\ar[d]\\
 &  & \textsc{diagonalized-\ensuremath{n}}
}\]

When the requirement is in state \textsc{next-subrequirement}, the subrequirements will begin acting. The requirement will then monitor them for a chance to diagonalize. There are two ways in which the requirement can diagonalize, either by going to state \textsc{waiting-for-change} if the element $p$ described above does not exist, or \textsc{waiting-for-first-change-$n$} if some element $q^n$ described above does not exist. Recall that in the second case, the values $b^n_{e,i,j}$, $\mu^n_{e,i,j}$, $\nu^n_{e,i,j}$, and $\tau^n_{e,i,j}$ are protected by the requirement $\mc{R}_{e,i,j}$ at its priority.

Each subrequirement $\mc{R}^n_{e,i,j}$ will be in one of three states: \textsc{initialized}, \textsc{waiting-for-computation}, or \textsc{next-subrequirement}. The subrequirement will move through these stages in order. When one subrequirement is in state \textsc{next-subre\-quirement}, it has finished acting and the next one can begin.

We are now ready to describe the construction.

\vspace*{10pt}
\noindent\textit{Construction.}
\vspace*{10pt}

At stage 0, let $F_s = \varnothing$ and for each $e$, $i$, and $j$ let $a_{e,i,j}[0]$, $u_{e,i,j}[0]$, $v_{e,i,j}[0]$, and $t_{e,i,j}[0]$ be $\varnothing$. For each $n$, let $b^n_{e,i,j}[0]$, $\mu^n_{e,i,j}[0]$, $\nu^n_{e,i,j}[0]$, and $\tau^n_{e,i,j}[0]$ be $0$ as well.

At a stage $s+1$, let $F_s:\{0,\ldots,\xi_s\} \to \mc{A}$ be the partial isomorphism determined in the previous stage, and let $D(\mc{B}_s)$ be the finite part of the diagram of $\mc{B}$ which has been determined so far. We have an approximation $R^{\mc{B}}_s$ to $R^\mc{B}$ which we get by taking $k \in R^{\mc{B}}_s$ if $F_s(k) \in R$. For each $e$, we have a guess $S^{\mc{N}_{e}}_s$ at $S^{\mc{N}_{e}}$ using the diagram of the finite structure $\mc{N}_{e,s}$, given by $x \in S^{\mc{N}_e}_s$ if and only if in $\mc{N}_{e,s}$, $x$ is in the first sort and is related by $V^{\mc{N}_{e,s}}$ to exactly one element of the second sort.

We will deal with a single requirement---the highest priority requirement which requires attention at stage $s+1$. A requirement $\mc{S}_i$ \textit{requires attention at stage $s+1$} if the $i$th element of $\mc{A}$ is not in the image of $F_s$. If $\mc{S}_i$ is the highest priority requirement which requires attention, then let $c$ be the $i$th element of $\mc{A}$. Let $F_{s+1}$ extend $F_s$ with $c$ in its image. Injure each requirement of higher priority.

The conditions for a requirement $\mc{R}_{e,i,j}$ or a subrequirement $\mc{R}^n_{e,i,j}$ to \textit{require attention at stage $s+1$} depend on the state of the requirement. Below, we will list for each possible state of $\mc{R}_{e,i,j}$, the conditions for $R_{e,i,j}$ to require attention, and the action that the requirement takes if it is the highest priority requirement that requires attention. The subrequirements will follow afterward.
\begin{description}[font=\sc]
\item[{Initialized}]
The requirement has been initialized, so $a_{e,i,j}[0]$, $u_{e,i,j}[0]$, and so on are all $0$.
\medskip{}
\begin{description}[font=\it,labelindent=-15pt,leftmargin = 0pt]
	\item[Requires attention] The requirement always requires attention.
	\item[Action] Let $F_{s+1}$ extend $F_s$ by adding to its image the element $0^{\ell}$, where $\ell$ is large enough that $0^{\ell}$ has no children in $\ran(F_{s})$. Let $a_{e,i,j}[s+1]$ be such that $F_{s+1}({a_{e,i,j})[s+1]} = 0^\ell$. Change the state to \textsc{waiting-for-computation}.
\end{description}

\item[{Waiting-for-computation}]
We have set $F(a_{e,i,j}) = 0^\ell \notin R$ a d-free element. We wait for the computations \eqref{eq3:R-S} and \eqref{eq3:S-R} below. Then we can begin to satisfy the subrequirements.
\medskip{}
\begin{description}[font=\it,labelindent=-15pt,leftmargin = 0pt]
	\item[Requires attention] The requirement requires attention if there is a computation
\begin{equation}\label{eq3:R-S} R^\mc{B}_s[0,\ldots,a_{e,i,j}[s]] = \Phi_i^{S^{\mc{N}_e}_s}[0,\ldots,a_{e,i,j}[s]] \end{equation}
with use $u < s$, and
\begin{equation}\label{eq3:S-R} S^{\mc{N}_e}_s[0,\ldots,u] = \Phi_j^{R^\mc{B}_s}[0,\ldots,u] \end{equation}
with use $v < s$.
	\item[Action] Let $u$ and $v$ be the uses of the computations which witness that this requirement requires attention. Set $u_{e,i,j}[s+1] = u$, $v_{e,i,j}[s+1] = u$, and $t_{e,i,j}[s+1] = s$. We have $F_{s+1} = F_s$. Change the state to \textsc{next-subrequirement}.
\end{description}

\item[{Next-subrequirement}]
While in this state, we begin trying to satisfy the subrequirements, building elements $b_{e,i,j}^1,b_{e,i,j}^2$, and so on. At the same time, we look for a way to immediately satisfy $\mc{R}_{e,i,j}$. The requirement requires attention during this state if we see such a way to satisfy $\mc{R}_{e,i,j}$. There are two possibly ways that we might immediately diagonalize. The first is that we can diagonalize using only $a_{e,i,j}$ and the computations \eqref{eq3:R-S} and \eqref{eq3:S-R}, because none of the elements of $\mc{N}_e$ below the use $u_{e,i,j}$ of \eqref{eq3:R-S} are d-free. The second is that we can diagonalize using $a_{e,i,j}$ and some $b^n_{e,i,j}$, because we use $a_{e,i,j}$ to force $S^{\mc{N}_e}$ to change below the use $u_{e,i,j}$ of \eqref{eq3:R-S}, and this will mean that we can diagonalize by changing $b^n_{e,i,j}$ from being in $R^\mc{B}$ to being out of $R^\mc{B}$. If we see a chance to diagonalize, we modify $F$ to put $a_{e,i,j}$ into $R^\mc{B}$, breaking the computation \eqref{eq3:R-S}.
\medskip{}
\begin{description}[font=\it,labelindent=-15pt,leftmargin = 0pt]
	\item[Requires attention] There are two possible ways that this requirement might require attention. The requirement requires \textit{attention of the first kind} if in $\mc{N}_{e,s}$:
\begin{enumerate}
	\item each of the elements $0,\ldots,u_{e,i,j}[s]$ of $\mc{N}_e$ which is in the first sort is of type $\langle 1,0 \rangle$, type $\langle 1,1 \rangle$, or type $\langle 2,0 \rangle$,
	\item we still have
	\[ S^{\mc{N}_e}_{s}[0,\ldots,u_{e,i,j}[s]] = S^{\mc{N}_e}_{t_{e,i,j}[s]}[0,\ldots,u_{e,i,j}[s]]. \]
\end{enumerate}
The requirement requires \textit{attention of the second kind} if for some $n$:
\begin{enumerate}
	\item each of the subrequirements $\mc{R}^m_{e,i,j}$ is in state \textsc{next-requirement} for all $m \leq n$,
	\item each of the elements $\mu^{n-1}_{e,i,j}[s]+1,\ldots,\mu^n_{e,i,j}[s]$ of $\mc{N}_{e,s}$ (with $\mu^{n-1}_{e,i,j}[s]$ replaced by $u_{e,i,j}[s]$ if $n = 0$) which is in the first sort and is either of type $\langle 0,0 \rangle$ or type $\langle 1,0 \rangle$ in $\mc{N}_{e,s}$ is in the subtree below those elements from among $0,\ldots,u_{e,i,j}[s]$ which are not related to any elements of the second sort, 
	\item we still have
\[ S^{\mc{N}_e}_{s}[0,\ldots,\mu^n_{e,i,j}[s]] = S^{\mc{N}_e}_{\tau^n_{e,i,j}[s]}[0,\ldots,\mu^n_{e,i,j}[s]]. \]
\end{enumerate}

	\item[Action] There were two different ways in which this requirement might require attention. The only difference in the action we take is which state we move to. Let $a = a_{e,i,j}[s]$. We have $F_s(a) = 0^\ell$, where $\ell$ is large enough that no child of $0^\ell$ appears earlier in the image of $F_s$. Choose $x$ odd and larger than any odd number we have encountered so far, and define $F_{s+1}$ by
\[ F_{s+1}(w)=\begin{cases}
0^{\ell-1}x\concat\sigma & F_{s}(w)=0^{\ell}\concat\sigma\\
F_{s}(w) & otherwise
\end{cases}. \]
What we have done is taken every element of $\mc{C}$ which was mapped to the subtree below $0^\ell$, and moved it to the subtree below $0^{\ell - 1} x$. Any existential formula which was true of the tree below $0^\ell$ is also true of the tree below $0^{\ell-1}x$ (but not vice versa, since $0^{\ell - 1} x$ is connected by $U$ to an element of the second sort, but $0^\ell$ is not). Also, for $w \in \dom (F_s)$, $F_{s}(w) \in R$ if and only if $F_{s+1}(w) \in R$ with the single exception of $w = a$. In that case, $F_{s}(a) \notin R$ and $F_{s+1}(a) \in R$. If the requirement required attention of the first kind, change to state \textsc{waiting-for-change}. Otherwise, if it required attention of the second kind, change to state \textsc{waiting-for-first-change-$n$} where $n$ is the least witness to the fact that this requirement required attention of the second kind.
\end{description}

\item[{Waiting-for-change}]
In this state, we are trying to diagonalize against $\mc{R}_{e,i,j}$ in the first way described above. The computation \eqref{eq3:R-S} was broken, and so as usual, $S^{\mc{N}_e}$ must change below the use $u_{e,i,j}$. When we first entered this state, all of the elements $0,\ldots,u_{e,i,j}$ of $\mc{N}_{e}$ were of types $\langle 1,0 \rangle$, $\langle 1,1 \rangle$, or $\langle 2,0 \rangle$ (i.e., were not d-free). So in order for $S^{\mc{N}_e}$ to change below the use $u_{e,i,j}$, one of these elements (call it $p$) which was connected to one element of the second sort must become connected to two elements of the second sort. We then modify $F$ to make $R^\mc{B}$ the same as it was originally (below $v_{e,i,j}$). This will successfully satisfy the requirement, because $S^{\mc{N}_e}$ cannot return to the way it was originally because $p$ cannot return to being in $S^{\mc{N}_e}$, and so the computation \eqref{eq3:S-R} from the state \textsc{waiting-for-computation} means that we cannot have $S^{\mc{N}_e} = \Phi_j^{R^{\mc{B}}}$. 
\medskip{}
\begin{description}[font=\it,labelindent=-15pt,leftmargin = 0pt]
	\item[Requires attention] This requirement requires attention if
\[ S^{\mc{N}_e}_{s}[0,\ldots,u_{e,i,j}[s]] \neq S^{\mc{N}_e}_{t_{e,i,j}[s]}[0,\ldots,u_{e,i,j}[s]].\]
	\item[Action] Let $a = a_{e,i,j}[s]$. We have $F_s(a) = 0^{\ell-1} x$, where $x$ is odd and no child of $0^{\ell-1} x$ appears earlier in the image of $F_s$. Choose $y > 0$ even and larger than any even number we have encountered so far, and define $F_{s+1}$ by
\[ F_{s+1}(w)=\begin{cases}
0^{\ell-1}y\concat\sigma & F_{s}(w)=0^{\ell-1}x\concat\sigma\\
F_{s}(w) & otherwise
\end{cases}. \]
This is moving the subtree below $0^{\ell - 1}x$ to the subtree below $0^{\ell-1}y$. For $w \in \dom(F_s)$, $w \neq a$, we have $F_{s}(w) \in R$ if and only if $F_{s+1}(w) \in R$. For $w = a$, we have $F_{s}(a) \in R$ and $F_{s+1}(a) \notin R$. Change the state to \textsc{diagonalized}.
\end{description}

\item[{Diagonalized}]
In this state, we have successfully satisfied $\mc{R}_{e,i,j}$ in the first way described above.
\medskip{}
\begin{description}[font=\it,labelindent=-15pt,leftmargin = 0pt]
	\item[Requires attention] The requirement never requires attention.
	\item[Action] None.
\end{description}

\item[{Waiting-for-first-change-}$n$]
In this state, we are trying to diagonalize against $\mc{R}_{e,i,j}$ in the second way described above. The computation \eqref{eq3:R-S} was broken, and so as usual, $S^{\mc{N}_e}$ must change below the use $u_{e,i,j}$. So some such element (which we call $p$) which was connected to no elements of the second sort (i.e., of type $\langle 0,0\rangle$) must become connected to one element of the second sort\footnote{It is possible for some element which was connected to one element of the second sort to become connected to two elements, but in this case we will successfully satisfy $\mc{R}_{e,i,j}$ in much the same way as above as a byproduct of our general construction.} (i.e., it is now some other type). We then modify $F$ to make $R^\mc{B}$ the same as it was originally (below $v_{e,i,j}$).
\medskip{}
\begin{description}[font=\it,labelindent=-15pt,leftmargin = 0pt]
	\item[Requires attention] This requirement requires attention if
\[ S^{\mc{N}_e}_{s}[0,\ldots,u_{e,i,j}[s]] \neq S^{\mc{N}_e}_{t_{e,i,j}[s]}[0,\ldots,u_{e,i,j}[s]].\]
	\item[Action] Do the same thing as in state \textsc{waiting-for-change}, except that instead of moving to state \textsc{diagonalized}, change to state \textsc{waiting-for-second-change-$n$}.
\end{description}

\item[{Waiting-for-second-change}-$n$]
In the previous state, $F$ was modified so that $R^\mc{B}$ is the same as it was originally below $v_{e,i,j}$. By the computation \eqref{eq3:S-R} from state \textsc{waiting-for-computation}, $S^{\mc{N}_e}$ must return to the same as it was originally below $u_{e,i,j}$, i.e., the element $p$ from the previous state must become connected to two elements of the second sort. Now, in state \textsc{next-subrequirement}, we had the condition (2). This condition implied that each of the elements $\mu_{e,i,j}^{n-1}+1,\ldots,\mu_{e,i,j}^n$ of $\mc{N}_e$ which is not either forced to be in $S^{\mc{N}_e}$ (by being of type $\langle 1,1 \rangle$) or forced to be not in $S^{\mc{N}_e}$ (by being of type $\langle 2,0 \rangle$) was in the subtree below $p$. But now $p$ is of type $\langle 2,0 \rangle$, and so there are no such elements below $p$. So each of $\mu_{e,i,j}^{n-1}+1,\ldots,\mu_{e,i,j}^n$ is either forced to be in $S^{\mc{N}_e}$ or forced to not be in $S^{\mc{N}_e}$. Now $b^n_{e,i,j}$ in $\mc{B}$ is currently in $R^{\mc{B}}$, and we can modify $F$ so that $b^n_{e,i,j}$ is not in $R^{\mc{B}}$. By the computation \eqref{eq4:R-S} below from the $n$th subrequirement, we cannot have $R^{\mc{B}} = \Phi_i^{S^{\mc{N}_e}}$.
\medskip{}
\begin{description}[font=\it,labelindent=-15pt,leftmargin = 0pt]
	\item[Requires attention] This requirement requires attention if
\[ S^{\mc{N}_e}_{s}[0,\ldots,\mu^n_{e,i,j}[s]] = S^{\mc{N}_e}_{t_{e,i,j}[s]}[0,\ldots,\mu^n_{e,i,j}[s]]\]
and also in $\mc{N}_{e,s}$, each of the elements $\mu^{n-1}_{e,i,j}[s]+1,\ldots,\mu^n_{e,i,j}[s]$ of $\mc{N}_{e,s}$ is either of type $\langle 1,1 \rangle$ or of type $\langle 2,0 \rangle$ in $\mc{N}_{e,s}$.
	\item[Action] Let $a = a_{e,i,j}[s]$ and $b = b^n_{e,i,j}[s]$. We have $F_s(a) = 0^{\ell_1-1} y$ for some even $y$, and $F_s(b) = 0^{\ell_1 - 1} y 0^{\ell_2} 1$. Let $\rho = 0^{\ell_1 - 1} y 0^{\ell_2}$ so that $F_s(b) = \rho 1$. Choose $z > 0$ even and larger than any even number we have encountered so far, and define $F_{s+1}$ by
\[ F_{s+1}(w)=\begin{cases}
\rho \concat z  \concat\sigma & F_{s}(w)= \rho \concat 1 \concat \sigma\\
F_{s}(w) & otherwise
\end{cases}. \]
This is moving the subtree below $\rho \concat 1$ to the subtree below $\rho \concat z$. For $w \in \dom(F_s)$, $w \neq b$, we have $F_{s}(w) \in R$ if and only if $F_{s+1}(w) \in R$. For $w = b$, we have $F_{s}(b) \in R$ and $F_{s+1}(b) \notin R$. Change the state to \textsc{diagonalized-$n$}.
\end{description}

\item[{Diagonalized-}$n$]
In this state, we have successfully satisfied $\mc{R}_{e,i,j}$ in the second way using $b^n_{e,i,j}$.
\medskip{}
\begin{description}[font=\it,labelindent=-15pt,leftmargin = 0pt]
	\item[Requires attention] The requirement never requires attention.
	\item[Action] None.
\end{description}

\end{description}

In order for a subrequirement $\mc{R}^n_{e,i,j}$ to require attention (in any state), there is a necessary (but not sufficient) condition: the parent requirement $\mc{R}_{e,i,j}$ must be in state \textsc{{next-subrequirement}}.
If this condition is satisfied, then the whether the requirement requires attention depends on its state:

\begin{description}[font=\sc]
\item[{Initialized}]
The subrequirement has been initialized, so $b^n_{e,i,j}$, $\mu^n_{e,i,j}$, and so on are all $0$. We define $b^n_{e,i,j}[s+1]$.
\medskip{}
\begin{description}[font=\it,labelindent=-15pt,leftmargin = 0pt]
	\item[Requires attention] The subrequirement always requires attention.
	\item[Action] Let $F_{s+1}$ extend $F_s$ by adding to its image the element $0^{\ell} 1$, where $\ell$ is large enough that $0^{\ell}$ has no children in $\ran(F_{s})$. Let $b^n_{e,i,j}[s+1]$ be such that $F_{s+1}({b^n_{e,i,j}[s+1]}) = 0^\ell 1$. Change the state to \textsc{waiting-for-computation}.
\end{description}

\item[{Waiting-for-computation}]
In the previous state, we defined $b^n_{e,i,j}$. We now wait for the computations below.
\medskip{}
\begin{description}[font=\it,labelindent=-15pt,leftmargin = 0pt]
	\item[Requires attention] This subrequirement requires attention if there are computations
\begin{equation}\label{eq4:R-S} R^\mc{B}_s[0,\ldots,b^n_{e,i,j}[s]] = \Phi_i^{S^{\mc{N}_e}_{s}}[0,\ldots,b^n_{e,i,j}[s]] \end{equation}
with use $\mu < s$, and
\begin{equation}\label{eq4:S-R} S^{\mc{N}_e}_{s}[0,\ldots,\mu] = \Phi_j^{R^\mc{B}_s}[0,\ldots,\mu] \end{equation}
with use $\nu < s$.
	\item[Action] Set $\mu^n_{e,i,j}[s+1] = \mu$, $\nu^n_{e,i,j}[s+1] = \nu$, and $\tau^n_{e,i,j}[s+1] = s$. We have $F_{s+1} = F_s$. Change the state to \textsc{next-subrequirement}.
\end{description}

\item[{Next-subrequirement}]
In the previous state, we found the computations \eqref{eq4:R-S} and \eqref{eq4:S-R}. This subrequirement is done acting, and the next subrequirement can begin.
\medskip{}
\begin{description}[font=\it,labelindent=-15pt,leftmargin = 0pt]
	\item[Requires attention] The subrequirement never requires attention.
	\item[Action] None.
\end{description}

\end{description}

Now we will say what happens when we say that we injure a requirement or subrequirement. When a requirement $\mc{R}_{e,i,j}$ is injured, it is returned to state \textsc{initialized} and its values $a_{e,i,j}$, $u_{e,i,j}$, $v_{e,i,j}$, and $t_{e,i,j}$ are set to $0$. Moreover, if it is in one of the states \textsc{waiting-for-first-change-$n$}, \textsc{waiting-for-second-change-$n$}, or \textsc{diagonalized-$n$}, then for $m \leq n$ set $b^m_{e,i,j}$, $\mu^m_{e,i,j}$, $\nu^m_{e,i,j}$, and $\tau^m_{e,i,j}$ to $0$.

When a subrequirement $\mc{R}^n_{e,i,j}$ is injured, it is returned to state \textsc{initialized}. Unless its parent requirement $\mc{R}_{e,i,j}$ is in one of the states \textsc{waiting-for-first-change-$m$}, \textsc{waiting-for-second-change-$m$}, or \textsc{diagonalized-$m$} for $m \geq n$, set $b^n_{e,i,j}$, $\mu^n_{e,i,j}$, $\nu^n_{e,i,j}$, and $\tau^n_{e,i,j}$ to $0$.  In this way, by being in one of these three states \textsc{waiting-for-first-change-$n$}, \textsc{waiting-for-second-change-$n$}, or \textsc{diagonalized-$n$} the parent requirement can take over control of the values associated to the subrequirements $\mc{R}^m_{e,i,j}$ for $m \leq n$ and protect them with its own priority level.

Set $D(\mc{B}_{s+1})$ to be the pullback along $F_{s+1}$ of the atomic and negated atomic formulas true of $\ran(F_{s+1})$ with Godel number at most $s$.

\vspace*{10pt}
\noindent\textit{End construction.}
\vspace*{10pt}

Each requirement and subrequirement, if it is not injured, only acts finitely many times. We must show that each requirement is satisfied.  Suppose not. Then there is a least requirement which is not satisfied. It is easy to see that a requirement $\mc{S}_i$ is always eventually satisfied, so let $\mc{R}_{e,i,j}$ be the least requirement which is not satisfied. Then $\mc{N}_e$ is a computable structure isomorphic to $\mc{B}$, $R^{\mc{B}} = \Phi_i^{S^{\mc{N}_e}}$, and $S^{\mc{N}_e} = \Phi_j^{R^{\mc{B}}}$.

There is some last stage at which $\mc{R}_{e,i,j}$ is injured. At this stage $\mc{R}_{e,i,j}$ and its subrequirements are in state \textsc{initialized}.

We will use the following fact implicitly throughout the rest of the proof. It is easy to prove.
\begin{lem}
If a requirement $\mc{R}_{e,i,j}$ is never injured after the stage $s$, then after the stage $s$, $F$ is only changed on the domain $[0, \ldots, v_{e,i,j}]$ by $\mc{R}_{e,i,j}$. If a subrequirement $\mc{R}^n_{e,i,j}$ is never injured after the stage $s$, then after the stage $s$, $F$ is only changed on the domain $[0, \ldots, \nu^n_{e,i,j}]$ by $\mc{R}_{e,i,j}$. Also, if a requirement $\mc{R}_{e,i,j}$ is in one of the states \textsc{waiting-for-first-change-$n$}, \textsc{waiting-for-second-change-$n$}, or \textsc{diagonalized-$n$} and is never injured after the stage $s$, then $F$ is only changed on $[0,\ldots,\nu^n_{e,i,j}]$ by $\mc{R}_{e,i,j}$.
\end{lem}

We will show that eventually $\mc{R}_{e,i,j}$ enters state \textsc{diagonalized} or \textsc{diagona\-lized-$n$} and diagonalizes against the two computations above, a contradiction. We will write $a$ for $a_{e,i,j}[s]$ since the $e$, $i$, and $j$ are fixed, and the value is never redefined since $\mc{R}_{e,i,j}$ is never injured. Similarly, we write $u$ for $u_{e,i,j}[s]$, $\mu^n$ for $\mu^n_{e,i,j}[s]$, and so on.

In state \textsc{initialized}, $\mc{R}_{e,i,j}$ always requires attention, so we will always define $a$ such that $F_s(a) = 0^\ell$ for some $\ell$ and move on to state \textsc{waiting-for-computation}.

Now because $R^{\mc{B}} = \Phi_i^{S^{\mc{N}_e}}$ and $S^{\mc{N}_e} = \Phi_j^{R^{\mc{B}}}$,
at some later stage $t$ we will have computations
\begin{equation}\label{R-S-t} R^\mc{B}_t[0,\ldots,a_{e,i,j}[t]] = \Phi_{i,t}^{S^{\mc{N}_e}_t}[0,\ldots,a_{e,i,j}[t]] \end{equation}
with use $u < t$, and
\begin{equation}\label{S-R-t} S^{\mc{N}_e}_t[0,\ldots,u] = \Phi_{j,t}^{R^\mc{B}_t}[0,\ldots,u]\end{equation}
with use $v < t$.
Then $u$, $v$, and $t$ will be defined to be these values and the requirement will move to state \textsc{next-subrequirement}.

Now we have three cases. First, it might be that at some later stage, $\mc{R}_{e,i,j}$ leaves state \textsc{next-subrequirement} and enters state \textsc{waiting-for-change}. Second, it might be that it enters the state \textsc{waiting-for-first-change-$n$}. Third, the requirement might never leave state \textsc{next-subrequirement}. We have to find a contradiction in each case.

\begin{case}
$\mc{R}_{e,i,j}$ leaves state \textsc{next-subrequirement} and enters state \textsc{waiting-for-change}.
\end{case}

At some stage $s_1 + 1 > t$, $R_{e,i,j}$ requires attention of the first kind. We change $F$ so that $F_{s_1+1}(a) \in R$ and change to state \textsc{waiting-for-change}.

Now since $R^{\mc{B}} = \Phi_i^{S^{\mc{N}_e}}$, at some stage $s_2 > s_1$, we have
\[ R^{\mc{B}}_{s_2}[0,\ldots,a] = \Phi_{i,s_2}^{S_{s_2}^{\mc{N}_e}}[0,\ldots,a]\]
and hence, by \eqref{R-S-t} and the fact that $F_{s_2}(a) = F_{s_1+1}(a) \neq F_t(a)$, we have
\begin{equation} \label{S-s2-t} S_{s_2}^{\mc{N}_e}[0,\ldots,u] \neq S_{t}^{\mc{N}_e}[0,\ldots,u]. \end{equation}
So the requirement requires attention at stage $s_2+1$.

We change $F$ so that $F_{s_2+1}(a) \notin R$ and change to state \textsc{diagonalized}. We make sure that
\begin{equation}\label{R-R-t} R_{s_2 + 1}^{\mc{B}} [0,\ldots,v] = R_{t}^{\mc{B}} [0,\ldots,v] \end{equation}
Since $S^{\mc{N}_e} = \Phi_j^{R^\mc{B}}$, at some stage $s_3 > s_2$, we have
\[ S^{\mc{N}_e}_{s_3}[0,\ldots,u] = \Phi_{j,s_3}^{R_{s_3}^{\mc{B}}}[0,\ldots,u]. \]
Then by \eqref{R-R-t} and \eqref{R-S-t} we have
\begin{equation}\label{S-s3-t} S^{\mc{N}_e}_{s_3}[0,\ldots,u] = S^{\mc{N}_e}_{t}[0,\ldots,u] \end{equation}
Combining this with \eqref{S-s2-t} we see that there is some $p$ in $[0,\ldots,u]$ with $p \notin S_t^{\mc{N}_e}$, $p \in S_{s_2}^{\mc{N}_e}$, and $p \notin S_{s_3}^{\mc{N}_e}$.

Now $s_1+1$ was the stage at which $\mc{R}_{e,i,j}$ required attention of the first kind while in state \textsc{next-subrequirement}. First of all, this means that
\[ S_{s_1}^{\mc{N}_e} [0,\ldots,u] = S_{t}^{\mc{N}_e} [0,\ldots,u] \]
and so $p \notin S_{s_1}^{\mc{N}_e}$.

Also, in $\mc{N}_{e,s_1}$, $p$ must be related by $V^{\mc{N}_e}$ to at least one element of the second sort (i.e., of one of the types $\langle 1,0 \rangle$, $\langle 1,1 \rangle$, or $\langle 2,0 \rangle$). Since $p \notin S_{s_1}^{\mc{N}_e}$, $p$ must be related to two elements of the second sort, so of type $\langle 2,0 \rangle$. But then the same is true at stage $s_2$, which contradicts the fact that $p \in S_{s_2}^{\mc{N}_e}$.

\begin{case}
$\mc{R}_{e,i,j}$ leaves state \textsc{next-subrequirement} and enters state \textsc{waiting-for-first-change-$n$} for some $n$.
\end{case}

The beginning of the proof of this case is the same as the beginning of the last case (with the states \textsc{waiting-for-change} and \textsc{diagonalized} replaced by \textsc{waiting-for-first-change-$n$} and by \textsc{waiting-for-second-change-$n$} respectively). Only the part of the proof after we conclude that $p \notin S_{s_1}^{\mc{N}_e}$, $p \in S_{s_2}^{\mc{N}_e}$, and $p \notin S_{s_3}^{\mc{N}_e}$ is different---this no longer leads to a contradiction. The requirement is in state \textsc{waiting-for-second-change-$n$}.

Since $\mc{R}_{e,i,j}$ required attention of the second kind at stage $s_1+1$, the subrequirement $R^n_{e,i,j}$ must have been in state \textsc{next-subrequirement}. It will have defined, for $m \leq n$, $b^m$, $\mu^m$, $\nu^m$, and $\tau^m$ with
\begin{equation}\label{R-S-tau} R^{\mc{B}}_{\tau^m}[0,\ldots,b^m] = \Phi_{i,\tau}^{S^{\mc{N}_e}_{\tau^m}}[0,\ldots,b^m] \end{equation}
with use $\mu^m$, and
\begin{equation}\label{S-R-tau-2} S^{\mc{N}_e}_{\tau^m}[0,\ldots,\mu^m] = \Phi_{j,\tau}^{R^{\mc{B}}_{\tau^m}}[0,\ldots,\mu^m] \end{equation}
with use $\nu^m$.

Now if $q$ is an element of $\mc{N}_e$ from among $\mu^{n-1} + 1,\ldots,\mu^n$, and in $\mc{N}_{e,s_1}$ it looks like $q$ is either of type $\langle 1,1 \rangle$ or type $\langle 2,0 \rangle$, then $q$ has the same type in the diagram at stage $s_3$. Now since $\mc{R}_{e,i,j}$ required attention of the second kind while in state \textsc{next-subrequirement} at stage $s_1+1$, any other element $q$ from among $\mu^{n-1} + 1,\ldots,\mu^n$ not satisfying either of the above conditions was in the subtree below $p$. At stage $s_3$, $p$ is related by is of type $\langle 2,0 \rangle$, and so we can see from the definition of $\mc{M}$ that any $q$ in the subtree below $p$ must be of type $\langle 1,1 \rangle$ or type $\langle 2,0 \rangle$.

Thus at some stage $s_4 > s_3$, the requirement requires attention. Each element $q$ from among $\mu^{n-1}+1,\ldots,\mu^n$ has been determined to be in either not in $S^{\mc{N}_e}$ if it is of type $\langle 2,0 \rangle$, or  in $S^{\mc{N}_e}$ if it is of type $\langle 1,1 \rangle$. So for all stages $s > s_4$, we have
\begin{equation}\label{S-s-s4-tau} S_s^{\mc{N}_e}[\mu^{n-1} + 1,\ldots,\mu^n] = S_{s_4}^{\mc{N}_e}[\mu^{n-1} + 1,\ldots,\mu^n] = S_{\tau^n}^{\mc{N}_e}[\mu^{n-1} + 1,\ldots,\mu^n] \end{equation}

We change $F$ so that $F_{s_4 + 1}(b^n) \notin R$ (while $F_{\tau^n}(b^n)$ was in $R$) and
\[ R_{s_4+1}[0,\ldots,\nu^{n-1}] = R_{\tau^n}[0,\ldots,\nu^{n-1}]. \]
This is also true with $s_4$ replaced by any $s \geq s_4$. By \eqref{S-R-tau-2}, for sufficiently large $s > s_4$ we have
\begin{equation}\label{S-s-t-1} S_{s}^{\mc{N}_e}[0,\ldots,\mu^{n-1} ] = S^{\mc{N}_e}_{\tau^n}[0,\ldots,\mu^{n-1}]. \end{equation}
Then since for all $s > s_4$ we have
\[ R_s^{\mc{B}}[0,\ldots,b^n] \neq R_{\tau^n}^{\mc{B}}[0,\ldots,b^n] \]
by \eqref{R-S-tau} for sufficiently large $s > s_4$ we have
\[ S_{s}[0,\ldots,\mu^{n} ] \neq S_{\tau^n}[0,\ldots,\mu^{n}]. \]
From this and \eqref{S-s-t-1}, we see that
\[ S_{s}[\mu^{n-1}+1,\ldots,\mu^{n} ] \neq S_{\tau^n}[\mu^{n-1}+1,\ldots,\mu^{n}]. \]
This contradicts \eqref{S-s-s4-tau}.

\begin{case}
$\mc{R}_{e,i,j}$ never leaves state \textsc{next-subrequirement}.
\end{case} 

Suppose to the contrary that $\mc{R}_{e,i,j}$ never requires attention while in state \textsc{next-subrequirement}. Then for all stages $s > t$, we have
\[ R_s^\mc{B}[0,\ldots,v] = R_t^\mc{B}[0,\ldots,v]. \]
Since $S^{\mc{N}_e} = \Phi_j^{R^\mc{B}}$ and using \eqref{S-R-t}, for sufficiently large stages $s$ we have
\begin{equation}\label{S-s-t-2} S_s^{\mc{N}_e}[0,\ldots,u] = S_t^{\mc{N}_e}[0,\ldots,u].\end{equation}
So for sufficiently large stages $s$, $\mc{R}^1_{e,i,j}$ always requires attention in state \textsc{initialized}. At some stage, each requirement of higher priority than $\mc{R}^1_{e,i,j}$ has acted, and so $\mc{R}^1_{e,i,j}$ is never injured after this point.

Then $\mc{R}^1_{e,i,j}$ will require attention and we will define $b^1$ such that $F_s(b^1) = 0^\ell 1$ for some $\ell$ and move on to state \textsc{waiting-for-computation}.

Because $R^{\mc{B}} = \Phi_i^{S^{\mc{N}_e}}$ and $S^{\mc{N}_e} = \Phi_j^{R^{\mc{B}}}$, at some later stage $\tau^1$ we will have computations
\[ R^\mc{B}_{\tau^1}[0,\ldots,b^1] = \Phi_{i,\tau^1}^{S^{\mc{N}_e}_{\tau^1}}[0,\ldots,b^1]\]
with use $\mu^1 < \tau^1$, and
\[ S^{\mc{N}_e}_{\tau^1}[0,\ldots,\mu^1] = \Phi_{j,\tau^1}^{R^\mc{B}_{\tau^1}}[0,\ldots,\nu^1] \]
with use $\nu^1 < \tau^1$.
Then $\mu^1$, $\nu^1$, and $\tau^1$ will be defined to be these values and the requirement will move to state \textsc{next-subrequirement}.

Continuing a similar argument, each of the subrequirements defines $b^n$, $\mu^n$, $\nu^n$, and $\tau^n$ such that
\begin{equation}\label{S-R-tau} S^{\mc{N}_e}_{\tau^n}[0,\ldots,\mu^n] = \Phi_{j,\tau^1}^{R^{\mc{B}_{\tau^n}}}[0,\ldots,\mu^n] \end{equation}
with use $\nu^n$.

Now we claim that $\mc{N}_e$ is not isomorphic to $\mc{M}$. First, $\mc{R}_{e,i,j}$ never requires attention of the first kind. Because of \eqref{S-s-t-2}, the only way this is possible is if there is an element $p$ from among $0,\ldots,u$ which is of type $\langle 0,0 \rangle$. That is, $p$ is the isomorphic image of $0^\ell$ in $\mc{M}$ for some $\ell$.

Let $n$ be arbitrary. By \eqref{S-R-tau} and the fact that $S^{\mc{N}_e} = \Phi_j^{R^{\mc{B}}}$, for sufficiently large stages $s$, we have
\[ S_{s}^{\mc{N}_e}[0,\ldots,\nu^n] = S_{\tau^n}^{\mc{N}_e}[0,\ldots,\nu^n]. \]
Then since $\mc{R}_{e,i,j}$ does not ever require attention of the second kind, there is $q_n$ from $\nu^{n-1} + 1,\ldots,\nu^n$ which is of type $\langle 0,0\rangle$ or type $\langle 1,0 \rangle$ and not in the subtree below $p$.

But one can easily see form the definition of $\mc{M}$ that there cannot be infinitely many such elements $q_n$, a contradiction.
\end{proof}

This proposition relativizes as follows:
\begin{corollary}\label{cor:prop-2-rel}
For every degree $\textbf{d}$, there is a copy $\mc{B}$ of $\mc{A}$ with $\mc{B} \leq_T \textbf{d}$ such that no copy $\mc{N}$ of $\mc{M}$ with $\mc{N} \leq_T \textbf{d}$ has $R^{\mc{B}} \oplus \textbf{d} \equiv_T S^{\mc{N}} \oplus \textbf{d}$.
\end{corollary}

From the relativized versions of these two propositions (Corollaries \ref{cor:prop-1-rel} and \ref{cor:prop-2-rel}), we get Theorem \ref{thm:incomp-dce}.

\chapter[Relations on the Naturals]{Degree Spectra of Relations on the Naturals}
\label{OmegaSection}
In this chapter, we will consider the special case of the structure $(\omega,<)$. We will generally be working with relations on the standard computable copy of this structure. Downey, Khoussainov, Miller, and Yu \cite{DowneyKhoussianovMillerYu09} studied relations on $\omega$ and though they were mostly interested in the degree spectra of non-computable relations, they showed that any computable unary relation $R$ on $(\omega,<)$ has a maximal degree in its degree spectrum, and this degree is either $0$ or $0'$. Knoll \cite{Knoll09} (and later independently Wright \cite{Wright13})  extended this to show:
\begin{thm}[{Knoll \cite[Theorem 2.2]{Knoll09}, Wright \cite[Theorem 1.2]{Wright13}}]
Let $R$ be a computable unary relation on $(\omega,<)$. Then either $R$ is intrinsically computable, or its degree spectrum consists of all $\Delta^0_2$ degrees.
\end{thm}

For relations which are not unary, Wright was able to show:
\begin{thm}[{Wright \cite[Theorem 1.3]{Wright13}}]
Let $R$ be a computable $n$-ary relation on $(\omega,<)$ which is not intrinsically computable. Then the degree spectrum of $R$ contains all of the c.e.\ degrees.
\end{thm}
Note that this is the same as the conclusion of Harizanov's Theorem \ref{Harizanov} for this particular structure. One could adapt Wright's proof to check Harizanov's effectiveness condition. All of these results relativize.

In the case of unary relations on the standard copy of $(\omega,<)$, Knoll was able to classify the possible degree spectra completely---they are either just the computable degree, or all $\Delta^0_2$ degrees. This suggests the following idea: study the partial order of degree spectra (on a cone) for relations on a single fixed structure (or class of structures). While in general, we know from \S \ref{DCESection} that there are incomparable d.c.e.\ degree spectra, this is not the case for unary relations on $(\omega,<)$; in fact, there are only two possible degree spectra for such relations. We know that there are at least three possible degree spectra for arbitrary relations on $(\omega, <)$: the computable degree, the c.e.\ degrees, and the $\Delta^0_2$ degrees. Are there more, and if so, is there a nice classification of the possible degree spectra?

To begin, we will study the intrinsically $\alpha$-c.e.\ relations on $\omega$. We need some definitions and a lemma about $\mc{L}_{\omega_1,\omega}$-definable sets which is implicit in Wright's work. A significant portion of the lemma coincides with Lemma 2.1 of \cite{Montalban09}.

A partial order is a \textit{well-quasi-order} if it is well-founded and has no infinite anti-chains (see, for example, \cite{Kruskal72} where Kruskal first noticed that the same notion of a well-quasi-order had been used in many different places under different names). Note that a total ordering which is well-ordered is a well-quasi-order. There are two simple constructions that we will use which produce a new well-quasi-ordering from an existing one. First, if $(A_i,\leq_i)$ for $i = 1,\ldots,n$ are partial orders, then their product $A_1 \times \cdots \times A_n$ is partially ordered by $\bar{a} \leq \bar{b}$ if and only if $a_i \leq_i b_i$ for each $i$. If each $(A_i,\leq_i)$ is a well-quasi-order, then so is the product. Second, if $(A,\leq)$ is a partial order, then let $\mc{P}_{< \omega}(A)$ be the set of finite subsets of $A$. Define $\leq_\mc{P}$ on this set by $U \leq_\mc{P} V$ if and only if for all $a \in U$, there is $b \geq a$ in $V$. If $(A,\leq)$ is a well-quasi-order, then $(\mc{P}_{\leq \omega}(A),\leq_{\mc{P}})$ is a well-quasi-order.

\begin{lem}\label{well-quasi-order}
The set of $\Sigma^{\infi}_1$-definable $n$-ary relations on $(\omega,<)$ over a fixed set of parameters $\bar{c}$ is a well-quasi-order under (reverse) inclusion. That is, there is no infinite strictly increasing sequence, and there are no infinite anti-chains.
\end{lem}
\begin{proof}
We will begin by considering the relations defined without parameters. Let $R$ be a $\Sigma^{\infi}_1$ $n$-ary relation. There are various possible orderings of an $n$-tuple, for example the entries may be increasing, decreasing, some entries may be equal to other entries, or many other possible orderings. But however many ways an $n$-tuple may be ordered, there are only finitely possibilities. We can write $R$ as the disjoint union of its restrictions to each of these orderings. Each of these restrictions is also $\Sigma^{\infi}_1$-definable. For any two such relations, we have $S \subseteq R$ if and only if each of these restrictions of $S$ is contained in the corresponding restriction of $R$. Hence inclusion on $n$-ary relations is the product order of the inclusion order on each of these restrictions, and so it suffices to show that each of these restrictions is a well-quasi-order. Without loss of generality, it suffices to show this for increasing tuples.

Let $R$ be $\Sigma^{\infi}_1$-definable relation on increasing $n$-tuples. Then $R$ is defined by a $\Sigma^{\infi}_1$ formula $\varphi(\bar{x})$. We may assume that $\varphi(\bar{x})$ can be written in the form
\[ \varphi(\bar{x}) = (x_1 < x_2 < \cdots < x_n) \wedge   \bigdoublevee_i (\exists \bar{y}) \psi_i(\bar{x},\bar{y}).\]
Now each of these disjuncts can be written in turn as the disjunct of finitely many formulas $\chi_{\bar{p}}(\bar{x})$ where $\bar{p}=(p_1,\ldots,p_n) \in \omega^{n}$ and $\chi_{\bar{p}}(\bar{x})$ is the formula which say that $x_1 < x_2 < \cdots < x_n$ and that there are at least $p_1$ elements less than $x_1$, $p_2$ elements between $x_1$ and $x_2$, and so on. So we can write $\varphi(\bar{x})$ as the disjunction of such formulas $\chi_{\bar{p}}$. For each $\bar{p} \in \omega^{n}$, let $D_{\bar{p}}$ be the set of solutions to $\chi_{\bar{p}}(\bar{x})$ in $\omega^n$. Then $R$ is the union of some of these sets $D_{\bar{p}}$.

Then note that $\bar{p} \mapsto D_{\bar{p}}$ is an order-maintaining bijection between the product order $\omega^{n}$ and the relations $D_{\bar{p}}$ ordered by \textit{reverse} inclusion (that is, $\bar{p} \leq \bar{q}$ in the product order if and only if $D_{\bar{p}} \supseteq D_{\bar{q}}$). Since $\omega^{n}$ is well-quasi-ordered, the set of relations $D_{\bar{p}}$ is also well-quasi-ordered. Thus any set of such relations $D_{\bar{p}}$ contains finitely many maximal elements ordered under inclusion (or minimal elements ordered under reverse inclusion), and each other relation is contained in one of those maximal elements. In particular, $R$ is the union of finitely many sets of the form $D_{\bar{p}}$. What we have done so far is the main content of Montalb\'an's work in \cite[Lemma 2.1]{Montalban09}.

Let $\bar{p} = (p_0,\ldots,p_n)$. Now note that the element $\hat{p} = (p_0,p_0+p_1+1,\ldots,p_0+\cdots+p_n+n-1)$ is in $D_{\bar{r}}$ if and only if $p_i \geq r_i$ for each $i$, in which case $D_{\bar{p}} \subseteq D_{\bar{r}}$. Now suppose that $D_{\bar{p}} \subseteq D_{\bar{q}_1} \cup \cdots \cup D_{\bar{q}_\ell}$. Then $\hat{p} \in D_{\bar{p}}$ and so $\hat{p} \in D_{\bar{q}_k}$ for some $k$. Hence for this $k$, $D_{\bar{p}} \subseteq D_{\bar{q}_k}$. Thus $D_{\bar{p}} \subseteq D_{\bar{q}_1} \cup \cdots \cup D_{\bar{q}_\ell}$ if and only if, for some $k$, $D_{\bar{p}} \subseteq D_{\bar{q}_k}$.

From this it follows that the partial order on $\Sigma^{\infi}_1$-definable relations on increasing $n$-tuples is isomorphic to the finite powerset order on $\omega^n$. This is a well-quasi-order.

Now we must consider formulas over a fixed tuple of parameters $\bar{c}$. Note that if $\varphi(\bar{x},\bar{y})$ and $\psi(\bar{x},\bar{y})$ are $\Sigma^{\infi}_1$ formulas with no parameters, and every solution of $\varphi(\bar{x},\bar{y})$ is a solution of $\psi(\bar{x},\bar{y})$, then every solution of $\varphi(\bar{c},\bar{y})$ is a solution of $\psi(\bar{c},\bar{y})$. Thus there is no infinite anti-chain of sets $\Sigma^{\infi}_1$-definable over $\bar{c}$ (as any such anti-chain would yield an anti-chain of sets definable without parameters).

Now suppose that there is a strictly increasing chain of sets definable over $\bar{c}$, $A_1 \subsetneq A_2 \subsetneq \cdots$, which are definable by $\Sigma^{\infi}_1$ formulas $\varphi(\bar{c},\bar{y})$. Let $B_1,B_2,\ldots$ be the corresponding sets definable by the formulas $\varphi(\bar{x},\bar{y})$. Then $B_1,B_2,\ldots$ cannot be a strictly increasing sequence, nor can it be an anti-chain. Thus for some $i < j$, $B_i \supseteq B_j$. But then $A_i \supseteq A_j$, and so $A_i = A_j$. This is a contradiction. Hence there is no strictly increasing chain of sets $\Sigma^{\infi}_1$-definable over $\bar{c}$. This completes the proof.
\end{proof}

From now on, by $\Sigma^{\infi}_1$-definable we mean definable with finitely many parameters. Then:

\begin{corollary}
Let $R$ be a $\Sigma^{\infi}_1$-definable relation on $(\omega,<)$. Then $R$ is defined by a finitary existential formula and $R$ is computable (in the standard copy of $(\omega,<)$). Moreover, $R$ is computable uniformly in the finitary existential formula and the tuple $\bar{c}$ over which it is defined.
\end{corollary}

Note that while $R$ is computable as a relation on $(\omega,<)$, it may not be intrinsically computable.

\begin{proof}
Suppose that $R$ is definable by a $\Sigma^{\infi}_1$ formula $\varphi(\bar{c},\bar{y})$. By the proof of the previous theorem, $\varphi(\bar{x},\bar{y})$ is a disjunction of finitely many formulas of the form $\chi_{\bar{p}}(\bar{x},\bar{y})$. Since in $(\omega,<)$ we can compute the number of elements between any two particular elements, the solution set in $\omega$ of each of these formulas $\chi_{\bar{p}}(\bar{x},\bar{y})$ is a computable set. Then $\varphi(\bar{x},\bar{y})$ is equivalent to a finitary existential formula and its solutions are computable; thus the same is true of $\varphi(\bar{c},\bar{y})$.
\end{proof}

Recall that $R$ is said to be intrinsically $\alpha$-c.e.\ if in all computable copies $\mc{B}$ of $\mc{A}$, $R^\mc{B}$ is $\alpha$-c.e. There is a theorem due to Ash and Knight (Propositions 3.2 and 3.3 of \cite{AshKnight96}), like that of Ash and Nerode \cite{AshNerode81}, which relates the notion of intrinsically $\alpha$-c.e.\ to formally $\alpha$-c.e.\ definitions. The theorem uses the following notion of $\alpha$-free for a particular (computable presentation of) an ordinal $\alpha$. Given tuples $\bar{c}$ and $\bar{a}$ in a structure, we say that $\bar{a}$ is $\alpha$-free over $\bar{c}$ if for any finitary existential formula $\varphi(\bar{c},\bar{x})$ true of $\bar{a}$, and any $\beta < \alpha$, there is a $\bar{a}'$ satisfying $\varphi(\bar{c},\bar{x})$ which is $\beta$-free over $\bar{c}$ and such that $\bar{a} \in R$ if and only if $\bar{a}' \notin R$. Then there are two theorems which together describe the intrinsically $\alpha$-c.e.\ relations:

\begin{thm}[{Ash-Knight \cite[Proposition 3.3]{AshKnight96}}]\label{formally-alpha-ce}
Let $\alpha$ be a computable ordinal. Let $\mc{A}$ be a computable structure and $R$ a computable relation on $\mc{A}$. Let $\bar{c}$ be a tuple. Suppose that no $\bar{a} \in R$ is $\alpha$-free over $\bar{c}$, and for each tuple $\bar{a}$ we can find a formula $\varphi(\bar{c},\bar{a})$ which witnesses this. Also, suppose that for any $\beta < \alpha$, we can effectively decide whether a tuple $\bar{a}$ is $\beta$-free over $\bar{c}$, and if not then we can find the witnessing formula $\varphi(\bar{c},\bar{a})$. Then $R$ is formally $\alpha$-c.e., that is, there are computable sequences $(\varphi_\beta(\bar{c},\bar{x}))_{\beta \leq \alpha}$ and $(\psi_\beta(\bar{c},\bar{x}))_{\beta \leq \alpha}$ such that
\begin{enumerate}
	\item for all $\beta \leq \alpha$ and tuples $\bar{a}$, if
	\[ \mc{A} \models \varphi_\beta(\bar{c},\bar{a}) \wedge \psi_\beta(\bar{c},\bar{a}) \]
	then for some $\gamma < \beta$,
	\[ \mc{A} \models \varphi_\gamma(\bar{c},\bar{a}) \vee \psi_\gamma(\bar{c},\bar{a}), \text{ and}\]

	\item $R$ is defined by
	\[ \bigdoublevee_{\beta < \alpha} (\varphi_\beta(\bar{c},\bar{x}) \wedge \neg \bigdoublevee_{\gamma < \beta} \psi_\beta(\bar{c},\bar{x}))\]
\end{enumerate}
\end{thm}

\begin{thm}[{Ash-Knight \cite[Proposition 3.2]{AshKnight96}}]\label{not-intrinsically-alpha-ce}
Let $\alpha$ be a computable ordinal, $\mc{A}$ be a computable structure, and $R$ a computable relation on $\mc{A}$. Suppose that for each tuple $\bar{c}$, we can find a tuple $\bar{a} \in R$ which is $\alpha$-free over $\bar{c}$. Suppose if $\beta \leq \alpha$, $\bar{a}$ is $\beta$-free, and given an existential formula $\varphi$, we can find the witness $\bar{a}'$ to $\beta$-freeness. Then $R$ is not intrinsically $\alpha$-c.e.
\end{thm}

There is also the usual relationship between relationship between relatively intrinsically $\alpha$-c.e.\ and formally $\alpha$-c.e. This was shown in the d.c.e.\ case by McCoy \cite{McCoy02}, and also independently McNicholl \cite{McNicholl00} who proved the result for all $n$-c.e.\ degrees. The general theorem for any ordinal $\alpha$ was shown by Ash and Knight in \cite{AshKnight00}.

\begin{thm}[{Ash-Knight \cite[Theorem 10.11]{AshKnight00}}]
Let $\alpha$ be a computable ordinal, $\mc{A}$ be a computable structure, and $R$ a computable relation on $\mc{A}$. Then $R$ is relatively intrinsically $\alpha$-c.e.\ if and only if it is formally $\alpha$-c.e.\ in the sense of Theorem \ref{formally-alpha-ce}.
\end{thm}

Finally, note that a relation $R$ is intrinsically $\alpha$-c.e.\ on a cone if and only if it has a formally $\alpha$-c.e.\ definition in the sense of Theorem \ref{formally-alpha-ce} for any countable ordinal $\alpha$ (not necessarily computable) and where the sequences $\varphi_\beta$ and $\psi_\beta$ are also not necessarily computable.

We will show that for the structure $(\omega,<)$, all of these notions coincide. We begin by showing that intrinsically $\alpha$-c.e.\ implies relatively intrinsically $\alpha$-c.e. We do this by checking in the next two lemmas that in $(\omega,<)$, the effectiveness conditions from Theorems \ref{not-intrinsically-alpha-ce} and \ref{formally-alpha-ce} are always satisfied.

\begin{lem}\label{lem:effect1}
The structure $(\omega,<)$ satisfies the effectiveness conditions of Theorem \ref{not-intrinsically-alpha-ce} for any computable relation $R$ and computable ordinal $\alpha$ (i.e., if for each tuple $\bar{c}$ there is $\bar{a}$ $\alpha$-free over $\bar{c}$, then we can find such an $\bar{a}$, etc.).
\end{lem}
\begin{proof}
Our argument will be very similar to the proof of Theorem 1.3 of \cite{Wright13}. Using similar arguments, we may assume that $R$ is a relation on increasing $n$-tuples. Suppose that for each tuple $\bar{c}$, there is a tuple $\bar{a}$ which is $\alpha$-free over $\bar{c}$.

We will show that if $\bar{a}$ is $\alpha$-free over the empty tuple, and each element of $\bar{a}$ is greater than each element of a tuple $\bar{c}$, then $\bar{a}$ is $\alpha$-free over $\bar{c}$. Let $\bar{c} = (c_1,\ldots,c_n)$ and let $m = \max(c_i)$. Then for any existential formula $\varphi(\bar{c},\bar{x})$ true of $\bar{a}$, there is a corresponding formula $\psi(\bar{x})$ which says that there are elements $y_0 < \cdots < y_m$ smaller than each element of $\bar{x}$, and that $y_{c_1},\ldots,y_{c_n},\bar{x}$ satisfy $\varphi$. Any solution of $\varphi(\bar{c},\bar{x})$ is also a solution of $\psi(\bar{x})$, and vice versa (note that if $\psi(\bar{b})$ holds for some $\bar{b}$, then $y_0 = 0,\ldots,y_m = m$ witness the existential quantifier, since $\varphi$ is existential).

Consider for each $\beta \leq \alpha$ the following sets:
\[ C_{\beta} = \{ \bar{a} : \bar{a} \text{ is not $\beta$-free over $\varnothing$}\} \]
and its complement
\[ F_{\beta} = \{ \bar{a} : \bar{a} \text{ is $\beta$-free over $\varnothing$}\}. \]
Now suppose that the set
\[ L_\alpha = \{ \min(\bar{a}) : \bar{a} \in F_{\alpha}\} \]
is not unbounded, say its maximum is $m$. (Here, $\min(\bar{a})$ is the least entry of $\bar{a}$.) Then consider the $(n-1)$-ary relations $R(0,x_1,\ldots,x_{n-1})$, $R(1,x_1,\ldots,x_{n-1})$, and so on up to $R(m,x_1,\ldots,x_{n-1})$. One of these relations must have $\alpha$-free tuples over any tuple $\bar{c}$. We may replace $R$ with this new relation. Continuing in this way, eventually we may assume that $L_\alpha$ is unbounded.

Now if $\bar{a}$ is not $\beta$-free, this is because there is a finitary existential formula $\varphi(\bar{x})$ true of $\bar{a}$ which witnesses that $\bar{a}$ is not $\beta$-free. Thus $C_{\beta}$ can be written in the form
\[ C_{\beta} = (D_{\beta} \cap R) \cup (E_{\beta} \cap \neg R) \]
where $D_{\beta}$ and $E_{\beta}$ are $\Sigma^{\infi}_1$-definable (and hence computable and definable by a finitary existential formula). Since $R$ is computable, $C_{\beta}$ is computable, and hence $F_{\beta}$ is computable as well for each $\beta$. Moreover, these are uniformly computable, because the sets $C_\beta$ are increasing and the $\Sigma^{\infi}_1$-definable sets are well-quasi-ordered. So there are $\beta_1,\ldots,\beta_m \leq \alpha$ such that $C_0 = \cdots = C_{\beta_1}$, $C_{\beta_1+1} = \cdots = C_{\beta_2}$, and so on until $C_{\beta_m+1} = \cdots = C_\alpha$, and each of these sets are computable.

Thus, for any $\bar{c}$, we can find some tuple $\bar{a}$ which is $\alpha$-free over $\bar{c}$ (by searching through $F_\alpha$ for a tuple $\bar{a}$ all of whose elements are greater than each element of $\bar{c}$).

Now suppose that $\bar{a}$ is $\beta$-free over a tuple $\bar{c}$ for some $\beta \leq \alpha$. Then for any $\gamma < \beta$ and existential formula $\varphi(\bar{c},\bar{x})$ true of $\bar{a}$, there is $\bar{b}$ satisfying $\varphi(\bar{c},\bar{x})$ and $\gamma$-free over $\bar{c}$. Note that there must be such a $\bar{b}$ all of whose elements are greater than each element of $\bar{c}$, since this true of $\bar{a}$, and that any such element of $F_\gamma$ is $\gamma$-free over $\bar{c}$. So we can compute such a $\bar{b}$.
\end{proof}

\begin{lem}\label{lem:effect2}
$(\omega,<)$ satisfies the effectiveness condition of Theorem \ref{formally-alpha-ce} for any computable relation $R$ and ordinal $\alpha$.
\end{lem}
\begin{proof}
Fix $\bar{c}$. Suppose that there are no $\alpha$-free tuples over $\bar{c}$. For each $\beta \leq \alpha$, let
\[ C_{\beta} = \{ \bar{a} : \bar{a} \text{ is not $\beta$-free over $\bar{c}$}\}. \]
Once again, $C_{\beta} \cap R$ and $C_\beta \cap \neg R$ are $\Sigma^{\infi}_1$-definable over $\bar{c}$, and so by the well-quasi-ordering of such sets, they are uniformly computable and the finitary existential definitions can be uniformly determined. This is enough to have the effectiveness condition of Theorem \ref{not-intrinsically-alpha-ce}.
\end{proof}

Now as a corollary of the previous two lemmas, we can prove the following fact.

\begin{corollary}\label{intrinsically-implies-relatively-intrinsically}
Suppose that $R$ is a computable relation on $(\omega,<)$ which is intrinsically $\alpha$-c.e. Then $R$ is relatively intrinsically $\alpha$-c.e.
\end{corollary}
\begin{proof}
Suppose that $R$ is computable and intrinsically $\alpha$-c.e.\ for some computable ordinal $\alpha$. We must show that $R$ is relatively intrinsically $\alpha$-c.e., that is, that $R$ is formally $\alpha$-c.e.

\begin{claim} There is a tuple $\bar{c}$ such that no $\bar{a} \in R$ is $\alpha$-free over $\bar{c}$.
\end{claim}
\begin{proof}
Suppose for a contradiction that for each tuple $\bar{c}$, there is $\bar{a} \in R$ which is $\alpha$-free over $\bar{c}$. By Theorem \ref{not-intrinsically-alpha-ce} and Lemma \ref{lem:effect1}, $R$ is not intrinsically $\alpha$-c.e. This contradicts the fact that $R$ is intrinsically $\alpha$-c.e.
\end{proof}

Now let $\bar{c}$ be as in the claim. By Theorem \ref{formally-alpha-ce} and Lemma \ref{lem:effect2}, $R$ is formally $\alpha$-c.e., and hence relatively intrinsically $\alpha$-c.e.
\end{proof}

Now we use this to show that the notions of intrinsically $\alpha$-c.e., relatively intrinsically $\alpha$-c.e., and intrinsically $\alpha$-c.e.\ on a cone all coincide for $(\omega,<)$. One can view this as saying that $(\omega,<)$ and every relation on it are ``natural.''

\begin{prop}
If $R$ is a relation on $(\omega, <)$ and $\alpha$ is any ordinal, then if $R$ is intrinsically $\alpha$-c.e.\ on a cone then $R$ is computable and intrinsically $m$-c.e.\ for some finite $m$.
\end{prop}
\begin{proof}
Now suppose that $R$ is a possibly non-computable relation on $(\omega,<)$, $\alpha$ is a possibly non-computable countable ordinal, and $R$ is intrinsically $\alpha$-c.e.\ on a cone. Then there are sets $A_\beta$ and $B_\beta$ for $\beta < \alpha$ which are $\Sigma^{\infi}_1$-definable over a tuple $\bar{c}$ such that
\[
R = \bigcup_{\beta < \alpha} (A_\beta - \bigcup_{\gamma < \beta} B_\gamma)
\]
and if $x \in A_\beta \cap B_\beta$ then for some $\gamma < \beta$, $X \in A_\gamma \cup B_\gamma$. We may replace $A_\beta$ by $\bigdoublevee_{\gamma \leq \beta} A_\beta$ for each $\beta$, and similarly for $B_\beta$. Then the sequences $A_\beta$ and $B_\beta$ are increasing in $\beta$. Since the $\Sigma^{\infi}_1$-definable relations form a well-quasi-order under inclusion, there is some sequence $0 = \beta_1,\ldots,\beta_m \leq \alpha$ such that $A_\gamma$ and $B_\gamma$ are constant on the intervals $[\beta_1 = 0,\beta_2)$, $[\beta_2,\beta_3)$, and so on up to $[\beta_m,\alpha]$. Otherwise, we could construct an infinite strictly increasing chain. So
\[
R = (A_{\beta_m} - B_{\beta_{m-1}}) \cup (A_{\beta_{m-1}} - B_{\beta_{m-2}}) \cup \cdots \cup A_{\beta_1}
\]
Suppose that $x \in A_{\beta_i} \cap B_{\beta_i}$. Then for some least $\gamma < \beta_i$, $x \in A_{\gamma} \cup B_{\gamma}$. By the minimality of $\gamma$, $\gamma \leq \beta_{i-1}$. Thus, for some $j < i$, $x \in A_{\beta_j} \cup B_{\beta_j}$.

Since each of these sets $A_{\beta_i}$ and $B_{\beta_j}$ is $\Sigma^{\infi}_1$-definable, they are all computable subsets of $\omega$ which are definable by a finitary existential (and hence $\Sigma^\comp_1$) formula. Thus $R$ is intrinsically $m$-c.e.\ and $R$ is computable.
\end{proof}

\begin{prop}\label{omega-equivalences}
Let $R$ be a relation on $(\omega,<)$. Then the following are equivalent for any computable ordinal $\alpha$:
\begin{enumerate}
	\item $R$ is intrinsically $\alpha$-c.e.\ and computable in $(\omega,<)$,
	\item $R$ is relatively intrinsically $\alpha$-c.e.,
	\item $R$ is intrinsically $\alpha$-c.e.\ on a cone.
\end{enumerate}
In this case, $R$ is intrinsically $m$-c.e.
\end{prop}
\begin{proof}
Corollary \ref{intrinsically-implies-relatively-intrinsically} shows that if $R$ is intrinsically $\alpha$-c.e. and computable in $(\omega,<)$, then $R$ is relatively intrinsically $\alpha$-c.e.

If $R$ is relatively intrinsically $\alpha$-c.e., then $R$ is formally $\alpha$-c.e., and hence intrinsically $\alpha$-c.e.\ on a cone.

The previous proposition shows that if $R$ is intrinsically $\alpha$-c.e.\ on a cone, then it is intrinsically $\alpha$-c.e.\ and computable in $(\omega,<)$.
\end{proof}

Now we will show that any intrinsically $\alpha$-c.e.\ relation on $(\omega,<)$ (which must be $m$-c.e.\ for some $m$) is intrinsically of c.e.\ degree. One example of such a relation is the intrinsically d.c.e.\ relation $S$ consisting of pairs $(a,b)$ which are separated by exactly one element. In any computable copy, $S$ computes the adjacency relation (which is always co-c.e.). Two elements $a$ and $b$ are adjacent if and only if there is some $c > b$ such that $c$ and $a$ are separated by a single element (which is $b$), and so the adjacency relation is c.e.\ in $S$ in any computable copy of $(\omega,<)$. Since it is always co-c.e., it is always computable in $S$. On the other hand, the adjacency relation in any computable copy computes an isomorphism between that copy and $(\omega,<)$, and hence computes $S$ in that copy. Thus we see that $S$ is intrinsically of c.e.\ degree. The proof of the following proposition is just a generalization of this idea.

Note that it is possible to have a relation which is formally $\Delta^0_2$ but not formally $\alpha$-c.e.\ for any computable ordinal $\alpha$. There is a formally $\Delta^0_2$ relation $R$ on a structure $\mc{A}$ whose degree spectrum, relativized to any degree \textbf{d}, consists of all of the $\Delta^0_2(\textbf{d})$ degrees. The degree spectra of any formally $\alpha$-c.e.\ relation does not consist, relative to any degree \textbf{d}, of all of the $\Delta^0_2(\textbf{d})$ degrees. This is in contrast to the $\Delta^0_2$ degrees, all of which are $\omega^2$-c.e.\ for some computable presentation of $\omega^2$ (see Theorem 8 of \cite{EpsteinHaasKramer81}).

\begin{prop}\label{intrinsically-ce}
Let $R$ be an  intrinsically $m$-c.e.\ $n$-ary relation on $(\omega,<)$ for some finite $m$. Then $R$ is intrinsically of c.e.\ degree.
\end{prop}
\begin{proof}
Let $A_1 \supsetneq \cdots \supsetneq A_m$ be intrinsically $\Sigma^0_1$ sets such that (depending on whether $m$ is odd or even) either
\[
R = (A_1 - A_2) \cup (A_3 - A_4) \cup \cdots \cup (A_{m-1} - A_m)
\]
or
\[
R = (A_1 - A_2) \cup (A_3 - A_4) \cup \cdots \cup (A_{m-2} - A_{m-1}) \cup A_m.
\]
We claim that in any computable copy $\mc{B} \cong (\omega,<)$, from $R^\mc{B}$ we can compute the successor relation on $\mc{B}$. The successor relation computes the isomorphism between $\mc{B}$ and $(\omega,<)$ and hence computes $R^\mc{B}$. The successor relation is intrinsically $\Pi^0_1$ and hence of c.e.\ degree, so this will suffice to complete the proof.

We will begin with a simple case to which we will later reduce the general case. For $\bar{p} \in \mc{N}^n$, let $E_{\bar{p}} = D_{0 \concat \bar{p}}$, that is, $E_{\bar{p}}$ consists of those $\bar{x}=(x_1,\ldots,x_n)$ such that there are $p_1$ points between $x_1$ and $x_2$, $p_2$ point between $x_2$ and $x_3$, and so on (but no restriction on the number of points before $x_1$).

\begin{claim}
Suppose that $B \subseteq A$ are intrinsically $\Sigma^0_1$ sets of the form
\[ A = E_{\bar{p}_1} \cup \cdots \cup E_{\bar{p}_{\ell_1}}\]
and
\[ B= E_{\bar{q}_1} \cup \cdots \cup E_{\bar{q}_{\ell_2}} .\]
Furthermore, suppose that $A \neq B$ and $B \neq \varnothing$, $A - B \subseteq R$, and $B \subseteq \neg R$. Then $R$ computes the successor relation.
\end{claim}
\begin{proof}
Choose $j$ such that $E_{\bar{q}_j}$ is not contained in any of the others, and is not equal to $E_{\bar{p}_i}$ for any $i$. Some such $j$ exists because $A - B$ is non-empty.

Now, in the proof of Lemma \ref{well-quasi-order} we showed that if
\[ E_{\bar{q}_j} \subseteq E_{\bar{p}_1} \cup \cdots \cup E_{\bar{p}_{\ell_2}} \]
then $E_{\bar{q}_j} \subset E_{\bar{p}_i}$ for some $i$ (though we showed this with $E$ replaced by $D$, the same result still applies here). Fix such an $i$. There must be some index $t \in \{1,\ldots,n\}$ such that $\bar{p}_i(t) < \bar{q}_j(t)$.

Then let
\[ \bar{u} = (u_1,\ldots,u_n) = (\bar{q}_j(1),\bar{q}_j(2),\ldots,\bar{q}_j(t)-1,\ldots,\bar{q}_j(n))\]
and
\[ \bar{v} = (v_1,\ldots,v_n) = \bar{q}_j. \]
Thus $u_i = v_i$ except for $i = t$, in which case $v_i = u_i + 1$. So $E_{\bar{v}} \subseteq E_{\bar{u}}$.

Now $E_{\bar{v}} = E_{\bar{q}_j} \subseteq B$ and $E_{\bar{u}} \subseteq E_{\bar{p}_i} \subseteq A$. Moreover, if, for any $\bar{w}$, $E_{\bar{v}} \subseteq E_{\bar{w}} \subseteq E_{\bar{u}}$, then either $E_{\bar{w}} = E_{\bar{u}}$ or $E_{\bar{w}} = E_{\bar{v}}$.

So, by the choice of $j$, we have
\[ E_{\bar{u}} - E_{\bar{v}} \subseteq A - B \subseteq R \]
and
\[ E_{\bar{v}} \subseteq B \subseteq \neg R. \]

Let $\mc{B}$ be a computable copy of $(\omega,<)$. Let $S$ be the successor function on $\mc{B}$. We claim that, using $R^{\mc{B}}$, we can compute $S$. Suppose that we wish to compute whether an element $y$ is the successor of $x$. We can, in a c.e.\ way, find out if $y$ is not the successor of $x$, so we just need to show how to recognize that $y$ is the successor of $x$ if this is the case. We can non-uniformly assume that we know some initial segment of $\mc{B}$, say the first $t + u_1 + \cdots + u_t + 1$ elements.

First, we must have $x < y$. If $x \leq t + u_1 + \cdots + u_t$, then we can non-uniformly decide whether $y$ is the successor of $x$. Otherwise, search for $z_1 < z_2 < \cdots < z_{t-1} < x < y = z_t < \cdots < z_n$ such that
\begin{enumerate}
	\item there $u_1$ elements between $z_1$ and $z_2$, $u_2$ elements between $z_2$ and $z_3$, and so on,
	\item there are $u_t-1$ elements between $z_{t-1}$ and $x$, and
	\item $(z_1,\ldots,z_n) \in R^\mc{B}$.
\end{enumerate}
Then since $(z_1,\ldots,z_n) \in D_{\bar{u}}$, and $D_{\bar{v}} \subseteq \neg R$,
\[ (z_1,\ldots,z_n) \in D_{\bar{u}} - D_{\bar{v}}. \]
In particular, there cannot be more than $v_{t}$ elements between $z_{t-1}$ and $z_t$. As $z_{t-1} < x < y = z_t$, and there are $u_t = v_t - 1$ elements between $z_{t-1}$ and $x$, $y$ is the successor of $x$. If $y$ is the successor of $x$ and $x \geq t + u_1 + \cdots + u_t$, then it is possible to find such elements $z_1,\ldots,z_n$.
\end{proof}

Now we will finish the general case. Let $\mc{B}$ be a computable copy of $(\omega,<)$ and let $S$ be the successor relation on $\mc{B}$. Suppose that
\[
R = (A_1 - A_2) \cup (A_3 - A_4) \cup \cdots \cup (A_{m-1} - A_m)
\]
or
\[
R = (A_1 - A_2) \cup (A_3 - A_4) \cup \cdots \cup (A_{m-2} - A_{m-1}) \cup A_m
\]
where each of these sets is an intrinsically $\Sigma^0_1$ set definable over constants $\bar{c}$.

Let $M$ be a constant greater than each entry of $\bar{c}$. Then we can non-uniformly know the successor relation restricted to $\{0,\ldots,M\}$. It remains to compute the successor relation on $\{M,M+1,\ldots\}^n$. View $R$ as a relation on $\{M,M+1,\ldots\}^n$, which is essentially the same as $\omega^n$ under the natural identification $M \to 0$, $M+1 \to 1$, and so on. So we have reduced to the case where (for some possibly smaller $m$) we have
\[
R = (A_1 - A_2) \cup (A_3 - A_4) \cup \cdots \cup (A_{m-1} - A_m)
\]
or
\[
R = (A_1 - A_2) \cup (A_3 - A_4) \cup \cdots \cup (A_{m-2} - A_{m-1}) \cup A_m
\]
and each of these sets is intrinsically $\Sigma^0_1$ and definable without parameters.

Now each $A_i$ is a union of sets of the form $D_{\bar{p}}$ for $\bar{p} \in \mathbb{N}^{n+1}$. Let $N$ be larger than the first entry $p_0$ of each of these tuples $\bar{p}$. Then make the same reduction as before to reduce to the case where each $A_i$ is a union of sets $E_{\bar{q}}$. 

Now if 
\[
R = (A_1 - A_2) \cup (A_3 - A_4) \cup \cdots \cup (A_{m-1} - A_m)
\]
then $A = A_{m-1}$ and $B = A_m$ are both as in the claim above. If 
\[
R = (A_1 - A_2) \cup (A_3 - A_4) \cup \cdots \cup (A_{m-2} - A_{m-1}) \cup A_m
\]
then $A = A_{m-1}$ and $B = A_m$ are as in the claim above, except with $R$ and $\neg R$ interchanged.
\end{proof}

So far, we still only know of three possible degree spectra for a relation on $(\omega,<)$: the computable degree, the c.e.\ degrees, and the $\Delta^0_2$ degrees. It is possible that there is another degree spectrum in between the c.e.\ degree and the $\Delta^0_2$ degrees, but we do not know whether such a degree spectrum exists. This is the main open question of this section:

\begin{question}\label{possible-spectra-omega}
Is there a relation on $(\omega,<)$ whose degree spectrum is strictly contained between the c.e.\ degrees and the $\Delta^0_2$ degrees on a cone?
\end{question}

This question appears to be a difficult one. If the answer to the question is no, that such a degree spectrum cannot exist, and the proof was not too hard, then it would probably be of the following form. Let $R$ be a relation on $(\omega,<)$. Working on a cone, there is a computable function $f$ such that given an index $e$ for a computable function $\varphi_e(x,s)$ of two variables which is total and gives a $\Delta^0_2$ approximation of a set $C$, $f(e)$ is an index for a structure $\mc{A}$ isomorphic to $(\omega,<)$ with $R^\mc{A} \equiv_T C$. If $\varphi_e(x,s)$ does not give a $\Delta^0_2$ approximation, then we place no requirements on $f(e)$. Moreover, if the coding is simple, then we will have indices $g_1(e)$ and $g_2(e)$ for the computations $R^\mc{A} \leq_T C$ and $C \leq_T R^\mc{A}$ respectively.  We capture this situation with the following definition:

\begin{defn}
Let $\mc{A}$ be a structure and $R$ a relation on $\mc{A}$. Let $\Gamma$ be a class of degrees indexed by some subset of $\omega$ which relativizes. Then the degree spectrum of $R$ is \textit{uniformly equal to $\Gamma$ on a cone} if, on a cone, there are computable functions $f$, $g_1$, and $g_2$ such that given an index $e$ for $C \in \Gamma$, the computable structure $\mc{A}$ with index $f(e)$ has 
\[ R^\mc{A} = \Phi_{g_1(e)}^C \text{ and } C = \Phi_{g_2(e)}^{R^\mc{A}}. \]
\end{defn}
\noindent We do not know of any relations which are not uniformly equal to their degree spectrum. 
\begin{question}\label{spectra-uniformly}
Must every relation obtain its degree spectrum uniformly on a cone?
\end{question}
This question is not totally precise, because we have not classified all possible degree spectra on a cone and so we do not have an indexing for each of them. But this question is precise for relations with particular degree spectra, such as the c.e.\ degrees or the $\Delta^0_2$ degrees.

Theorem \ref{thm:weird-relation-on-omega} below says that one of the two questions we have just introduced is resolved in an interesting way: either there is a relation $R$ on $(\omega,<)$ with degree spectrum all of the $\Delta^0_2$ degrees, but not uniformly, or the relation $R$ has a degree spectrum contained strictly in between the c.e.\ degrees and the $\Delta^0_2$ degrees. 

If the answer to Question \ref{possible-spectra-omega} is no, that there are only three possible degree spectra for a relation on $(\omega,<)$, then by Theorem \ref{thm:weird-relation-on-omega} the answer to Question \ref{spectra-uniformly} must also be no. If the answer to Question \ref{possible-spectra-omega} is yes, then the answer to Question \ref{spectra-uniformly} might be either yes or no.

If the answer to Question \ref{possible-spectra-omega} is yes, there are more than three possible degree spectra, then we can ask what sort of degree spectra are possible. We know that any relation which is intrinsically $\alpha$-c.e.\ has degree spectra consisting only of the c.e.\ degrees, but can there be a relation which is intrinsically of $\alpha$-c.e.\ \textit{degree}, while not being intrinsically of c.e.\ degree?

We are now ready to prove Theorem \ref{thm:weird-relation-on-omega}.

\begin{thm}\label{thm:weird-relation-on-omega}
There is a relation $R$ on $(\omega,<)$ whose degree spectrum is not uniformly equal to the $\Delta^0_2$ degrees on a cone, but strictly contains the c.e.\ degrees.
\end{thm}
\begin{proof}
We will begin by working computably, and everything will relativize. We will exhibit the relation $R$ and show that there is a computable function $h$ which, given indices $e,i,j$, produces a $\Delta^0_2$ set $C$ with index $h(e,i,j)$ such that if $e$ is the index for a computable structure $\mc{A}$ isomorphic to $(\omega,<)$, then either $R^\mc{A} \neq \Phi_{i}^C$ or $C \neq \Phi_{j}^{R^\mc{A}}$. Note that $h(e,i,j)$ will be total and will always give an index for a $\Delta^0_2$ set, even if $e$ is not an index of the desired type. We will also show that $R$ is not intrinsically of c.e.\ degree. The construction of $R$ and $h$ will follow later, but first we will show that such an $R$ has a degree spectrum which is not uniformly equal to the $\Delta^0_2$ degrees.

Suppose to the contrary that there are computable functions $f$, $g_1$, and $g_2$ such that, given an index $e$ for a function $\varphi_e(x,s)$ which is total and gives a $\Delta^0_2$ approximation of a set $C$, $f(e)$ is the index of a computable copy $\mc{A}$ of $(\omega,<)$ with \[ R^\mc{A} = \Phi_{g_1(e)}^C \text{ and } C = \Phi_{g_2(e)}^{R^\mc{A}}.\]
Then we would like to take the composition $\theta(e) = h(f(e),g_1(e),g_2(e))$, except that this may not be total when applied to indices $e$ where $\varphi_e(x,s)$ is either not total or not a $\Delta^0_2$ approximation. So instead define $h$ as follows. Given an input $e$, we will define a $\Delta^0_2$ set uniformly by giving an approximation by stages. Try to compute $f(e)$, $g_1(e)$, and $g_2(e)$, and while these does not converge, make the $\Delta^0_2$ approximation equal to zero. When they do converge, compute $h(f(e),g_1(e),g_2(e))$ (since $h$ is total, this will always converge and give an index for a $\Delta^0_2$ set) and have our approximation instead follow the $\Delta^0_2$ set with index $h(f(e),g_1(e),g_2(e))$. Thus $\theta(e) = h(f(e),g_1(e),g_2(e))$ whenever the right hand side is defined, and is an index for the empty set otherwise.

Now, by the recursion theorem, there is a fixed point $e$ of $\theta$, so that
\[ \varphi_{\theta(e)}(x,s) = \varphi_e(x,s). \]
Now the left hand side is always total and is a $\Delta^0_2$ approximation, so the same is true of the right hand side. Thus $\theta(e) = h(f(e),g_1(e),g_2(e))$. Let $C$ be the $\Delta^0_2$ set with this approximation, so that $C$ has indices $e$ and $\theta(e)$. Thus $f(e)$ is the index of a computable copy $\mc{A}$ of $(\omega,<)$ with \[ R^\mc{A} = \Phi_{g_1(e)}^C \text{ and } C = \Phi_{g_2(e)}^{R^\mc{A}}.\] But then by definition of $h$, \[ R^\mc{A} \neq \Phi_{g_1(e)}^C \text{ or } C \neq \Phi_{g_2(e)}^{R^\mc{A}}.\] This is a contradiction. Hence no such functions $f$, $g_1$, and $g_2$ can exist, and the degree spectrum of $R$ is not uniformly equal to the $\Delta^0_2$ degrees on a cone.

\vspace*{10pt}
\noindent \emph{Construction of $R$.}
\vspace*{10pt}

We will begin by defining $a_n$, $b_n$, and $c_n$ with $a_1 < b_1 = c_1 < a_2 < b_2 < c_2 < a_3 < \cdots$. Begin with $a_1 = 0$. Let $b_n = a_n + n + 1$, $c_n = b_n + a_n$, and $a_{n+1} = c_n + n + 1$. This divides $\omega$ up into disjoint intervals $[a_n,b_n)$, $[b_n,c_n)$, and $[c_n,a_{n+1})$ as $n$ varies. Note that $b_1 = 2 = c_1$, and so the interval $[b_1,c_1)$ is empty. Every other interval is non-empty. Also, the length of the interval $[b_n,c_n)$ is the same as the length of the interval $[0,a_n)$.

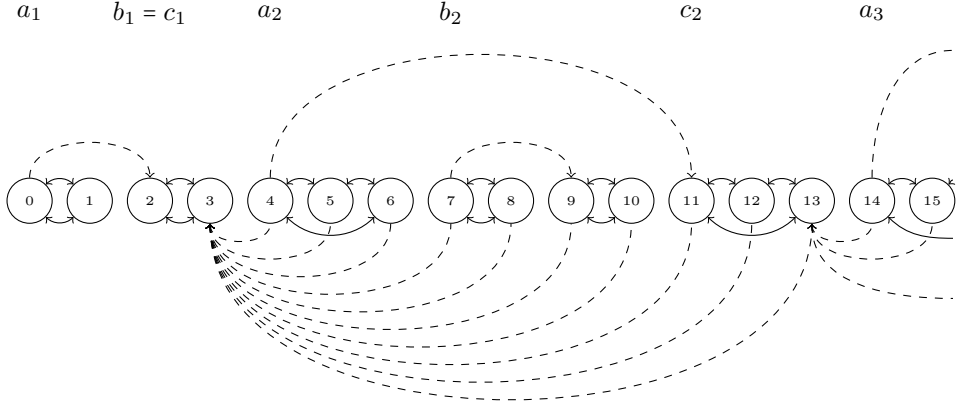
\begin{figure}[t]
\begin{center}
\begin{tikzpicture}[->]
\node[circle,draw=black,label=below:$$,minimum width=.6cm, minimum height=.6cm] (box0) at (0,0) {{\tiny$0$}};
\node (a1) at (0,2.5) {{$a_1$}};
\node[circle,draw=black,label=below:$$,minimum width=.6cm, minimum height=.6cm] (box1) at (0.8,0) {{\tiny$1$}};
\node[circle,draw=black,label=below:$$,minimum width=.6cm, minimum height=.6cm] (box2) at (1.6,0) {{\tiny$2$}};
\node (c1) at (1.6,2.5) {{$b_1 = c_1$}};
\node[circle,draw=black,label=below:$$,minimum width=.6cm, minimum height=.6cm] (box3) at (2.4,0) {{\tiny$3$}};
\node[circle,draw=black,label=below:$$,minimum width=.6cm, minimum height=.6cm] (box4) at (3.2,0) {{\tiny$4$}};
\node (a2) at (3.2,2.5) {{$a_2$}};
\node[circle,draw=black,label=below:$$,minimum width=.6cm, minimum height=.6cm] (box5) at (4.0,0) {{\tiny$5$}};
\node[circle,draw=black,label=below:$$,minimum width=.6cm, minimum height=.6cm] (box6) at (4.8,0) {{\tiny$6$}};
\node[circle,draw=black,label=below:$$,minimum width=.6cm, minimum height=.6cm] (box7) at (5.6,0) {{\tiny$7$}};
\node (b2) at (5.6,2.5) {{$b_2$}};
\node[circle,draw=black,label=below:$$,minimum width=.6cm, minimum height=.6cm] (box8) at (6.4,0) {{\tiny$8$}};
\node[circle,draw=black,label=below:$$,minimum width=.6cm, minimum height=.6cm] (box9) at (7.2,0) {{\tiny$9$}};
\node[circle,draw=black,label=below:$$,minimum width=.6cm, minimum height=.6cm] (box10) at (8.0,0) {{\tiny$10$}};
\node[circle,draw=black,label=below:$$,minimum width=.6cm, minimum height=.6cm] (box11) at (8.8,0) {{\tiny$11$}};
\node (c2) at (8.8,2.5) {{$c_2$}};
\node[circle,draw=black,label=below:$$,minimum width=.6cm, minimum height=.6cm] (box12) at (9.6,0) {{\tiny$12$}};
\node[circle,draw=black,label=below:$$,minimum width=.6cm, minimum height=.6cm] (box13) at (10.4,0) {{\tiny$13$}};
\node[circle,draw=black,label=below:$$,minimum width=.6cm, minimum height=.6cm] (box14) at (11.2,0) {{\tiny$14$}};
\node (a3) at (11.2,2.5) {{$a_3$}};
\node[circle,draw=black,label=below:$$,minimum width=.6cm, minimum height=.6cm] (box15) at (12,0) {{\tiny$15$}};

\path[bend left=45] (box0) edge[<->] node {{$$}} (box1);
\path[bend left=45] (box1) edge[<->] node {{$$}} (box0);

\path[bend left=45] (box3) edge[<->] node {{$$}} (box2);
\path[bend left=45] (box2) edge[<->] node {{$$}} (box3);

\path[bend left=45] (box4) edge[<->] node {{$$}} (box5);
\path[bend left=45] (box5) edge[<->] node {{$$}} (box6);
\path[bend left=45] (box6) edge[<->] node {{$$}} (box4);

\path[bend left=45] (box7) edge[<->] node {{$$}} (box8);
\path[bend left=45] (box8) edge[<->] node {{$$}} (box7);

\path[bend left=45] (box9) edge[<->] node {{$$}} (box10);
\path[bend left=45] (box10) edge[<->] node {{$$}} (box9);

\path[bend left=45] (box11) edge[<->] node {{$$}} (box12);
\path[bend left=45] (box12) edge[<->] node {{$$}} (box13);
\path[bend left=45] (box13) edge[<->] node {{$$}} (box11);

\path[bend left=90,dashed] (box0) edge[->] node {{$$}} (box2);
\path[bend left=90,dashed] (box4) edge[->] node {{$$}} (box11);
\path[bend left=90,dashed] (box7) edge[->] node {{$$}} (box9);

\path[bend right=90,dashed] (box3) edge[<-] node {{$$}} (box4);
\path[bend right=90,dashed] (box3) edge[<-] node {{$$}} (box5);
\path[bend right=90,dashed] (box3) edge[<-] node {{$$}} (box6);
\path[bend right=90,dashed] (box3) edge[<-] node {{$$}} (box7);
\path[bend right=90,dashed] (box3) edge[<-] node {{$$}} (box8);
\path[bend right=90,dashed] (box3) edge[<-] node {{$$}} (box9);
\path[bend right=90,dashed] (box3) edge[<-] node {{$$}} (box10);
\path[bend right=90,dashed] (box3) edge[<-] node {{$$}} (box11);
\path[bend right=90,dashed] (box3) edge[<-] node {{$$}} (box12);
\path[bend right=90,dashed] (box3) edge[<-] node {{$$}} (box13);


\path[bend left=45] (box14) edge[<->] node {{$$}} (box15);
\node (box16up) at (12.4,0.275) {{$$}};
\path[out=45,in=180] (box15) edge[<-] node {{$$}} (box16up);
\node (box16down) at (12.4,-0.5) {{$$}};
\path[out = 180, in = 315] (box16down) edge[->] node {{$$}} (box14);
\node (box17up) at (12.4,2) {{$$}};
\path[out=90,in=180,dashed] (box14) edge[-] node {{$$}} (box17up);

\path[bend right=90,dashed] (box13) edge[<-] node {{$$}} (box14);
\path[bend right=90,dashed] (box13) edge[<-] node {{$$}} (box15);
\node (box16down2) at (12.4,-1.3) {{$$}};
\path[out=270,in=180,dashed] (box13) edge[<-] node {{$$}} (box16down2);

\end{tikzpicture}
\end{center}
\caption{The relation $R$ on the first sixteen elements of $(\omega,<)$. The arrows $\leftrightarrow$ are those from the cycles $a_n \leftrightarrow a_n + 1 \leftrightarrow a_n + 2 \leftrightarrow \cdots \leftrightarrow a_n + n \leftrightarrow a_n$
and
$c_n \leftrightarrow c_n + 1 \leftrightarrow c_n + 2 \leftrightarrow \cdots \leftrightarrow c_n + n \leftrightarrow c_n.$
The arrows $\dashrightarrow$ (which curve above) are those between $a_n$ and $c_n$ for some $n$. The arrows $\dashleftarrow$ (which curve below) are those from $y$ to $c_n+n$ for $y \geq a_{n+1}$.}
\label{fig:omega}
\end{figure}

The relation $R$ will be a binary relation which we can interpret as a directed graph. For each $n$, we will have cycles of edges
\[ a_n \leftrightarrow a_n + 1 \leftrightarrow a_n + 2 \leftrightarrow \cdots \leftrightarrow a_n + n \leftrightarrow a_n \]
and
\[ c_n \leftrightarrow c_n + 1 \leftrightarrow c_n + 2 \leftrightarrow \cdots \leftrightarrow c_n + n \leftrightarrow c_n. \]
These edges all go in both directions; i.e., there is an edge from $a_n$ to $a_n + 1$ and from $a_n+1$ to $a_n$.

Now add an edge from $a_n$ to $c_n$. These edges are directed, and go from the smaller element to the larger element. Also, add edges from $y$ to $c_n+n$ for all $y \geq a_{n+1}$. These edges are also directed, but go from the larger element to the smaller element. By looking at whether an edge goes in both directions, in the increasing direction, or in the decreasing direction, we can decide what type of edge it is (i.e., is it from a cycle, from $a_n$ to $c_n$ for some $n$, or from $y$ to $c_n+n$ for some $n$ and some $y \geq a_{n+1}$).

For $x,y \in [b_n,c_n)$, put an edge from $x$ to $y$ if and only if there is an edge from $(x - b_n)$ to $(y - b_n)$. So the relation $R$ on the interval $[b_n,c_n)$ looks the same as it does on the interval $[0,a_n)$. This completes the definition of the relation $R$. Note that $R$ is computable. Figure \ref{fig:omega} shows the relation $R$ on an initial segment of $\omega$.

We say that elements $\bar{y} = (y_0,y_1,\ldots,y_m)$ of $(\omega,<)$ form an \textit{$m$-cycle} (for $m \geq 1$) if $y_0 < y_1 < \cdots < y_m$ and there is a cycle
\[ y_0 \leftrightarrow y_1 \leftrightarrow \cdots \leftrightarrow y_m \leftrightarrow y_0 \]
and there are no other edges between any of the $y_i$. An $m$-cycle is really an $m+1$-cycle of the graph. Note that every $m$-cycle is either
\[ a_m \leftrightarrow a_m + 1 \leftrightarrow \cdots \leftrightarrow a_m + m \leftrightarrow a_m,\]
or
\[ c_m \leftrightarrow c_m + 1 \leftrightarrow \cdots \leftrightarrow c_m + m \leftrightarrow c_m,\]
or contained in $[b_n,c_n)$ for some $n > m$.

\begin{remark}
Each element $x \in \omega$ is part of exactly one $m$-cycle for some $m$.
\end{remark}
\begin{proof}
This can easily be seen by an induction argument on the $n$ such that $x \in [a_n,c_n)$. If $x$ is in $[a_n,b_n)$ or $[c_n,a_{n+1})$ then this is obvious, and if $x$ is in $[b_n,c_n)$ then this follows by the induction hypothesis.
\end{proof}

If $\bar{x} = (x_0,\ldots,x_m)$ and $\bar{y} = (y_0,\ldots,y_m)$ are $m$-cycles and there are edges between $x_0$ and $y_0$ and vice versa, then we say that $\bar{x}$ and $\bar{y}$ are a \textit{matching pair} of $m$-cycles. Each $m$-cycle is part of a matching pair of $m$-cycles.

\begin{remark}\label{rem:k-cycle-1}
If $\bar{x} = (x_0,\ldots,x_m)$ and $\bar{y} = (y_0,\ldots,y_m)$ are a matching pair of $m$-cycles, then there is some $n$ such that both $\bar{x}$ and $\bar{y}$ are contained in $[a_n,a_{n+1})$.
\end{remark}
\begin{proof}
Since $\bar{x}$ and $\bar{y}$ form a matching pair, there is an edge from $x_0$ to $y_0$ but not vice versa. We can see by the definition of $R$ that either $\bar{x} = (a_m,\ldots,a_m+m)$ and $\bar{y} = (c_m,\ldots,c_m+m)$, or $\bar{x}$ and $\bar{y}$ are both contained in $[b_n,c_n)$ for some $n$.
\end{proof}

\begin{remark}\label{rem:k-cycle-2}
If $\bar{u} = (u_0,\ldots,u_m)$ and $\bar{v} = (v_0,\ldots,v_m)$ are a matching pair of $m$-cycles, and $\bar{x} = (x_0,\ldots,x_m)$ and $\bar{y} = (y_0,\ldots,y_m)$ are another matching pair of $m$-cycles (for the same $m$), then the relation $R$ restricted to the interval $[u_0,v_m]$ is the same as the relation $R$ restricted to the interval $[x_0,y_m]$. In particular, the lengths of these intervals are the same: $v_m - u_0 = y_m - x_0$.
\end{remark}
\begin{proof}
We may assume that $\bar{u} = (a_m,\ldots,a_m+m)$ and $\bar{v} = (c_m,\ldots,c_m+m)$. Suppose to the contrary that there are $\bar{x} = (x_0,\ldots,x_m)$ and $\bar{y} = (y_0,\ldots,y_m)$ a matching pair of $m$-cycles such that the relation $R$ restricted to the interval $[u_0,\ldots,v_m]$ is \textit{not} the same as the relation $R$ restricted to the interval $[x_0,\ldots,y_m]$. Assume that $x_0$ is least with this property.

Now by the previous fact, $\bar{x}$ and $\bar{y}$ are contained in $[a_n,a_{n+1})$ for some $n$. We must have $x_0 > c_m+m$, and so $n > m$. Then $\bar{x}$ and $\bar{y}$ are contained in $[b_n,c_n)$. Then by definition of $R$, $\bar{x}' = (x_0 - b_n,\ldots,x_m-b_n)$ and $\bar{y}' = (y_0 - b_n,\ldots,y_m - b_n)$ are a matching pair of $m$-cycles and $R$ restricted to the interval $[x_0,\ldots,y_m]$ is the same as $R$ restricted to the interval $[x_0-b_n,\ldots,y_m-b_n]$. But by the induction hypothesis, $R$ restricted to the interval $[x_0-b_n,\ldots,y_m-b_n]$ is the same as $R$ restricted to the interval $[u_0,\ldots,v_m]$. This contradiction finishes the proof.
\end{proof}

In any computably copy $\mc{A}$ of $(\omega,<)$, using $R^\mc{A}$ as an oracle, we can compute for each $x \in \mc{A}$ the unique $m$-cycle in which $x$ is contained, $x$'s position in that cycle, and we can compute the other $m$-cycle with which this $m$-cycle forms a matching pair.

\vspace*{10pt}
\noindent \emph{$R$ is not intrinsically of c.e.\ degree.}
\vspace*{10pt}

Let $\bar{c}$ be a tuple and $a_n$ be such that $\bar{c} < a_n$. We will show that the tuple $\bar{a} = (a_n,\ldots,a_n + n)$ is d-free over $\bar{c}$. First, we will introduce some notation. Given a tuple $\bar{x} = (x_0,\ldots,x_n)$, and $r \in \omega$, write $\bar{x} + r$ for $(x_0 + r,\ldots,x_n + r)$.

Recall what it means for $\bar{a}$ to be d-free over $\bar{c}$: for every $\bar{b}$ and existential formula $\varphi(\bar{c},\bar{u},\bar{v})$ true of $\bar{a},\bar{b}$, there are $\bar{a}'$ and $\bar{b}'$ which satisfy $\varphi(\bar{c},\bar{u},\bar{v})$ such that $R$ restricted to tuples from $\bar{c}\bar{a}'$ is not the same as $R$ restricted to tuples from $\bar{c}\bar{a}$ and also such that for every existential formula $\psi(\bar{c},\bar{u},\bar{v})$ true of them, there are $\bar{a}'',\bar{b}''$ satisfying $\psi$ and such that $R$ restricted to $\bar{c}\bar{a}''\bar{b}''$ is the same as $R$ restricted to $\bar{c}\bar{a}\bar{b}$.

Figure \ref{fig:dfree} shows the choices of $\bar{a}'$, $\bar{b}'$, etc. in one particular example. Given $\bar{b}$ and some existential formula $\varphi(\bar{c},\bar{u},\bar{v})$ true of $\bar{a},\bar{b}$, we may assume that $\varphi$ is quantifier-free by expanding $\bar{b}$. Then $\varphi$ just says that $\bar{c}\bar{a}\bar{b}$ are ordered in some particular way. Now, some of the entries of $\bar{b}$ are less than $a_n$, and the rest are greater than or equal to $a_n$. Rearranging $\bar{b}$ as necessary, write $\bar{b} = \bar{b}_1 \bar{b}_2$ where each entry of $\bar{b}_1$ is less than $a_n$, and each entry of $\bar{b}_2$ is greater than or equal to $a_n$. Let $\bar{a}' = \bar{a}+1$ (recall that this is shifting each entry by one) and $\bar{b}'$ be $\bar{b}_1 \bar{b}_2'$ where $\bar{b}_2' = \bar{b}_2 + 1$. Then $\bar{c}\bar{a}'\bar{b}'$ are ordered in the same way as $\bar{c}\bar{a}\bar{b}$ and hence still satisfy $\varphi$. Also, the valuation of $R$ on $\bar{a}$ is different from that of $R$ on $\bar{a}'$, since there are edges between $a_n$ and $a_n+n$, but not between $a_n + 1$ and $a_n + n + 1 = b_n$.

Now suppose that $\psi(\bar{c},\bar{u},\bar{v})$ is some further existential formula true of $\bar{c}\bar{a}'\bar{b}'$. Let $\bar{e}$ be the witnesses to the existential quantifiers, and $\chi(\bar{c},\bar{u},\bar{v},\bar{w})$ be the quantifier-free formula which holds of $\bar{c}\bar{a}'\bar{b}'\bar{e}$. Write $\bar{e} = \bar{e}_1 \bar{e}_2$ with each entry of $\bar{e}_1$ less than $a_n$, and each entry of $\bar{e}_2$ greater than or equal to $a_n$. Let $k$ be such that $b_k$ is larger than all of the entries of $\bar{a}$ and $\bar{b}$. Let $\bar{a}'' = \bar{a} + b_{k}$, $\bar{b}'' = \bar{b}_1\bar{b}_2''$ where $\bar{b}_2'' = \bar{b}_2 + b_k$, and $\bar{e}' = \bar{e}_1 \bar{e}_2'$ where $\bar{e}_2' = \bar{e} + b_k$. Then $\bar{c} \bar{a}'' \bar{b}'' \bar{e}'$ is ordered in the same way as $\bar{c} \bar{a}' \bar{b}' \bar{e}$, and so $\bar{c} \bar{a}'' \bar{b}'' \bar{e}'$ satisfies $\chi$. Thus $\bar{c} \bar{a}'' \bar{b}''$ satisfies $\psi$. We need to show that the relation $R$ restricted to $\bar{c} \bar{a}'' \bar{b}''$ is the same as $R$ restricted to $\bar{c}\bar{a}\bar{b}$.

Note that $\bar{a}''$ and $\bar{b}_2''$ are contained in the interval $[b_k,c_k)$. By definition of $R$, there is an edge between $x$ and $y$ in $[b_k,c_k)$ if and only if there is an edge between $x - b_k$ and $y - b_k$. Since $\bar{a}'' = \bar{a}+b_k$ and $\bar{b}_2'' = \bar{b}_2 + b_k$, $R$ restricted to $\bar{a}''$ and $\bar{b}_2''$ is the same as $R$ restricted to $\bar{a}$ and $\bar{b}_2$. Now $\bar{c}$ and $\bar{b}_1$ are contained in the interval $[0,a_n)$. There are no edges from some $x < a_n$ to some $y \geq a_n$. Note from the construction of $R$ that there is an edge from $x \geq a_n$ to $y < a_n$ if and only if there is an edge from \textit{every} $z \geq a_n$ to $y$. This completes the proof that $R$ restricted to $\bar{c}\bar{a}\bar{b}$ is the same as $R$ restricted to $\bar{c}\bar{a}''\bar{b}''$.

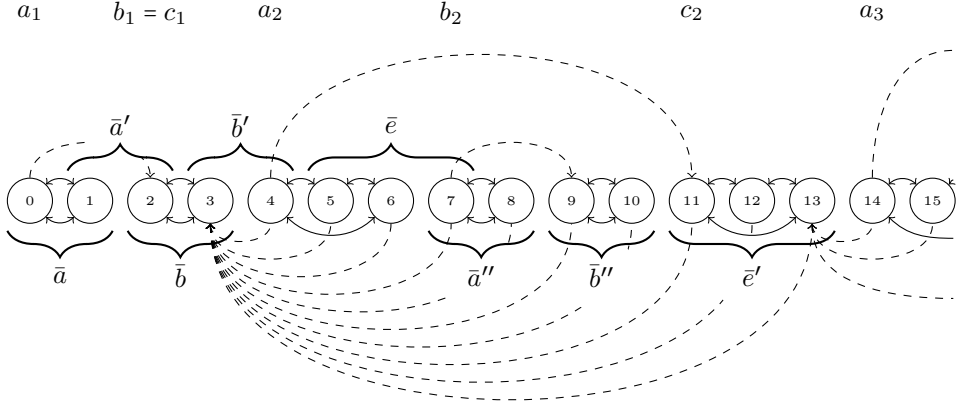
\begin{figure}[t]
\begin{center}
\begin{tikzpicture}
\node[circle,draw=black,label=below:$$,minimum width=.6cm, minimum height=.6cm] (box0) at (0,0) {{\tiny$0$}};
\node (a1) at (0,2.5) {{$a_1$}};
\node[circle,draw=black,label=below:$$,minimum width=.6cm, minimum height=.6cm] (box1) at (0.8,0) {{\tiny$1$}};
\node[circle,draw=black,label=below:$$,minimum width=.6cm, minimum height=.6cm] (box2) at (1.6,0) {{\tiny$2$}};
\node (c1) at (1.6,2.5) {{$b_1 = c_1$}};
\node[circle,draw=black,label=below:$$,minimum width=.6cm, minimum height=.6cm] (box3) at (2.4,0) {{\tiny$3$}};
\node[circle,draw=black,label=below:$$,minimum width=.6cm, minimum height=.6cm] (box4) at (3.2,0) {{\tiny$4$}};
\node (a2) at (3.2,2.5) {{$a_2$}};
\node[circle,draw=black,label=below:$$,minimum width=.6cm, minimum height=.6cm] (box5) at (4.0,0) {{\tiny$5$}};
\node[circle,draw=black,label=below:$$,minimum width=.6cm, minimum height=.6cm] (box6) at (4.8,0) {{\tiny$6$}};
\node[circle,draw=black,label=below:$$,minimum width=.6cm, minimum height=.6cm] (box7) at (5.6,0) {{\tiny$7$}};
\node (b2) at (5.6,2.5) {{$b_2$}};
\node[circle,draw=black,label=below:$$,minimum width=.6cm, minimum height=.6cm] (box8) at (6.4,0) {{\tiny$8$}};
\node[circle,draw=black,label=below:$$,minimum width=.6cm, minimum height=.6cm] (box9) at (7.2,0) {{\tiny$9$}};
\node[circle,draw=black,label=below:$$,minimum width=.6cm, minimum height=.6cm] (box10) at (8.0,0) {{\tiny$10$}};
\node[circle,draw=black,label=below:$$,minimum width=.6cm, minimum height=.6cm] (box11) at (8.8,0) {{\tiny$11$}};
\node (c2) at (8.8,2.5) {{$c_2$}};
\node[circle,draw=black,label=below:$$,minimum width=.6cm, minimum height=.6cm] (box12) at (9.6,0) {{\tiny$12$}};
\node[circle,draw=black,label=below:$$,minimum width=.6cm, minimum height=.6cm] (box13) at (10.4,0) {{\tiny$13$}};
\node[circle,draw=black,label=below:$$,minimum width=.6cm, minimum height=.6cm] (box14) at (11.2,0) {{\tiny$14$}};
\node (a3) at (11.2,2.5) {{$a_3$}};
\node[circle,draw=black,label=below:$$,minimum width=.6cm, minimum height=.6cm] (box15) at (12,0) {{\tiny$15$}};

\path[bend left=45] (box0) edge[<->] node {{$$}} (box1);
\path[bend left=45] (box1) edge[<->] node {{$$}} (box0);

\path[bend left=45] (box3) edge[<->] node {{$$}} (box2);
\path[bend left=45] (box2) edge[<->] node {{$$}} (box3);

\path[bend left=45] (box4) edge[<->] node {{$$}} (box5);
\path[bend left=45] (box5) edge[<->] node {{$$}} (box6);
\path[bend left=45] (box6) edge[<->] node {{$$}} (box4);

\path[bend left=45] (box7) edge[<->] node {{$$}} (box8);
\path[bend left=45] (box8) edge[<->] node {{$$}} (box7);

\path[bend left=45] (box9) edge[<->] node {{$$}} (box10);
\path[bend left=45] (box10) edge[<->] node {{$$}} (box9);

\path[bend left=45] (box11) edge[<->] node {{$$}} (box12);
\path[bend left=45] (box12) edge[<->] node {{$$}} (box13);
\path[bend left=45] (box13) edge[<->] node {{$$}} (box11);

\path[bend left=90,dashed] (box0) edge[->] node {{$$}} (box2);
\path[bend left=90,dashed] (box4) edge[->] node {{$$}} (box11);
\path[bend left=90,dashed] (box7) edge[->] node {{$$}} (box9);

\path[bend right=90,dashed] (box3) edge[<-] node {{$$}} (box4);
\path[bend right=90,dashed] (box3) edge[<-] node {{$$}} (box5);
\path[bend right=90,dashed] (box3) edge[<-] node {{$$}} (box6);
\path[bend right=90,dashed] (box3) edge[<-] node {{$$}} (box7);
\path[bend right=90,dashed] (box3) edge[<-] node {{$$}} (box8);
\path[bend right=90,dashed] (box3) edge[<-] node {{$$}} (box9);
\path[bend right=90,dashed] (box3) edge[<-] node {{$$}} (box10);
\path[bend right=90,dashed] (box3) edge[<-] node {{$$}} (box11);
\path[bend right=90,dashed] (box3) edge[<-] node {{$$}} (box12);
\path[bend right=90,dashed] (box3) edge[<-] node {{$$}} (box13);


\path[bend left=45] (box14) edge[<->] node {{$$}} (box15);
\node (box16up) at (12.4,0.275) {{$$}};
\path[out=45,in=180] (box15) edge[<-] node {{$$}} (box16up);
\node (box16down) at (12.4,-0.5) {{$$}};
\path[out = 180, in = 315] (box16down) edge[->] node {{$$}} (box14);
\node (box17up) at (12.4,2) {{$$}};
\path[out=90,in=180,dashed] (box14) edge[-] node {{$$}} (box17up);

\path[bend right=90,dashed] (box13) edge[<-] node {{$$}} (box14);
\path[bend right=90,dashed] (box13) edge[<-] node {{$$}} (box15);
\node (box16down2) at (12.4,-1.3) {{$$}};
\path[out=270,in=180,dashed] (box13) edge[<-] node {{$$}} (box16down2);


\node[circle,fill=white,label=below:$$,minimum width=1cm, minimum height=1cm] (a) at (0.4,-1) {{$\bar{a}$}};
\draw[decorate,thick,decoration={brace,mirror,amplitude=10pt}] (-0.3,-0.4) -- (1.1,-0.4);
\node[circle,fill=white,label=below:$$,minimum width=1cm, minimum height=1cm] (b) at (2,-1) {{$\bar{b}$}};
\draw[decorate,thick,decoration={brace,mirror,amplitude=10pt}] (1.3,-0.4) -- (2.7,-0.4);

\node[circle,fill=white,label=below:$$,minimum width=1cm, minimum height=1cm] (ap) at (1.2,1) {{$\bar{a}'$}};
\draw[decorate,thick,decoration={brace,amplitude=10pt}] (0.5,0.4) -- (1.9,0.4);
\node[circle,fill=white,label=below:$$,minimum width=1cm, minimum height=1cm] (bp) at (2.8,1) {{$\bar{b}'$}};
\draw[decorate,thick,decoration={brace,amplitude=10pt}] (2.1,0.4) -- (3.5,0.4);
\node[circle,fill=white,label=below:$$,minimum width=1cm, minimum height=1cm] (e) at (4.8,1) {{$\bar{e}$}};
\draw[decorate,thick,decoration={brace,amplitude=10pt}] (3.7,0.4) -- (5.9,0.4);

\node[circle,fill=white,label=below:$$,minimum width=1cm, minimum height=1cm] (app) at (6,-1) {{$\bar{a}''$}};
\draw[decorate,thick,decoration={brace,mirror,amplitude=10pt}] (5.3,-0.4) -- (6.7,-0.4);
\node[circle,fill=white,label=below:$$,minimum width=1cm, minimum height=1cm] (bpp) at (7.6,-1) {{$\bar{b}''$}};
\draw[decorate,thick,decoration={brace,mirror,amplitude=10pt}] (6.9,-0.4) -- (8.3,-0.4);
\node[circle,fill=white,label=below:$$,minimum width=1cm, minimum height=1cm] (ep) at (9.6,-1) {{$\bar{e}'$}};
\draw[decorate,thick,decoration={brace,mirror,amplitude=10pt}] (8.5,-0.4) -- (10.7,-0.4);

\end{tikzpicture}
\end{center}
\caption{The choice of $\bar{a}'$, $\bar{b}'$, $\bar{a}''$, $\bar{b}''$, and $\bar{e}'$. The figure shows the case when $\bar{c} = \varnothing$, we choose $\bar{a} = (0,1)$, and $\bar{b} = \bar{b}_2 = (2,3)$. We choose $k = 2$ so that $b_k = b_2 = 7$.}
\label{fig:dfree}
\end{figure}

So for any tuple $\bar{c}$, there is a tuple $\bar{a}$ which is d-free over $\bar{c}$. Moreover, everything was effective. Thus, by Proposition \ref{ce-condition}, the degree spectrum of $R$ contains a non-c.e.\ degree.

\vspace*{10pt}
\noindent \emph{Construction of $h$.}
\vspace*{10pt}

Let $e$ be an index for a computable function $\varphi_e$, which we attempt to interpret as the diagram of a computable structure $\mc{A}$. Let $i$ and $j$ be indices for the Turing functionals $\Phi_i$ and $\Phi_j$. We will build a $\Delta^0_2$ set $C$ such that if $\mc{A}$ is a computable copy of $(\omega,<)$, then either \[ C \neq \Phi_{i}^{R^\mc{A}} \text{ or } R^\mc{A} \neq \Phi_{j}^C.\] By constructing $C$ via a $\Delta^0_2$ approximation uniformly from $e$, $i$, and $j$ we will obtain the required function $h$.

At each stage $s$, we get a finite linear order $\mc{A}_s$ which approximates $\mc{A}$. Now, the domain of $\mc{A}_s$ is contained in $\omega$, but $\mc{A}$ may also be isomorphic to $\omega$. To avoid confusion, when we say $\omega$, we refer to the underlying domain of $\mc{A}_s$, and by $(\omega,<)$ we mean the structure as a linear order. We may assume that the elements of $\mc{A}_s$ form a finite initial segment of the domain $\omega$. To differentiate between the ordering (as part of the language of the structure) on $\mc{A}_s$ and the underlying order on the domain as a subset of $\omega$, we will use $\preceq_{\mc{A}}$ for the former and $\leq$ for the latter. We can write the elements of $\mc{A}_s$ as $x_1^s \prec_{\mc{A}} x_2^s \prec_{\mc{A}} \cdots \prec_{\mc{A}} x_n^s$, and given $z \in \omega$, we write $N_s(z) = k$ if $z = x^s_k$. So $N_s(z)$ is a guess at which element of $\omega$ the element $z$ represents. For $z \in \mc{A}$, let $N(z) \in (\omega,<)$ be the element which $z$ is isomorphic to (if $\mc{A}$ is isomorphic to $(\omega,<)$). We also get an approximation $R^\mc{A}_s$ of $R^\mc{A}$ by setting $z \in R^\mc{A}_s$ if and only if $N_s(z) \in R$. Then $R^\mc{A}_s$ is a $\Delta^0_2$ approximation of $R^\mc{A}$ in the case that $\mc{A}$ is isomorphic to $(\omega,<)$.

The general idea of the construction is as follows. At each stage $s$, we will have finite set $C_s$ such that $C(n) = \lim_{s\to\infty} C_s(n)$. If we do not explicitly say that we change $C$ from stage $s$ to stage $s+1$, then we will have $C_{s+1} = C_s$. Suppose that at stage $s$ we have $C_s(0) = 0$ and we have computations
\[ C_s(0) = 0 = \Phi_{i,s}^{R_s^{\mc{A}}}(0) \]
with use $u$ and
\[ R_s^{\mc{A}}[0,\ldots,u] = \Phi_{j,s}^{C_s}[0,\ldots,u] \]
with use $v$. By putting $0$ into $C$, we destroy the first computation, forcing $R_s^{\mc{A}}$ to change below the use $u$ (or else we are done); then, by removing $0$ from $C$, because of the second computation we force $R_s^{\mc{A}}$, at some later stage $t$, to change back to the way it was before  (i.e., $R_s^{\mc{A}}[0,\ldots,u] = R_t^{\mc{A}}[0,\ldots,u]$). This means that some $q \leq u$ in $\mc{A}$ must have had some element enumerated below it, so that $N_t(q) > N_s(q)$. By moving $0$ in and out of $C$ in this way, we can force arbitrarily many elements to be enumerated below one of $[0,\ldots,u]$. If we could enumerate infinitely many such elements, then we would have prevented $\mc{A}$ from being isomorphic to $(\omega,<)$. However, this would require moving $0$ in and out of $C$ infinitely many times, which would make $C$ not $\Delta^0_2$. We must be more clever.

Let $p_0 = 0$. We will wait for computations as above (with uses $u_0$ and $v_0$), and then choose $p_1 > v$. We wait for computations
\[ C_s[0,\ldots,p_1] = 0 = \Phi_{i,s}^{R_s^{\mc{A}}}[0,\ldots,p_1] \]
with use $u_1$ and
\[ R_s^{\mc{A}}[0,\ldots,u_1] = \Phi_{j,s}^{C_s}[0,\ldots,u_1] \]
with use $v_1$. We will move $p_0 = 0$ into and then out of $C$ as above to enumerate an element in $\mc{A}$ below one of $[0,\ldots,u]$. At the same time, we will create a ``link'' between $[0,\ldots,u_0]$ and $[u_0+1,\ldots,u_1]$ so that if some element gets enumerated below one of $[u_0+1,\ldots,u_1]$, then some element will also get enumerated below one of $[0,\ldots,u_0]$. (Exactly how these links work will be explained later.) We have moved $p_0$ into and then out of $C$, but from now on it will stay out of $C$. We will find a $p_2$, and move $p_1$ into and out of $C$. On the one hand, this will cause an element to be enumerated below one of $[u_0+1,\ldots,u_1]$, and hence below one of $[0,\ldots,u_0]$. On the other hand, we will create a ``link'' between $[u_0+1,\ldots,u_1]$ and $[u_1+1,\ldots,u_2]$. Now when an element gets enumerated below one of $[u_1+1,\ldots,u_2]$, an element gets enumerated below one of $[u_0+1,\ldots,u_1]$, and thus some element gets enumerated below one of $[0,\ldots,u_0]$. Continuing in this way, defining $p_3$, $p_4$, and so on, and maintaining these links, we force infinitely many elements to be enumerated below some element of $[0,\ldots,u_0]$. We will describe exactly how these links work in the construction. Figure \ref{fig:thm5-19} shows the computations which we use during the construction.

\begin{figure}[htb]
	\includegraphics{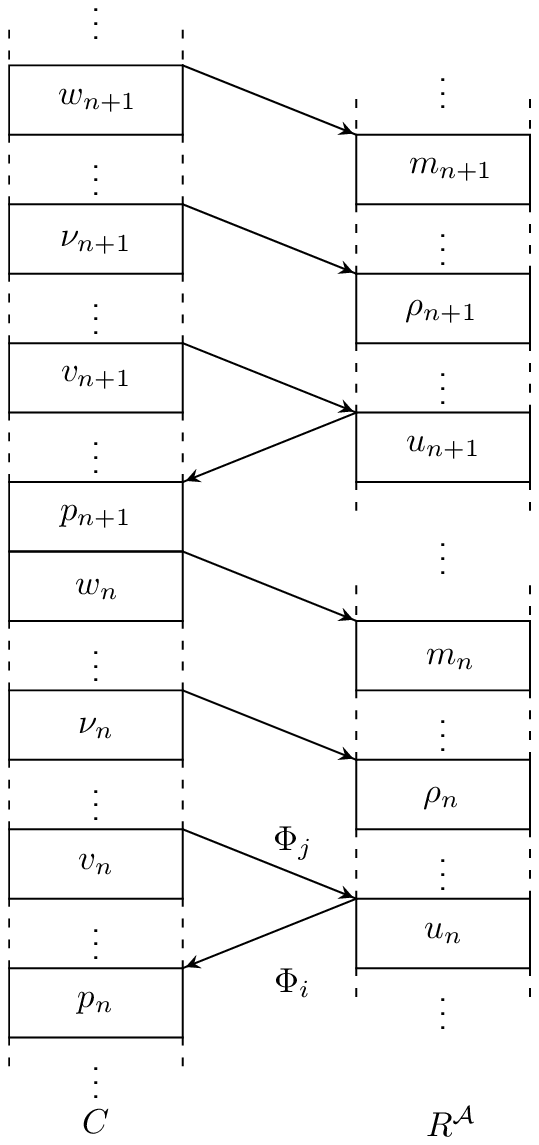}
	\caption{The values associated to a requirement for Proposition \ref{thm:weird-relation-on-omega}. An arrow shows a computation converging. The computations use an oracle and compute some initial segment of their target. The tail of the arrow shows the use of the computation, and the head shows the length.}
	\label{fig:thm5-19}
\end{figure}

The construction will consist of three steps defined below, which are repeated for each of $n = 0,1,\ldots$. The proof of the following lemma will be interspersed with the construction below:
\begin{lem}\label{lem:omega-not-stuck}
If $\mc{A}$ is an isomorphic copy of $(\omega,<)$, and
\[ C = \Phi_{i}^{R^\mc{A}} \text{ and } R^\mc{A} = \Phi_{j}^{C} \]
then the construction does not get stuck in any step.
\end{lem}
After describing the action taking place at each step, we will prove that if $\mc{A}$ is an isomorphic copy of $(\omega,<)$, and
\[ C = \Phi_{i}^{R^\mc{A}} \text{ and } R^\mc{A} = \Phi_{j}^{C} \]
then the construction eventually finishes that step.

Begin the construction with $p_0 = 0$ and $C_0 = \varnothing$. Before beginning to repeat the three steps, wait for a stage $s$ where we have computations
\begin{equation}\label{eq:a0-C-R} C_s[p_0] = 0 = \Phi_{i,s}^{R^\mc{A}_s}[p_0] \end{equation}
with use $u_0$ and
\begin{equation}\label{eq:a0-R-C} R^\mc{A}_s[0,\ldots,u_0] = \Phi_{j,s}^{C_s}[0,\ldots,u_0] \end{equation}
with use $v_0$. Let $s_0 = s$. If $C = \Phi_{i}^{R^\mc{A}}$ and $R^\mc{A} = \Phi_{j}^{C}$, we eventually find computations as in (\ref{eq:a0-C-R}) and (\ref{eq:a0-R-C}).

Now repeat, in order, the following steps. We call each repetition of these steps a \textit{Rep}. Begin at \textit{Rep 0} with $n=0$. At the beginning of \textit{Rep $n$}, we will have defined values $p_i$, $s_i$, $u_i$, and $v_i$ for $0 \leq i \leq n$, $m_i$, $w_i$, and $q_i$ for $0 \leq i < n$, and $\rho_i$ and $\nu_i$ for $1 \leq i \leq n$ (note that we never define $\rho_0$ or $\nu_0$). At the beginning of each repetition, we will have $C = \varnothing$; in \textit{Step One}, we will add an element to $C$, and in \textit{Step Two} we will remove that element from $C$ returning it to the way it was before.

\medskip{}
\noindent\textit{Step One.} Wait for a stage $s$ and an $m > \rho_n$ such that 
\begin{enumerate}
	\item at this stage $s$ we still have the computation
\begin{equation}\label{eq:an-C-R} R^\mc{A}_s[0,\ldots,\rho_n] = \Phi_{j,s}^{C_s}[0,\ldots,\rho_n] \end{equation}
with use $\nu_n$ from \textit{Step 3} of the previous step (see equation \eqref{eq:rho}). If $n = 0$, then instead we ask that
\begin{equation}\label{eq:an-C-R-2} R^\mc{A}_s[0,\ldots,u_0] = \Phi_{j,s}^{C_s}[0,\ldots,u_0], \end{equation}
	\item all of the elements $x$ of $\mc{A}_s$ with $x \preceq_{\mc{A}} p_n$ come from among $[0,\ldots,m]$,
	\item there is a computation
	\begin{equation}\label{eq:m-R-C} R^\mc{A}_s[0,\ldots,m] = \Phi_{j,s}^{C_s}[0,\ldots,m] \end{equation}
	with use $w$,
	\item there are computations
	\begin{equation}\label{eq:w-C-R} C_s[0,\ldots,w + 1] = \Phi_{i,s}^{R^\mc{A}_s}[0,\ldots,w + 1] \end{equation}
with use $u$ and
\begin{equation}\label{eq:w-R-C} R^\mc{A}_s[0,\ldots,u] = \Phi_{j,s}^{C_s}[0,\ldots,u] \end{equation}
with use $v$,
	\item for each $z \in [0,\ldots,u_n]$, among the elements $[0,\ldots,m]$ of $\mc{A}_s$ there is a matching pair of $k$-cycles (for some $k \in \omega$) $\bar{x} = (x_0,\ldots,x_k)$ and $\bar{y} = (y_0,\ldots,y_k)$ in $\mc{A}_s$ with $x_0 \preceq_{\mc{A}} z \preceq_{\mc{A}} y_k$ and there are edges from each element of $[m+1,\ldots,u]$ to $y_k$ (and an edge from $x_0$ to $y_0$ witnessing that the $k$-cycles form a matching pair).
\end{enumerate}
Set $m_n = m$, $w_n = w$, $p_{n+1} = w + 1$, $u_{n+1} = u$, $v_{n+1} = v$, and $s_{n+1} = s$. Set $C_{s+1}(p_n) = 1$ to break the computation .

\medskip{}

The idea at this step is to find, for each $z \in [0,\ldots,u_n]$, a matching pair of $k$-cycles $\bar{x}$ and $\bar{y}$ which contain $z$ between them. These $k$-cycles are all contained in the elements $[0,\ldots,m]$. Moreover, $\bar{x}$ and $\bar{y}$ look like they correspond to the $k$-cycles $(a_k,\ldots,a_k+k)$ and $(c_k,\ldots,c_k+k)$ respectively from the standard copy of $(\omega,<)$ in the sense that there is an edge from each element of $[m_n+1,\ldots,u_{n+1}]$ to $y_k$. We also define the next value $p_{n+1}$ during this step.  All of this is to set up the ``link'' between $p_{n+1}$ to $p_n$ (but the link will not be completed until \textit{Step Three}).

Note that, since we set $p_{n+1} = w + 1$, $u_{n+1} = u$, and $s_{n+1} = s$, the computation \eqref{eq:w-C-R} is really
\begin{equation}\label{eq:w-C-R-5} C_{s_{n+1}}[0,\ldots,p_{n+1}] = \Phi_{i,s_{n+1}}^{R^\mc{A}_{s_{n+1}}}[0,\ldots,p_{n+1}] \end{equation}
with use $u_{n+1}$. At the end of this step, we set $C_{s+1}(p_n) = 1$ to break the computation
\[ C_s[0,\ldots,p_n] = \Phi_{i,s}^{R^\mc{A}_s}[0,\ldots,p_n] \]
with use $u_n$ from \eqref{eq:w-C-R-5} of \textit{Rep $n-1$}.

\begin{proof}[Proof of Lemma \ref{lem:omega-not-stuck} for Step One]
Suppose that we never leave \textit{Step One}. Let $t$ be the stage at which we entered \textit{Step One}. Then $C = C_t$. For sufficiently large $s > t$, we have the true computation
\[ R^\mc{A}_s[0,\ldots,\rho_n] = \Phi_{j,s}^{C_s}[0,\ldots,\rho_n] \]
as in \eqref{eq:an-C-R}. Thus we satisfy (1). Recall that $N$ is the isomorphism $\mc{A} \to (\omega,<)$. Since $\mc{A}$ is an isomorphic copy of $(\omega,<)$, there are only finitely many elements $x \in \mc{A}$ with $x \preceq_{\mc{A}} a_n$ and so for sufficiently large $m$, they all come from $[0,\ldots,m]$. So (2) is satisfied as well. For each $z \in [0,\ldots,u_n]$, there is $k$ such that $N(z)$ is in the interval $[a_k,a_{k+1})$. Let $\bar{x}$ and $\bar{y}$ be the matching pair of $k$-cycles in $\mc{A}$ which are the pre-images, under the isomorphism $N$, of $(a_k,\ldots,a_k+k)$ and $(c_k,\ldots,c_k+k)$. Increasing $m$ if necessary, we may assume that these $k$-cycles are contained in $[0,\ldots,m]$ in $\mc{A}$, and that all of the elements of $\mc{A}$ which are $\preceq_\mc{A}$-less than any entry of $\bar{x}$ and $\bar{y}$ are contained in $[0,\ldots,m]$. Do this for each $z$.

Since $R^\mc{A} = \Phi_{j}^{C}$, for sufficiently large $s$ we have true computations as in \eqref{eq:m-R-C}, \eqref{eq:w-C-R}, and \eqref{eq:w-R-C} defining $w$, $u$, and $v$. So (3) and (4) are satisfied.

Finally, for each $z \in [0,\ldots,u_n]$ the $k$-cycles $\bar{x}$ and $\bar{y}$ chosen above satisfy $x_0 \preceq_\mc{A} z \preceq_\mc{A} y_k$ and there is an edge from $x_0$ to $y_0$. Since all of the elements of $\mc{A}$ which are $\preceq_\mc{A}$-less than any entry of $\bar{x}$ and $\bar{y}$ are contained in $[0,\ldots,m]$, each of $m+1,\ldots,u$ are $\preceq_\mc{A}$-greater than $y_k$. So there is an edge from each of these to $y_k$. Thus (5) is satisfied. This contradicts our assumption that we never leave \textit{Step One}.
\end{proof}

\medskip{}
\noindent\textit{Step Two.} Wait for a stage $s$ such that
\begin{equation}\label{eq:omega-diff} R^\mc{A}_s[m_{n-1}+1,\ldots,u_n] \neq  R^\mc{A}_{s_n}[m_{n-1}+1,\ldots,u_n] = R^\mc{A}_{s_{n+1}}[m_{n-1}+1,\ldots,u_n]. \end{equation}
Let $q_n$ be the $\preceq_{\mc{A}}$-greatest element in $[m_{n-1}+1,\ldots,u_n]$; note that $N_s(q_n) > N_{s_{n+1}}(q_n)$. Set $C_{s+1}(p_n) = 0$.

\medskip{}

In the previous step, we broke the computation 
\[ C_{s_n}[0,\ldots,p_n] = \Phi_{i,{s_n}}^{R^\mc{A}_{s_n}}[0,\ldots,p_n]. \]
In order for this computation to again hold at some $s$, $R^\mc{A}$ must change below the use $u_n$ of that computation. That means that some element must be enumerated in $\mc{A}$ $\preceq_{\mc{A}}$-below one of $0,\ldots,u_n$ (in fact, it will have to be enumerated below one of $m_{n-1}+1,\ldots,u_n$). We let $q$ be the $\preceq_{\mc{A}}$-greatest such element, so we know that some an element has been enumerated below $q_n$. At the end of this stage, we set $C_{s+1}(p_n) = 0$ (so that $C_{s+1} = C_{s_n}$) to restore the computation above and force $R^\mc{A}$ to return to the way it was before.

\begin{proof}[Proof of Lemma \ref{lem:omega-not-stuck} for Step Two]
Suppose that the construction does not leave \textit{Step Two}. Then for all $t > s_{n+1}$, $C_t(a_n) = C_{s_{n+1}+1}(a_n) \neq C_{s_{n+1}}(a_n) = C_{s_n}(a_n)$. Since $C = \Phi_{i}^{R^\mc{A}}$, at some stage $t > s_{n+1}$, we have a true computation
\[ C_t[0,\ldots,a_n] = \Phi_{i}^{R^\mc{A}_t}. \]
By \eqref{eq:w-C-R-5} of the previous repetition, and since $C_t(a_n) \neq C_{s_{n}}(a_n)$, we have
\begin{equation} R^\mc{A}_s[0,\ldots,u_n] \neq  R^\mc{A}_{s_n}[0,\ldots,u_n]. \end{equation}
Since $C_t[0,\ldots,w_{n-1}] = C_{s_{n}}[0,\ldots,w_{n-1}]$, by \eqref{eq:m-R-C} we have
\[ R^\mc{A}_s[0,\ldots,m_{n-1}] =  R^\mc{A}_{s_n}[0,\ldots,m_{n-1}]. \]
Thus
\[ R^\mc{A}_s[m_{n-1}+1,\ldots,u_n] \neq  R^\mc{A}_{s_n}[m_{n-1}+1,\ldots,u_n]. \]
This contradicts our assumption that we never leave \textit{Step Two}.
\end{proof}

\medskip{}
\noindent\textit{Step Three.} Wait for a stage $s$ and $\rho > u_{n+1}$ such that 
\begin{enumerate}
	\item we have
	\begin{equation}\label{eq:back-to-same} R^\mc{A}_s[0,\ldots,u_{n+1}] =  R^\mc{A}_{s_{n+1}}[0,\ldots,u_{n+1}], \end{equation}
	\item among the elements $[0,\ldots,\rho]$ of $\mc{A}_s$, there is a matching pair of $\ell$-cycles (for some $\ell$) $\bar{\sigma} = (\sigma_0,\ldots,\sigma_k)$ and $\bar{\tau} = (\tau_0,\ldots,\tau_k)$ in $\mc{A}_s$ with $\sigma_k \prec_{\mc{A}_s} q_n$ and $z \prec_{\mc{A}_s} \tau_0$ for each $z \in [m_n+1,\ldots,u_{n+1}]$ (and an edge from $\sigma_0$ to $\tau_0$ but not vice versa witnessing that the $\ell$-cycles are matching),
	\item we have the computation
	\begin{equation}\label{eq:rho} R^\mc{A}_s[0,\ldots,\rho] = \Phi_{j,s}^{C_s}[0,\ldots,\rho]\end{equation}
	with use $\nu$.
\end{enumerate}
Set $\rho_{n+1} = \rho$ and $\nu_{n+1} = \nu$. Return to \textit{Step One} for \textit{Rep $n+1$}.

\medskip{}

In this step, we wait for the computation \eqref{eq:back-to-same} to be restored. Now, this forces $R^\mc{A}$ to be the same as it was during \textit{Step One}:
\[ R^\mc{A}_s[0,\ldots,u_{n+1}] = R^\mc{A}_{s_n}[0,\ldots,u_{n+1}]. \]
At \textit{Step One}, there was a matching pair of $k$-cycles $\bar{x}$ and $\bar{y}$ which contained $q_n$ between them. Since $R^\mc{A}$ is the same at this stage as it was at that stage, $\bar{x}$ and $\bar{y}$ are still a matching pair of $k$-cycles. But some element has been enumerated below $q_n$ since then, and by Remark \ref{rem:k-cycle-2}, it must be enumerated below $\bar{x}$. So $\bar{x}$ and $\bar{y}$ are not the $\preceq_{\mc{A}}$-least matching pair of $k$-cycles. Thus, they correspond to some elements in an interval $[b_\ell,c_\ell)$ from $(\omega,<)$ for some $\ell$. The $\bar{\sigma}$ and $\bar{\tau}$ in (2) above are supposed to correspond to $(a_\ell,\ldots,a_\ell+\ell)$ and $(c_\ell,\ldots,c_\ell+\ell)$ from $(\omega,<)$ respectively. Since, in \textit{Step One}, there was an edge from each $z \in [m_n+1,\ldots,u_{n+1}]$ to $y_k$, each such $z$ must also correspond to some element from the interval $[b_\ell,c_\ell)$ (one can see from the definition of the relation $R$ that there are no such edges from an element $z \geq c_\ell$ to some $y \in [b_\ell,c_\ell)$).

This establishes the desired ``link''. By Remark \ref{rem:k-cycle-2}, since the interval $[\sigma_0,\tau_\ell]$ is of a fixed length determined by $\ell$, no new elements can be enumerated between $\sigma_0$ and $\tau_\ell$. So if some new element is enumerated below one of $[m_n+1,\ldots,u_{n+1}]$, it must be enumerated below $\sigma_0$ and hence below $q_n$.

\begin{proof}[Proof of Lemma \ref{lem:omega-not-stuck} for Step Three]
Suppose that the construction never leaves \textit{Step Three}. Suppose that the construction entered \textit{Step Three} at stage $s$. For $t \geq s$, we have $C_t = C_{s_{n+1}} = \varnothing$. Since $R^\mc{A} = \Phi_{j}^{C}$, and using (\ref{eq:w-R-C}), for sufficiently large $t > s$ we have
\begin{equation} R^\mc{A}[0,\ldots,u_{n+1}] = R^\mc{A}_t[0,\ldots,u_{n+1}] =  R^\mc{A}_{s_{n+1}}[0,\ldots,u_{n+1}]. \end{equation}
Since $q_n \in [m_{n-1}+1,\ldots,u_n]$, at stage $s_{n+1}$, there was a matching pair of $k$-cycles $\bar{x} = (x_0,\ldots,x_k)$ and $\bar{y} = (y_0,\ldots,y_k)$ among the elements $[0,\ldots,m_n]$ as in (5) of \textit{Step One} with $x_0 \preceq_\mc{A} q_n \preceq_{\mc{A}} y_k$. Since
\[ R^\mc{A}[0,\ldots,u_{n+1}] = R^\mc{A}_{s_{n+1}}[0,\ldots,u_{n+1}], \]
$\bar{x}$ and $\bar{y}$ are actually $k$-cycles in $\mc{A}$, and is actually an edge from $x_0$ to $y_0$. Also, $N_{t}(q_n) > N_{s_{n+1}}(q_n)$. By Remark \ref{rem:k-cycle-2}, $N_t(x_0) > N_{s_{n+1}}(x_0)$. Let $\bar{x}' \in (\omega,<)$ be $(N_{s_{n+1}}(x_0),\ldots,N_{s_{n+1}}(x_k))$ and similarly for $\bar{y}'$. Let $\bar{x}'' \in (\omega,<)$ be $(N_{t}(x_0),\ldots,N_{t}(x_k))$ and similarly for $\bar{y}''$. Then in $(\omega,<)$ we know that $\bar{x}'$ and $\bar{y}'$ are a matching pair of $k$-cycles, and so are $\bar{x}''$ and $\bar{y}''$. So $\bar{x}''$ and $\bar{y}''$ are not the first matching pair of $k$-cycles, and so they are contained in some interval $[b_\ell,c_\ell)$ for some $\ell$.  Moreover, in $R^{\mc{A}}$, each element $z$ of $[m_n,\ldots,u_{n+1}]$ has an edge from it to $x_0$. This is only possible if $N_t(z) \in [b_\ell,c_\ell)$ for each such $z$. Let $\bar{\sigma}$ and $\bar{\tau}$ be the matching pair of $\ell$-cycles in $\mc{A}$ whose images in $(\omega,<)$ under the isomorphism $N$ are $(a_\ell,\ldots,a_\ell + \ell)$ and $(c_\ell,\ldots,c_\ell+\ell)$ respectively. Then $\bar{\sigma}$ and $\bar{\tau}$ the required $\ell$-cycles in \textit{Step 3}. We get the computation for (\ref{eq:rho}) because $R^\mc{A} = \Phi_{j}^{C}$. This contradicts our assumption that we never leave \textit{Step Three}.
\end{proof}

This ends the construction. In the process, we have proved Lemma \ref{lem:omega-not-stuck}. The next two lemmas complete the proof that the construction works as desired.

\begin{lem}
$C$ is a d.c.e.\ set and the approximation $C_s$ is a d.c.e.\ approximation
\end{lem}
\begin{proof}
We change the approximation $C_s(x)$ at most twice for each $x$---since $x = p_n$ for at most one $n$, we change $C_s(x)$ at most once in \textit{Step One} and once in \textit{Step Two}.
\end{proof}

\begin{lem}\label{lem:moves-inf-times}
If the construction does not get stuck in any stage, then $\mc{A}$ is not isomorphic to $(\omega,<)$.
\end{lem}
\begin{proof}
In the construction above we remarked that $N_{s_{n+2}}(q_n) > N_{s_{n+1}}(q_n)$. We claim that for all $n \geq 1$, $N_{s_{n+1}}(q_0) > N_{s_n}(q_0)$. This would imply that $\mc{A}$ is not isomorphic to $(\omega,<)$, as $q_n \in \mc{A}$ would have infinitely many predecessors. 

The key to the proof will be to use the ``links'' that we created during the construction. We will show that if, for $n' > n + 1$, $N_{s_{n'+1}}(q_{n+1}) > N_{s_{n'}}(q_{n+1})$ then $N_{s_{n'+1}}(q_n) > N_{s_{n'}}(q_n)$. This will suffice to prove the lemma.

During \textit{Step Three} of the $n$th repetition, we saw that among the elements $[0,\ldots,\rho_{n+1}]$, there is a matching pair of $\ell$-cycles $\bar{\sigma} = (\sigma_0,\ldots,\sigma_\ell)$ and $\bar{\tau} = (\tau_0,\ldots,\tau_\ell)$ in $\mc{A}_s$ with $\sigma_\ell \prec_{\mc{A}} q_n$ and $z \prec_{\mc{A}} \tau_0$ for each $z \in [m+1,\ldots,v]$. Moreover, by \eqref{eq:an-C-R}, at every stage $s_{n'}$ for $n' > n+1$, $\bar{\sigma}$ and $\bar{\tau}$ are $\ell$-cycles in $\mc{A}_{s_{n'}}$.

Then $q_{n+1} \in [m+1,\ldots,v]$, so $q_{n+1} \prec_{\mc{A}} \tau_0$. Thus if $N_{s_{n'+1}}(q_{n+1}) > N_{s_{n'}}(q_{n+1})$, then $N_{s_{n'+1}}(\tau_0) > N_{s_{n'}}(\tau_0)$. Since $\bar{\sigma}$ and $\bar{\tau}$ are a matching pair of $\ell$-cycles at stages $s_{n'}$ and $s_{n'+1}$, by Remark \ref{rem:k-cycle-2} no new elements are added between $\bar{\sigma}$ and $\bar{\tau}$ in between these stages. So $N_{s_{n'+1}}(\sigma_0) > N_{s_{n'}}(\sigma_0)$, and since $\sigma_0 \preceq_{\mc{A}} q_n$, $N_{s_{n'+1}}(q_n) > N_{s_{n'}}(q_n)$.
\end{proof}

If $\mc{A}$ is an isomorphic copy of $(\omega,<)$, then Lemma \ref{lem:omega-not-stuck} and Lemma \ref{lem:moves-inf-times} combine to show that \[ C \neq \Phi_{i}^{R^\mc{A}} \text{ or } R^\mc{A} \neq \Phi_{j}^{C}.\]
This completes the proof of Theorem \ref{thm:weird-relation-on-omega}.
\end{proof}

One can view the proof as a strategy for satisfying a single requirement $\mc{R}_{e,i,j}$. For a fixed $e_0$, it does not add too much difficulty to satisfy multiple requirements of the form $\mc{R}_{e_0,i,j}$ at the same time---since these requirement are all working with the same structure $\mc{A}_e$, only one requirement has to force $\mc{A}_e$ to not be isomorphic to $(\omega,<)$. However, if one tries to satisfy every requirement $\mc{R}_{e,i,j}$ for different $e$'s at the same time, one runs into a problem. Each requirement tries to restrain infinitely much of $\omega$, and in order to build $p_{n+1}$, the requirement must move $p_n$. Thus if $p_{n+1}$ is injured, $p_n$ may injure other requirements.

\chapter{A ``Fullness'' Theorem for 2-\cea Degrees}
\label{SigmaTwoSection}
In this section, we will prove Theorem \ref{relativized-2}. Recall that a set $A$ is 2-\cea in a set $B$ if there is $C$ such that $A$ is c.e.\ in and above $C$ and $C$ is c.e.\ in and above $B$.

We will prove the theorem in the following form:

\begin{thm}\label{relativized-1}
Let $\mc{C}$ be a computable structure, and let $R$ be an additional computable relation on $\mc{C}$. Suppose that $R$ is not formally $\Delta^0_2(0''')$. Then for all degrees $\textbf{d} \geq 0^{(\omega + 1)}$ and sets $A$ 2-\cea in $\textbf{d}$ there is an isomorphic copy $\mc{D}$ of $\mc{C}$ with $\mc{D} \equiv_T \textbf{d}$ and
\[ R^{\mc{D}} \oplus \mc{D} \equiv_T A . \]
\end{thm}

Our construction will actually build $\mc{D} \leq_T \textbf{d}$. We can use Knight's theorem on the upwards closure of the degree spectrum of a structure (see \cite{Knight86}) to obtain $\mc{D} \equiv_T \textbf{d}$ as follows. Suppose that we have built $\mc{D} \leq_T \textbf{d}$ as in the theorem. There is an isomorphic copy $\mc{D}^{*} \equiv_T \textbf{d}$ of $\mc{D}$. Moreover, $\mc{D}^{*}$ is obtained by applying a permutation $f \equiv_T \textbf{d}$ to $\mathcal{D}$. Then $f \oplus \mc{D} \equiv_T f \oplus \mc{D}^{*} \equiv_T \mc{D}^*$ and $f \oplus R^{\mc{D}} \equiv_T f \oplus R^{\mc{D}^{*}}$. Hence
\[ A \equiv_T \mc{D}^{*} \oplus R^{\mc{D}^{*}}. \]

Theorem \ref{relativized-2} is obtained from Theorem \ref{relativized-1} by relativizing the proof. Given any structure $\mc{C}$ and relation $R$ on $\mc{C}$, we can build a computable copy $\mc{D} \equiv_T \textbf{d}$ as in Theorem \ref{relativized-1} for any $\textbf{d}$ in the cone above $(\mc{C} \oplus R)^{\omega + 1}$. We could also give effectiveness conditions on computable $\mc{C}$ and $R$ which would suffice to take $\textbf{d} = 0$, but these would be quite complicated.

Finally, the simplest way to state the theorem is as follows:
\begin{corollary}
Let $\mc{C}$ be a structure and $R$ a relation on $\mc{C}$. Then either
\[ \dgSp_{rel}(\mc{C},R) \subseteq \Delta^0_2 \]
or
\[ \text{2-\cea} \subseteq \dgSp_{rel}(\mc{C},R). \]
\end{corollary}

The proof of the theorem will use free elements as in Barker's proof that formally $\Sigma_\alpha$ is the same as intrinsically $\Sigma_\alpha$ \cite{Barker88}. It will probably be helpful to understand the proof of that result at least for the case $\alpha = 2$.
\begin{defn}
We say that $\bar{a} \notin R$ is $2$\textit{-free over} $\bar{c}$ if for all $\bar{a}_1$, there are $\bar{a}' \in R$ and $\bar{a}_{1}'$ such that
\[ \bar{c},\bar{a},\bar{a}_1 \leq_1 \bar{c},\bar{a}',\bar{a}_1. \]
\end{defn}
\noindent Recall that $\leq_0$ and $\leq_1$ are the first two back-and-forth relations; $\bar{a} \leq_0 \bar{b}$ if all of the quantifier-free formulas with G\"odel number less than $|\bar{a}|$ which are true of $\bar{a}$ are true of $\bar{b}$, while $\bar{a} \leq_1 \bar{b}$ if every $\Sigma^0_1$ formula true of $\bar{b}$ is true of $\bar{a}$ (see Chapter 15 of \cite{AshKnight00}). If $F: \{0,\ldots,m\} \to \mc{C}$ and $G: \{0,\ldots,n\} \to \mc{C}$ are functions with $n > m$, then $F \leq_i G$ means that
\[ F(0),\ldots,F(m) \leq_i G(0),\ldots,G(m). \]

If $R$ is not formally $\Delta^0_2(0''')$, then either $R$ is not defined by a $0'''$-computable $\Pi^0_2$ formula or $\neg R$ is not defined by a $0'''$-computable $\Pi^0_2$ formula. We may suppose without loss of generality that it is $R$ which is not defined by a $0'''$-computable $\Pi^0_2$ formula. We will relativize Proposition 16.1 of \cite{AshKnight00} to show that for any tuple $\bar{c}$, there is a tuple $\bar{a}$ which is 2-free over $\bar{c}$. Moreover, using $0^{(4)}$ we can check whether a tuple is $2$-free, and hence find these $2$-free tuples. 

\begin{prop}
Let $\mc{C}$ be a computable structure, and let $R$ be a further computable relation on $\mc{C}$. Suppose that $\bar{c}$ is a tuple over which no $\bar{a} \notin R$ is 2-free. Then there is a $0'''$-computable $\Sigma_2$ formula $\varphi(\bar{c},\bar{x})$ defining $\neg R$.
\begin{proof}
We have $\bar{x} \leq_1 \bar{y}$ exactly when all $\exists_1$ formulas true of $\bar{y}$ are true of $\bar{x}$; so $\bar{x} \leq_1 \bar{y}$ if and only if
\[ \bigdoublewedge_{\varphi \text{ a } \exists_1 \text{ formula}} [\varphi(\bar{y}) \Rightarrow \varphi(\bar{x})]. \]
This is a computable $\Pi_2$ formula. In particular, $\mc{C}$ is 2-friendly relative to $0''$. By Theorem 15.2 of Ash-Knight, for each $\bar{a} \notin R$ and $\bar{a}_1$ there is (uniformly in $\bar{a}$ and $\bar{a}_1$) a $0''$-computable $\Pi_1$ formula $\varphi_{\bar{c},\bar{a},\bar{a}_1}(\bar{c},\bar{x},\bar{u})$ which says that $\bar{c},\bar{a},\bar{a}_1 \leq_1 \bar{c},\bar{x},\bar{u}$. 

Since there are no tuples in $\neg R$ which are 2-free over $\bar{c}$, for each $\bar{a} \notin R$ there is $\bar{a}_1$ such that for every $\bar{a}' \in R$ and $\bar{a}_{1}'$, $\bar{c},\bar{a},\bar{a}_1 \nleq_1 \bar{c},\bar{a}',\bar{a}_{1}'$. Since $\leq_1$ is computable in $0''$, we can find such a $\bar{a}_1$ for each $\bar{a} \in \notin R$ using $0'''$. For each $\bar{a}$, using this $\bar{a}_1$, define
\[ \psi_{\bar{a}}(\bar{c},\bar{x}) = (\exists \bar{u}) \varphi_{\bar{c},\bar{a},\bar{a}_1}(\bar{c},\bar{x},\bar{u}). \]
This formula is true of $\bar{a}$, but it is not true of any element of $R$. Thus $\neg R$ is defined by
\[ \bigdoublevee_{\bar{a} \in R} \psi_{\bar{a}}(\bar{c},\bar{x}). \]
This is a $0'''$-computable $\Sigma_2$ formula.
\end{proof}
\end{prop}

The proof of Theorem \ref{relativized-1} is quite complicated and will take the rest of this chapter.

\section{Approximating a 2-\cea Set}

Let $B$ be c.e.\ and let $A$ be c.e.\ in and above $B$. As $A$ is $\Sigma^0_2$, there is a computable approximation $f(x,s)$ for $A$ such that $x \in A$ if and only if $f(x,s) = 0$ for sufficiently large $s$, and $x \notin A$ if and only if $f(x,s) = 1$ for infinitely many $s$. However, if $A$ is an arbitrary $\Sigma^0_2$ set, and $x \in A$, then $A$ cannot necessarily compute a stage $s$ after which $f(x,t) = 0$. We will begin by showing that $A$ in fact has such an approximation by virtue of computing $B$. Everything in this section relativizes.

\begin{lem}
Let $B$ be c.e.\ and let $A$ be c.e.\ in and above $B$. There is a computable approximation $f:\omega^2 \to \{0,1\}$ such that $x \in A$ if and only if $f(x,s) = 0$ for sufficiently large $s$, and $x \notin A$ if and only if $f(x,s) = 1$ for infinitely many $s$. Moreover, $B$ (and hence $A$) can compute uniformly whether for all $t \geq s$, $f(x,t) = 0$.
\begin{proof}
As $B$  is c.e., it has a computable approximation $B_s$. Let $e$ be such that $A = W_e^B$. Set $f(x,0) = 0$. Suppose that we have defined $f(x,s)$. If $x \notin W_e^{B_s}$, then let $f(x,s+1) = 1$. If $x \in W_e^{B_s}$, and $f(x,s) = 1$, set $f(x,s+1) = 0$. If $x \in W_e^{B_s}$ and $f(x,s) = 0$, then $x \in W_e^{B_{s-1}}$ with some use $u$. We set $f(x,s+1) = 0$ if the computations $x \in W_e^{B_s}$ and $x \in W_e^{B_{s-1}}$ are the same (that is, if $B_s$ and $B_{s-1}$ agree up to the use $u$). Otherwise, set $f(x,s+1) = 1$.

If $x \in W_e^B$, there is some computation with finite use that witnesses this, and hence $f(x,s) = 0$ for sufficiently large $s$. If $x \notin W_e^B$, then there is no such computation; so either there are infinitely many stages $s$ at which $x \notin W_e^{B_s}$, or there are infinitely many pairs of stages $s$ and $s+1$ such that $x \in W_e^{B_s}$ and $x \in W_e^{B_{s-1}}$, but the computations are different, and hence $f(x,s+1) = 1$. So there are infinitely many stages $s$ at which $f(x,s) = 1$ and $f$ is a $\Sigma^0_2$ approximation for $A$.

Finally, suppose $s$ is a stage at which $f(x,s) = 0$; we wish to know whether, for all $t > s$, $f(x,t) = 0$. There is some computation $x \in W_e^{B_{s-1}}$ with use $u$ witnessing that $f(x,s) = 0$. Using $B$, we can decide whether $B_{s-1}$ correctly approximates $B$ up to $u$. If so, $f(x,t) = 0$ for all $t > s$. Otherwise, there is some first stage $t > s$ at which $B_{t-1}$ is different from $B_{s-1}$ below the use $u$; either $x \notin W_e^{B_{t-1}}$, or $x \in W_e^{B_{t-1}}$ but the computation is different from that which witnesses $x \in W_e^{B_{s-1}}$. In either case, $f(x,t) = 1$.
\end{proof}
\end{lem}

We want to put an approximation like the one in the previous lemma on a tree so that the true path is the leftmost path visited infinitely often, while still maintaining the ability of $A$ to compute stages at which it stabilizes. We will also keep track of how many separate times a particular node in $2^{< \omega}$ is visited. Let $\omegaext$ be $\{\infty\} \cup \omega$ (with $\infty$ viewed as being to the left of each element of $\omega$). Let $T$ be the tree $(\omegaext)^{< \omega}$. If $\sigma \in T$, we write $v(\sigma)$ for the string in $2^{< \omega}$ of the same length as $\sigma$ which replaces each $\infty$ in $\sigma$ with $0$, and each other entry with $1$. For $f \in [T]$, $v(f)$ is defined similarly. We denote by $T_\infty$ the set of nodes which end in $\infty$, and by $T_\omega$ the set of nodes which end in an element of $\omega$. For $\sigma \in T$, we denote by $\sigma^-$ the proper initial segment of $\sigma$ of length one less. We denote by $\ell(\sigma)$ the last entry of $\sigma$.

\begin{lem}\label{approx}
Let $B$ be c.e.\ and $A$ be c.e.\ in and above $B$. There is a computable approximation $(\sigma_s)_{s\in\omega} \in T^\omega$ such that there is a unique $g \in [T]$ with $\rho \subset g$ if and only if $\rho \subset \sigma_s$ for infinitely many $s$. From $A$ we can compute $g$, and $v(g) = A$. Moreover,
\begin{enumerate}[label=(\roman*)]
	\item for each $\tau \in T$ of length $n$, if $s_0,s_1,\ldots$ are the stages $s$ at which $\tau \subset \sigma_s$, then $\sigma_{s_0} = \tau$ and the sequence $a_0 = \sigma_{s_1}(n),a_1 = \sigma_{s_2}(n),\ldots$ has the property that $a_0 = \infty$ and if $a_i \in \omega$, then $a_i$ counts the number of $j < i - 1$ such that $a_j = \infty$ and $a_{j+1} \in \omega$, and
	\item if $s < t$ are such that $\sigma_s$ and $\sigma_t$ are compatible, then $\sigma_t$ is a strict extension of $\sigma_s$, and if $\sigma_s$ is the largest initial segment of $\sigma_t$ which has appeared before stage $t$, then $\sigma_t$ extends $\sigma_s$ by a single element.
\end{enumerate}
\end{lem}

It is a consequence of (i) that if we visit some node $\sigma$, and then move further left in the tree than $\sigma$, we will never again return to $\sigma$. We may however return to a node further to the right of $\sigma$. For example, if $f(0) = 0$, $f(1) = 1$, $f(2) = 0$, and $f(3) = 1$ then we might have $\sigma_0 = \infty$, $\sigma_1 = 0$, $\sigma_2 = \infty \infty$, and $\sigma_3 = 1$. Since $\sigma_2$ is to the left of $\sigma_1$, we can no longer visit $\sigma_1$; instead, we visit $\sigma_3$ which is further to the right.

\begin{proof}
Let $f$ be the approximation from the previous lemma. We may assume that $f(x,0) = 0$ for all $x$. We will construct the approximation $(\sigma_s)_{s\in\omega}$ in stages. We will actually build a sequence of pairs $(\sigma_s,\tau_s)$ with $\tau_s \in \omega^{< \omega}$. For each $s$, $\tau_s$ will be the same length as $\sigma_s$, and $\tau_s(n)$ will point to a stage of the approximation $f$ to which $\sigma_s(n)$ corresponds in the sense that $\sigma_s(n) = f(n,\tau_s(n))$. Begin with $\sigma_0 = \infty$ and $\tau_0 = 0$.

Suppose that we have defined $(\sigma_s,\tau_s)$ for $s = 0,1,\ldots,n$. Now we will define $(\sigma_{n+1},\tau_{n+1})$, beginning with $\sigma_{n+1}(0)$ and $\tau_{n+1}(0)$ and then proceeding recursively. Let $t_0 = n+1$ and $a_0 = f(0,n+1)$. Suppose that we have defined $t_0,\ldots,t_\ell$ and $a_0,\ldots,a_\ell$. Let $m_{\ell+1}$ be the most recent stage at which the string $\sigma_{m_{\ell+1}}(0),\ldots,\sigma_{m_{\ell+1}}(\ell)$ agrees with the string $f(0,t_0),f(1,t_1),\ldots,f(\ell,t_\ell)$ in the sense that
\[ v(\sigma_{m_{\ell+1}}\upharpoonright_{\ell+1}) = f(0,t_0),f(1,t_1),\ldots,f(\ell,t_\ell). \]
Let $t_{\ell+1} = \tau_{m_{\ell+1}}(\ell+1) + 1$ (or $0$ if $|\tau_{m_\ell+1}| \leq \ell+1$). Let $a_{\ell+1} = \infty$ if $f(\ell+1,t_{\ell+1}) = 0$. If $f(\ell+1,t_{\ell+1}) = 1$, let $s_0,s_1,\ldots,s_r$ be the stages $s \leq \ell$ at which $(a_0,\ldots,a_\ell) \subset \sigma_s$; let $a_{\ell+1} \in \omega$ count the number of $j < r$ such that $\sigma_{s_j}(\ell+1) = \infty$ and $\sigma_{s_{j+1}}(\ell+1) \in \omega$. Continue defining the $t_i$ and $a_i$ until we reach some string $a_0,\ldots,a_{\ell+1}$ which has not appeared before as a $\sigma_s$ with $s \leq n$. Let $\sigma_{n+1} = a_0,\ldots,a_{\ell+1}$ and $\tau_{n+1} = t_0,\ldots,t_{\ell+1}$. This completes the construction.

We need to show that there is a unique $g \in [T]$ with $\rho \subset g$ if and only if $\rho \subset \sigma_s$ for infinitely many $s$. We show this by induction on the length of $\rho$, building $g$ as we go. At the same time, we will see that $v(g) = A$. For $|\rho| = 1$, we always have $\sigma_s(0) = f(0,s)$, so $\sigma_s(0) = \infty$ for infinitely many $s$ if and only if $f(0,s) = 0$ for infinitely many $s$. On the other hand, if $f(0,s) = 1$ for sufficiently large $s$, then for sufficiently large $s$, $\sigma_s(0) = k$ for some particular $k \in \omega$.

Now we consider the case $|\rho| = n+1$, supposing that there is a unique $\rho'$ of length $n$ with $\rho' \subset \sigma_s$ for infinitely many $s$. Let $s_0,s_1,\ldots$ be the stages $s$ at which $\rho \subseteq \sigma_s$. Now, from the construction, we see that $\sigma_{s_0} = \rho$, since the construction of $\sigma_{s_0}$ stops when $\sigma_{s_0}$ is distinct from each $\sigma_s$ for $s < s_0$. In defining $\sigma_{s_1}$, $s_0$ is the value for the previous stage $m_{n+1}$ in the construction; and $t_0 = \tau_{s_1}(n+1) = 0$. In defining $\sigma_{s_2}$, $s_1$ is the previous stage, and so $t_{n+1} = \tau_{s_2}(n+1) = 1$; and so on. So $\tau_{s_{i+1}}(n+1) = i$. Moreover, $\sigma_{s_{i+1}}(n+1) = \infty$ if $f(n+1,i) = 0$ and is in $\omega$ otherwise, and if $f(0,s) = 1$ for sufficiently large $s$, then for sufficiently large $i$, $\sigma_{s_i}(n+1) = k$ for some particular $k \in \omega$.

Finally, (i) and (ii) can be seen directly from the construction: (i) from the choice of $k$ and (ii) from the fact that the $t_i$ and $a_i$ are constructed until $a_0,\ldots,a_{\ell+1} = \sigma_{n+1}$ is distinct from $\sigma_0,\ldots,\sigma_n$; thus $\sigma_{n+1}$ is an extension by a single element of the longest previous string it is compatible with.
\end{proof}

The lemma relativizes as follows.

\begin{corollary}\label{approx-cor}
Let $A$ be 2-\cea relative to $\textbf{d}$. There is a $\textbf{d}$-computable approximation $(\sigma_s)_{s\in\omega} \in T^\omega$ and a unique $g \in [T]$ as in the lemma with $A \geq_T g$ and $v(g) = A$.
\end{corollary}

\section{Basic Framework of the Construction}\label{sec:framework}

In this section, we will describe what we are building during the construction and nine properties that we will maintain at each stage. We will also show how this will allow us to compute the relation from the structure and vice versa. Let $\mc{C}$, \textbf{d}, $R$, and $A$ be as in Theorem \ref{relativized-1}: $\mc{C}$ is a computable structure, $R$ is an additional computable relation on $\mc{C}$, $\textbf{d}$ is a degree above $0^{(\omega+1)}$, and $A$ is 2-\cea in $\textbf{d}$. For simplicity, assume that $R$ is unary.

We will build a model $\mc{D} \leq \textbf{d}$ with domain $\omega$. The construction will be by stages, at each of which we define increasing finite portions $\mc{D}_s$ of the diagram.  At each stage we will define a finite partial isomorphism $F_s$. This partial isomorphism will map $\{0,\ldots,n\}$ into $\mc{C}$ for some $n$. These partial isomorphisms will not necessarily extend each other; however, they will all be consistent with the partial diagram $\mc{D}_s$ we are building. Moreover, there will be a bijective limit $F: \omega \to \mc{C}$ along the true stages of the construction. Then $\mc{D}$ will be given by the pullback, along $F$, of the diagram of $\mc{C}$. $\mc{D}$ will be computable because its diagram will be $\bigcup \mc{D}_s$. To simplify the notation, denote by $F_s(a_1,\ldots,a_\ell)$ the tuple $F_s(a_1),\ldots,F_s(a_\ell)$.

Here is the basic idea. To code that some element $x$ is not in $A$, we will put an element $a \notin R$ which is 2-free into the image of $F$ (there may already be some elements in the range of $F$; $a$ should be free over them as well). At following stages, we may add more elements $\bar{a}_1$ into the image of $F$. If, at some later stage, we think that $x \in A$ then we will replace $a,\bar{a}_1$ with $a',\bar{a}_{1}'$ where $a' \in R$ and so that every existential formula true of $a',\bar{a}_{1}'$ is true of $a,\bar{a}_1$. Then it is possible that at some later stage we again think that $x \notin A$. We can replace $a',\bar{a}_{1}'$ with $a,\bar{a}_1$, returning to the previous stage in the construction at which we thought that $x \notin A$. We have probably added some more elements to the image of $F$ while we thought that $x \in A$, but the fact that $a,\bar{a}_1 \leq_1 a',\bar{a}_{1}'$ means that we can find corresponding witnesses over $a,\bar{a}_1$. Using $R^\mc{D}$, we can figure out what happened during the construction, and hence whether or not $x \in A$. Now we know how to code a single fact.

Now we will describe how to code two elements $x < y$. To code the fact that $x \notin A$, add to the range of $F$ an element $a \notin R$ which is 2-free. To code that $y \notin A$, add to the range of $F$ another element $b \notin R$ which is 2-free over $a$. Now if at some later stage, we think that $y \in A$, we can act as above. But if we think that $x \in A$, then we replace $a,b$ with some $a',b'$ with $a' \in R$ and $ab \leq_1 a' b'$. Now $b'$ is not necessarily $2$-free; so to code that $y \notin A$ while we think that $x \in A$, we need to add a new element $c$ which is 2-free over $a' b'$. At some later stage, we might think again that $x \notin A$, so we need to return to $a b$; we can do this because $ab \leq_1 a' b'$ (and $c$ is replaced by some $c'$ which is not necessarily 2-free over $ab$, but that does not matter because now $b$ is doing the coding again). So far, this is essentially describing the argument for Theorem 2.1 of \cite{AshKnight97}. This succeeds in coding $A$ into $R^{\mc{D}}$, but we require $\textbf{d}'$ in order to decode it. The problem is the following situation. Suppose that we have done the construction as described above, coding whether $x$ or $y$ is in $A$, and we currently think that neither is. Our function $F$ looks like $abc'$ so far. But then at some later stage we might think for a second time that $x \in A$.  We replace $abc'$ by $a''b''c''$, with $a'' \in R$ and $abc' \leq_1 a''b''c''$. Now $c''$ may not be 2-free over $a''b''$, so we need to add a new element $d$ which is 2-free over $a''b''c''$, and use $d$ to code whether $y \in A$. Using just $R^\mc{D}$ and $\textbf{d}$, we cannot distinguish between the case $a'b'c$ where $y$ is being coded by the third element, or $a''b''c''d$ where it is being coded by the fourth element ($a'$ and $a''$ are both in $R$, and we cannot control whether or not $b'$, $b''$, or $c''$ are in $R$). We \textit{can} differentiate between the two cases using $\textbf{d}'$, which is the basis of the proof by Ash and Knight.

What we do to solve this is the first time we think that $x \in A$, after choosing $a'$ and $b'$, we add a new element $a_0$ before adding $c$ which is $2$-free over $a'b'a_0$. When we later believe that $x \notin A$ again, we return to $aba_{0}'c'$ for some $a_{0}'$; and then later when we believe that $x \in A$ once more, we choose $a_{0}''$ as well as $a''$, $b''$, and $c''$. The trick will be to ensure that $a_0 \in R$ if and only if $a_{0}'' \notin R$. Thus $R$ can differentiate between the two cases. The argument as to why we can choose an appropriate $a_0$ is complicated, and is the main difficulty in this proof. The element $a_0$ will actually need to be a finite list of elements. For now, we do not need to worry about how we choose $a_0$---this will be done in the following sections.

In the previous section we had a computable approximation for $A$ which kept track of how many times we switched between believing an element was in $A$ and not in $A$. We will assigns labels from the tree $T$ to positions in the partial isomorphism to say what those positions are coding. In the situation above, we will assign to the position in which $a$ is the label $\infty$; it codes whether $\infty$ (i.e. $x \notin A$) is the correct approximation to $A$. The position of $b$ would code $\infty \infty$ since it codes, if the correct approximation for whether or not $x \in A$ is $\infty$ (or $x \notin A$), that $y \notin A$. The position of $a_0$ codes $0$, since it represents the first time that we think that $x \in A$. The position of $c$ codes $0 \infty$ since, if $0$ is the correct approximation for $x$, it tries to code $y \notin A$. When we add in $d$ after $a''b''c''a_{0}''$, we will actually first add another element $a_{1}$ in case at a further stage, we think that $x \notin A$, and then later think that $x \in A$ for a third time; $a_{1}$ will code $1$, and $d$ will code $1 \infty$.

To keep track of these labels, we will define a partial injection $\loc_s: T \to \dom(F_s)$. While $\loc_s$ is dependent on the stage $s$, once we set $\loc_s(\sigma)$ at some stage $s$, we will never change the image of $\sigma$. So if $s < t$, we will have $\loc_s \subseteq L_t$. At each stage, we will increase the domain of $\loc_s$ by a finite amount, and so the domain of each $\loc_s$ will be finite. In fact, the domain of $\loc_s$ will be $\{\sigma_i : 0 \leq i \leq s \}$. Because of this, $\loc_s$ will satisfy two closure properties on its domain. First, if $\loc_s$ is defined at some string, the it is defined at every initial segment; and second, if $\loc_s$ is defined at some string $\sigma \concat b$ ending in $b$, and $a < b$ in $\omegaext$, then $\loc_s$ is defined at $\sigma \concat a$. Each $\sigma \in T_\omega$ will label not just $\loc_s(\sigma)$, but a whole tuple $\loc_s(\sigma),\ldots,\loc_s(\sigma) + k - 1$ where $k = k_s(\sigma) > 0$ is a value defined at stage $s$. For each $\sigma \in \dom(\loc_s) \cap T_\omega$, we will also maintain a valuation $m_s(\sigma) \in \{-1,1\}^{k}$ which represents a choice of $R$ or $\neg R$ for the $k$ elements labeled by $\sigma$. We will write $\bar{a} \in R^{m_s(\sigma)}$ if $a_i \in R$ whenever $m_s(\sigma) = 1$, and $a_i \notin R$ otherwise. Like $\loc_s$, once we set $k_s(\sigma)$ or $m_s(\sigma)$, the value will be fixed.

Let $\sigma_s$ and $f$ be as in Corollary \ref{approx-cor} with $(\sigma_s)_{s \in \omega} \leq \textbf{d}$ and $A \equiv_T f$. We will take a moment to show how we will compute $f$ from $R^\mc{D}$ and $\textbf{d}$. Let $\loc$ be the union of all the $\loc_s$, and similarly for $m$ and $k$. $\loc$ contains each initial segment of $f$. These are partial $\textbf{d}$-computable functions. We will build $\mc{D}$ in such a way that
\begin{description}
	\item[(C1\textsuperscript{*})] if $\sigma \in T_\infty$ and $\sigma \subset f$ then $\loc(\sigma) \notin R^\mc{D}$;
	\item[(C2\textsuperscript{*})] if $\sigma \in T_\infty$ and $\sigma \nsubset f$ but $\sigma^- \subset f$, then $\loc(\sigma) \in R^\mc{D}$;
	\item[(C3\textsuperscript{*})] if $\sigma \in T_\omega$ and  $\sigma \subset f$ then $(\loc(\sigma),\ldots,\loc(\sigma) + k(\sigma)-1)) \in (R^\mc{D})^{m(\sigma)}$; and
	\item[(C4\textsuperscript{*})] if $\sigma \in T_\omega$ and  $\sigma \nsubset f$ but $\sigma^- \subset f$ then $(\loc(\sigma),\ldots,\loc(\sigma) + k(\sigma)-1) \notin (R^\mc{D})^{m(\sigma)}$.
\end{description}
We will use $R^\mc{D}$ to recursively compute longer and longer initial segments of $f$. Suppose that we have computed $\tau \subset f$. First, check whether $\loc(\tau \concat \infty) \in \notin R^\mc{D}$; if it is, then by \CTwoStar we must have $\tau \concat \infty \subset f$. Otherwise, by \COneStar there is some $x \in \omega$ such that $\tau \concat x \subset f$. For each of $x=0,1,\ldots$, let $a_x = \loc(\tau \concat x)$ and check whether $(a_x,a_{x}+1,\ldots,a_x + k(\tau \concat x)) \in (R^\mc{D})^{m(\tau \concat x)}$; if so for some $x$, then by \CFourStar, $\tau \concat x \subset f$; otherwise, continue searching. By \CThreeStar, we will eventually find the correct initial segment of $f$. Thus we will have $f \leq_T R^\mc{D} \oplus \loc \leq_T R^\mc{D} \oplus \textbf{d}$.

Now we will describe some properties which the partial isomorphisms $F_s$ will have. We required above that $\dom(\loc)$ consisted of $\sigma_0,\sigma_1,\sigma_2,\ldots$. So at each stage $s$, we must make sure that $\sigma_s$ is assigned a coding location:
\begin{description}
	\item[(CLoc)] \label{G-defined} $\dom(\loc_s)$ contains $\sigma_s$.
\end{description}
We also need to make sure the four properties \COneStar, \CTwoStar, \CThreeStar, and \CFourStar of $G$ and $R^\mc{D}$ are true by doing the correct coding during the construction. At each stage $s$ and for each $\sigma \in \dom(\loc_s)$, we will ensure that:
\begin{description}
	\item[(C1)] if $\sigma \in T_\infty$ and $\sigma \subset \sigma_s$ then $F_s(\loc_s(\sigma)) \in \notin R$;
	\item[(C2)] if $\sigma \in T_\infty$ and $\sigma \nsubset \sigma_s$ but $\sigma^- \subset \alpha_s$, then $F_s(\loc_s(\sigma)) \in R$;
	\item[(C3)] if $\sigma \in T_\omega$ and  $\sigma \subset \sigma_s$ then $F_s(\loc_s(\sigma),\ldots,\loc_s(\sigma) + k_s(\sigma)-1)) \in R^{m_s(\sigma)}$; and
	\item[(C4)] if $\sigma \in T_\omega$ and  $\sigma \nsubset \sigma_s$ but $\sigma^- \subset \sigma_s$ then $F_s(\loc_s(\sigma),\ldots,\loc_s(\sigma) + k_s(\sigma)-1) \notin R^{m_s(\sigma)}$.
\end{description}
The $F_s$ will maintain the same atomic diagram, even if they do not agree on particular elements. If $s < t$ then
\begin{description}
	\item[(At)] $F_s \leq_0 F_t$.
\end{description}
However, if two stages $s$ and $t$ agree on some part of the approximation, then they will also agree on how that part of the approximation is being coded. This will ensure that we can construct a limit $F$ by looking at the values of the $F_s$ at stages where the approximation is correct.
\begin{description}[resume*]
	\item[(Ext)] If $\alpha \subset \sigma_s$ and $\alpha \subset \sigma_t$, then for all $x \leq \loc(\alpha)$, $F_s(x) = F_t(x)$; if $\sigma_s \subset \sigma_t$, then in fact $F_s \subset F_t$.
\end{description}
We want the final isomorphism $F$ to be total. To do this, we need to ensure that we continue to add new elements into its image.
\begin{description}
	\item[(Surj)] The first $|\sigma_s| - 1$ elements of $\mc{C}$ appear in $\ran(F_s \upharpoonright_{\loc_s(\sigma_s)})$.
\end{description}
It may be helpful to remember what the labels stand for: \CLoc for \textit{coding location}; \COne, \CTwo, \CThree, and \CFour for \textit{coding 1}, \textit{coding 2}, \textit{coding 3}, and \textit{coding 4} respectively; \At for \textit{atomic agreement}; \Ext for \textit{extension}; and \Surj for \textit{surjective}.

There is one last condition which ensures that we can complete the construction while still satisfying the above conditions. Before stating it in the next section, we have already done enough to describe the $A$-computable isomorphism $F: \mc{D} \to \mc{C}$ and see that $A \geq_T R^\mc{D}$. Recall that $f$ is the path approximated by the $\sigma_s$ as in Corollary 1.3. Let $s_1,s_2,\ldots$ be the list of stages $s_n$ such that $\sigma_{s_n} \subset f$. Then $\sigma_{s_1} \subsetneq \sigma_{s_2} \subsetneq \cdots$ is a proper chain and $f = \bigcup \sigma_{s_n}$. As before, let $\loc = \bigcup \loc_s$ and similarly for $k$ and $m$. Now $\loc(\alpha_{s_1}),\loc(\alpha_{s_2}),\ldots$ is a strictly increasing sequence in $\omega$, and for each $i < j$, $F_{s_i} \subseteq F_{s_j}$ by \Ext. Define $F(n) = \bigcup F_{s_i}$; this is a total function because the $F_{s_i}$ are defined on increasingly large initial segments of $\omega$. By \Surj, $F$ is onto as $|\sigma_{s_i}| \geq i$ and so the first $i-1$ elements of $\mc{C}$ appear in the range of $F_{s_i}$ below $G(\sigma_{s_i})$, and hence appear in the range of $F$. $F$ is injective since each $F_s$ is; and the pullback along $F$ gives an isomorphic structure $\mc{D}$ whose diagram is the union of the diagrams of the $F_s$, and these diagrams agree with each other because of \At. So the atomic diagram of $\mc{D}$ is computable in $\textbf{d}$. The sequence $s_0,s_1,\ldots$ can be computed by $A$ because $A$ can compute $f$; and, knowing the sequence $s_0,s_1,\ldots$, we can compute $F$. Hence $A$ can compute $F$ and so $A \geq_T R^\mc{D}$ (recall that $\mc{C}$ and $R$ were computable). Also, \COne, \CTwo, \CThree, and \CFour imply respectively \COneStar, \CTwoStar, \CThreeStar, and \CFourStar. Earlier we argued that as a consequence, $f \leq_T R^\mc{D} \oplus \textbf{d}$. Hence $A \equiv_T R^\mc{D} \oplus \textbf{d}$ as required.

\section{An Informal Description of the Construction}

Recall the intuitive picture of how the coding is done in the previous section, but now using some of the more precise notation just developed. We will begin by looking at coding a single element, but now choosing the element $a_0$ which we had to add when we were coding two or more elements. Things will start to get more complicated than they were before, so Figure \ref{fig:CodingOne} shows the isomorphism $F_s$ as it changes. We began by choosing an element $a \notin R$ which is 2-free. We labeled $a$ with $\infty$, coding that $0 \notin A$. Now, at some stage, we might think that $0 \in A$, so we replace $a$ by $a'$ with $a' \in R$. At this point, we must choose some $a_0$ to code $0$. We are now concerned with the issue, which we ignored before, of how to choose $a_0$. What properties does $a_0$ need to have? If at some later stage after we have added $\bar{b}$ to the image of $F$, we think that it is actually the case that $0 \notin A$, we replace $a' a_0 \bar{b}$ by $a a_0' \bar{b}'$. Then, if at some further later stage after we have added more elements $\bar{c}$ to the image of $F$, we once more think that $0 \in A$, we replace $a a_0' \bar{b}' \bar{c}$ by $a'' a_0'' \bar{b}'' \bar{c}'$ where
\[ a a_0' \bar{b}' \bar{c} \leq_1 a'' a_0'' \bar{b}'' \bar{c}' \]
and $a'' \in R$. We need to have $a_0 \in R$ if and only if $a_0'' \notin R$. So what we need to know is that, no matter what elements $\bar{b}$ are added to the image of $F$, we can choose $a_0'$ so that no matter which elements $\bar{c}$ are then added to the image of $F$, we can choose $a''$, $a_0''$ and so on such that $a_0 \in R$ if and only if $a_0'' \notin R$. This is a sort of game where each player gets two moves---we are choosing $a_0'$, $a''$, $a_0''$, etc. satisfying the required properties while both we and our opponent together choose the tuples $\bar{b}$ and $\bar{c}$. By this, we mean that we choose a tuple, and then the tuple $\bar{b}$ (or $\bar{c}$) that our opponent plays must extend the tuple which we chose. We can do this because at any point, we can add any elements we want to the isomorphism. We want to know that we have a winning move in this game.

\begin{figure}[htb]
\begin{center}
\begin{tabular}{l l | *{5}{c}}
\multicolumn{2}{c}{} & \multicolumn{5}{c}{Coding locations} \\
& $\sigma_s$  & $\infty$ & $0$ & & & $1$ \\
\cline{2-7}
\multirow{6}{*}{\rotatebox[origin=r]{90}{Approximation}} & $\infty$ 					& $\yes{a}$ &  &  &  & \\
& $0$   & $\no{a'}$ & $\yes{a_0}$ &  & &  \\
& $0 \cdots$   & $\no{a'}$ & $\yes{a_0}$ & $\bar{b}$ & & \\
& $\infty$        & $\yes{a}$ & $a_0'$ & $\bar{b}'$ &  &  \\
& $\infty \cdots$        & $\yes{a}$ & $a_0'$ & $\bar{b}'$ & $\bar{c}$ &  \\
& $1$        & $\no{a''}$ & $\no{a_0''}$ & $\bar{b}''$ & $\bar{c}'$ & $\yes{a_1}$ \\
\end{tabular}
\end{center}
	\caption{The values of the isomorphism $F_s$ when coding a single element of $A$. The column $\sigma_s$ shows the approximation at a particular stage, and the coding location shows an indexing via $\loc$. An entry is surrounded by brackets ``$[\quad]$'' to show that it is coding ``yes'' (i.e., if it is at coding location $\tau$, then it is coding $\tau$). If an entry is ``active'' in the sense that it is of the form $\tau \concat x$ and $\tau$ is being coded as ``yes'', but $\tau \concat x$ is coding ``no'' then it is marked as ``$\no{\quad}$.'' Any other entries which are not active are unmarked.}
	\label{fig:CodingOne}
\end{figure}

Now let $y \notin R$ be 2-free over $a'$. So there is $x \in R$ with $a' x \leq_1 a' y$. Now we can try choosing $a_0 = y$. If this works (in the sense that for $a_0 = y$, no matter which $\bar{b}$ and then $\bar{c}$ our opponent chooses we can choose $a_0'$ and $a_0''$ as required), then we can just choose $a_0 = y$ and we are done. If this choice does not work, then I want to argue that choosing $a_0 = x$ does work. If $y$ does not work, that means that for some $\bar{b}$ which our opponent plays, every $a_0'$ we choose puts us in a losing position. This means that there is some existential formula $\varphi(u,v)$ (which is witnessed by the elements $\bar{b}$) so that $y$ satisfies $\varphi(a',v)$ but that every $a_0'$ which works satisfies $\varphi(a,v)$ puts us in a losing position. This means that for every such $a_0'$ there is a tuple $\bar{c}$ which our opponent can play so that any $a''$, $a_0''$, etc. we choose with
\[ a a_0' \bar{b}' \bar{c} \leq_1 a'' a_0'' \bar{b}'' \bar{c}' \]
has $a_0'' \in R$ since $y$ was also in $R$. Now if we instead started with $a_0 = x$, then that same existential formula $\varphi(u,v)$ which was true of $a',y$ is also true of $a',x$ since $a',x \leq_1 a'y$. We can add to the isomorphism a tuple witnessing that $\varphi(u,v)$ holds of $a',x$. So no matter which tuple $\bar{b}$ our opponent actually plays, the existential formula $\varphi(a',x)$is witnessed by elements from $\bar{b}$. Then, since $a \leq_1 a'$, there are $a_0' = x'$ and $\bar{b}'$ such that $a' x \bar{b} \leq_0 a x' \bar{b}'$; thus $\varphi(a,a_0)$ holds. But then this $a_0'$ is one which did not work for the choice $a_0 = y$. Now we add to the isomorphism the tuple which our opponent used to beat us at this point when we chose $a_0 = y$ and then also chose this value of $a_0'$. So no matter which tuple $\bar{c}$ our opponent actually plays, it contains the tuple which they used to win when $a_0$ was $y$. Thus, for every $a''$, $a_0''$, etc. we choose with
\[ a a_0' \bar{b}' \bar{c} \leq_1 a'' a_0'' \bar{b}'' \bar{c}' \]
we have $a_0'' \in R$. This did not work for $a_0 = y$, but since $x \notin R$, it does work for $a_0 = x$.

What we have done is taken our opponent's strategy from the case $a_0 = y$, and forced them to use it when $a_0 = x$. Their strategy consists only of choosing the tuples $\bar{c}$ and $\bar{d}$, and these tuples are chosen by us together with our opponent (because we can add them to the isomorphism before our opponent does). So when $a_0 = x$ we can force our opponent to play tuples $\bar{c}$ and $\bar{d}$ which extend the tuples they played when $a_0 = y$. But because $y \in R$ and $x \notin R$, the winning conditions are different for the different choices of $a_0$, so what won our opponent the game for $a_0 = y$ now loses him the game for $a_0 = x$.

Choosing $a_0$ can begin to get more complicated when we are coding two elements. Figure \ref{fig:CodingTwo} shows the isomorphism when coding two elements. For example, suppose that we choose $a$ and $b$, both not in $R$, with $a$ 2-free and $b$ 2-free over $a$. We label $a$ with $\infty$, coding $0 \notin A$, and $b$ with $\infty \infty$, coding $1 \notin A$. We think that $1 \in A$, and replace $b$ by some $b'$ with $a b \leq_1 a b'$. Then we add a new element $\bar{b}_0$ which is labeled by $\infty 0$. At some later stage, we think that $0 \in A$, and replace $a b' \bar{b}_0$ with $a'b''\bar{b}_0'$ with $a b' \bar{b}_0 \leq_1 a' b'' \bar{b}_0'$. Now we need to add a tuple $\bar{a}_0$ which is labeled by $0$, followed by an element $c$ which is 2-free over $a'b''\bar{b}_0'\bar{a}_0$.

\begin{figure}[htb]
\begin{center}
\begin{tabular}{l l | *{9}{c}}
\multicolumn{2}{c}{} & \multicolumn{9}{c}{Coding locations} \\
& $\sigma_s$ & $\infty$ & $\infty \infty$ & $\infty 0$ & $0$ & $0 \infty$ & & & & 1  \\
\cline{2-11}
 & $\infty$ 					& $\yes{a}$ &  &  &  &  &  &  & &\\
& $\infty \infty$   & $\yes{a}$ & $\yes{b}$ &  & &   & & &   &  \\
& $\infty 0$        & $\yes{a}$ & $\no{b'}$ & $\yes{\bar{b}_0}$ &  &  &   &  &    & \\
& $0$					      & $\no{a'}$ & $b''$ & $\bar{b}_0'$ & $\yes{\bar{a}_0}$ & &   &  &    & \\
& $0 \infty$				& $\no{a'}$ & $b'' $ & $\bar{b}_0'$ & $\yes{\bar{a}_0}$ & $\yes{c}$ & $\bar{d}$  &  & &    \\
& $\infty 0$	& $\yes{a}$ & $\no{b'}$ &  $\yes{\bar{b}_0}$ & $\bar{a}_0'$ & $c'$ &  $\bar{d}'$ & $\bar{e}$ & & \\
& \multicolumn{9}{ l }{} \\
& \multicolumn{9}{ c }{First case: opponent plays $1$ immediately} \\

& $1$					      & $\no{a''}$ & $b''' $ & $\bar{b}_0''$ & $\no{\bar{a}_0''}$ & $c''$ & $\bar{d}''$ & $\bar{e}'$ & & $\yes{\bar{a}_1}$    \\

& \multicolumn{9}{ l }{} \\
& \multicolumn{9}{ c }{Second case: opponent plays $\infty \infty$ then $1$} \\

& $\infty \infty$   & $\yes{a}$ & $\yes{b}$ & $\bar{b}_0''$ & $\bar{a}_0''$ & $c''$ &  $\bar{d}''$ & $\bar{e}'$ & $\bar{f}$ &    \\
& $1$					      & $\no{a''}$ & $b''' $ & $\bar{b}_0'''$ & $\no{\bar{a}_0'''}$ & $c''$ & $\bar{d}'''$ & $\bar{e}''$ & $\bar{f}'$ & $\yes{\bar{a}_1}$    \\
\end{tabular}
\end{center}
	\caption{The values of the isomorphism $F_s$ when coding two elements of $A$. Two possibilities are shown, depending on what the opponent in the game described plays as the approximation---$1$, or $\infty \infty$ followed by $1$.}
	\label{fig:CodingTwo}
\end{figure}

Now, as before, we need to see what properties we want to be true of $\bar{a}_0$. The tuple $\bar{a}_0$ will have two entries. Suppose that we have added some tuple $\bar{d}$ to the isomorphism, and then we believe that $0 \notin A$ and $1 \in A$, so now we want to code $\infty 0$. We make our isomorphism $a b' \bar{b}_0 \bar{a}_0' c' \bar{d}'$ for some $\bar{a}_0'$, $c'$, and $\bar{d}'$, and then an additional tuple $\bar{e}$ gets added to the isomorphism. Now at some later stage, we believe that $0 \in A$ once again, so we find $a'' \in R$, $b'''$, $\bar{b}_0''$, $\bar{a}_0''$, $c''$, $\bar{d}''$, and $\bar{e}'$ such that
\[ ab'\bar{b}_0\bar{a}_0'c'\bar{d}'\bar{e} \leq_1 a''b'''\bar{b}_0''\bar{a}_0'''c''\bar{d}''\bar{e}'. \]
In order to do our coding, for any $\bar{d}$, we must be able to choose our elements so that for any $\bar{e}$, we can find such $a'' \in R$, $b'''$, etc. with the first entry of $\bar{a}_0''$ in $R$ if and only if the first coordinate of $\bar{a}$ is not in $R$.

On the other hand, the approximation might turn out to be different. After adding $\bar{d}$ to the isomorphism, and again believing that $0 \notin A$ and $1 \in A$, and making our isomorphism $a b' \bar{b}_0 \bar{a}_0' c' \bar{d}'$, we add a new tuple $\bar{e}$ to the isomorphism. In the previous case, the approximation next told us that $0 \in A$ once again; it might instead be the case that first, the approximation says that $0 \notin A$ and $1 \notin A$. Then we must change the isomorphism to $a b \bar{b}_0''\bar{a}_0''\bar{c}''\bar{d}''\bar{e}'$. After we add some tuple $\bar{f}$ to the isomorphism, then we later believe that $0 \in A$, and so we must code $1$. Now we need to find $a'' \in R$, $b'''$, $\bar{b}_0'''$, $\bar{a}_0'''$, $\bar{c}''$, $\bar{d}'''$, $\bar{e}''$, and $\bar{f}'$ such that
\[ a b \bar{b}_0'' \bar{a}_0'' c'' \bar{d}'' \bar{e}' f \leq_1 a''b'''\bar{b}_0'''\bar{a}_0'''\bar{c}''\bar{d}'''\bar{e}''\bar{f}'. \]
To do our coding, we must have the second element of $\bar{a}_0'''$ in $R$ if and only if the second element of $a_0$ is not in $R$.

When we choose the tuple $\bar{a}_0$, we do now know which case we will be in, and so we must be able to handle both cases. There are actually more possibilities than this (for example, we could think that $1 \notin A$, then $1 \in A$, then $1 \notin A$, then $1 \in A$, and so on), but it will turn out that we get these possibilities for free, and so for now we will just consider the two possibilities outlined above.

Looking at this as a game again, in addition to adding tuples to the range of the isomorphism, our opponent can now choose whether the first possibility for the approximation of $A$ described above will happen, or whether the second possibility will happen (he chooses the approximation stage by stage---so he chooses, at the same time as choosing the tuple $\bar{e}$, which possibility we must respond to). We will choose a pair $\bar{a}_0 = (a_0^1,a_0^2)$ to defeat our opponent: $a_0^1$ to defeat our opponent when he chooses to first possibility for the approximation, and $a_0^2$ for the second. We can choose $a_0^1$ exactly as before when we were only coding a single element, in order to defeat our opponent if he uses the first possibility; that is, if he chooses $\bar{d}$ and $\bar{e}$, and the approximation says that $0 \notin A$, and then $0 \in A$ (while saying that $1 \in A$ the whole time), we can choose $a_0^{1 \prime \prime}$ which is in $R$ if and only if $a_0^1$ is not. Now we have to argue that we can choose $a_0^2$ so that not only do we defeat our opponent if he uses the second possibility for the approximation of $A$, but that we still beat our opponent if he uses the first possibility. Choose $y \notin R$ which is $2$-free over $a' b'' \bar{b}_0' a_0^1$, and $x \in R$ such that $a' b'' \bar{b}_0' a_0^1 x \leq_1 a' b'' \bar{b}_0' a_0^1 y$.  Suppose that we cannot beat our opponent if we choose $a_0^2 = y$, so that he has some winning strategy for this game. Call our opponent's winning strategy for $a_0^2 = y$ their $y$-strategy.

Then we will choose $a_0^2 = x$. We will use our winning strategy for the first approximation to ensure that the only way in which we can lose is if our opponent uses the second approximation ($\infty \infty$ followed by $1$) and forces us to have $a^{2\prime\prime\prime}_0 \in R$ (recall that $a^2_0 = x \in R$). Now, other than choosing the approximation, the only plays our opponent can make are to choose $\bar{d}$, $\bar{e}$, and $\bar{f}$. Now these tuples are played by us together with our opponent, so we can force our opponent to use their $y$-strategy by forcing them to play tuples extending those they used in the $y$-strategy (note that if our opponent plays a larger tuple, it puts more of a restriction on what we can play, and so even though our opponent is not, strictly speaking, using his $y$-strategy, he is using a strategy which is even stronger). We had a winning strategy when our opponent could only choose the first approximation. We will still follow this winning strategy, playing against the tuples our opponent plays which extend the tuples which we add to the isomorphism.

If our opponent uses the first approximation, then we will win because we used the strategy that we already had to beat them in this case. If they use the second approximation, then we force them to play their winning strategy from the case $a^2_0 = y$. So no matter what we play, we are forced to choose $a_0''' \notin R$ because this is the only way that our winning strategy from before with just the first approximation could lose in this new game. But now $x \in R$, so we win.

Now we said above that we do not have to worry about more complicated possibilities for the approximation, like if we think that the approximation is $\infty 0$, then $1$, then $\infty 0$, then $2$, then $\infty 0$, and so on. This is because every time that we think that the approximation is $\infty 0$, our partial isomorphism looks like $a b' \bar{b}_0 \bar{a}_0' c' \bar{d}' \bar{e}$ for some additional tuple $\bar{e}$. Because our opponent can play any tuple they like, and also we can respond in the same way whether the approximation becomes $1$, $2$, and so on, these are all essentially the same position in the game---we can play from any of these positions in the same way that we would play from any other. The values $1$, $2$, etc. are effectively the same for our purposes in the games above because we have not added a coding location for $1$, $2$, etc. and so they all put only the requirement that $0$ not be coded as ``yes.'' If, for example, $1$, $2$, and $3$ all had coding locations as well, then it would be $4$, $5$, and so on that were all equivalent. We play the game starting at a stage $s$ only for those coding locations that exist at the stage $s$. If we are only worrying about coding finitely many elements of $A$, then after some bounded number of steps, the approximation our opponent plays will have to repeat in this way. So there is some bound $N$ such that if we can beat our opponent when he plays only $N$ stages of the approximation, then we can beat him when he plays any number of stages.

The process starts to get more complicated when we are coding more than two elements. There become even more possibilities for what could happen with the coding which our opponent could play. If we are only coding finitely many elements of $A$, the tuple $\bar{a}_0$ will be exactly as long as the number of possibilities which we have to consider. Viewing the requirements on $\bar{a}_0$ as a game means that we can ignore the exact details of all of the possibilities for the approximation, while keeping the important properties, like the fact that there are finitely many possibilities. In the next section, we will formally define the game which we have used informally throughout this section.

Now in general, we are trying to code all of the elements of $A$. At each stage $s$, we code only finitely many facts, each labeled by the function $\loc_s$, and we add only finitely many new coding locations at each stage. We will maintain the property, at each stage $s$, that we can win the game described above for those coding locations (by which we mean that we must maintains properties \COne, \CTwo, \CThree, and \CFour for those coding locations). Then whenever we add new coding locations, we must show that our winning strategy for the game at the previous stage gives rise to a winning strategy which includes these new coding locations. In this way, even though there are infinitely many coding locations which we will have to deal with eventually, at each point we only consider a game where we deal with finitely many of them. The choice of the new coding locations will have to take into account the winning strategy for the game at the previous stage.

\section{The Game \texorpdfstring{$\game{s}$}{Gs} and the Final Condition}
\label{G-s}

We return to giving the last condition \WS. So far, we have put no requirement on the construction to reflect the fact that some of the elements have to be chosen to be 2-free, or that we can mark how many separate occasions we have believed that some $x$ is in $A$. At each stage $s$ of the construction, we will associate a game $\game{s}$ with two players, \POne and \PTwo. Condition \WS will simply be
\begin{description}
	\item[(WS)] \label{extension-condition} \POne has an arithmetic winning strategy for $\game{s}$.
\end{description}
This stands for \textit{winning strategy}.

In the game $\game{s}$, \POne goes first. On their turns, \POne plays a partial injection $G: \omega \to \mc{C}$, interpreted as a partial isomorphism $\mc{D} \to \mc{C}$; on the first turn, \POne is required to play a partial isomorphism extending $F_s$. In response to $G$, \PTwo plays some elements $\bar{c} \in \mc{C}$ which are not in the range of $G$, viewed as elements in the range of an extension of $G$, and a string $\alpha$ which is a possible value for $\sigma_{t}$ for $t > s$. The string $\alpha$ must either be a string in $\dom(\loc)$ which has no proper extensions in $\dom(\loc)$, or $\sigma \concat \eta$ where $\sigma \concat x \in \dom(\loc)$ for some $x \in \omegaext$. Let $\dom(\loc)^*$ be the set of these strings. Here, $\eta$ is a symbol which we can think of as representing some $k \in \omega$ for which $\sigma \concat k$ has not yet been visited, but we do not want to differentiate between different values of $k$. Thus $\sigma \concat \eta$ should be viewed as being to the right of every extension of $\sigma$ in $\dom(\loc)$. By convention, $\tau \concat \eta$ is not in $T$. The approximation $(\sigma_t)_{t \in \omega}$ must satisfy the properties from Lemma \ref{approx}. So we also restrict \PTwo and force him to play strings which form subsequences of sequences with the properties from the lemma. So if \PTwo plays some string to the left of $\tau$ after playing $\tau$, then \PTwo can never play any string extending $\tau$ again. Also, for any string $\tau$ and $x \in \omega$, if \PTwo plays a string $\tau \concat y$ for $y > x$ (or $y = \eta$), then \PTwo can never again play a string extending $\tau \concat x$. Note that this does not apply to strings which end with $\eta$ (conceptually, \PTwo should be thought of as playing $\sigma \concat k$ for increasingly large values of $k$).

At stage $s$, we do not yet know the actual values of the approximation $(\sigma_t)_{t\in\omega}$ after stage $s$. We could compute finitely many future stages, but not all of them. In playing the $\alpha \in T$, \PTwo plays a possible future value of the approximation which we have to be able to handle. When \POne plays a partial isomorphism $G$ in response, it is an attempt to continue the construction assuming that the approximation is as \PTwo has played it. But \POne only has to continue the construction in a limited manner: they must maintain the coding given by $\loc_s$, but they do not need to add more coding locations to $\loc_s$. Since we will not be adding new elements to $\loc_s$, we will let $\loc = \loc_s$, $k = k_s$, and $m = m_s$ for the rest of this section and the next.

Now during the construction there will be certain elements which we will have to add to our partial isomorphism. For example, condition  \Surj requires us to add elements in order to make the isomorphism bijective. We will also have to add free elements in order to code new strings; we have some control over these in that we can choose which free element we choose, but not total control in that we are restricted to the free elements. This is the role of the tuples $\bar{c}$ which are played by \PTwo in response to a play $G$ by \POne: they are elements which \POne is required to add to $G$ before continuing the construction. They will also be useful for more technical reasons in the next section.

Now we will state the ways in which \POne can lose. If \POne does not lose at any finite stage, then they win (thus it is a closed game). First, there are some conditions on the coding by $G$. If \POne plays $G$ in response to $\alpha$, then \POne must ensure that for each $\sigma \in \dom(\loc)$:
\begin{description}
	\item[(C1\textsuperscript{\dag})] if $\sigma \in T_\infty$ and $\sigma \subset \alpha$ then $G(\loc(\sigma)) \in \notin R$;
	\item[(C2\textsuperscript{\dag})] if $\sigma \in T_\infty$ and $\sigma \nsubset \alpha$ but $\sigma^- \subset \alpha$, then $G(\loc(\sigma)) \in R$;
	\item[(C3\textsuperscript{\dag})] if $\sigma \in T_\omega$ and  $\sigma \subset \alpha$ then $G(\loc(\sigma),\ldots,\loc(\sigma) + k(\sigma)-1)) \in R^{m(\sigma)}$; and
	\item[(C4\textsuperscript{\dag})] if $\sigma \in T_\omega$ and  $\sigma \nsubset \alpha$ but $\sigma^- \subset \alpha$ then $G(\loc(\sigma),\ldots,\loc(\sigma) + k(\sigma)-1) \notin R^{m(\sigma)}$.
\end{description}
These are conditions which ensure that $G$ codes $\alpha$ using the coding locations given by $\loc$; they are the equivalents of conditions \COne, \CTwo, \CThree, and \CFour respectively for the $F_s$.

Now we also have global agreement conditions which are the equivalents of \At and \Ext. There is a slight modification to \At and \Ext, which is that if \POne plays $G$, and $\PTwo$ responds by playing $\bar{c}$ and $\alpha$, then we use $G \concat \bar{c}$ rather than $G$ because the $\bar{c}$ are elements that must be added to the image of $G$ before the next move. Suppose that so far, \POne has played $G_0 \supset F_s,G_1,\ldots,G_n$ and \PTwo has played $(\bar{c}_0,\alpha_1),\ldots,(\bar{c}_n,\alpha_{n+1})$. Note that the two indices of a move by \PTwo differ by one, so that \PTwo plays $(\bar{c}_{i},\alpha_{i+1})$ rather than $(\bar{c}_i,\alpha_i)$. This will turn out to be more convenient later. By convention, we let $\alpha_0 = \sigma_s$ and $\bar{c}_{n+1}$ the empty tuple (or, if \PTwo plays the strings $\beta_i$, then $\beta_0 = \sigma_s$, and so on). Now \POne must play a partial isomorphism $G_{n+1}$ which codes $\alpha_{n+1}$. In addition to the four requirements above, \POne must also ensure that:
\begin{description}[resume*]
	\item[(At\textsuperscript{\dag})] $G_i \concat \bar{c}_i \leq_0 G_{i+1}$,
	\item[(Ext1\textsuperscript{\dag})] for each $0 \leq i < n+1$, if $\sigma \in T$ (so $\sigma$ does not end in $\eta$), $\sigma \subset \alpha_i$, and $\sigma \subset \alpha_{n+1}$, then for all $x \leq \loc(\sigma)$, $G_i(x) = G_{n+1}(x)$; if $\alpha_i = \alpha_{n+1}$ and they do not end in $\eta$, then in fact $G_i \concat \bar{c}_i \subset G_{n+1}$, and
	\item[(Ext2\textsuperscript{\dag})] for each $t \leq s$ and $0 \leq i < n+1$, if $\sigma \in T$ (so $\sigma$ does not end in $\eta$), $\sigma \subset \sigma_t$, and $\sigma \subset \alpha_{n+1}$, then for all $x \leq \loc(\sigma)$, $F_t(x) = G_{n+1}(x)$; if $\sigma_t = \alpha_{n+1}$ and they do not end in $\eta$, then in fact $F_t \subset G_{n+1}$.
\end{description}

A winning strategy for \POne is a just way to continue the construction, but without having to add any new strings to $\loc_s$. Because \PTwo can play any appropriate string $\alpha$, the strategy is independent of the future values $\sigma_{s+1},\sigma_{s+2},\ldots$ of the approximation.

\section{Basic Plays and the Basic Game\texorpdfstring{ $\basicgame{s}$}{}}

The game $\game{s}$ requires \POne to play infinitely many moves in order to win. However, \PTwo has only finitely many different strings in $\dom(\loc)^*$ which they can play, and so if they extend the approximation for infinitely many stages, they must repeat some strings infinitely many times. In this section, we will define a game which is like $\game{s}$, except that \PTwo is not allowed to have the approximation loop more than once. The main lemma here will be that if \POne can beat \PTwo when \PTwo is restricted to only playing one loop, then \POne can win in general. The idea is that at the end of a loop in the approximation, the game ends up in essentially the same place it was before the loop. So if \POne does not lose to any single loop, they do not lose to any sequence of loops. These plays with only one loop will be called the \textit{basic plays}, and the game with no loops the basic game $\basicgame{s}$.

Whether or not a play by \PTwo is a basic play depends only on the strings $\alpha$ in the play, and is independent of the tuples $\bar{c}$. We say that a play $(\bar{c}_0,\alpha_1),\ldots,(\bar{c}_{\ell-1},\alpha_\ell)$ by \PTwo is \textit{based on} the list $\alpha_1,\ldots,\alpha_\ell$. A play based on $\alpha_1,\ldots,\alpha_\ell$ is a basic play (and the list of strings it is based on is a \textit{basic list}) if it satisfies:
\begin{description}
	\item[(B1)] for $i < l - 1$, $\alpha_i \neq \alpha_{i+1}$ (though $\alpha_{\ell - 1}$ may be equal to $\alpha_\ell$), and
	\item[(B2)] if for some $i < j$ there is some $\tau \in T$ such that $\tau \concat \infty \subset \alpha_i$, $\tau \concat \infty \subset \alpha_j$, and for all $k$ with $i < k < j$, $\alpha_k = \tau \concat \eta$, then $j = \ell$ and $\alpha_i = \alpha_j$.
\end{description}
These conditions include $\alpha_0 = \sigma_s$; so, for example, if $\alpha_1 = \sigma_s$, then $\ell = 1$ by \BOne. Note that all of these definitions are dependent on the stage $s$. So really, we are defining what it meas to be a basic play \textit{at stage $s$}.

The first of two important facts about the basic plays is the following lemma.

\begin{lem}
\label{finite-basic-list}
There are finitely many basic lists.
\begin{proof}
The domain of $\loc$ is finite, and hence $\dom(\loc)^*$ is finite. Since \PTwo can only play strings from $\dom(\loc)^*$, there are only finitely many different $\sigma$ which can appear as one of the $\alpha_i$ in a basic list. Let $\alpha_1,\ldots,\alpha_\ell$ be a basic list. We will show that $\ell$ is bounded, and hence there are only finitely many basic lists. Let $M$ be the size of $\dom(\loc)^*$. In any sufficiently long basic list, say $\alpha_1,\ldots,\alpha_{N}$ of length at least $N$ (depending on $M$), there must be three indices $i_1 < i_2 < i_3 < N$ such that $\alpha_{i_1} = \alpha_{i_2} = \alpha_{i_3}$. Since $i_3 < N$, by \BOne there must be $j_1$ and $j_2$ with $i_1 < j_1 < i_2$ and $i_2 < j_2 < i_3$ with $\alpha_{j_1} \neq \alpha_{i_1}$ and $\alpha_{j_2} \neq \alpha_{i_2}$. Let $\tau_1$ and $\tau_2$ be the greatest common initial segments of the $j$ with $i_1 \leq j \leq i_2$ and $i_2 \leq j \leq i_3$ respectively. Let $x$ be such that $\tau_1 \concat x \subseteq \alpha_{i_1} = \alpha_{i_2}$. Since there is $j_1$ with $i_1 < j_1 < i_2$ and $\tau_1 \concat x \nsubseteq \alpha_{j_1}$, we cannot have $x \in \omega$. So $x = \infty$ or $x = \eta$. First, suppose that $x = \infty$. Then for all $j_1$ with $i_1 < j_1 < i_2$, we cannot have $\tau_1 \concat z \subseteq \alpha_{j_1}$ for some $z \in \omega$, since for any such $z$ with $\tau_1 \concat z \in \loc$, some string $\tau' \supseteq \tau_1 \concat z$ appeared before $i_1$ and hence $\tau_1 \concat z$ can never appear again. So, by decreasing $i_2$ and increasing $i_1$, we may assume that $\tau_1 \concat \infty \subseteq \alpha_{i_1}$, $\tau_1 \concat \infty \subseteq \alpha_{i_2}$, and for all $j_1$ with $i_1 < j_1 < i_2$, $\tau_1 \concat \eta \subseteq \alpha_{j_1}$ (though we may no longer have $\alpha_{i_1} = \alpha_{i_2}$). This contradicts \BTwo. So we must have $x = \eta$, and so $\tau_1 \concat \eta = \alpha_{i_1} = \alpha_{i_2}$. Similarly, we must have $\tau_2 \concat \eta = \alpha_{i_3} = \alpha_{i_2}$. Thus $\tau_1 = \tau_2$---call this $\tau$. Then there are $j_1$ and $j_2$ with $i_1 < j_1 < i_2 < j_2 < i_3$ and $\tau \concat \infty \subseteq \alpha_{j_1}$ and $\tau \concat \infty \subseteq \alpha_{j_2}$. For all $i$ with $j_1 < i < j_2$, we have $\tau \subseteq \alpha_i$. Thus, increasing $j_1$ and decreasing $j_2$, we may assume that for all such $i$, $\tau \concat \eta = \alpha_i$. This contradicts $\BTwo$. So there is no basic list of length greater than $N$.
\end{proof}
\end{lem}

The basic game $\basicgame{s}$ is the same as the game $\game{s}$, except that we add a new requirement for \PTwo: any play by \PTwo must be a basic play. If, at any point in the game, \PTwo has violated one of the conditions of the basic plays, then \PTwo loses. The next lemma is the second important fact about the basic plays; it says that they are the only plays which \POne has to know how to win against.

\begin{lem}
\label{basic-is-all}
If \POne has a winning strategy for the basic game $\basicgame{s}$, then \POne has a winning strategy for the game $\game{s}$. Moreover, if \POne has an arithmetic winning strategy for $\basicgame{s}$, then they have an arithmetic winning strategy for $\game{s}$.
\begin{proof}

Let $\mc{S}^b$ be a winning strategy for \POne in the basic game $\basicgame{s}$. We must give a winning strategy $\mc{S}$ for \POne in the game $\game{s}$. To each play $P$ by \PTwo in the game $\game{s}$, $\mc{S}$ must give $\POne$'s response $\mc{S}(P)$. To each play $P$, we will associate a basic play $P^*$. \POne will respond to $P$ in the same way that they responded to the corresponding basic play $P^*$; that is, $\mc{S}(P)$ will be $\mc{S}^b(P^*)$. If $P$ is $(\bar{c}_0,\alpha_1),\ldots,(\bar{c}_{m-1},\alpha_m)$ and $P^*$ is $(\bar{d}_0,\beta_1),\ldots,(\bar{d}_{n-1},\beta_n)$ then we will have $\alpha_m = \beta_n$. Thus since \COneDagger, \CTwoDagger, \CThreeDagger, \CFourDagger, and \ExtTwoDagger hold for \POne playing $\mc{S}^b(P^*)$ in response to $P^*$, they will also hold for \POne playing $\mc{S}(P) = \mc{S}^b(P^*)$ in response to $P$.

The general strategy to construct $P^*$ from $P$ will be to build $P^*$ up inductively. If $P$ is not a basic play, it is because it fails to satisfy \BOne and \BTwo. In the first case, this means that for some $i$, $\alpha_i = \alpha_{i+1}$, and so in $P^*$ we will omit $\alpha_i$. In the second case, we will be able to omit everything between the $i$ and $j$ witnessing the failure of \BTwo. The difficulty is to do this in a well-defined way, so that we have a nice relationship between $Q^*$ and $P^*$ when $P$ is a longer play which includes $Q$. This is necessary to ensure that $\mc{S}^b(Q)$ and $\mc{S}^b(P)$ are related in the correct way, e.g. by \ExtOneDagger. It will be sufficient to build up $P^*$ inductively from the definition of $Q^*$ in a natural way, but there are a number of things to check.

There is a condition ($*$) relating $P$ and $P^*$ which, intuitively, says that $P^*$ captures the essence of $P$ (i.e., $P^*$ omits only non-essential plays from $P$). Let $G_0 \supset F_s,G_1,\ldots,G_{m}$ be \POne's response to $P$ and $H_0 \supset F_s,H_1,\ldots,H_{n}$ be \POne's response to $P^*$. Let $\tilde{F}: \omega \to \mc{C}$ be a partial isomorphism and $\gamma \in \dom(\loc)^*$. Then denote by $\Phi(P^*,\tilde{F},\gamma,i)$ the statement:
\begin{itemize}
	\item for all $\sigma \in T$, if $\sigma \subset \beta_i$ and $\sigma \subset \gamma$, then for all $x \leq \loc(\sigma)$, $H_i(x) = \tilde{F}(x)$. If $\beta_i = \gamma$ and they do not end in $\eta$, then in fact $H_i \concat \bar{d}_i \subset \tilde{F}$.
\end{itemize}
and by $\Psi(P,\tilde{F},\gamma,j)$ the statement:
\begin{itemize}
	\item for all $\sigma \in T$, if $\sigma \subset \alpha_j$ and $\sigma \subset \gamma$, then for all $x \leq \loc(\sigma)$, $G_j(x) = \tilde{F}(x)$. If $\alpha_j = \gamma$ and they do not end in $\eta$, then in fact $G_j \concat \bar{c}_j \subset \tilde{F}$.
\end{itemize}
Then ($*$) says that for any $\tilde{F}$ and $\gamma$ such that for all $i$ with $0 \leq i \leq n$ the statement $\Phi(P^*,\tilde{F},\gamma,i)$ holds, then for the same $\tilde{F}$ and $\gamma$ and for all $j$ with $0 \leq j \leq m$ the statement $\Psi(P,\tilde{F},\gamma,j)$ holds.

Now suppose that ($*$) holds for $P$ and $P^*$, and that we want to check \ExtOneDagger for the response $H_n = G_m = \mc{S}(P) = \mc{S}^b(P^*)$ to $P$. Choose $\tilde{F} = H_n = G_m$ and $\gamma = \alpha_m = \beta_n$. We know that $\mc{S}^b$ is a winning strategy in the basic game, and so \POne does not lose by playing $H_n$. Thus \ExtOneDagger must hold in the basic game. This implies, for all $i \leq n$, the statement $\Phi(P^*,H_n,\gamma,i)$. Then by ($*$), we know that for all $i \leq m$, we have $\Psi(P,G_m,\gamma,i)$. This implies \ExtOneDagger for the full game at $P$. Thus, instead of checking \ExtOneDagger, it suffices to check property ($*$).

We will define the operation $P \to P^*$ by induction on the length of $P$. At the same time, we will show that the strategy $\mc{S}$ for \POne does not lose at any finite point in the game (and hence, it must win), and also that $P$ and $P^*$ have ($*$). If we show ($*$), then the only thing remaining to see that \POne does not lose is to check \AtDagger. Let $P$ be $(\bar{c}_0,\alpha_1),\ldots,(\bar{c}_{\ell},\alpha_{\ell+1})$ and suppose that the operation $Q \to Q^*$ is defined for plays of length up to $\ell$. Let $Q$ be first $\ell$ plays in $P$, that is, $(\bar{c}_0,\alpha_1),\ldots,(\bar{c}_{\ell-1},\alpha_\ell)$, so that $Q^*$ has already been defined. We have three possibilities.

\begin{case}
$Q^*$ followed by $(\bar{c}_{\ell},\alpha_{\ell+1})$ is a basic play.
\end{case}

Let $P^*$ be $Q^*$ followed by $(\bar{c}_{\ell},\alpha_{\ell+1})$; this is already a basic play.

First we check ($*$). Fix $\gamma$ and $\tilde{F}$. Let $H_0,H_1,\ldots,H_n$ be \POne's response to $Q^*$ and $G_0,G_1,\ldots,G_\ell$ be \POne's response to $Q$. Let $H_{n+1} = G_{\ell+1} = \mc{S}^b(P^*) = \mc{S}(P)$. Suppose that, for each $i \leq n+1$, we have $\Phi(P^*,\tilde{F},\gamma,i)$. Then, from $\Phi(P^*,\tilde{F},\gamma,i)$ for $i \leq n$ and ($*$) for $Q$ and $Q^*$, we get $\Psi(P,\tilde{F},\gamma,j)$ for $i \leq \ell$. Note that $\Phi(P^*,\tilde{F},\gamma,\ell+1)$ is the same as $\Psi(P,\tilde{F},\gamma,n+1)$ since $H_{n+1} = G_{\ell+1}$. So, for all $j \leq \ell+1$, we have $\Psi(P,\tilde{F},\gamma,j)$. Thus we have ($*$).

Now we need to check \AtDagger to see that \POne does not lose $\game{s}$ by responding to $P$ with $\mc{S}^b(P^*) = H$. So we need to show that $\mc{S}(Q) \concat \bar{c}_\ell \leq_0 H_{n+1}$. Since \POne does not lose $\basicgame{s}$ when responding to $P^*$ with $H_{n+1}$, we have $\mc{S}^b(Q^*) \concat \bar{c}_{\ell} \leq_0 H_{n+1}$. But $\mc{S}^b(Q^*) = \mc{S}(Q)$, so \AtDagger holds for $\mc{S}$.

\begin{case}
$Q^*$ followed by $(\bar{c}_{\ell},\alpha_{\ell+1})$ does not satisfy \BOne.
\end{case}

Let $Q^*$ be $(\bar{d}_0,\beta_1),\ldots,(\bar{d}_{n-1},\beta_n)$. It must be the case that $\beta_{n-1} = \beta_n$. Let $H_0,H_1,\ldots,H_n$ be \POne's response to $Q^*$; then $H_{n-1} \concat \bar{d}_{n-1} \subset H_n$. Let $\bar{e}$ be a tuple of elements from $\mc{C}$ such that $H_{n-1} \concat \bar{d}_{n-1} \concat \bar{e} = H_n$ and let $P^*$ be the play $(\bar{d}_0,\beta_1),\ldots,(\bar{d}_{n-2},\beta_{n-1}),(\bar{d}_{n-1} \concat \bar{e} \concat \bar{c}_\ell,\alpha_{\ell+1})$. Since $Q^*$ satisfied $\BTwo$ and $\beta_{n-1} = \beta_{n}$, $P^*$ satisfies $\BTwo$ and hence is a basic play.

Now we will check ($*$). Fix $\gamma$ and $\tilde{F}$. Recall that $H_0,H_1,\ldots,H_n$ is \POne's response to $Q^*$ in the basic game, and let $G_0,G_1,\ldots,G_\ell$ be \POne's response to $Q$; let $G_{\ell+1} = H_n'$ be $\mc{S}^b(P^*) = \mc{S}(P)$. Then $H_0,H_1,\ldots,H_{n-1},H_n'$ is \POne's response to $P^*$. Suppose that for each $i \leq n$ we have $\Phi(P^*,\tilde{F},\gamma,i)$. Then since $G_{\ell+1} = H_n'$ and $\alpha_{\ell+1}$ is the last play by \PTwo in both $P$ and $P^*$, we immediately have $\Psi(P,\tilde{F},\gamma,\ell+1)$. To show $\Psi(P,\tilde{F},\gamma,j)$ for $j \leq \ell$, it suffices to show $\Psi(Q,\tilde{F},\gamma,j)$ for $j \leq \ell$, and hence (by ($*$) for $Q$ and $Q^*$) to show $\Phi(Q^*,\tilde{F},\gamma,i)$ for $i \leq n$.

Now $P^*$ and $Q^*$ agree on $(\bar{d}_0,\beta_1),\ldots,(\bar{d}_{n-2},\beta_{n-1})$, so we have $\Phi(Q^*,\tilde{F},\gamma,i)$ for $i \leq n-2$. Now $\Phi(P^*,\tilde{F},\gamma,n-1)$ says that: for all $\sigma \in T$, if $\sigma \subset \beta_{n-1}$ and $\sigma \subset \gamma$, then for all $x \leq \loc(\sigma)$, $H_{n-1}(x) = \tilde{F}(x)$; and moreover, if $\beta_{n-1} = \gamma$ and they do not end in $\eta$, then in fact $H_{n-1} \concat \bar{d}_{n-1} \concat \bar{e} \concat \bar{c}_\ell = H_n \concat \bar{c}_\ell \subset \tilde{F}$. This implies both $\Phi(Q^*,\tilde{F},\gamma,n-1)$ and $\Phi(Q^*,\tilde{F},\gamma,n)$ since $H_{n-1} \subset H_n$. Thus, for all $i \leq n$, we have $\Phi(Q^*,\tilde{F},\gamma,i)$. This completes the proof of ($*$).

Now we need to check \AtDagger. Since \POne does not lose $\basicgame{s}$ when responding to $P^*$ with $H_n' = G_{\ell+1}$, by \AtDagger we have that $H_{n-1} \concat \bar{d}_{n-1} \concat \bar{e} \concat \bar{c}_{\ell} \leq_0 H_n'$. But $H_{n-1} \concat \bar{d}_{n-1} \concat \bar{e} = \mc{S}^b(Q^*)=\mc{S}(Q)$, so $\mc{S}(Q) \leq_0 \mc{S}(P)$. So \AtDagger holds for $\mc{S}$.

\begin{case}
$Q^*$ followed by $(\bar{c}_{\ell},\alpha_{\ell+1})$ satisfies \BOne but does not satisfy \BTwo.
\end{case}

Let $Q^*$ be $(\bar{d}_0,\beta_1),\ldots,(\bar{d}_{n-1},\beta_n)$. Now $Q^*$ satisfies \BTwo, so there are two possible reasons that $Q^*$ followed by $(\bar{c}_{\ell},\alpha_{\ell+1})$ might fail to satisfy \BTwo.

\begin{subcase}
$\tau \concat \infty \subset \beta_m = \beta_n$ but for each $k$ with $m < k < n$, $\beta_k = \tau \concat \eta$
\end{subcase}

Choose $m$ to be least with the above property. For any $k$ with $m < k < n$, $\beta_k = \tau \concat \eta$. Let $H_0,H_1,\ldots,H_n$ be \POne's play in response to $Q^*$. We have $H_{m} \subset H_n$. Let $\bar{e}$ be a tuple of elements of $\mc{C}$ such that $H_{m} \concat \bar{e} = H_n$. Then let $P^*$ be $(\bar{d}_0,\beta_1),\ldots,(\bar{d}_{m-1},\beta_{m}),(\bar{e} \concat \bar{c}_{\ell},\alpha_{\ell+1})$. This is a basic play since $Q^*$ was. Note that \POne's play in response to $P^*$ is $H_0,H_1,\ldots,H_{m},H'$ for some partial isomorphism $H'$.

Now we will check ($*$). Let $G_0,G_1,\ldots,G_{\ell}$ be \POne's response to $Q$, so that \POne's response to $P$ is $G_0,G_1,\ldots,G_{\ell},G_{\ell+1} = H'$. Fix $\gamma$ and $\tilde{F}$. Suppose that for each $i \leq n$ we have $\Phi(P^*,\tilde{F},\gamma,i)$. Since both $P$ and $P^*$ end in $\alpha_{\ell+1}$, $\Phi(P^*,\tilde{F},\gamma,i)$ implies $\Psi(P,\tilde{F},\gamma,\ell + 1)$. Now for $j \leq \ell$, $\Psi(Q,\tilde{F},\gamma,j)$ implies $\Psi(P,\tilde{F},\gamma,j)$, and so by ($*$) for $Q$ and $Q^*$, it suffices to show that for all $i \leq n$, we have $\Phi(Q^*,\tilde{F},\gamma,i)$.

For the first $m$ turns, $Q^*$ and $P^*$ agree, and so for $i \leq m$, $\Phi(P^*,\tilde{F},\gamma,i)$ implies $\Phi(Q^*,\tilde{F},\gamma,i)$. So we have established $\Phi(Q^*,\tilde{F},\gamma,i)$ for $i \leq m$.

For $m < i < n$, we have $\beta_i = \tau \concat \eta$. To show $\Phi(Q^*,\tilde{F},\gamma,i)$, it suffices to check that if $\sigma \in T$ has $\sigma \subset \tau$ and $\sigma \subset \gamma$, then for all $x \leq \loc(\sigma)$, $G_{j}(x) = \tilde{F}(x)$. Now $\sigma \subset \tau \subset \alpha_{i}$, so $G_{i}(x) = \tilde{F}(x)$ for all $x \leq G(\sigma)$. By \ExtOneDagger, for all $x \leq \loc(\sigma)$, $G_{i}(x) = G_{j}(x)$, and hence $G_{j}(x) = \tilde{F}(x)$.

Finally, we have the case $i = n$. We have $\beta_n = \beta_{i}$. If $\sigma \in T$ has $\sigma \subset \beta_n$ and $\sigma \subset \gamma$, then $\sigma \subset \beta_{i}$ and so for all $x \leq \loc(\sigma)$, $G_{n}(x) = G_{i}(x) = \tilde{F}(x)$. We also need to consider the case where $\gamma = \beta_n$. By $\Phi(P^*,\tilde{F},\gamma,i)$, we have that $G_{i}\concat \bar{e} \subset \tilde{F}$; but $G_{i} \concat \bar{e} = G_n$ and so $G_n \subset \tilde{F}$ as desired. Thus we have $\Psi(Q^*,\tilde{F},\gamma,n)$. This completes the proof of ($*$).

So \POne responds to $P$ with $\mc{S}^b(P^*) = H$. To see that this does not lose the game for \POne, we just need to check \AtDagger. Since \POne does not lose $\basicgame{s}$ when responding to $P^*$ with $H$, by \AtDagger we have $G_{i} \concat \bar{e} \concat \bar{c}_{\ell+1} \leq_0 H$. But $G_{i} \concat \bar{e} = G_n$, so $G_n \concat \bar{c}_{\ell+1} \leq_0 H$ as required.

\begin{subcase}
There is $i \leq n$ and $\tau \in T$ such that $\tau \concat \infty \subset \beta_i$ and $\tau \concat \infty \subset \alpha_{\ell+1}$ but for each $k$ with $i < k \leq n$, $\tau \concat \eta = \beta_k$. Also, $\beta_i \neq \alpha_{\ell+1}$.
\end{subcase}

Let $\hat{Q}$ be $(\bar{d}_0,\beta_1),\ldots,(\bar{d}_{n-1},\beta_n),(\bar{c}_\ell,\beta_{i})$. We can now use the same argument as in the previous subcase, with $\hat{Q}$ being extended by $(\varnothing,\alpha_{\ell+1})$.

\medskip{}

The whole construction of $\mc{S}$ from $\mc{S}^b$ by transforming plays $P$ into basic plays $P^*$ is arithmetic. \qedhere
\end{proof}
\end{lem}

Now any winning strategy for \POne in $\game{s}$ is also a winning strategy in $\basicgame{s}$; in particular, if \POne has an arithmetic winning strategy for $\game{s}$, then they have an arithmetic winning strategy in $\basicgame{s}$. However, it would be nice if we did not have to worry about the computability-theoretic properties of the winning strategy during our arguments. The following lemma lets us do exactly that.

\begin{corollary}
The following are equivalent:
\begin{enumerate}
	\item \POne has a winning strategy for $\game{s}$.
	\item \POne has an arithmetic winning strategy for $\game{s}$.
	\item \POne has a winning strategy for $\basicgame{s}$.
	\item \POne has an arithmetic winning strategy for $\basicgame{s}$.
\end{enumerate}
\begin{proof}
We have $(4) \Rightarrow (2)$ from the previous lemma. $(2) \Rightarrow (1)$ is immediate. $(1) \Rightarrow (3)$ is because any winning strategy for \POne in $\game{s}$ is a winning strategy for $\basicgame{s}$. So it remains to prove $(3) \Rightarrow (4)$. In $\basicgame{s}$, each player makes only finitely many moves in every play of the game, and there is a computable bound on the number of moves. Any such game with a winning strategy for \POne has an arithmetic winning strategy for \POne.
\end{proof}
\end{corollary}

So in order to check \WS, it suffices to show that \POne has a winning strategy without worrying about whether it is arithmetic.

\section{The Construction}

Begin at stage $s = -1$ with $F_{-1} = \loc_{-1} = \varnothing$.

At each subsequent stage $s+1$, we will first update the partial isomorphism from the previous stage according to the new approximation. To do this, we will use the winning strategy from the game $\game{s}$ which we had at the previous stage. Then we will add new elements to the image of the isomorphism, and add a new coding location. We must add these new elements so that all of the properties from Section \ref{sec:framework} are satisfied.

At stage $s+1$, $\sigma_{s+1} \notin \dom(\loc_s)$ but $\sigma_{s+1}^- \in \dom(\loc_s)$. (Recall that $\sigma_{s+1}^-$ is $\sigma_{s+1}$ with the last entry removed.) We begin by using our winning strategy for $\game{s}$ to code $\sigma_{s+1}^-$ correctly. If there is some $x \in \omegaext$ such that $\sigma_{s+1}^- \concat x \in \dom(\loc_s)$, then $\ell(\sigma_{s+1}) \in \omega$; let $\tau = \sigma_{s+1}^- \concat \eta$. Otherwise, let $\tau = \sigma_{s+1}^-$. \POne has an arithmetic winning strategy $\mc{S}$ for the game $\game{s}$ from the previous stage. Consider these first couple moves of the game where \POne uses the strategy $\mc{S}$: \POne plays $G_0 \supset F_s$, \PTwo plays $(\varnothing,\tau)$, and \POne responds with $G$ according to their winning strategy $\mc{S}$. $F_{s+1}$ will be an extension of $G$. Since $G$ was part of the winning strategy for \POne, it did not lose the game. So, automatically, $F_{s+1}$ will satisfy \COne, \CTwo, \CThree, and \CFour for $\sigma \in \dom(\loc_s)$, \At, and \Ext since $G$ already satisfies these.

Suppose that the domain of $G$ is $\{0,\ldots,n\}$. Let $t < s+1$ be the previous stage at which $\sigma_t = \sigma_{s+1}^-$. Then by \Surj at stage $t$, the first $|\sigma_t|-1$ elements of $\mc{C}$ appear in $\ran(F_t \upharpoonright_{\loc_t(\sigma_t)})$. By \Ext, $F_t \subset G$. If the first $|\sigma_{s+1}| - 1$ elements of $\mc{C}$ do not appear in $\ran(G)$, then it is because the $(|\sigma_{s+1}| - 1)$th element is not in $\ran(G)$; let this element be $a$, and define $G'$ extending $G$ by $G'(n+1) = a$. Set $\loc_{s+1}$ to be the extension of $\loc_s$ with $\loc_{s+1}(\sigma_{s+1}) = n+2$. Then any $F_{s+1}$ extending $G'$ will satisfy \Surj. Also, since $\sigma_{s+1} \in \dom(\loc_{s+1})$, we satisfy \CLoc.

While we have set $\loc_{s+1}(\sigma_{s+1}) = n+2$, we have not yet added an element into that position. What remains to be done is to add some element in the position $n+2$ (or a tuple in the positions $n+2,n+3,\ldots$ if $\sigma_{s+1} \in T_\omega$) to satisfy \COne, \CTwo, \CThree, and \CFour for $\sigma = \sigma_{s+1}$ and also to give a winning strategy witnessing \WS. There are two cases depending on the last entry of $\sigma_{s+1}$. The first is relatively easy while the second is much harder.

\subsection{The \texorpdfstring{}{first }case\texorpdfstring{ $\sigma_{s+1} \in T_\infty$}{}}

Let $b \notin R$ be an element which is $2$-free over $\ran(G')$. Let $F_{s+1}$ extend $G'$ with $F_{s+1}(n+2) = c$. Define $m_{s+1}(\sigma_{s+1}) = -1$ and $k_{s+1}(\sigma_{s+1}) = 1$. Then $F_{s+1}$ satisfies \COne, \CTwo, \CThree, and \CFour for $\sigma = \sigma_{s+1}$, which was the last remaining case of those properties.

We must show that \POne has a winning strategy in the game $\basicgame{s+1}$. The only difference between $\basicgame{s+1}$ and $\basicgame{s}$ is that in $\basicgame{s+1}$ we have added a new coding location $\sigma_{s+1}$ to $\loc$, and $\basicgame{s}$ starts a turn earlier than $\basicgame{s+1}$. We can accommodate the latter by considering plays in $\basicgame{s}$ which begin with $(\varnothing,\sigma_{s+1})$, thus making \POne play $G$ as their first play. So every part of the approximation that our opponent plays in $\basicgame{s+1}$ is an approximation that our opponent could have played in $\basicgame{s}$, except that in $\basicgame{s+1}$ \PTwo can also play $\sigma_{s+1}$ and $\sigma_{s+1}^- \concat \eta$.

As a first approximation, we could use the strategy $\mc{S}$ from the previous stage (and when our opponent plays $\sigma_{s+1}$ or $\sigma_{s+1}^- \concat \eta$, we respond instead to $\sigma_{s+1}^-$). This works except for one thing, which is that when our opponent plays either $\sigma_{s+1}$ or $\sigma_{s+1}^- \concat \eta$, we are not guaranteed to code $\sigma_{s+1}$ correctly. Now, when our opponent plays $\sigma_{s+1}$, our response using $\mc{S}$ will extend $F_{s+1} = G a b$ with $b \notin R$ coding $\sigma_{s+1}$ (this is because by \ExtOneDagger and \ExtTwoDagger our response must extend our response at the previous stage where the approximation was $\sigma_{s+1}^-$). The only problem is that when our opponent plays $\sigma_{s+1}^- \concat \eta$, $\mc{S}$ will also have us respond with an extension of $F_{s+1} = G a b$, but now $b$ is coding the wrong thing. This is relatively easy to fix. Say $\mc{S}$ has us respond with $G a b \bar{c}$. Since $b$ is 2-free over $Ga$, we can find $b' \in R$ and $\bar{c}'$ with $Gab \bar{c} \leq_1 G a b' \bar{b}'$. Then we will play $Gab' \bar{c}'$.

In order to continue to follow the strategy $\mc{S}$ at later stages, we cannot tell $\mc{S}$ that our opponent plays $\sigma_{s+1}$ or $\sigma_{s+1}^- \concat \eta$ because this is an illegal play in $\basicgame{s}$. However, we can tell $\mc{S}$ that our opponent played $\sigma_{s+1}^-$ and pretend that we responded with $G a b \bar{c}$. We will keep track of these corresponding plays for the purposes of using $\mc{S}$.

So along with defining the strategy $\mc{T}$ for $\basicgame{s}$, we will describe a correspondence between plays by \PTwo in $\basicgame{s+1}$ and $\basicgame{s}$. If \PTwo plays $(\bar{c}_0,\alpha_1),\ldots,(\bar{c}_{n-1},\alpha_{n})$, the corresponding play in $\basicgame{s}$ will be of the form $(\varnothing,\sigma_{s+1}^-),(ab\bar{d}_0,\beta_1),\ldots,(\bar{d}_{n-1},\beta_{n})$. Note that the length of the play in $\basicgame{s}$ is one more than of that in $\basicgame{s+1}$; this is because $\basicgame{s}$ begins at the previous stage, and so we need to begin by playing $G$ and adding $a$ and $b$ to the image of the isomorphism. Thus the first play using the strategy $\mc{S}$ will be $G$, and before the second play \POne will be forced to add $ab$ to $G$, and thus \POne is essentially playing $F_{s+1}$. We can already define $\beta_1,\ldots,\beta_{n-1}$, but the $\bar{d}_i$ will be defined at the same time as we define $\mc{T}$. If $\alpha_i$ is in $\dom(\loc_s)^*$, then $\beta_i$ and $\alpha_i$ will be equal. Otherwise, $\alpha_i$ is either $\sigma_{s+1}$ or $\sigma_{s+1}^- \concat \eta$, and $\beta_i$ will be the longest initial segment which is in $\dom(\loc_s)^*$, which is $\beta_i = \sigma_{s+1}^-$. 

Now we will define $\mc{T}$ and the $\bar{d}_i$ by an alternating inductive definition. Suppose that so far we have defined \POne's response $G_0 \supset F_{s+1},G_1,\ldots,G_{n-1}$ when \PTwo plays $(\bar{c}_0,\alpha_1),\ldots,(\bar{c}_{n-2},\alpha_{n-1})$. We will also have defined a corresponding play by \PTwo in $\basicgame{s}$, $(\varnothing,\sigma_{s+1}^-),(ab \concat \bar{d}_0,\beta_1),\ldots,(\bar{d}_{n-2},\beta_{n-1})$. Let $F' \supset F_s,G,H_1,\ldots,H_{n-1}$ be \POne's response to this using $\mc{S}$.  We will have ensured that if $\alpha_i = \sigma_{s+1}^- \concat \eta$ then $H_i \leq_1 G_i$, and otherwise that $H_i = G_i$ and  for each $i$ and $\bar{d}_i = \bar{c}_i$. Recall that $G \concat ab = F_{s+1}$.

It is now \PTwo's turn, and suppose that \PTwo plays $(\bar{c}_{n-1},\alpha_{n})$, and that this is a basic play by \PTwo. We must define $\bar{d}_{n-1}$ and then define \POne's response $G_n$. Note that if $\alpha_i = \alpha_{i+1} = \sigma_{s+1}^- \concat \eta$, then $i+1 = n$ by \BOne.

There are four cases depending on the values of $\alpha_{n-1}$ and $\alpha_{n}$. When neither of $\alpha_{n-1}$ nor $\alpha_{n}$ are $\sigma_{s+1}^- \eta$, then we can just follow $\mc{S}$. Then we have three more cases depending on whether one (or both) of $\alpha_{n-1}$ and $\alpha_{n}$ are $\sigma_{s+1}^- \eta$.

\begin{itemize}
	\item[Case 1:] $\alpha_{n-1}, \alpha_n \neq \sigma_{s+1}^- \concat \eta$. In this case we can simply follow $\mc{S}$.	Let $\bar{d}_{n-1} = \bar{c}_{n-1}$. Let $G_n = H_n$ be \POne's response, using $\mc{S}$, to
	\[ (\varnothing,\sigma_{s+1}^-),(ab \concat \bar{d}_0,\beta_1),\ldots,(\bar{d}_{n-1},\beta_{n}). \]

	\item[Case 2:] $\alpha_{n-1} = \sigma_{s+1}^- \concat \eta$ and $\alpha_{n} \neq \sigma_{s+1}^- \concat \eta$. We have $H_{n-1} \leq_1 G_{n-1}$. Let $\bar{d}_{n-1}$ be such that $G_{n-1} \bar{c}_{n-1} \leq_0 H_{n-1}\bar{d}_{n-1}$. Now let $G_n = H_n$ be \POne's response, using $\mc{S}$, to
	\[ (\varnothing,\sigma_{s+1}^-),(ab \concat \bar{d}_0,\beta_1),\ldots,(\bar{d}_{n-1},\beta_{n}). \]
		
	\item[Case 3:] $\alpha_{n-1} \neq \sigma_{s+1}^- \concat \eta$ and $\alpha_{n} = \sigma_{s+1}^- \concat \eta$. Let $\bar{d}_{n-1} = \bar{c}_{n-1}$. Now let $H_n$ be \POne's response, using $\mc{S}$, to
	\[ (\varnothing,\sigma_{s+1}^-),(ab \concat \bar{d}_0,\beta_1),\ldots,(\bar{d}_{n-1},\beta_{n}). \]
Note that since $\alpha_{n} = \sigma_{s+1}^- \concat \eta$, $\beta_n = \sigma_{s+1}^-$. Then $H_n \supseteq F_{s+1} = Gab$, say $H_n = Gab\bar{e}$. The using the fact that $b$ is free over $Ga$, choose $b' \in R$ and $\bar{e}'$ such $H_n \leq_1 Gab'\bar{e}'$. Then set $G_n = Gab'\bar{e}'$. 
	
	\item[Case 4:] $\alpha_{n-1} = \alpha_n = \sigma_{s+1}^- \concat \eta$. Set $G_n = G_{n-1} \bar{c}_{n-1}$. Since $\alpha_{n-1} = \alpha_n$, there are no longer basic plays than this, and so we do not need to define $\bar{d}_{n-1}$.
\end{itemize}

It is tedious but easy to see that none of these plays by \POne is a losing play. Hence \POne has a winning strategy $\mc{T}$ in $\basicgame{s+1}$.

\subsection{The \texorpdfstring{}{second }case\texorpdfstring{ $\sigma_{s+1} \in T_\omega$}{}}

By Lemma \ref{finite-basic-list}, there are only finitely many basic lists on which \PTwo's basic plays are based. Let $\mathfrak{b}^1,\ldots,\mathfrak{b}^m$ be these basic lists, where $\mathfrak{b}^i$ is the list $\beta_1^i,\beta_2^i,\ldots,\beta_{\ell_i}^i$.

We must add a tuple $\bar{b}$ to the image of our partial isomorphism, setting $F_{s+1} = G \bar{b}$. The tuple $\bar{b}$ will be made up of tuples $\bar{b}_1,\ldots,\bar{b}_m$ from $\mc{C}$. Let $\bar{\epsilon}_1,\ldots,\bar{\epsilon}_m$ be tuples of elements in $\{-1,1\}$ be such that $\bar{b}_j \in R^{\bar{\epsilon}_j}$. If we set $m_{s+1}(\sigma)=\bar{\epsilon}_1 \concat \cdots \concat \bar{\epsilon}_m$ and $k_{s+1}(\sigma_{s+1}) = |\bar{b}_1| + \cdots + |\bar{b}_m|$, then \COne-\CFour will be satisfied. So we just have to make sure that $\bar{b}_1,\ldots,\bar{b}_m$ are chosen such that \WS is satisfied, that is, \POne has a winning strategy for the game $\game{s+1}$ (or, equivalently, $\basicgame{s+1}$).

To choose the tuples $\bar{b}_1,\ldots,\bar{b}_m$, we will define a new class of games. For each $r \leq m$, $\bar{b}_1,\ldots,\bar{b}_r$ tuples of elements of $\mc{C}$, and tuples $\bar{\epsilon}_1,\ldots,\bar{\epsilon}_r$ which are tuples of $1$s and $-1$s (with $\epsilon_i$ the same length as $\bar{b}_i$), we have a game $\mc{H}(\bar{b}_1,\bar{\epsilon}_1;\ldots;\bar{b}_r,\bar{\epsilon}_r)$. We do not require, for the definition of the game, that $\bar{b}_j \in R^{\bar{\epsilon}_j}$. We allow the case $r = 0$; that is, $\mc{H}(\varnothing)$ is a game. $\mc{H}(\varnothing)$ will be essentially the game $\basicgame{s}$, and we will be able to easily turn a winning strategy for \POne in $\basicgame{s}$ into a winning strategy for \POne in $\mc{H}(\varnothing)$. Then we will use the winning strategy for $\mc{H}(\varnothing)$ to show that we can choose $\bar{b}_1$ and $\bar{\epsilon}_1$ so that we have a winning strategy for \POne in $\mc{H}(\bar{b}_1,\bar{\epsilon}_1)$, and we will use that winning strategy to show that we can choose $\bar{b}_2$ and $\bar{\epsilon}_2$ so that we have a winning strategy for \POne in $\mc{H}(\bar{b}_1,\bar{\epsilon}_1;\bar{b}_2,\bar{\epsilon}_2)$, and so on. Eventually we will be able to choose $\bar{b}_1,\ldots,\bar{b}_m$ and $\bar{\epsilon}_1,\ldots,\bar{\epsilon}_m$ so that we have a winning strategy for $\mc{H}(\bar{b}_1,\bar{\epsilon}_1;\ldots;\bar{b}_m,\bar{\epsilon}_m)$. We will choose the tuples so that if we make the definition of $F_{s+1}$, $m_{s+1}$, and $k_{s+1}$ above, this winning strategy will immediately yield a winning strategy for \POne in $\basicgame{s+1}$.

Here are the rules for the game $\mc{H}(\bar{b}_1,\bar{\epsilon}_1;\ldots;\bar{b}_r,\bar{\epsilon}_r)$. \POne begins by playing a partial isomorphism $G_0$ which extends $Ga\bar{b}_1\cdots\bar{b}_i$. Then, \PTwo and \POne alternate, with \PTwo playing a tuple of elements from $\mc{C}$ and a string in $\dom(\loc_{s+1})^*$, and \POne playing a partial isomorphism. As in $\basicgame{s+1}$, \PTwo must make a play which is based on one of the basic lists. \POne can lose by violating one of \COneDagger-\CFourDagger for $\sigma \in \dom(\loc_s)^*$ (not $\dom(\loc_{s+1}^*)$, since we have not yet defined $m_{s+1}(\sigma_{s+1})$ and $k_{s+1}(\sigma_{s+1})$). \POne can also lose by violating \AtDagger-\ExtTwoDagger using $\loc = \loc_{s+1}$. Finally, \POne must ensure that
\begin{description}
 \item[(CB\textsuperscript{\dag})] whenever \POne is responding with a partial isomorphism $H$ to a basic play $(\bar{c}_0,\beta^i_1),\ldots,(\bar{c}_{\ell_i-1},\beta_{\ell_i}^i)$ based on $\mathfrak{b}^i$ for $i \leq r$, and $\beta_{\ell_i} = \sigma_{s+1}^- \concat \eta$, then
	\[H(n+|\bar{b}_1|+\cdots+|\bar{b}_{i-1}|+1),\ldots,H(n+|\bar{a}_1|+\cdots+|\bar{b}_{i-1}|+|\bar{b}_i|) \notin R^{\bar{\epsilon}_i}. \]
\end{description}
This condition \CB will be what is required to ensure \CFourDagger for $\sigma = \sigma_{s+1}$.

The choice of $\bar{b}_1,\ldots,\bar{b}_m$ involves three lemmas. We will begin by stating the first two lemmas and using them to prove induction step of the third lemma before returning to the proofs of the first two.

\begin{lem}
\label{one-or-other}
Suppose that \POne wins $\mc{H}(\bar{b}_1,\bar{\epsilon}_1;\ldots;\bar{b}_r,\bar{\epsilon}_r)$. Let $\bar{c}$ be such that \POne's first play using their winning strategy for this game is $Ga\bar{b}_1\cdots\bar{b}_r\bar{c}$. Let $\bar{\nu}$ be such that $\bar{c} \in R^{\bar{\nu}}$. Let $x \in \mc{C}$. Then for one of $\iota = 1$ or $\iota = -1$, \POne wins $\mc{H}(\bar{b}_1,\bar{\epsilon}_1;\ldots;\bar{b}_r,\bar{\epsilon}_r;\bar{c}x,\bar{\nu}\iota)$.
\end{lem}

\begin{lem}
\label{transfer-win}
Let $\bar{c}$ and $\bar{\nu}$ be as above. Let $\iota = 1$ or $\iota = -1$. If \POne wins the game $\mc{H}(\bar{b}_1,\bar{\epsilon}_1;\ldots;\bar{b}_r,\bar{\epsilon}_r;\bar{c} x,\bar{\nu}\iota)$ and
\[ G a \bar{b}_1\cdots\bar{b}_r\bar{c} x \leq_1 G a \bar{b}_1\cdots\bar{b}_r\bar{c}y,\]
then \POne wins $\mc{H}(\bar{b}_1,\bar{\epsilon}_1;\ldots;\bar{b}_r,\bar{\epsilon}_r;\bar{c} y,\bar{\nu}\iota)$.
\end{lem}

\begin{lem}
There are $\bar{b}_1,\ldots,\bar{b}_m \in \mc{C}$ and $\bar{\epsilon}_1,\ldots,\bar{\epsilon}_m$ such that \POne wins the game $\mc{H}(\bar{b}_1,\bar{\epsilon}_1;\ldots;\bar{b}_m,\bar{\epsilon}_m)$ and $(\bar{b}_1,\ldots,\bar{b}_m) \in R^{(\bar{\epsilon}_1,\ldots,\bar{\epsilon}_m)}$.
\begin{proof}
The proof is by induction on $r$. The first step is to show that \POne has a winning strategy for the game $\mc{H}(\varnothing)$. Recall that \POne has a winning strategy $\mc{S}$ for $\game{s}$, and $\mc{S}$ responds to $(\varnothing,\sigma_{s+1}^-\concat \eta)$ with $G$. To any play $P = (\bar{d}_0,\alpha_1),\ldots,(\bar{d}_{\ell-1},\alpha_{\ell})$ by \PTwo in $\mc{H}(\varnothing)$, associate the play $P^* = (\varnothing,\sigma_{s+1}^- \concat \eta),(a\bar{d}_0,\alpha_1),\ldots,(\bar{d}_{\ell-1},\alpha_{\ell})$ in $\game{s}$. A response $F_s,G,H_1,\ldots,H_\ell$ to $P'$ according to $\mc{S}$ which does not lose $\game{s}$ gives rise to a response $Ga,H_1,\ldots,H_\ell$ which does not lose $\mc{H}(\varnothing)$ (the conditions for \POne losing in each game are essentially the same; the only difference is that $\game{s}$ involves an additional turn at the beginning). So the winning strategy $\mc{S}$ for $\game{s}$ gives rise to a winning strategy for $\mc{H}(\varnothing)$.

Now for the induction step, suppose that we have tuples $\bar{b}_1,\ldots,\bar{b}_r$ and $\bar{\epsilon}_1,\ldots,\bar{\epsilon}_r$ such that \POne wins $\mc{H}(\bar{b}_1,\bar{\epsilon}_1;\ldots;\bar{b}_r,\bar{\epsilon}_r)$ and $\bar{b}_i \in R^{\bar{\epsilon}_i}$. Let $\bar{c}$ and $\bar{\nu}$ be as in the previous lemmas, that is, \POne's winning strategy begins by playing $Ga\bar{b}_1\cdots\bar{b}_r\bar{c}$. Now choose an element $x \notin R$ which is $2$-free over $Ga\bar{b}_1 \cdots \bar{b}_r \bar{c}$. Then choose $y \in R$ such that
\[ G a \bar{b}_1\cdots \bar{b}_r \bar{c} x \leq_1 G a \bar{b}_1\cdots \bar{b}_r \bar{c} y \]
Now by Lemma \ref{one-or-other}, \POne wins $\mc{H}(\bar{b}_1,\bar{\epsilon}_1;\ldots;\bar{b}_r,\bar{\epsilon}_r;\bar{c}x,\bar{\nu}\iota)$ for either $\iota = 1$ or $\iota = -1$. If $x \in R^\iota$ (i.e., $\iota = -1$), then we are done. Otherwise, if $\iota = 1$, then by Lemma \ref{transfer-win}, \POne also wins $\mc{H}(\bar{b}_1,\bar{\epsilon}_1;\ldots;\bar{b}_r,\bar{\epsilon}_r;\bar{c}y,\bar{\nu} \iota)$. But $y \in R$ and $\iota = 1$, so this completes the induction step.
\end{proof}
\end{lem}

The winning strategy for $\mc{H}(\bar{b}_1,\bar{\epsilon}_1;\ldots;\bar{b}_m,\bar{\epsilon}_m)$ is a winning strategy for $\basicgame{s+1}$. We need to check that this strategy satisfies \CThreeDagger and \CFourDagger for $\sigma_{s+1}$. The first follows from the choice of $m_{s+1}$, and the fact that once \PTwo plays some string other than $\sigma_{s+1}$, they can never again play $\sigma_{s+1}$. The second is because the winning strategy for $\mc{H}(\bar{b}_1,\bar{\epsilon}_1;\ldots;\bar{b}_m,\bar{\epsilon}_m)$ satisfies \CB for each basic list.

Now we return to the omitted proofs.

\begin{proof}[Proof of Lemma \ref{one-or-other}]
In order to simplify the notation, denote by $\mc{G}$ the game $\mc{H}(\bar{b}_1,\bar{\epsilon}_1;\ldots;\bar{b}_r,\bar{\epsilon}_r)$ and by $\mc{G}(\iota)$ the game $\mc{H}(\bar{b}_1,\bar{\epsilon}_1;\ldots;\bar{b}_r,\bar{\epsilon}_r;\bar{c}x,\bar{\nu}\iota)$ for $\iota = 1$ and $\iota = -1$. We must show that \POne has a winning strategy for either $\mc{G}(-1)$ or $\mc{G}(1)$.

Suppose that \POne does not have a winning strategy for $\mc{G}(1)$. Then \PTwo has a winning strategy for $\mc{G}(1)$. We also know that \POne has a winning strategy in $\mc{G}$. We will show that \POne has a winning strategy for $\mc{G}(-1)$.

The strategy for \POne will try to do two things. First, it will try to be the same as \POne's winning strategy for $\mc{G}$. Then the only way for \PTwo to lose while using such a strategy will be for \PTwo to use a basic play based on $\mathfrak{b}^{r+1}$, the $r+1$st basic list, and to have \POne fail to satisfy \CB for this basic list. The second thing that \POne will try to do is, if \PTwo follows the basic list $\mathfrak{b}^{r+1}$, to try and force \PTwo to use their winning strategy from $\mc{G}(1)$. Since this is a winning strategy for \PTwo, \POne will fail to satisfy \CB in $\mc{G}(1)$ for $\mathfrak{b}^{r+1}$. But failing to satisfy \CB in $\mc{G}(1)$ for some basic list is the same as satisfying \CB in $\mc{G}(-1)$ for that basic list. So \POne will win $\mc{G}(-1)$.

Let $\mc{S}$ be \POne's winning strategy for $\mc{G}$, and $\mc{T}$ be \PTwo's winning strategy for $\mc{G}(1)$. We will define $\mc{S}'$, a winning strategy for \POne in $\mc{G}(-1)$.

\POne must play first. The first move in $\mc{G}$ according to $\mc{S}$ is $H_0 = G a \bar{b}_1\cdots \bar{b}_r \bar{c}$. Let $(\bar{d}_0,\beta_1)$ be \PTwo's response to $H_0x$ according to their winning strategy $\mc{T}$ for $\mc{G}(1)$. Then the strategy $\mc{S}'$ for $\mc{G}(-1)$ will play $G_0 = H_0 b \bar{d}_0$ as \POne's initial play. Note that \PTwo has not actually played $(\bar{d}_0,\beta_1)$; \POne has just looked ahead at what \PTwo \textit{would} play if they were following the strategy $\mc{T}$.

Now \PTwo must actually respond to $G_0$. Suppose that they respond with $(\bar{e}_0,\alpha_1)$. If $\alpha_1 \neq \beta_1^{r+1}$, then since $\alpha_1$ is not part of the basic list $\beta^{r+1}$, the winning and losing conditions are the same as in $\mc{G}$; have \POne respond to $(\bar{e}_0,\alpha_1)$ as they would, using $\mc{S}$, to $(\bar{d}_0\bar{e}_0,\alpha_1)$ in $\mc{G}$. After this, \POne just continues to use $\mc{S}$ to win.

If instead $\alpha_1 = \beta_1^{r+1}$, then let $H_1$ be \POne's response, using $\mc{S}$, to $(\bar{d}_0\bar{e}_0,\beta_1^{r+1})$. Let $(\bar{d}_1,\beta_2)$ be \PTwo's response to $H_1$ using the strategy $\mc{T}$. Then once again the strategy $\mc{S}'$ will tell \POne to play $G_1 = H_1\bar{d}_1$.

If $\POne$ ever plays an $\alpha_i \neq \beta^{r+1}_i$, then $\PTwo$ can win by following the strategy $\mc{S}$. Otherwise, $\mc{S}'$ will have \POne play $G_0,G_1,\ldots$ in response to \PTwo playing $(\bar{e}_0,\beta_1),(\bar{e}_1,\beta_2)\ldots$. There will be $H_0,H_1,\ldots$ which are plays according to $\mc{S}$ in response to \PTwo playing $(\bar{d}_0\bar{e}_0,\beta_1),(\bar{d}_1\bar{e}_1,\beta_2),\ldots$. Moreover, we will have $H_i \subset G_i$.

Since $H_0,H_1,\ldots$ is a winning play against $(\bar{d}_0\bar{e}_0,\beta_1),(\bar{d}_1\bar{e}_1,\beta_2),\ldots$ in $\mc{G}$, we can see that because all for the conditions \COneDagger-\ExtTwoDagger are satisfied for $H_0,H_1,\ldots$, they are also satisfied for $G_0,G_1,\ldots$ (the extra elements $\bar{d}_i$ in $G_i$ but not $H_i$ are included in \AtDagger, \ExtOneDagger, and \ExtTwoDagger as the tuples which \PTwo plays). Also, \CB is satisfied for the basic lists $\mathfrak{b}^1,\ldots,\mathfrak{b}^r$.

So the only way $G_0,G_1,\ldots$ could be a losing play against $(\bar{e}_0,\beta_1),(\bar{e}_1,\beta_2)\ldots$ is if $\beta_1,\beta_2,\ldots$ is the basic play $\mathfrak{b}^{r+1}$ and \CB fails for this basic play. $H_0,H_1,\ldots$ is a play by \POne against which \PTwo wins in $\mc{G}(1)$ using the strategy $\mc{T}$. Then \CB fails in $\mc{G}(1)$ for the basic list $\mathfrak{b}^{r+1}$, and so \CB must be satisfied in $\mc{G}(-1)$ for this basic list.

So \POne wins this play of $\mc{G}(-1)$, and $\mc{S}'$ is a winning strategy.
\end{proof}

To prove Lemma \ref{transfer-win}, we first need a quick technical lemma.

\begin{lem}
Suppose that $\bar{x} \leq_1 \bar{x}'$. Let $\bar{y}$ be a tuple so that for no $y \in \bar{y}$ is there a $y'$ with $\bar{x} y \leq_1 \bar{x}' y'$. Then for each tuple $\bar{z}'$, there is $\bar{z}$ such that $\bar{x}' \bar{z}' \leq_0 \bar{x} \bar{z}$ and moreover $\bar{z}$ is disjoint from $\bar{y}$.
\begin{proof}
For each $y \in \bar{y}$, and each $y' \in \mc{C}$, there is some existential fact true about $y'$ which is not true of $y$. Let $\bar{u}'$ be tuple of elements witnessing these existential formulas for each $z' \in \bar{z}'$. Then there are $\bar{z}$ and $\bar{u}$ such that $\bar{x}' \bar{z}' \bar{u}' \leq_0 \bar{x} \bar{z} \bar{u}$. So $\bar{x}' \bar{z}' \leq_0 \bar{x} \bar{z}$ and each $z \in \bar{z}$ satisfies an existential formula which no $y \in \bar{y}$ satisfies. Hence $\bar{z}$ is disjoint from $\bar{y}$.
\end{proof}
\end{lem}

\begin{proof}[Proof of Lemma \ref{transfer-win}]
Once again we will simplify the notation. For $z = x$ or $z = y$, let $\mc{G}(z)$ be the game $\mc{H}(\bar{b}_1,\bar{\epsilon}_1;\ldots;\bar{b}_r,\bar{\epsilon}_r;\bar{c} z,\bar{\nu}\iota)$.

Let $\mc{T}$ be a winning strategy for \POne in $\mc{G}(x)$. We need to find a winning strategy $\mc{S}$ for \POne in $\mc{G}(y)$. We will use the fact that
\[ G a \bar{b}_1\cdots\bar{b}_r\bar{c} x \leq_1 G a \bar{b}_1\cdots\bar{b}_r\bar{c} y \]
to convert $\mc{T}$ into the desired strategy $\mc{S}$.

Let $H_0 \supset G a \bar{b}_1\cdots\bar{b}_r \bar{c} x$ be the initial play for \POne according to $\mc{T}$. Let $\bar{u}$ be the tuple of elements such that $H_0 = G a \bar{b}_1\cdots\bar{b}_r \bar{c} x \bar{u}$.

Let $I = \{i_1,\ldots,i_n\}$ be a maximal set of indices in $\bar{u}$ such that there is a tuple $\bar{v}$ such that
\[ G a \bar{b}_1\cdots\bar{b}_r\bar{c} x \bar{u}_I \leq_1 G a \bar{b}_1\cdots\bar{b}_r\bar{c} y \bar{v} \]
where $\bar{u}_I$ is $(u_{i_1},\ldots,u_{i_n})$. Let $J$ be the rest of the indices.

The first play according to $\mc{S}$ will be $G_0 = G a \bar{b}_1\cdots\bar{b}_r \bar{c} y \bar{v}$. We have $H_0^{I} \leq_1 G_0$ where $H_0^{I}$ denotes $H_0$ with the entries at indices in $J$ removed.

Now suppose that \PTwo responds with $(\bar{d}_0,\alpha_1)$.  If $\alpha_1 = \sigma_{s+1}$, then \POne can play $G_0 \bar{d}_0$ in response to win (by property \BOne of basic lists). Otherwise, $\alpha_1 \neq \sigma_{s+1}$.

Using the fact that $H_0^I \leq_1 G_0$, choose $\bar{e}_0$ such that $G_0 \bar{d}_0 \leq_0 H_0^{I} \bar{e}_0$. By the previous lemma, $\bar{e}_0$ is disjoint from $H_0$. Let $H_1$ be the response, according to $\mc{T}$, to $(\bar{e}_0,\alpha_1)$. Let $s$ be the permutation which moves those entries of $H_1$ with indices in $J$ to the end (the permutation $s$ fixes anything not in the domain of $H_1$, so that for example if $H' \supset H_1$, then applying the permutation $s$ to $H'$ does not move the indices $J$ to the end of $H'$, but rather to somewhere in the middle). Then $\mc{S}$ will tell \POne to respond to $(\bar{d}_0,\alpha_1)$ with $H_1^s$, the application of the permutation $s$ to $H_1$.

To any further play $(\bar{c}_1,\alpha_2),(\bar{c}_2,\alpha_3),\ldots$, $\mc{S}$ will respond in the same way as $\mc{T}$, except that it will again apply the permutation $s$.

The games $\mc{G}(x)$ and $\mc{G}(y)$ are the same, except for the initial move; otherwise, the ways in which \POne can lose are the same. The permutation $s$ does not affect any of the conditions, since it only permutes indices which do not do any coding (i.e. above everything in the image of $\loc_s$, and also above the position of $x$ and $y$ in $G_0$ and $H_0$ respectively). Since $G_0 \bar{d}_0 \leq_0 H_0^{I} \bar{e}_0 \leq_0 H_1^s$, and if $H_i \leq_0 H_{i+1}$ then $H_i^s \leq_0 H_{i+1}^s$, \AtDagger still holds. And as none of $\alpha_1,\ldots,\alpha_\ell$ are $\sigma_{s+1}$, \ExtOneDagger is not affected by changing the change from $H_0$ to $G_0$, and \ExtTwoDagger is also not affected by the application of the permutation $s$. Thus \POne has a winning strategy for $\mc{G}(y)$.
\end{proof}

\chapter{Further Questions}
\label{QuestionSection}
In this section we will list some of the unresolved questions from our investigation. This is a new investigation and so there are many questions to be answered. We have seen that there are many nice properties that degree spectra on a cone must have. Many degree spectra are well-known classes of degrees and satisfy many ``fullness'' properties. But on the other hand, there are some interesting degree spectra that are not so nicely behaved, like the incomparable degree spectra from Theorem \ref{thm-incomparable-dce} and the relation on $(\omega,<)$ from \ref{thm:weird-relation-on-omega}. The general question is: to what extent do the degree spectra avoid pathological behaviour? Of course, there are many specific questions about particular types of pathological behaviour and many questions arising directly out of results in this paper. We will take the opportunity to list some of them here.

In Chapter \ref{DCESection} we gave a condition, involving d-free tuples, which was equivalent to being intrinsically of c.e.\ degree, but the equivalence only held for relations which are relatively intrinsically d.c.e. (see Proposition \ref{ce-degree-condition}). A relation which is not intrinsically $\Delta^0_2$ cannot be intrinsically of c.e.\ degree. We ask:
\begin{question}
Which intrinsically $\Delta^0_2$ but not relatively intrinsically d.c.e.\ relations are intrinsically of c.e.\ degree?
\end{question}

Also in Chapter \ref{DCESection}, we gave an example of two relations with incomparable degree spectra on a cone, but whose degree spectra are strictly contained within the d.c.e.\ degrees and strictly contain the c.e.\ degrees (see Theorem \ref{thm-incomparable-dce}, Proposition \ref{M-not-A}, and Proposition \ref{prop:incomp-2}). We ask whether it is possible to find other such degree spectra:
\begin{question}
How many different possible degree spectra on a cone are there strictly containing the c.e.\ degrees and strictly contained in the d.c.e.\ degrees? How are they ordered?
\end{question}

Many of the degree spectra on a cone have a ``name,'' that is, some sort of description of degrees which relativizes. For example, the $\Delta^0_\alpha$ degrees, the $\Sigma^0_\alpha$ degrees, $\alpha$-c.e.\ degrees, and $\alpha$-\cea degrees. We do not know of any such description of the degree spectra from Proposition \ref{M-not-A} and Proposition \ref{prop:incomp-2}. In general, one would hope that any degree spectrum on a cone has a nice description of some form.
\begin{question}
Is there a good degree-theoretic description of the degree spectra from Examples \ref{exam1} and \ref{exam2}?
\end{question}

In Chapter \ref{OmegaSection}, we show that every relation on $(\omega,<)$ which is intrinsically $\alpha$-c.e.\ is intrinsically of c.e.\ degree (see Propositions \ref{omega-equivalences} and \ref{intrinsically-ce}). It might, however, be possible to have a relation which is intrinsically of $\alpha$-c.e.\ degree but not intrinsically $\alpha$-c.e.
\begin{question}
Is there a computable relation on the standard copy of $(\omega,<)$ which is not intrinsically of c.e.\ degree, but is intrinsically of $\alpha$-c.e.\ degree for some fixed ordinal $\alpha$?
\end{question}

In Theorem \ref{thm:weird-relation-on-omega}, we show that there is a relation $R$ on $(\omega,<)$ such that either the degree spectrum of $R$ on a cone is strictly contained between the c.e.\ degrees and the $\Delta^0_2$ degrees, or $R$ has degree spectrum $\Delta^0_2$ but not uniformly. It would be interesting to know if either of these behaviours is possible.
\begin{question}
Is there a relation on $(\omega,<)$ whose degree spectrum on a cone is strictly contained between the c.e.\ degrees and the $\Delta^0_2$ degrees?
\end{question}
\begin{question}
Is there a relation with $\dgSp_{rel}=\Delta^0_2$ but not uniformly? Is there such a relation on $(\omega,<)$?
\end{question}

In Chapter \ref{SigmaTwoSection} we proved a ``fullness'' result by showing that any degree spectra on a cone which strictly contains the $\Delta^0_2$ degrees contains all of the 2-\cea degrees. Fullness results are very interesting because they show that degree spectra must contain all of the degrees of a particular type, and hence can provide a good description of the degree spectrum. Our result for 2-\cea degrees is an answer to a question of Ash and Knight from \cite{AshKnight95} and \cite{AshKnight97}. The general question, when stated in our framework, is as follows:
\begin{question}
If a degree spectrum on a cone strictly contains the $\Delta^0_3$ degrees, must it contain the 3-\cea degrees? What about for general $\alpha$?
\end{question}
Recall that a positive answer to this question implies a positive answer to the following questions of Montalb\'an which first appeared in \cite{Wright13}:
\begin{question}[Montalb\'an]\leavevmode
\begin{enumerate}
	\item Is it true that for any relation $R$, for all degrees \textbf{d} on a cone, $\dgSp(R)_{\leq \textbf{d}}$ has a maximal element?
	\item Does the function which takes a degree \textbf{d} to the maximal element of $\dgSp(R)_{\leq \textbf{d}}$ satisfy Martin's conjecture?
\end{enumerate}
We also add the question:
\begin{enumerate}[resume]
	\item Is this function uniformly degree invariant (in the sense of the footnote at the end of Chapter \ref{Preliminaries})?
\end{enumerate}
\end{question}

In \cite{Slaman05}, Slaman defines a \textit{$\Sigma$-closure operator} to be a map $M: 2^{\omega} \to 2^{2^{\omega}}$ such that:
\begin{enumerate}
	\item for all $X$, $X \in M(X)$,
	\item for all $X$ and $Y,Z \in M(X)$, $Y \oplus Z \in M(X)$,
	\item for all $X$ and $Y \in M(X)$, if $Z \equiv_T Y$ then $Z \in M(X)$,
	\item for all $X \leq_T Y$, $M(X) \subseteq M(Y)$.
\end{enumerate}
Note that for any relation $R$ on a structure $\mc{A}$, the map which takes a set $X$ to $\dgSp(\mc{A},R)_{\leq \textbf{d}}$ (where $\textbf{d}$ is the degree of $X$) satisfies (1), (3), and (4) of this definition.

Slaman shows:
\begin{thm}[{Slaman \cite[Theorem 5.2]{Slaman05}}]
Let $M$ be a Borel $\Sigma$-closure operator. If, on a cone, $M(X) \nsubseteq \Delta^0_2(X)$, then there is a cone on which $M(X)$ contains all of the sets which are \cea in $X$.
\end{thm}
One can see this as an analogue of Harizanov's Theorem \ref{Harizanov}. We ask whether the degree spectrum on a cone is a $\Sigma$-closure operator:
\begin{question}
On a cone, do the degree spectra form an upper semi-lattice under joins?
\end{question}

In Chapter \ref{Preliminaries}, we defined the alternate degree spectrum
\[ \dgSp^*(\mc{A},R)_{\leq  \textbf{d}} = \{ d(R^\mc{B}) : (\mc{B},R^\mc{B}) \text{ is an isomorphic copy of } (\mc{A},R) \text{ with } \mc{B} \leq_T \textbf{d}\}\]
and used this to define the alternate degree spectrum on a cone, $\dgSp^*_{rel}$. In Appendix \ref{CESection}, we showed that Harizanov's Theorem \ref{Harizanov} on c.e.\ degrees (in the form of Corollary \ref{cor:Harizanov}) holds for $\dgSp^*_{rel}$. The general question is whether anything we can prove about $\dgSp_{rel}$ is also true of $\dgSp^*_{rel}$. More formally:
\begin{question}
Is it always the case that restricting $\dgSp^*(\mc{A},R)_{\leq  \textbf{d}}$ to the degrees above \textbf{d} gives $\dgSp(\mc{A},R)_{\leq  \textbf{d}}$?
\end{question}

\appendix
\chapter[Relativizing Harizanov's Theorem]{Relativizing Harizanov's Theorem on C.E. Degrees}
\label{CESection}
Theorem \ref{Harizanov} had, as a consequence, Corollary \ref{cor:Harizanov} which said that for any structure $\mc{A}$ and relation $R$, either $\dgSp_{rel}(\mc{A},R) = \Delta^0_1$ or $\dgSp_{rel}(\mc{A},R) \supseteq \Sigma^0_1$. Now $\dgSp_{rel}(\mc{A},R)$ was defined using the behaviour, on a cone, of the degree spectrum relativized to a degree \textbf{d}:
\[ \dgSp(\mc{A},R)_{\leq \textbf{d}} = \{ d(R^\mc{B}) \oplus \textbf{d} : (\mc{B},R^\mc{B}) \text{ is an isomorphic copy of } (\mc{A},R) \text{ with } \mc{B} \leq_T \textbf{d}\}. \]
In Chapter \ref{Preliminaries}, we briefly considered the alternate definition
\[ \dgSp^*(\mc{A},R)_{\leq  \textbf{d}} = \{ d(R^\mc{B}) : (\mc{B},R^\mc{B}) \text{ is an isomorphic copy of } (\mc{A},R) \text{ with } \mc{B} \leq_T \textbf{d}\}\]
from which one could define an alternate degree spectrum on a cone, $\dgSp^*_{rel}(\mc{A},R)$. The relativization of Theorem \ref{Harizanov} does not suffice to prove that $\dgSp^*_{rel}(\mc{A},R) = \Delta^0_1$ or $\dgSp^*_{rel}(\mc{A},R) \supseteq \Sigma^0_1$. This is because, as we will see later, the relativization to a degree \textbf{d} only shows that $R^\mc{B} \oplus \textbf{d} \equiv_T C \oplus \textbf{d}$ (for $C$ c.e.\ in \textbf{d}).

Our goal in this appendix is to give a new proof of Theorem \ref{Harizanov} whose revitalization is strong enough to apply to the alternative degree spectrum $\dgSp^*_{rel}(\mc{A},R)$. We will show that it is possible, though the proof is significantly more complicated. The reader may skip this appendix without any impact on their understanding of the rest of this work. The results in this section are in a similar style, but much easier, than Theorem \ref{relativized-1} and hence may be read as an introduction to the proof of that result.

\begin{thm}\label{HarizanovRel}
Suppose that $\mc{A}$ is a computable structure and $R$ is a computable relation which is not intrinsically computable. Suppose that $\mc{A}$ satisfies the following effectiveness condition: the $\exists_1$-diagram of $(\mc{A},R)$ is computable, and given a finitary existential formula $\varphi(\bar{c},\bar{x})$, we can decide whether there are finitely many or infinitely many solutions. Then for any sets $X \leq_T Y$ with $Y$ c.e.\ in $X$, there is an $X$-computable copy $\mc{B}$ of $\mc{A}$ with
\[ R^\mc{B} \equiv_T Y \]
Moreover, $Y$ can compute the isomorphism between $\mc{A}$ and $\mc{B}$.
\end{thm}

From this theorem, we get the following corollary:
\begin{corollary}
Suppose that $\mc{A}$ is a structure and $R$ is a relation on $\mc{A}$ which is not intrinsically computable on a cone. Then $\dgSp^*_{rel}(\mc{A},R)$ contains the c.e.\ degrees.
\end{corollary}

Our goal in proving Theorem \ref{HarizanovRel} is to give evidence towards two ideas. First, we want to give evidence that results that can be proved for $\dgSp_{rel}$ can also be proved for $\dgSp^*_{rel}$. Second, we will see that if we try to prove results for $\dgSp^*_{rel}$, we have to deal with many complications which distract from the heart of the proof.

We will begin by describing Harizanov's proof of Theorem \ref{Harizanov}. This will both show us why it does not relativize in the way we desire, and also guide us as to what we need to do. Let $\mc{A}$ be a structure and $R$ a computable relation which is not intrinsically computable, say $R$ is not intrinsically $\Pi^0_1$. Harizanov's construction uses the following definition from Ash-Nerode \cite{AshNerode81}.

\begin{defn}
Let $\bar{c}$ be a tuple from $\mc{A}$. We say that $\bar{a} \notin R$ is \textit{free} over $\bar{c}$ if for any finitary existential formula $\psi(\bar{c},\bar{x})$ true of $\bar{a}$ in $\mc{A}$, there is $\bar{a}' \in R$ which also satisfies $\psi(\bar{c},\bar{x})$.
\end{defn}

If $\mc{A}$ is assumed to have an effectiveness condition---namely that for each $\bar{c}$ in $\mc{A}$ and finitary existential formula $\varphi(\bar{x},\bar{c})$, we can decide whether there is $\bar{a} \notin R$ such that $\mc{A} \models \varphi(\bar{a},\bar{c})$---then for any tuple $\bar{c}$, we can effectively find a tuple $\bar{a} \notin R$ which is free over $\bar{c}$. Harizanov uses these free elements to code a c.e.\ set $C$ into a computable copy $\mc{B}$ of $\mc{A}$. Building $\mc{B}$ via a $\Delta^0_2$ isomorphism, for each $x \in \omega$, there is a tuple from $\mc{B}$ which codes, by being in $R$ or not in $R$, whether or not $x$ is in $C$. For a given $x$, she fixes a tuple $\bar{b}_x$ from $\mc{B}$ and maps it to a tuple $\bar{a}$ from $\mc{A}$ which is free; the fact that $\bar{a} \notin R$ codes that $x \notin C$. If, at a later stage, $x \in C$, then using the fact that $\bar{a}$ is free, she modifies the $\Delta^0_2$ isomorphism to instead map $\bar{b}$ to a tuple $\bar{a}' \in R$, coding that $x \in C$. The argument involves finite injury. Given $R^\mc{B}$, one can compute $\bar{b}_0$ and decide whether or not $0 \in C$; then, knowing this, we can wait until a stage at which $0$ enters $C$ (if necessary) and can compute $\bar{b}_1$. We need to wait, since the choice of $\bar{b}_1$ may be injured before this stage. Once we know $\bar{b}_1$, we can use $R^\mc{B}$ to decide whether or not $1 \in C$, and so on. On the other hand, given $C$, one understands the injury from the construction and can compute the isomorphism between $\mc{A}$ and $\mc{B}$.

Now consider the relativisation to a degree \textbf{d}. Let $C$ be \cea in \textbf{d} and try to use the same construction (building $\mc{B}$ computable in \textbf{d}). Given $C$, we can once again run the construction and compute the isomorphism between $\mc{A}$ and $\mc{B}$. However, given $R^\mc{B}$, we can not necessarily compute \textbf{d}, and hence do not necessarily have access to the enumeration of $C$. Without this, we cannot run through the construction and compute the coding locations $\bar{b}_i$. We do, however, get that $R^\mc{B} \oplus \textbf{d} \equiv_T C$.

To prove Theorem \ref{HarizanovRel}, we need a strategy to divorce the coding locations from the construction of $\mc{B}$. The trick we will use is as follows. Fix beforehand tuples from $\mc{B}$ to act as coding locations, and number them in increasing order using $\omega$. Choose a computable infinite-to-one bijection $g: \omega \to 2^{<\omega}$. A coding location may either be ``on'' or ``off'' depending on whether or not it is in $R^\mc{B}$ (though whether ``on'' means in $R^\mc{B}$ and ``off'' means out of $R^\mc{B}$, or vice versa, will depend on the particular structure and relation). We will show that we make two choices for a coding location: we can either choose for them to be permanently off no matter what happens with the other coding locations, or we can choose to have a coding location start on and later turn off (after which we not longer have control of the coding location---if some earlier location turns from on to off, the later coding location may turn back on again).

We will ensure that there is a unique increasing sequence $k_0 < k_1 < k_2 < \cdots$ of coding locations which are ``on'' such that $g(k_0)$ has length one, and $g(k_{i+1})$ extends $g(k_i)$ by a single element. Thus $\bigcup_{i \in \omega} g(k_i)$ will be a real, and we will ensure that it is the set $C$ which we are trying to code. We call such a sequence an \textit{active sequence}. Since this sequence is unique, we can compute it using $R^\mc{B}$ by looking for the first coding location $k_0$ with $g(k_0)$ of length one which is on, then looking for the next coding location $k_1$ with $g(k_1)$ an extension of $g(k_0)$ of length two and which is on, and so on.

We will illustrate how we build the sequence using the following example where we code whether two elements $0$ and $1$ are in $C$. To code that $0 \notin C$, start by choosing a coding location $k_0$ with $g(k_0) = 0$ and have $k_0$ be on. Set every smaller coding location to be permanently off. Then, to code that $1 \notin C$, find a coding location $k_1 > k_0$ with $g(k_1) = 00$ and have $k_1$ be on, while every coding location between $k_0$ and $k_1$ is off. Now if $0$ enters $C$, switch $k_0$ off; $k_1$ might be on, but every other coding location less than $k_1$ is permanently off. Because there is no coding location $i < k_1$ with $g(i)$ of length one and $g(i) \prec g(k_1)$, $k_1$ can never be part of an active sequence even if it is on. Find some $k_0'$ which has no appeared yet with $g(k_0') = 1$ and set $k_0'$ to be on, while every coding location between $k_1$ and $k_0'$ is permanently off. Thus $k_0'$ will be the first coding location in the active sequence.

In the remainder of this chapter, we give the proof of Theorem \ref{HarizanovRel}.

\section{Framework of the Proof}

Let $\mc{A}$, $R$, $X$, and $Y$ be as in the theorem. Note that the effectiveness condition is robust in the following sense: if $Q$ is definable from $R$ via both a finitary existential formula and a finitary universal formula, then the effectiveness conditions holds for $Q$ as well.

The proof of the theorem is by induction on the arity $r$ of $R$. We will argue that we can make three assumptions, (I), (II), and (III). Let $\noreps{A}^r$ denote the tuples in $A^r$ with no duplicate entries and let $A^r_{i=j}$ denote the set of tuples from $A^r$ with $i$th entry equal to the $j$th entry. Note that $\noreps{A}^r$ and $A^r_{i=j}$ are defined by both finitary existential and universal formulas. If the restriction $R \cap A^r_{i=j}$ of $R$ to some $A^r_{i=j}$ is not intrinsically computable, then as the restriction is essentially an $(r-1)$-ary relation, by the induction hypothesis, there is an $X$-computable copy $\mc{B}$ of $\mc{A}$ with $R^\mc{B} \cap B^r_{i=j} \equiv_T Y$. Now the set $Y$ computes the isomorphism between $\mc{A}$ and $\mc{B}$, and since $R$ is computable in $\mc{A}$, $Y$ computes $R^\mc{B}$. Also, $R^\mc{B} \cap B^r_{i=j} \leq_T R^\mc{B}$ and hence $R^\mc{B} \equiv_T Y$ and we are done. So we may assume that the restrictions $R \cap A^r_{i=j}$ are intrinsically computable and hence are defined by finitary existential and universal formulas. Since
\[
R = (R \cap \noreps{A}^r) \cup \bigcup_{i \neq j} (R \cap A^r_{i=j})
\]
and $R\cap \noreps{A}^r$ is disjoint from $\bigcup_{i \neq j} (R \cap A^r_{i=j})$, we must have that $R \cap \noreps{A}^r$ is not intrinsically computable. So we may replace $R$ by $R \cap \noreps{A}^r$. This is assumption (I): that $R \subseteq \noreps{A}^r$. When we say $\bar{a} \in R$ or $\bar{a} \notin R$, we really mean that $\bar{a} \in \noreps{A}^r$ as well.

An important aspect of the proof will be whether we can find ``large'' formally $\Sigma^0_1$ sets contained in $R$ or its complement. Let us formally define what we mean by large. We say that two tuples $\bar{a} = (a_1,\ldots,a_n)$ and $\bar{b} = (b_1,\ldots,b_n)$ are \textit{disjoint} if they do not share any entries, that is, $a_i \neq b_j$ for each $i$ and $j$. We say that a set $S \subseteq A^n$ is \textit{thick} if it contains an infinite set of pairwise disjoint tuples; otherwise, we say that $S$ is \textit{thin}.

\begin{lem}
If $S \subseteq A^k$ is a thin set, there is a bound on the size of sets of pairwise disjoint tuples from $S$.
\begin{proof}
Suppose that there is no such bound. We claim that $S$ is thick. Let $\bar{a}_1,\ldots,\bar{a}_n$ be a maximal pairwise disjoint subset of $S$. Since there is no bound on the size of sets of pairwise disjoint tuples form $S$, we can pick $\bar{b}_1,\ldots,\bar{b}_{n \cdot k + 1}$ all pairwise disjoint. There are $k \cdot n$ distinct entries that appear in $\bar{a}_1,\ldots,\bar{a}_n$; since each entry can only appear in a single $\bar{b}_i$, some $\bar{b}_i$ must be disjoint from each $\bar{a}_j$. This contradicts the maximality of $\bar{a}_1,\ldots,\bar{a}_n$.
\end{proof}
\end{lem}

We'll argue that we may assume that $R$ is thick. Suppose that $R$ is not thick. Then there are finitely many tuples $\bar{a}_1,\ldots,\bar{a}_n$ such that no more elements of $R$ are disjoint from $\bar{a}_1,\ldots,\bar{a}_n$. Write $\bar{a}_i = (a_i^1,\ldots,a_i^r)$. Let $R_{i,j,k}$ be the set of tuples $\bar{b} = (b^1,\ldots,b^n)$ in $R$ with $b^j = a_i^k$. Then
\[
	R = \bigcup_{i,j,k} R_{i,j,k}.
\]
Then one of the $R_{i,j,k}$ is not intrinsically computable, or else $R$ would be intrinsically computable. But $R_{i,j,k}$ is essentially an $(r-1)$-ary relation and $R_{i,j,k} \leq_T R$. So we can use the induction hypothesis as above to reduce to the case where $R$ and $\neg R$ are both thick. This is assumption (II).

Let $\bar{c}$ be any tuple. By a similar argument as above, we may assume that $R$, when restricted to tuples not disjoint from $\bar{c}$, is intrinsically computable, and that $R$ restricted to tuples disjoint from $\bar{c}$ is not intrinsically computable. This is assumption (III).

Now we have three cases to consider.
\begin{enumerate}
	\item $R$ is not intrinsically $\Pi^0_1$, but $R$ contains a thick set defined by a $\Sigma^\comp_1$ formula.
	\item $R$ is not intrinsically $\Sigma^0_1$, but $\neg R$ contains a thick defined by a $\Sigma^\comp_1$ formula.
	\item Neither $R$ nor $\neg R$ contains a thick set defined by a $\Sigma^\comp_1$ formula.
\end{enumerate}

These exhaust all of the possibilities. If we are not in the third case, then either $R$ or $\neg R$ must contain a thick set defined by a $\Sigma^\comp_1$ formula. If neither $R$ nor $\neg R$ are definable by a $\Sigma^\comp_1$ formula, then we must be one of the first two cases. Finally, it is possible that $R$ is definable by a $\Pi^\comp_1$ or $\Sigma^\comp_1$ formula. In the first case, $R$ is not definable by a $\Sigma^\comp_1$ formula, but $\neg R$ is (and, as $\neg R$ is thick, (2) holds). In the second case, $R$ is definable by a $\Sigma^\comp_1$ formula, but not by any $\Pi^\comp_1$ formula, and so (1) holds.

Note that cases one and two include the particular cases where $R$ is intrinsically $\Sigma^0_1$ but not intrinsically $\Pi^0_1$ and intrinsically $\Pi^0_1$ but not intrinsically $\Sigma^0_1$ respectively.

\section{The First Two Cases} 

The second case is similar to the first, but with $R$ replaced by $\neg R$. We will just consider the first case.

We assume that $R$ satisfies (I), (II), and (III). Now we have the following lemma which applies to $R$ (since our effectiveness condition is strong enough to imply the condition in the lemma, as well as to find the free elements from the lemma).

\begin{lem}[{Ash-Nerode \cite{AshNerode81}}]
Suppose that $R$ is not defined in $\mc{A}$ by $\Pi^\comp_1$ formula. Furthermore, suppose that for each tuple $\bar{c}$ in $\mc{A}$ and finitary existential formula $\varphi(\bar{c},x)$, we can decide whether there exists $\bar{a} \notin R$ such that $\mc{A} \models \varphi(\bar{a},\bar{c})$. Then for each tuple $\bar{c} \in \mc{A}$, there is $\bar{a} \notin R$ disjoint from $\bar{c}$ and free over $\bar{c}$.
\end{lem}

Let $\bar{c}$ be a tuple from $\mc{A}$. We say that $\bar{a} \in \noreps{A}^r$ is \textit{constrained over $\bar{c}$} if there is a finitary existential formula $\psi(\bar{c},\bar{x})$ true of $\bar{a}$ in $\mc{A}$ such that for any $\bar{a}' \in \noreps{A}^r$ which also satisfies $\psi(\bar{c},\bar{x})$, $\bar{a}' \in R$ if and only if $\bar{a} \in R$.

Let $\bar{d}$ be a tuple in $\mc{A}$ and $\varphi(\bar{d},\bar{x})$ a $\Sigma^\comp_1$ formula defining a thick subset of $R$. Now there are infinitely many disjoint tuples $\bar{b} \in \noreps{A}^r$ satisfying $\varphi(\bar{d},\bar{x})$ and hence in $R$. Each of these tuples satisfies a finitary existential formula which is a disjunct in $\varphi(\bar{d},\bar{x})$, and hence is constrained over $\bar{d}$. We may find such tuples computably.

The following proposition completes the theorem in the first case.

\begin{prop}\label{CEProp}
Suppose that there is a tuple $\bar{d}$ such that for any $\bar{c} \supseteq \bar{d}$, we can compute new elements $\bar{a} \notin R$ and $\bar{b} \in R$ disjoint from $\bar{c}$ such that $\bar{a}$ is free over $\bar{c}$ and $\bar{b}$ is constrained over $\bar{d}$. Then for any sets $X \leq_T Y$ with $Y$ c.e. in $X$ there is an $X$-computable copy $\mc{B}$ of $\mc{A}$ with $R^\mc{B} \equiv_T Y$. Moreover, $Y$ can compute the isomorphism between $\mc{A}$ and $\mc{B}$.
\begin{proof}
We may assume that the constants $\bar{d}$ are part of the language, and hence ignore them.

Let $\bar{c}$ be a tuple. Suppose that $a \notin R$ is free over $\bar{c}$, and $\bar{b}_1,\ldots,\bar{b}_m \in R$ are constrained, and $\bar{e}$ any elements. Let $\varphi(\bar{c},\bar{u},\bar{v}_1,\ldots,\bar{v}_m,\bar{w})$ be a quantifier-free formula true of $\bar{a},\bar{b}_1,\ldots,\bar{b}_m,\bar{e}$. We claim that there is $\bar{a}' \in R$, $\bar{b}_1',\ldots,\bar{b}_m' \in R$ constrained, and $\bar{e}$ such that $\mc{A} \models \varphi(\bar{c},\bar{a}',\bar{b}_1',\ldots,\bar{b}_m',\bar{e}')$. For $i=1,\ldots,m$ choose $\psi_i(\bar{d},\bar{v}_i)$ an existential formula true of $\bar{b}_i$ in $\mc{A}$ such that for any $\bar{b}_i' \in \mc{A}$ which also satisfies $\psi_i$, $\bar{b}_i' \in R$; note that any such $\bar{b}_i'$ is also constrained over $\bar{d}$ since it satisfies $\psi_i$. Then consider the existential formula
\[
	\phi(\bar{c},u) = \exists \bar{v}_1,\ldots,\bar{v}_m,\bar{w}(\varphi(\bar{c},\bar{u},\bar{v}_1,\ldots,\bar{v}_m,\bar{w}) \wedge \bigwedge_{i=1,\ldots,m} \psi_i(\bar{d},\bar{v}_i))
\]
Since $\bar{a}$ is free over $\bar{c}$, there is $\bar{a}' \in R$ with $\mc{A} \models \phi(\bar{c},\bar{a}')$. Then let $\bar{b}_1',\ldots,\bar{b}_m' \in R$ and $\bar{e}$ be the witnesses to the existential quantifier. These are the desired elements.

Let $Y_s$ be the enumeration of $Y$ relative to $X$. Using $X$ we will construct by stages a copy $\mc{B}$ of $\mc{A}$ with $R^\mc{B} \equiv_T Y$. Let $B$ be an infinite set of constants disjoint from $A$. We will construct a bijection $F: B \to \mc{A}$ and use $F^{-1}$ to define the structure $\mc{B}$ on $B$. At each stage, we will give a tentative finite part of the isomorphism $F$. It will be convenient to view $B$ as a list of $r$-tuples $B = \{\bar{b}_0,\bar{b}_1,\ldots\}$.

Fix a computable infinite-to-one bijection $g: \omega \to 2^{<\omega} \cup \{\varnothing\}$. We will code $Y$ into $R^\mc{B}$ in the following way. We will ensure that if $k_1 < \cdots < k_n$ are such that $\bar{b}_{k_1},\ldots,\bar{b}_{k_n} \notin R^\mc{B}$ and $|g(k_i)| = i$ and $g(k_1) \prec g(k_2) \prec \cdots \prec g(k_n)$ then $g(k_n) \prec Y$. Moreover, we will ensure that there is an infinite sequence $k_1,k_2,\ldots$ with this property. Thus $R^\mc{B}$ will be able to compute $Y$ by reconstructing such a sequence. On the other hand, since $Y$ will be able to compute the isomorphism, it will be able to compute $R^\mc{B}$. Note that the empty string $\varnothing$ has length zero, so it can never appear in such a sequence, and thus marks a position that never does any coding.

We will promise that if at some stage we are mapping some tuple $\bar{b} \in B$ to a constrained tuple in $R$, $\bar{b}$ will be mapped to a constrained tuple of $R$ at every later stage.

At a stage $s+1$, let $F_s:\{\bar{b}_0,\ldots,\bar{b}_{\ell}\} \to \mc{A}$ be the partial isomorphism determined in the previous stage, and let $\mc{B}_s$ be the finite part of the diagram of $\mc{B}$ which has been determined so far. We will also have numbers $i_{1,s}, \ldots, i_{n_s,s}$ with $i_{1,s} < \cdots < i_{n_s,s}$ which indicate tuples $\bar{b}_{i_{k,s}}$. For each $k$, $g(i_{k,s})$ will code a string of length $k$. We are trying to ensure that if $g(i_{k,s}) \prec Y_t$ for every stage $t \geq s$, then we keep $\bar{b}_{i_{k,s}} \notin R^\mc{B}$, and otherwise we put $\bar{b}_{i_{k,s}} \in R^\mc{B}$. Once the first $k$ entries of $Y$ have stabilized, $i_{k,s}$ will stabilize.

We define a partial isomorphism $G$ extending $F_s$ which, potentially with some corrections, will be $F_{s+1}$. Let $G(\bar{b}_i) = F_s(\bar{b}_i)$ for $0 \leq i \leq \ell$. We may assume, by extending $G$ by adding constrained tuples, that $g(\ell+1) = \varnothing$. Let $\bar{a}_{\ell + 1}$ be the first $r$ new elements not yet in the image of $F_s$. Now let $k > \ell$ be first such that $g(k) = Y_{s+1} \restriction_{n_s + 1}$. Find new tuples $\bar{a}_{\ell+2},\ldots,\bar{a}_{k} \in R$ which are constrained. Also Find $\bar{a}_{k+1} \in \neg R$ which is free over $G(\bar{b}_0),\ldots,G(\bar{b}_{\ell}),\bar{a}_{\ell+1},\ldots,\bar{a}_{k}$. Set $G(\bar{b}_i) = \bar{a}_i$ for $\ell+1 \leq i \leq k$.

Let $\mc{B}_{s+1} \supseteq \mc{B}_s$ be the atomic formulas of G\"odel number at most $s$ which are true of the images of $\bar{b}_0,\ldots,\bar{b}_k$ in $\mc{A}$.

Now we will act to ensure that for each $m$, $g(i_{m,{s+1}}) \prec Y_{s+1}$. Find the first $m$, if any exists, such that $g(i_{m,s}) \nprec Y_{s+1}$. If such an $m$ exists, using the fact that $g(i_{m,s})$ is free over the previous elements, choose $\bar{a}_{m}',\ldots,\bar{a}_{k}'$ such that:
\begin{enumerate}
	\item $G(\bar{b}_0),\ldots,G(\bar{b}_{m-1}),\bar{a}_{m}',\ldots,\bar{a}_{k}'$ satisfy the same existential formula that $\bar{b}_0,\ldots,\bar{b}_{k}$ does in $\mc{B}_{s+1}$,
	\item for any $m' > m$ such that $\bar{a}_{m'} \in R$ is constrained, so is $\bar{a}_{m'}' \in R$, and
	\item $\bar{a}_{i_{m,s}}' \in R$.
\end{enumerate}
Set $F_{s+1}(\bar{b}_j) = G(\bar{b}_j)$ for $j < m$ and $F_{s+1}(\bar{b}_j)=\bar{a}_j'$ for $m+1 \leq j \leq k$. Also set $n_{s+1} = m - 1$ and $i_{0,s+1} = i_{0,s},\ldots,i_{m-1,s+1} = i_{m-1,s}$.

On the other hand, if no such $m$ exists, set $F_{s+1} = G$, set $n_{s+1} = n_s + 1$, $i_{n_s+1,s+1} = k$, and $i_{j,s+1} = i_{j,s}$ for $0\leq j \leq n_s$.

This completes the construction. It is a standard finite-injury construction and it is easy to verify that the construction works as desired.
\end{proof}
\end{prop}

\section{The Third Case} 

We may suppose that, for each $n$ and restricting to tuples in $\noreps{A}^{rn}$, there is no thick subset of $R^{i_1} \times \cdots \times R^{i_n}$ where $i_1,\ldots,i_n \in \{-1,1\}$ (or the complement of such a set) definable by a $\Sigma^\comp_1$ formula. Moreover, we may assume that the same is true for particular fibers of such a set; we may assume that for any $\bar{c}_2,\ldots,\bar{c}_n$, the fiber
\[ S = \{(\bar{x},\bar{y}_2,\ldots,\bar{y}_n) | \bar{x} \in R^{i_1}, \bar{y}_j\bar{c}_j \in R^{i_j}\} \]
has no thick subset. If there was a thick subset, then since $R$ is not definable by any $\Sigma^\comp_1$ or $\Pi^\comp_1$ formula, $S$ is not definable by any $\Sigma^\comp_1$ or $\Pi^\comp_1$ formula (if it was definable in such a way, then as either $R$ or its complement is a projection of $S$ onto the initial coordinates, $R$ would be definable by either a $\Sigma^\comp_1$ or $\Pi^\comp_1$ formula as well). Since we considered only tuples of $A^{rn}$ with no repeated entries, $S$ satisfies assumption (I). As $R$ and $\neg R$ are thick by assumption (II), $S$ and $\neg S$ are also thick. Moreover, the restriction of $S$ to those tuples which are disjoint from some particular tuple $\bar{c}$ is not intrinsically computable as the restriction of $R$ is not. So $S$ satisfies assumptions (I), (II), and (III), and falls under either case one or case two. Note that we are not using the induction hypothesis here (and indeed it does not apply since $S$ is possibly of higher arity than $R$) because $S$ is already understood as a relation on tuples with no repeated entries satisfying the assumptions, so we can appeal directly to Proposition \ref{CEProp}. Thus we may make the assumption that each such set $S$ has no thick subset.

Now the remainder of the proof is an analysis of the definable sets in order to run the construction in the previous case even without constrained elements.

\begin{lem}\label{scott-family}
Let $\bar{c}$ be a tuple. Suppose that every tuple $\bar{a} \in \noreps{A}^{r}$ satisfies some finitary existential formula $\varphi(\bar{c},\bar{u})$ with $\varphi(\bar{c},\noreps{A}^r) = \{ \bar{b} \in \noreps{A}^r : \mc{A} \models \varphi(\bar{c},\bar{b})\}$ thin. Then there is a tuple $\bar{d}$ over which every $\bar{a} \in \noreps{A}^{r}$ satisfies a finitary existential formula with only finitely many solutions.
\end{lem}

First, we need another lemma.

\begin{lem}\label{fin-many-solns}
If $\varphi(\bar{c},u)$ is a finitary existential formula such that $S = \varphi(\bar{c},\noreps{A}^n)$ is a thin set, and $\bar{a} \in S$ where $\bar{a} = (a_1,\ldots,a_n)$, then for some $i$, $a_i$ satisfies a finitary existential formula over $\bar{c}$ with finitely many solutions.
\begin{proof}
The proof is by induction on $n$. For $n = 1$, a thin set is just a finite set and the result is clear. Now suppose that we know the result for $n$. Let $\varphi(\bar{c},\bar{u})$ be a finitary existential formula with $|\bar{u}| = n+1$, and suppose that $\varphi(\bar{c},\noreps{A}^{n+1})$ is thin. Let $k$ be maximal such that there are $k$ pairwise disjoint tuples satisfying $\varphi(\bar{c},\bar{u})$. Write $\bar{u} = \bar{v},w$ with $|\bar{v}| = n$.

If $\exists\bar{v} \varphi(\bar{c},\bar{v},\mc{A})$ is finite (including the case where $\varphi(\bar{c},\noreps{A}^{n+1})$ is finite) then we are done. If the set of solutions of $\exists w \varphi(\bar{c},\bar{v},w)$ is thin, then we are done by the induction hypothesis.

We claim that there are only finitely many (in fact at most $k$) elements $d$ with $\varphi(\bar{c},\noreps{A}^n,d)$ containing $n \cdot k+1$ or more disjoint tuples. If not, then choose $d_1,\ldots,d_{k+1}$ distinct. Then, for $d_1$, choose $\bar{e}_1$ satisfying $\varphi(\bar{c},\bar{v},d_1)$. Now $d_2$ has at least $n+1$ disjoint tuples satisfying $\varphi(\bar{c},\bar{v},d_2)$, so it must have some tuple disjoint from $\bar{e}_1$. Continuing in this way, we contradict the choice of $k$ by constructing $k+1$ disjoint solutions $d_i\bar{e}_i$ of $\varphi(\bar{c},\bar{u},\bar{v})$. So there are finitely many $d$ such that $\varphi(\bar{c},\noreps{A}^n,d)$ contains $n \cdot k+1$ or more disjoint tuples, and the set of such $d$ is definable by an existential formula. If for our given tuple $\bar{a}$, $a_{n+1}$ is one of these $d$, then we are done.

Otherwise, since the set is finite, say there are exactly $m$ such $d$, the set of $\bar{v},w$ with $\varphi(\bar{c},\bar{v},w)$ and $w$ not one of these $d$ is also existentially definable and thin. Add to $\varphi(\bar{c},\bar{v},w)$ the existential formula which says that there $d_1,\ldots,d_m$ with $\varphi(\bar{c},\noreps{A}^n,d_i)$ containing $n \cdot k + 1$ or more disjoint tuples, and that $w$ is not one of the $d_i$. So we may assume that for all $d$, $\varphi(\bar{c},\noreps{A}^n,d)$ contains at most $n \cdot k$ many disjoint tuples and that the set $\exists w \varphi(\bar{c},\bar{v},w)$ is not thin. Choose $\bar{e}_1,d_1$ satisfying $\varphi(\bar{c},\bar{v},w)$. Choose $\bar{f}_1,\ldots,\bar{f}_{n \cdot k}$ pairwise disjoint from each other and also disjoint from $\bar{e}_1$, and $g_1,\ldots,g_{n \cdot k}$ with $\bar{f}_i,g_i$ satisfying $\varphi$. Then we cannot have $g_1 = \cdots = g_{n \cdot k} = d_1$, so we can choose $\bar{e}_2,d_2$ pairwise disjoint and disjoint from from $\bar{e}_1,d_1$. Now choose $\bar{f}_1,\ldots,\bar{f}_{2 n \cdot k}$ disjoint from $\bar{e}_1,\bar{e}_2$, and $g_1,\ldots,g_{2 n \cdot k}$. Then some $g_i$ must be distinct from $d_1$ and $d_2$. Continue in this way; we contradict the fact that $\varphi(\bar{c},\noreps{A}^{n+1})$ is thin. This exhausts the possibilities.
\end{proof}
\end{lem}

We return to the proof of Lemma \ref{scott-family}.

\begin{proof}[Proof of Lemma \ref{scott-family}]
There must be fewer than $r$ elements of $\mc{A}$ which do not satisfy some existential formula over $\bar{c}$ with finitely many solutions. If not, then there are $a_1,\ldots,a_r \in \mc{A}$ that are not contained in any such existential formula. Let $\varphi(\bar{c},\bar{u})$ define a thin set containing $(a_1,\ldots,a_r)$; then by the previous lemma one of $a_1,\ldots,a_r$ must be contained in some existential formula over $\bar{c}$ with finitely many solutions.

Let $\bar{d}$ be $\bar{c}$ together with these finitely many exceptions. Then every tuple $\bar{a} \in \noreps{A}^r$ satisfies some existential formula $\varphi(\bar{d},\bar{u})$ with $\varphi(\bar{d},\noreps{A})$ finite.
\end{proof}

\begin{lem}
For each tuple $\bar{c}$, there is a tuple $\bar{a}$ such that the set of solutions of each existential formula over $\bar{c}$ satisfied by $\bar{a}$ is thick.
\end{lem}
\begin{proof}
Suppose not. Then Lemma \ref{scott-family} applies. For each $\bar{a} \in \noreps{A}^m$, let $\varphi_{\bar{a}}$ be such that $\varphi_{\bar{a}}(\bar{d},\noreps{A}^m)$ is finite and as small as possible and contains $\bar{a}$. Suppose that $\bar{b} \in \noreps{A}^m$ and $\mc{A} \models \varphi_{\bar{a}}(\bar{d},\bar{b})$. Also suppose that $\mc{A} \models \psi(\bar{d},\bar{b})$ but $\mc{A} \nmodels \psi(\bar{d},\bar{a})$ where $\psi$ is existential. Let $n = |\varphi_{\bar{a}}(\bar{d},\noreps{A}^m) \cap \psi(\bar{d},\noreps{A}^m)|$. Then
\[
	\mc{A} \models \varphi_{\bar{a}}(\bar{d},\bar{a}) \wedge \exists_{\geq n} \bar{u} (\varphi_{\bar{a}}(\bar{d},\bar{u}) \wedge \psi(\bar{d},\bar{u}) \wedge \bar{u} \neq \bar{a})
\]
but this formula has fewer solutions than $\varphi(\bar{d},\bar{x})$ since $\bar{b}$ is not a solution. So every tuple $\bar{a}$ in $\noreps{A}^m$ has some existential formula $\varphi_{\bar{a}}$ it satisfies over $\bar{d}$, with the property that each pair of tuples satisfying $\varphi_{\bar{a}}$ satisfy all the same existential formulas. By Proposition 6.10 of \cite{AshKnight00}, $\Phi = \{ \varphi_{\bar{a}} : \bar{a} \in \noreps{A}^m, m \in \omega\}$ is a Scott family. In particular, $\varphi_{\bar{a}}(\bar{d},\noreps{A}^r) \subseteq R$ or $\varphi_{\bar{a}}(\bar{d},\noreps{A}^r) \subseteq \neg R$ for each $\bar{a} \in \noreps{A}^r$, and so $R$ is defined by both a $\Sigma^\comp_1$ formula and a $\Pi^\comp_1$ formula. This is a contradiction.
\end{proof}

\begin{corollary}
For each tuple $\bar{c}$, there is a tuple $\bar{a} \in \noreps{A}^r$ such that the set of solutions of each existential formula over $\bar{c}$ satisfied by $\bar{a}$ is thick. Also, $\bar{a}$ is free over $\bar{c}$.
\end{corollary}
\begin{proof}
Let $\bar{c}$ be a tuple. We know that there is some $\bar{a}_1$ such that every existential formula over $\bar{c}$ satisfied by $\bar{a}_1$ is thick. If $|\bar{a}_1| \geq r$, then we can just truncate $\bar{a}_1$. Otherwise, if $|\bar{a}_1| < r$, we can find $\bar{a}_2$ such that every existential formula over $\bar{c}\bar{a}_1$ satisfied by $\bar{a}_2$ is thick. Let $\varphi(\bar{c},\bar{x},\bar{y})$ be an existential formula satisfied by $\bar{a}_1,\bar{a}_2$. We claim that the set of solutions of $\varphi$ is thick. Suppose not, and say that there are at most $k$ disjoint solutions. Now the solution set of $\varphi(\bar{c},\bar{a}_1,\bar{y})$ is thick, so there is an existential formula $\psi(\bar{c},\bar{x})$ true of $\bar{a}_1$ over $\bar{c}$ which says that there are at least $(k+1) \cdot (|\bar{a}_1| + |\bar{a}_2|) + 1$ disjoint solutions $\bar{y}$ to $\varphi(\bar{c},\bar{x},\bar{y})$. Then the solution set of $\psi(\bar{c},\bar{x})$ is thick, so we can choose $k+1$ disjoint solutions $\bar{b}_1,\ldots,\bar{b}_{k+1}$. Then choose $\bar{d}_1$ a solution to $\varphi(\bar{c},\bar{b}_1,\bar{y})$. Now there are $2 |\bar{a}_1| + |\bar{a}_2|$ entries in $\bar{b}_1 \bar{d}_1$, so one of the $(k+1) \cdot (|\bar{a}_1| + |\bar{a}_2|) + 1$ solutions to $\varphi(\bar{c},\bar{b}_2,\bar{y})$ is disjoint from $\bar{b}_1$, $\bar{d}_1$, and $\bar{b}_2$. We can pick some such solution $\bar{d}_2$. Continuing in this way, we get $k+1$ disjoint solutions $\bar{b}_1\bar{d}_1,\ldots,\bar{b}_{k+1},\bar{d}_{k+1}$ to $\varphi(\bar{c},\bar{x},\bar{y})$, a contradiction.

Now let $\bar{a} \in \noreps{A}^r$ be such that every existential formula over $\bar{c}$ satisfied by $\bar{a}$ is thick. If $\bar{a}$ is not free over $\bar{c}$, then there is some existential formula $\varphi_{\bar{a}}(\bar{c},\bar{x})$ true of $\bar{a}$ and not true of any $\bar{a}' \notin R$. Then $\varphi(\bar{c},\noreps{A}^r)$ is a $\Sigma^0_1$-definable subset of $R$, and hence thin, a contradiction.
\end{proof}

We also want another pair of corollaries of the above lemma.
\begin{corollary}
Let $S$ be any existentially defined set, and $\bar{a} \in S$. Then there is an existentially defined $S' \subseteq S$ containing $\bar{a}$ and $\varphi(\bar{c},\bar{u},\bar{v})$ defining $S'$ such that $\exists \bar{u} \varphi (\bar{c},\bar{u},\bar{v})$ is thick and $\exists \bar{v} \varphi(\bar{c},\bar{u},\bar{v})$ is finite.
\begin{proof}
Let $\bar{a} = \bar{a}' \bar{a}''$ where $\bar{a}'$ is contained in some finite existentially definable set over $\bar{c}$, and no entry of $\bar{a}''$ is. Let $\psi(\bar{c},\bar{u})$ be a defining formula of this finite set, and let $\chi(\bar{c},\bar{u},\bar{v})$ define $S$. Let $S'$ be the set of solutions to $\chi(\bar{c},\bar{u},\bar{v}) \wedge \psi(\bar{c},\bar{u})$. Then by Lemma \ref{fin-many-solns} since no entry of $\bar{a}''$ is in a finite existentially definable set over $\bar{c}$, $\exists \bar{u} (\chi(\bar{c},\bar{u},\bar{v}) \wedge \psi(\bar{c},\bar{u}))$ is thick.
\end{proof}
\end{corollary}

\begin{corollary}
Let $S$ be any existentially defined set, and $\bar{a} \in S$. Then we can write $\bar{a} = \bar{a}'\bar{a}''$ where $\bar{a}'$ is in an existential formula over $\bar{c}$ with finitely many solutions and there is an existential formula $\varphi(\bar{c},\bar{a}',\bar{u})$ which is thick and contains $\bar{a}''$.
\begin{proof}
Let $\bar{a} = \bar{a}' \bar{a}''$ be as in the previous corollary, and let $\chi(\bar{c},\bar{u},\bar{v})$ define $S$. Now, let $\psi(\bar{c},\bar{u})$ be the formula defining the finite set containing $\bar{a}'$; we may choose $\psi$ so that every solution has the same existential type over $\bar{c}$ by the argument on the previous page. Since $\exists \bar{u} (\chi(\bar{c},\bar{u},\bar{v}) \wedge \psi(\bar{c},\bar{u}))$ is thick, and is the union of $\chi(\bar{c},\bar{b},\bar{v})$ for each $\bar{b}$ satisfying $\psi(\bar{c},\bar{u})$, $\chi(\bar{c},\bar{b},\bar{v})$ is thick for some $\bar{b}$. This is witnessed by (a family of) existential formulas about $\bar{b}$ saying that there are arbitrarily many disjoint solutions to $\chi(\bar{c},\bar{b},\bar{v})$. But all such formulas are true of each other solution of $\psi(\bar{c},\bar{u})$, and in particular of $\bar{a}'$. So $\chi(\bar{c},\bar{a}',\bar{v})$ is thick.
\end{proof}
\end{corollary}

\begin{lem}
Let $\bar{c}$ be a tuple, and $\bar{a},\bar{b}_1,\ldots,\bar{b}_n \in \noreps{A}$, and $\bar{a}$ is contained in no thin set over $\bar{c}$. Let $\varphi(\bar{c},\bar{u},\bar{v}_1,\ldots,\bar{v}_n)$ be an existential formula true of $\bar{a},\bar{b}_1,\ldots,\bar{b}_n$. Then there are $\bar{a}',\bar{d}_1,\ldots,\bar{d}_n$ satisfying $\varphi$ with $\bar{a}' \in R \Leftrightarrow \bar{a} \notin R$ and $\bar{d}_i \in R \Leftrightarrow \bar{b}_i \in R$.
\begin{proof}
Using the above lemma, for each $i$ write $\bar{b}_i = \bar{b}_i' \bar{b}_i''$ be contained in finite definable sets over $\bar{c}$ with $\varphi(\bar{c},\bar{u},\bar{b}_1',\bar{v}_1,\ldots,\bar{b}_n',\bar{v}_n)$ defining a thick set. Note that $\bar{a}$ is not contained in any thin set over $\bar{c}$, and in particular, none of its entries are in any finite existentially definable sets over $\bar{c}$. Using the fact that $R \times S$ (where $S$ is a fiber over some tuple of a product of $R$ and $\neg R$) has no thick subset, we can choose $\bar{d}_i' = \bar{b}_i'$ for $1 \leq i \leq m$ and choose $\bar{a}'$ and $\bar{d}=\bar{d}_i'\bar{d}_i''$ to satisfy $\bar{a}' \in R \Leftrightarrow \bar{a} \notin R$ and $\bar{d}_i \in R \Leftrightarrow \bar{b}_i \in R$.
\end{proof}
\end{lem}

These lemmas are exactly what is required for the construction in Proposition \ref{CEProp}. Instead of choosing elements which are free, we choose elements which are not contained in any thin set, and use the above lemma to move them into $R$ while keeping later elements from $R$ in $R$. We need to know that we can effectively find such elements. We can do this using the effectiveness condition and Lemma \ref{fin-many-solns}.

\backmatter
\bibliographystyle{amsalpha}
\bibliography{harrisontrainor-references}

\end{document}